\documentclass[11pt]{article}
\usepackage[total={7in, 8in}]{geometry}
\usepackage{graphicx}
\usepackage{amsmath, amsthm, latexsym, amssymb, color,cite,enumerate}
\usepackage{caption,subcaption,}
\theoremstyle{theorem}
\newtheorem{theorem}{Theorem}[section]
\newtheorem{lemma}[theorem]{Lemma}
\newtheorem{definition}[theorem]{Definition}
\newtheorem{corollary}[theorem]{Corollary}
\newtheorem{proposition}[theorem]{Proposition}

\theoremstyle{remark}
\newtheorem{remark}{Remark}

\newcommand*{\myproofname}{Proof}
\newenvironment{subproof}[1][\myproofname]{\begin{proof}[#1]}{\end{proof}}
\renewcommand\Re{\operatorname{Re}}
\renewcommand\Im{\operatorname{Im}}
\newcommand\MdR{\mbox{M}_d(\mathbb{R})}
\newcommand\GldR{\mbox{Gl}_d(\mathbb{R})}
\newcommand\OdR{\mbox{O}_d(\mathbb{R})}
\newcommand\Sym{\operatorname{Sym}}
\newcommand\Exp{\operatorname{Exp}}
\newcommand\tr{\operatorname{tr}}
\newcommand\diag{\operatorname{diag}}
\newcommand\supp{\operatorname{Supp}}
\newcommand\Spec{\operatorname{Spec}}
\renewcommand\det{\operatorname{det}}
\newcommand\Ker{\operatorname{Ker}}

\author{Evan Randles and Laurent Saloff-Coste}

\title{Convolution powers of complex functions on $\mathbb{Z}^d$}
\date{}
\begin{document}
\maketitle

\begin{abstract}
The study of convolution powers of a finitely supported probability distribution $\phi$ on the $d$-dimensional square lattice is central to random walk theory. For instance, the $n$th convolution power $\phi^{(n)}$ is the distribution of the $n$th step of the associated random walk and is described by the classical local limit theorem. Following previous work of P. Diaconis and the authors, we explore the more general setting in which $\phi$ takes on complex values. This problem, originally motivated by the problem of Erastus L. De Forest in data smoothing, has found applications to the theory of stability of numerical difference schemes in partial differential equations. For a complex valued function $\phi$ on $\mathbb{Z}^d$, we ask and address four basic and fundamental questions about the convolution powers $\phi^{(n)}$ which concern sup-norm estimates, generalized local limit theorems, pointwise estimates, and stability. This work extends one-dimensional results of I. J. Schoenberg, T. N. E. Greville, P. Diaconis and the second author and, in the context of stability theory, results by V. Thom\'{e}e and M. V. Fedoryuk.
\end{abstract}


\section{Introduction}\label{sec:Introduction}

We denote by $\ell^1(\mathbb{Z}^d)$ the space of complex valued functions $\phi:\mathbb{Z}^d\rightarrow\mathbb{C}$ such that
\begin{equation*}
\|\phi\|_1=\sum_{x\in\mathbb{Z}^d}|\phi(x)|<\infty.
\end{equation*}
For $\psi,\phi\in\ell^1(\mathbb{Z}^d)$, the convolution product $\psi\ast\phi\in\ell^1(\mathbb{Z}^d)$ is defined by
\begin{equation*}
\psi\ast\phi(x)=\sum_{y\in\mathbb{Z}^d}\psi(x-y)\phi(y)
\end{equation*}
for $x\in\mathbb{Z}^d$.  Given $\phi\in\ell^1(\mathbb{Z}^d)$, we are interested in the convolution powers $\phi^{(n)}\in\ell^1(\mathbb{Z}^d)$ defined iteratively by $\phi^{(n)}=\phi^{(n-1)}\ast\phi^{(1)}$ for $n\in\mathbb{N}_+=:\{1,2,\dots\}$ where $\phi^{(1)}=\phi$. This study was originally motivated by problems in data smoothing, namely De Forest's problem, and it was later found essential to the theory of approximate difference schemes for partial differential equations \cite{Greville1966,Schoenberg1953, Thomee1969,Thomee1965}; the recent article \cite{Diaconis2014} gives background and pointers to the literature.\\

\noindent In random walk theory, the study of convolution powers is of central importance: Given an independent sequence of random vectors $X_1,X_2,\dots\in\mathbb{Z}^d$, all with distribution $\phi$ (here, $\phi\geq 0$), $\phi^{(n)}$ is the distribution of the random vector $S_n=X_1+X_2+\cdots+X_n$. Equivalently, a probability distribution $\phi$ on $\mathbb{Z}^d$ gives rise to a random walk whose $n$th-step transition kernel $k_n$ is given by $k_n(x,y)=\phi^{(n)}(y-x)$ for $x,y\in\mathbb{Z}^d$. For an account of this theory, we encourage the reader to see the wonderful and classic book of F. Spitzer \cite{Spitzer1964} and, for a more modern treatment, the recent book of G. Lawler and V. Limic \cite{Lawler2010} (see also Subsection \ref{subsec:ClassicalLLT}). In the more general case that $\phi$ takes on complex values (or just simply takes on both positive and negative values), its convolution powers $\phi^{(n)}$ are seen to exhibit rich and disparate behavior, much of which never appears in the probabilistic setting. Given $\phi\in\ell^1(\mathbb{Z}^d)$, we are interested in the most basic and fundamental questions that can be asked about its convolution powers. Here are four such questions:
\begin{enumerate}[(i)]
 \item\label{ques:SupremumDecay} What can be said about the decay of
\begin{equation*}
\|\phi^{(n)}\|_{\infty}=\sup_{x\in\mathbb{Z}^d}|\phi^{(n)}(x)|
\end{equation*}
as $n\rightarrow\infty$?

\item\label{ques:LLT} Is there a simple pointwise description of $\phi^{(n)}(x)$, analogous to the local (central) limit theorem, that can be made for large $n$?

\item\label{ques:PointwiseEstimates} Are global space-time pointwise estimates obtainable for $|\phi^{(n)}|$?

\item\label{ques:Stability} Under what conditions is $\phi$ \textit{stable} in the sense that
\begin{equation}\label{eq:PowerBoundedness}
\sup_{n\in\mathbb{N}_+}\|\phi^{(n)}\|_1<\infty?
\end{equation}
\end{enumerate}

\noindent The above questions have well-known answers in random walk theory. For simplicity we discuss the case in which $\phi$ is a probability distribution on $\mathbb{Z}^d$ whose associated random walk is symmetric, aperiodic, irreducible and of finite range. In this case, it is known that $n^{d/2}\phi^{(n)}(0)$ converges to a non-zero constant as $n\rightarrow\infty$ and this helps to provide an answer to Question \eqref{ques:SupremumDecay} in the form of the following two-sided estimate: For positive constants $C$ and $C'$,
\begin{equation*}
C n^{-d/2}\leq \sup_{x\in\mathbb{Z}^d}\phi^{(n)}(x)\leq C' n^{-d/2}
\end{equation*}
for all $n\in\mathbb{N}_+$. Concerning the somewhat finer Question \eqref{ques:LLT}, the classical local limit theorem states that
\begin{equation*}
\phi^{(n)}(x)=n^{-d/2}G_\phi(n^{-1/2}x)+o(n^{-d/2})
\end{equation*}
uniformly for $x\in\mathbb{Z}^d$, where $G_\phi$ is the generalized Gaussian density
\begin{align}\label{eq:Gaussian}\nonumber
G_\phi(x)&=\frac{1}{(2\pi)^d}\int_{\mathbb{R}^d}\exp\big(-\xi\cdot C_\phi\xi\big)e^{-ix\cdot\xi}\,d\xi\\
&=\frac{1}{(2\pi )^{d/2}\sqrt{\det C_\phi}}\exp\left(-\frac{x\cdot {C_\phi}^{-1}x}{2}\right);
\end{align}
here, $C_\phi$ is the positive definite covariance matrix associated to $\phi$ and $\cdot$ denotes the dot product. As an application of this local limit theorem, one can easily settle the question of recurrence/transience for random walks on $\mathbb{Z}^d$ which was originally answered by G. P\'{o}lya in the context of simple random walk \cite{Polya1921}.  For general complex valued functions $\phi\in\ell^1(\mathbb{Z}^d)$, Question \eqref{ques:LLT} is a question about the validity of (generalized) local limit theorems and can be restated as follows: Under what conditions can the convolution powers $\phi^{(n)}$ be approximated pointwise by a combination (perhaps a sum) of appropriately scaled smooth functions-- called \textit{attractors}? The answer for Question \eqref{ques:PointwiseEstimates} for a finite range, symmetric, irreducible and aperiodic random walk is provided in terms of the so-called Gaussian estimate: For positive constants $C$ and $M$,
\begin{equation*}
\phi^{(n)}(x)\leq C n^{-d/2}\exp(-M|x|^2/n)
\end{equation*}
for all $x\in\mathbb{Z}^d$ and $n\in\mathbb{N}_+$; here, $|\cdot|$ is the standard euclidean norm. Such estimates, with matching lower bounds on appropriate space-time regions, are in fact valid in a much wider context, see  \cite{Hebisch1993}. Finally, the conservation of mass provides an obvious positive answer to Question \eqref{ques:Stability} in the case that $\phi$ is a probability distribution.\\

\noindent Beyond the probabilistic setting, the study of convolution powers for complex valued functions has centered mainly around two applications, statistical data smoothing procedures and finite difference schemes for numerical solutions to partial differential equations; the vast majority of the existing theory pertains only to one dimension. In the context of data smoothing, the earliest (known) study was motivated by a problem of Erastus L. De Forest. De Forest's problem, analogous to Question \eqref{ques:LLT}, concerns the behavior of convolution powers of symmetric real valued and finitely supported functions on $\mathbb{Z}$ and was addressed by I. J. Schoenberg \cite{Schoenberg1953} and T. N. E. Greville \cite{Greville1966}. In the context of numerical solutions in partial differential equations, the stability of convolution powers (Question \eqref{ques:Stability}) saw extensive investigation following World War II spurred by advancements in numerical computing. For an approximate difference scheme to an initial value problem, the property \eqref{eq:PowerBoundedness} is necessary and sufficient for convergence to a classical solution; this is the so-called Lax equivalence theorem \cite{Richtmyer1967} (see Section \ref{sec:Stability}). Property \eqref{eq:PowerBoundedness} is also called \textit{power boundedness} and can be seen in the context of Banach algebras where $\phi$ is an element of the Banach algebra $(\ell^1(\mathbb{Z}^d),\|\cdot\|_1)$ equipped with the convolution product \cite{Schreiber1970,Kaniuth2011}. \\

\noindent In one dimension, Questions (\ref{ques:SupremumDecay}-\ref{ques:Stability}) were recently addressed in the articles \cite{Diaconis2014} and \cite{Randles2015}. For the general class of finitely supported complex valued functions on $\mathbb{Z}$, \cite{Randles2015} completely settles Questions \eqref{ques:SupremumDecay} and \eqref{ques:LLT}. For instance, consider the following theorem of \cite{Randles2015}.

\begin{theorem}[Theorem 1.1 of \cite{Randles2015}]\label{thm:1DSupDecay}
Let $\phi:\mathbb{Z}\rightarrow\mathbb{C}$ have finite support consisting of more than one point. Then there is a positive constant $A$ and a natural number $m\geq 2$ for which
\begin{equation*}
Cn^{-1/m}\leq A^n\|\phi^{(n)}\|_{\infty}\leq C'n^{-1/m}
\end{equation*}
for all $n\in\mathbb{N}_+$, where $C$ and $C'$ are positive constants.
\end{theorem}

\noindent In settling Question \eqref{ques:LLT}, the article \cite{Randles2015} gives an exhaustive account of local limit theorems in which the set of possible attractors includes the Airy function and the heat kernel evaluated at purely imaginary time. In addressing Question \eqref{ques:PointwiseEstimates}, the article \cite{Diaconis2014} contains a number of results concerning global space-time estimates for $\phi^{(n)}$ for a finitely supported function $\phi$ -- our results recapture (and extend in the case of Theorem \ref{thm:ExponentialEstimate}) these results of \cite{Diaconis2014}. The question of stability for finitely supported functions on $\mathbb{Z}$ was answered completely in 1965 by V. Thom\'{e}e \cite{Thomee1965} (see Theorem \ref{thm:Thomee} below). In fact, Thom\'{e}e's characterization is, in some sense, the light in the dark that gives the correct framework for the study of local limit theorems in one dimension and we take it as a starting point for our study in $\mathbb{Z}^d$. \\

\noindent Moving beyond one dimension, the situation becomes more interesting still, the theory harder and much remains open. As we illustrate, convolution powers exhibit a significantly wider range of behaviors in $\mathbb{Z}^d$ than is seen in $\mathbb{Z}$ (see Remark \ref{rmk:Attractors}). The focus of this article is to address Questions (\ref{ques:SupremumDecay}-\ref{ques:Stability}) under some strong hypotheses on the Fourier transform -- specifically, we work under the assumption that, near its extrema, the Fourier transform of $\phi$ is ``nice'' in a sense we will shortly make precise. To this end, we follow the article \cite{Diaconis2014} and generalize the results therein. A complete theory for finitely supported functions on $\mathbb{Z}^d$, in which the results of \cite{Randles2015} will fit, is not presently known. Not surprisingly, our results recapture the well-known results of random walk theory on $\mathbb{Z}^d$ (see Subsection \ref{subsec:ClassicalLLT}). \\

\noindent As a first motivating example, consider $\phi:\mathbb{Z}^2\rightarrow\mathbb{C}$ defined by
\begin{equation*}
\phi(x,y)=\frac{1}{22+2\sqrt{3}}\times\begin{cases}
	    8 & (x,y)=(0,0)\\
           5+\sqrt{3} & (x,y)=(\pm 1,0)\\
           -2 & (x,y)=(\pm 2,0)\\
            i(\sqrt{3}-1)& (x,y)=(\pm 1,-1)\\
            -i(\sqrt{3}-1)& (x,y)=(\pm 1,1)\\
            2\mp 2i & (x,y)=(0,\pm 1)\\
           0 & \mbox{otherwise}.
          \end{cases}
\end{equation*}
\begin{figure}[h!]
\begin{center}
\resizebox{\textwidth}{!}{
	    \begin{subfigure}[5cm]{0.5\textwidth}
		\includegraphics[width=\textwidth]{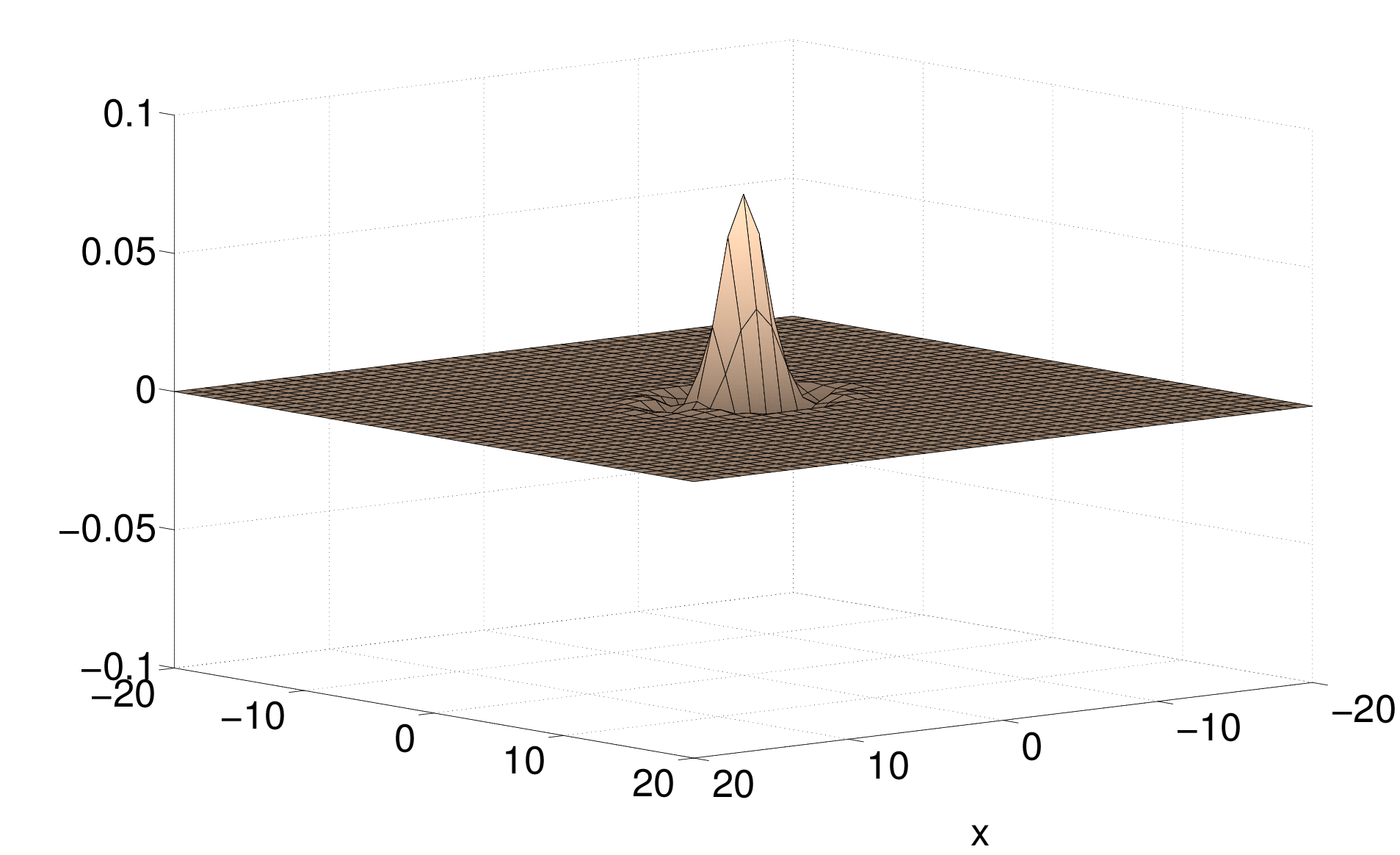}
		\caption{$\Re(\phi^{(n)})$ for $n=10$}
		\label{fig:ex_intro_10_1}
	    \end{subfigure}
\begin{subfigure}[6cm]{0.5\textwidth}
		\includegraphics[width=\textwidth]{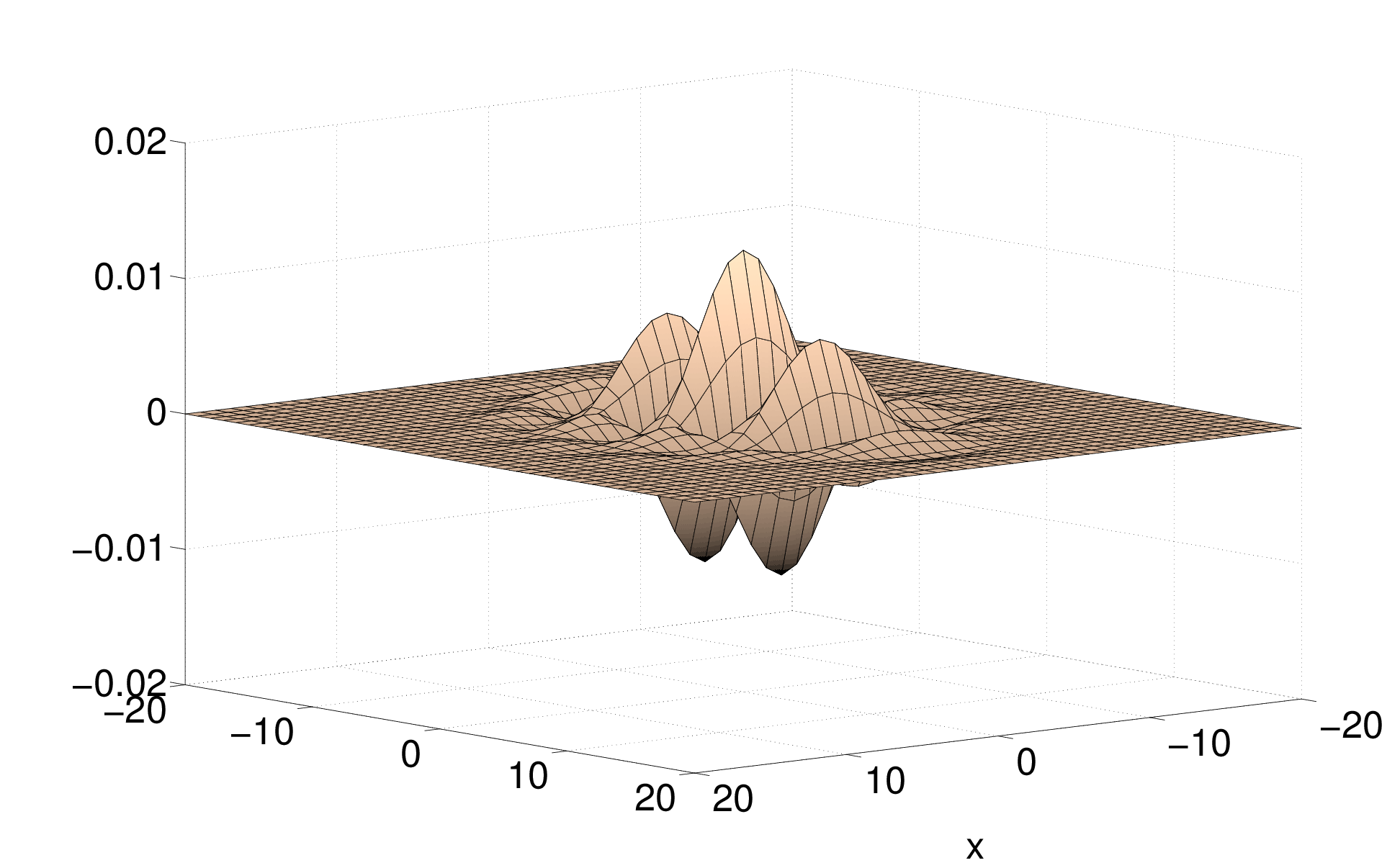}
		\caption{$\Re(\phi^{(n)})$ for $n=100$}
		\label{fig:ex_intro_100_1}
	    \end{subfigure}}
\caption{The graphs of $\Re(\phi^{(n)})$ for $n=10,100$.}
\label{fig:ex_intro}
\end{center}
\end{figure}

\noindent The graphs of $\Re(\phi^{(n)})$ for $(x,y)\in\mathbb{Z}^2$ for $-20\leq x,y\leq 20$ are displayed in Figures \ref{fig:ex_intro_10_1} and \ref{fig:ex_intro_100_1} for $n=10$ and $n=100$ respectively. By inspection, one observes that $\Re(\phi^{(n)})$ decays in absolute value as $n$ increases and, when $n=100$, there is an apparent oscillation of $\Re(\phi^{(n)})$ in the $y$-direction. Our results explain these observations.\\

\noindent For $\phi\in\ell^1(\mathbb{Z}^d)$, its Fourier transform $\hat\phi:\mathbb{R}^d\rightarrow\mathbb{C}$ is defined by
\begin{equation*}
\hat\phi(\xi)=\sum_{x\in\mathbb{Z}^d}\phi(x)e^{ix\cdot\xi}
\end{equation*}
for $\xi\in\mathbb{R}^d$; this series is absolutely convergent. The standard Fourier inversion formula holds for all $\phi\in\ell^1(\mathbb{Z}^d)$ and moreover, for each $n\in\mathbb{N}_+$,
\begin{equation}\label{eq:fourierinversion}
\phi^{(n)}(x)=\frac{1}{(2\pi)^d}\int_{\mathbb{T}^d}e^{-ix\cdot\xi}\hat\phi(\xi)^n d\xi
\end{equation}
for all $x\in\mathbb{Z}^d$ where $\mathbb{T}^d=(-\pi,\pi]^d$. Like the classical local limit theorem, our arguments are based on local approximations of $\hat\phi$ and such approximations require $\hat\phi$ to have a certain amount of smoothness. In our setting the order of smoothness needed in each case is not known a priori. For our purposes, it is sufficient (but not necessary) to consider only those $\phi\in\ell^1(\mathbb{Z}^d)$ with finite moments of all orders. That is, we consider the subspace of $\ell^1(\mathbb{Z}^d)$, denoted by $\mathcal{S}_d$, consisting of those $\phi$ for which
\begin{equation*}
\|x^{\beta}\phi(x)\|_1=\sum_{x\in\mathbb{Z}^d}|x^{\beta}\phi(x)|=\sum_{x\in\mathbb{Z}^d}|x_1^{\beta_1}x_2^{\beta_2}\cdots x_d^{\beta_d}\phi(x)|<\infty
\end{equation*}
for all multi-indices $\beta=(\beta_1,\beta_2,\dots,\beta_d)\in\mathbb{N}^d$. It is straightforward to see that $\hat\phi\in C^{\infty}(\mathbb{R}^d)$ whenever $\phi\in\mathcal{S}_d$. We note that $\mathcal{S}_d$ contains all finitely supported functions mapping $\mathbb{Z}^d$ into $\mathbb{C}$; of course, when $\phi$ is finitely supported, $\hat\phi$ extends holomorphically to $\mathbb{C}^d$. \\

\noindent Before we begin to formulate our hypotheses, let us introduce some important objects by taking motivation from probability. The quadratic form $\xi\mapsto \xi\cdot C_\phi\xi$ which appears in \eqref{eq:Gaussian} is a positive definite polynomial in $\xi$ and is homogeneous in the following sense. For all $t>0$ and $\xi\in\mathbb{R}^d$,
\begin{equation*}
(t^{1/2}\xi)\cdot C_\phi(t^{1/2}\xi)=t\,\xi\cdot C_\phi\xi.
\end{equation*}
The map $(0,\infty)\ni t\mapsto t^{1/2}I\in\GldR$ is a continuous (Lie group) homomorphism from the multiplicative group of positive real numbers into $\GldR$; here $I$ is the identity matrix in the set of $d\times d$ real matrices $\MdR$ and $\GldR\subseteq\MdR$ denotes the group of invertible matrices. For any such continuous homomorphism $t\mapsto T_t$,  $\{T_t\}_{t>0}$ is a Lie subgroup of $\GldR$, that is, a \textit{continuous one-parameter group}; the Hille-Yosida construction guarantees that all such groups are of the form
\begin{equation*}
T_t=t^{E}=\exp((\log t) E)=\sum_{k=0}^{\infty}\frac{(\log t)^k}{k!}E^k
\end{equation*}
for $t>0$ for some $E\in\MdR$. The Appendix (Section \ref{sec:Appendix}) amasses some basic properties of continuous one-parameter groups.

\begin{definition}\label{def:positivehomogeneous}
For a continuous function $P:\mathbb{R}^d\rightarrow\mathbb{C}$ and a continuous one-parameter group $\{T_t\}\subseteq\GldR$, we say that $P$ is homogeneous with respect to $T_t=t^E$ if
\begin{equation*}
tP(\xi)=P(T_t\xi)
\end{equation*}
for all $t>0$ and $\xi\in\mathbb{R}^d$. In this case $E$ is a member of the exponent set of $P$, $\Exp(P)$.

We say that $P$ is positive homogeneous if the real part of $P$, $R=\Re P$, is positive definite (that is, $R(\xi)\geq 0$ and $R(\xi)=0$ only when $\xi=0$) and if $\Exp(P)$ contains a matrix $E\in\MdR$ whose spectrum is real.
\end{definition}

\noindent Throughout this article, we concern ourselves with positive homogeneous multivariate polynomials $P:\mathbb{R}^d\rightarrow\mathbb{C}$; their appearance is seen to be natural, although not exhaustive, when considering local approximations of $\hat\phi$ for $\phi\in\mathcal{S}_d$. A given positive homogeneous polynomial $P$ need not be homogeneous with respect to a unique continuous one-parameter group. For example, for each $m\in\mathbb{N}_+$, $\xi\mapsto |\xi|^{2m}$ is a positive homogeneous polynomial and it can be shown directly that
\begin{equation*}
\Exp(|\cdot|^{2m})=(2m)^{-1}I+\mathfrak{o}(d),
\end{equation*}
where $\mathfrak{o}(d)\subseteq\MdR$ is the set of anti-symmetric matrices (these arise as the Lie algebra of the orthogonal group $\OdR\subseteq\GldR$). It will be shown however that, for a positive homogeneous polynomial $P$, $\tr E=\tr E'$ whenever $E,E'\in\Exp(P)$; this is Corollary \ref{cor:UniqueTrace}. To a given positive homogeneous polynomial $P$, the corollary allows us to uniquely define the number
\begin{equation}\label{eq:muP}
\mu_P:=\tr E
\end{equation}
for any $E\in\Exp(P)$. This number appears in many of our results; in particular, it arises in addressing the Question \eqref{ques:SupremumDecay} in which it plays the role of $1/m$ in Theorem \ref{thm:1DSupDecay}.\\

\noindent We now begin to discuss the framework and hypotheses under which our theorems are stated. Let $\phi\in\mathcal{S}_d$ be such that $\sup_{\xi\in\mathbb{R}^d}|\hat\phi(\xi)|=1$; this can always be arranged by multiplying $\phi$ by an appropriate constant. Set
\begin{equation*}
\Omega(\phi)=\{\xi\in \mathbb{T}^d:|\hat\phi(\xi)|=1\}
\end{equation*}
and, for $\xi_0\in\Omega(\phi)$, define $\Gamma_{\xi_0}:\mathcal{U}\subseteq\mathbb{R}^d\rightarrow\mathbb{C}$ by
\begin{equation*}
\Gamma_{\xi_0}(\xi)=\log\left(\frac{\hat\phi(\xi+\xi_0)}{\hat\phi(\xi_0)}\right)
\end{equation*}
where $\mathcal{U}$ is a convex open neighborhood of $0$ which is small enough to ensure that $\log$, the principal branch of logarithm, is defined and continuous on $\hat\phi(\xi+\xi_0)/\hat\phi(\xi_0)$ for $\xi\in\mathcal{U}$. Because $\hat\phi$ is smooth, $\Gamma_{\xi_0}\in C^{\infty}(\mathcal{U})$ and so we can use Taylor's theorem to approximate $\Gamma_{\xi_0}$ near $0$.  In this article, we focus on the case in which the Taylor expansion yields a positive homogeneous polynomial. The following definition, motivated by Thom\'{e}e \cite{Thomee1965}, captures this notion.
\begin{definition}\label{def:PosHomType}
Let $\phi\in\mathcal{S}_d$ be such that $\sup|\hat\phi(\xi)|=1$ and let $\xi_0\in\Omega(\phi)$. We say that $\xi_0$ is of positive homogeneous type for $\hat\phi$ if the Taylor expansion for $\Gamma_{\xi_0}$ about $0$ is of the form
\begin{equation}\label{eq:PosHomType}
\Gamma_{\xi_0}(\xi)=i\alpha_{\xi_0}\cdot\xi-P_{\xi_0}(\xi)+\Upsilon_{\xi_0}(\xi)
\end{equation}
where $\alpha_{\xi_0}\in\mathbb{R}^d$, $P_{\xi_0}$ is a positive homogeneous polynomial and $\Upsilon_{\xi_0}(\xi)=o(R_{\xi_0}(\xi))$ as $\xi\rightarrow 0$; here $R_{\xi_0}=\Re P_{\xi_0}$. We say that $\alpha_{\xi_0}$ is the drift associated to $\xi_0$.
\end{definition}

\noindent Though not obvious at first glance, $\alpha_{\xi_0}$ and $P_{\xi_0}$ of the above definition are necessarily unique. When looking at any given Taylor polynomial, it will not always be apparent when the conditions of the above definition are satisfied. In Section \ref{sec:PropertiesofHatPhi}, there is a discussion concerning this, and therein, necessary and sufficient conditions are given for $\xi_0\in\Omega(\phi)$ to be of positive homogeneous type for $\hat\phi$.\\

\noindent Our theorems are stated under the assumption that for $\phi\in\mathcal{S}_d$, $\sup|\hat\phi(\xi)|=1$ and each $\xi\in\Omega(\phi)$ is  of positive homogeneous type for $\hat\phi$. As we show in Section \ref{sec:PropertiesofHatPhi}, these hypotheses ensure that the set $\Omega(\phi)$ is finite and in this case we set
\begin{equation}\label{eq:MuPhiDefinition}
\mu_\phi=\min_{\xi\in\Omega(\phi)}\mu_{P_\xi}.
\end{equation}
This is admittedly a slight abuse of notation. We are ready to state our first main result.

\begin{theorem}\label{thm:SubDecay}
Let $\phi\in\mathcal{S}_d$ be such that $\sup|\hat\phi(\xi)|=1$ and suppose that each $\xi\in\Omega(\phi)$ is of positive homogeneous type for $\hat\phi$. Then
\begin{equation}\label{eq:SupDecay}
C'n^{-\mu_\phi}\leq\|\phi^{(n)}\|_{\infty}\leq C n^{-\mu_\phi}
\end{equation}
for all $n\in\mathbb{N}_+$, where $C$ and $C'$ are positive constants.
\end{theorem}

\noindent The theorem above is a partial answer to Question \eqref{ques:SupremumDecay} and nicely complements Theorem \ref{thm:1DSupDecay} and the results of \cite{Diaconis2014}.  We note however that, in view of the wider generality of Theorem \ref{thm:1DSupDecay}, Theorem \ref{thm:SubDecay} is obviously not the final result in $\mathbb{Z}^d$ on this matter (see the discussion of tensor products in Subsection \ref{subsec:ex4}).\\

\noindent Returning to our motivating example and with the aim of applying Theorem \ref{thm:SubDecay}, we analyze the Fourier transform of $\phi$. We have
\begin{align*}
\hat\phi(\eta,\zeta)&=\frac{1}{11+\sqrt{3}}\big(4-2\cos(2\eta)+\big(5+\sqrt{3}\big)\cos(\eta)+2\big(\cos(\zeta)\\
&\quad+\sin(\zeta)\big)+\big(2\sqrt{3}+2\big)\cos(\eta)\sin(\zeta)\big)
\end{align*}
for $(\eta,\zeta)\in\mathbb{R}^2$. One easily sees that $\sup|\hat\phi|=1$ and that $|\hat\phi|$ is maximized in $\mathbb{T}^2$ at only one point $(0,\pi/3)$ and here, $\hat\phi(0,\pi/3)=1$. As is readily computed,
\begin{align*}
\Gamma(\eta,\zeta)&=\log\left(\frac{\hat\phi(\eta,\zeta+\pi/3)}{\hat\phi(0,\pi/3)}\right)=\log(\hat\phi(\eta,\zeta+\pi/3))\\
&=-\frac{1}{11+\sqrt{3}}\eta^4+\frac{7-6\sqrt{3}}{118}\eta^2\zeta-\frac{2}{11+\sqrt{3}}\zeta^2\\
&\quad+O(|\eta|^5)+O(|\eta^4\zeta|)+O(|\eta\zeta|^2)+O(|\zeta|^3)
\end{align*}
as $(\eta,\zeta)\rightarrow 0$. Let us study the polynomial
\begin{equation*}
P(\eta,\zeta)=\frac{1}{22+2\sqrt{3}}\left(2\eta^4+\left(\sqrt{3}-1\right)\eta^2\zeta+4\zeta^2\right),
\end{equation*}
which leads this expansion. It is easily verified that $P=\Re P$ is positive definite and
\begin{equation*}
P(t^E(\eta,\zeta))=P(t^{1/4}\eta,t^{1/2}\zeta)=tP(\eta,\zeta) \hspace{.5cm}\mbox{ with}\hspace{.5cm}
E=\begin{pmatrix}
   1/4 & 0\\
   0  & 1/2
  \end{pmatrix}
\end{equation*}
for all $t>0$ and $(\eta,\zeta)\in\mathbb{R}^2$ and therefore $P$ is a positive homogeneous polynomial with $E\in\Exp(P)$. Upon rewriting the error in the Taylor expansion, we have
\begin{equation*}
\Gamma(\eta,\zeta)=-P(\eta,\zeta)+\Upsilon(\eta,\zeta)
\end{equation*}
where $\Upsilon(\eta,\zeta)=o(P(\eta,\zeta))$ as $(\eta,\zeta)\rightarrow (0,0)$ and so it follows that $(0,\pi/3)$ is of positive homogeneous type for $\hat\phi$ with corresponding $\alpha=(0,0)\in\mathbb{R}^2$ and positive homogeneous polynomial $P$. Consequently, $\phi$ satisfies the hypotheses of Theorem \ref{thm:SubDecay} with $\mu_\phi=\mu_P=\tr E=3/4$ and so
\begin{equation*}
C'n^{-3/4}\leq \|\phi^{(n)}\|_{\infty}\leq Cn^{-3/4}
\end{equation*}
for all $n\in\mathbb{N}_+$, where $C$ and $C'$ are positive constants. With the help of a local limit theorem, we will shortly describe the pointwise behavior of $\phi$.\\

\noindent Coming back to the general setting, we now introduce the attractors which appear in our main local limit theorem. For a positive homogeneous polynomial $P$, define $H_P^{(\cdot)}:(0,\infty)\times\mathbb{R}^d\to\mathbb{C}$ by
\begin{equation}\label{eq:HPdef}
H_P^t(x)=\frac{1}{(2\pi)^d}\int_{\mathbb{R}^d}e^{-tP(\xi)}e^{-ix\cdot\xi}\,d\xi
\end{equation}
for $t>0$ and $x\in\mathbb{R}^d$; we write $H_P(x)=H_P^1(x)$. As we show in Section \ref{sec:HomogeneousPolynomialsandAttractors}, for each $t>0$, $H_P^t(\cdot)$ belongs to the Schwartz space, $\mathcal{S}(\mathbb{R}^d)$, and moreover, for any $E\in\Exp(P)$,
\begin{equation}\label{eq:HPScale}
H_P^t(x)=\frac{1}{t^{\tr E}}H_P(t^{-E^*}x)=\frac{1}{t^{\mu_P}}H_P(t^{-E^*}x)
\end{equation}
for all $t>0$ and $x\in\mathbb{R}^d$; here $E^*$ is the adjoint of $E$. These function arise naturally in the study of  partial differential equations. For instance, consider the partial differential operator $\partial_t+\Lambda_P$ where $\Lambda_P:=P(D)$, called a \textit{positive homogeneous operator}, is defined by replacing the $d$-tuple $\xi=(\xi_1,\xi_2,\dots,\xi_d)$ in $P(\xi)$ by the $d$-tuple of partial derivatives $D=\left(i\partial_{x_1},i\partial_{x_2},\dots,i\partial_{x_d}\right)$.
The associated Cauchy problem for this operator can be stated thus: Given initial data $f$ (from a suitable class of functions), find $u(x,t)$ satisfying
\begin{equation}\label{eq:HeatTypeEquation}
\left\{ \begin{array}{rlll}
         (\partial_t+\Lambda_P)u(x,t)&=&0&x\in\mathbb{R}^d,\,t>0\\
			      u(0,x)&=&f(x)&x\in\mathbb{R}^d.
        \end{array}\right.
\end{equation}
In this context, $H_P^{(\cdot)}$ is a fundamental solution to \eqref{eq:HeatTypeEquation} in the sense that the representation
\begin{equation}\label{eq:SemigroupRepresentation}
u(x,t)=(e^{-t\Lambda_P}f)(x)=\int_{\mathbb{R}^d}H_P^{(t)}(x-y)f(y)dy
\end{equation}
satisfies $(\partial_t+\Lambda_P)u=0$ and has $u(t,\cdot)\rightarrow f$ as $t\rightarrow 0$ in an appropriate topology. Equivalently, $H_P^{(\cdot)}$ is the integral kernel of the semigroup $e^{-t\Lambda_P}$ with infinitesimal generator $\Lambda_P$. The Cauchy problem for the setting in which $\Lambda_P$ is replaced by an operator $H$ which depends on $x$ and is uniformly comparable to $(-\Delta)^{m}=\Lambda_{|\cdot|^{2m}}$ is the subject of (higher order) parabolic partial differential equations and its treatment can be found in the classic texts \cite{Eidelman1969} and \cite{Friedman1964} (see also \cite{Davies1995} and \cite{Davies1997}). A forthcoming article will treat extensively the case in which $H$ is uniformly comparable to a positive homogeneous operator. In the present article, we shall only need a few basic facts concerning $H_P^{(\cdot)}$.\\

\begin{remark}\label{rmk:Attractors}
When $d=1$, every positive homogeneous polynomial is of the form $P(\xi)=\beta \xi^{m}$ where $\Re \beta>0$ and $m$ is an even natural number. In this case, $H_P$ is equal to the function $H_m^{\beta}$ of \cite{Randles2015}. We note that the simplicity of the dilation structure in one dimension is in complete contrast with the natural complexity of the multi-dimensional analogue seen in this article.\\
\end{remark}

\noindent For our next main theorem which addresses Question \eqref{ques:LLT}, we restrict our attention to the set of points $\{\xi_1,\xi_2,\dots,\xi_A\}\subseteq \Omega(\phi)$ for which $\mu_{P_{\xi_q}}=\mu_\phi$ for $q=1,2,\dots,A$; the points $\xi\in\Omega(\phi)$ for which $\mu_{P_\xi}>\mu_\phi$ (if there are any) are not seen in local limits. Finally for each $\xi_q$ for $q=1,2,\dots,A$, we set $\alpha_q=\alpha_{\xi_q}$ and $P_q=P_{\xi_q}$. The following local limit theorem addresses Question \eqref{ques:LLT}.

\begin{theorem}\label{thm:MainLocalLimit}
Let $\phi\in\mathcal{S}_d$ be such that $\sup|\hat\phi(\xi)|=1$ and suppose that every point $\xi\in\Omega(\phi)$ is of positive homogeneous type for $\hat\phi$. Let $\mu_\phi$ be defined by \eqref{eq:MuPhiDefinition} and let $\xi_1,\xi_2,\dots,\xi_A$, $\alpha_1,\alpha_2,\dots,\alpha_A$, and $P_1,P_2,\dots,P_A$ be as in the previous paragraph. Then
\begin{equation}\label{eq:MainLocalLimit}
\phi^{(n)}(x)=\sum_{q=1}^Ae^{-ix\cdot \xi_q}\hat\phi(\xi_q)^n H^n_{P_q}(x-n\alpha_q)+o(n^{-\mu_\phi})
\end{equation}
uniformly for $x\in\mathbb{Z}^d$.
\end{theorem}

\noindent Let us make a few remarks about this theorem. First, the attractors $H^n_{P_q}$ appearing in \eqref{eq:MainLocalLimit} are rescaled versions of $H_{P_q}=H_{P_q}^1$ in view of \eqref{eq:HPScale}, and all decay in absolute value on the order $n^{-\mu_\phi}$ -- this is consistent with Theorem \ref{thm:SubDecay}. Second, the attractors $H_{P_q}(x)$ often exhibit slowly varying oscillations as $|x|$ increases (see Subsection \eqref{subsec:ex1}), however, the main oscillatory behavior, which is present in Figure \ref{fig:ex_intro_100_1}, is a result of the prefactor $e^{-ix\cdot\xi_q}\hat\phi(\xi_q)$. This is, of course, a consequence of $\hat\phi$ being maximized away from the origin. In Subsection \ref{subsec:ClassicalLLT}, we will see that when $\phi$ is a probability distribution, all of the attractors in \eqref{eq:MainLocalLimit} are identical and the prefactors collapse into a single function, $\Theta$, which nicely describes the support of $\phi^{(n)}$ and hence periodicity of the associated random walk (see Theorems \ref{thm:classicLLT} and \ref{thm:PeriodicityofRandomWalk}). \\


\noindent Taking another look at our motivating example, we note that the hypotheses of Theorem \ref{thm:SubDecay} are precisely the hypotheses of Theorem \ref{thm:MainLocalLimit} and so an application of the local limit theorem is justified, where, because $\Omega(\phi)$ is a singleton, the sum in \eqref{eq:MainLocalLimit} consists only of one term. We have
\begin{align*}
 \phi^{(n)}(x,y)&=e^{-i(x,y)\cdot(0,\pi/3)}\hat\phi((0,\pi/3))^nH_P^n(x,y)+o(n^{-\mu_\phi})\\
 &=e^{-i\pi y/3}H_P^n(x,y)+o(n^{-3/4})
\end{align*}
uniformly for $(x,y)\in\mathbb{Z}^2$. To illustrate this result, the graphs of $\Re(e^{-i\pi y/3}H_P^n)$ for $(x,y)\in\mathbb{Z}^2$ for $-20\leq x,y\leq 20$ are displayed in Figures \ref{fig:ex_intro_10_attractor_1} and \ref{fig:ex_intro_100_attractor_1} for $n=10$ and $n=100$ respectively for comparison against Figures \ref{fig:ex_intro_10_1} and \ref{fig:ex_intro_100_1}. The oscillation in the $y$-direction is now explained by the appearance of the multiplier $e^{-i\pi y/3}$ and is independent of $n$.\\
\begin{figure}[h!]
\begin{center}
\resizebox{\textwidth}{!}{
	    \begin{subfigure}[5cm]{0.5\textwidth}
		\includegraphics[width=\textwidth]{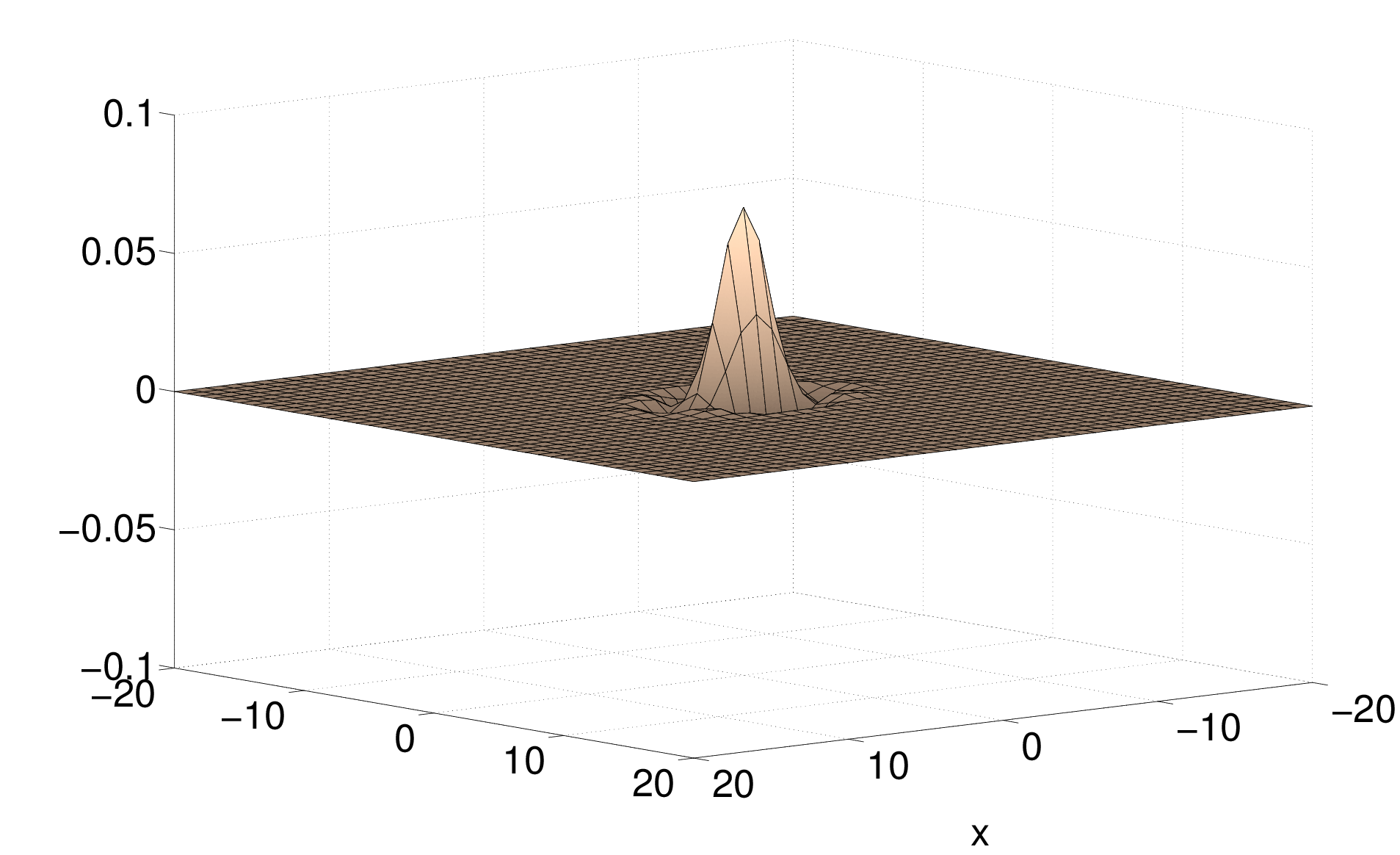}
		\caption{$\Re(e^{-i\pi y/3}H_P^n)$ for $n=10$}
		\label{fig:ex_intro_10_attractor_1}
	    \end{subfigure}
	    \begin{subfigure}[5cm]{0.5\textwidth}
		\includegraphics[width=\textwidth]{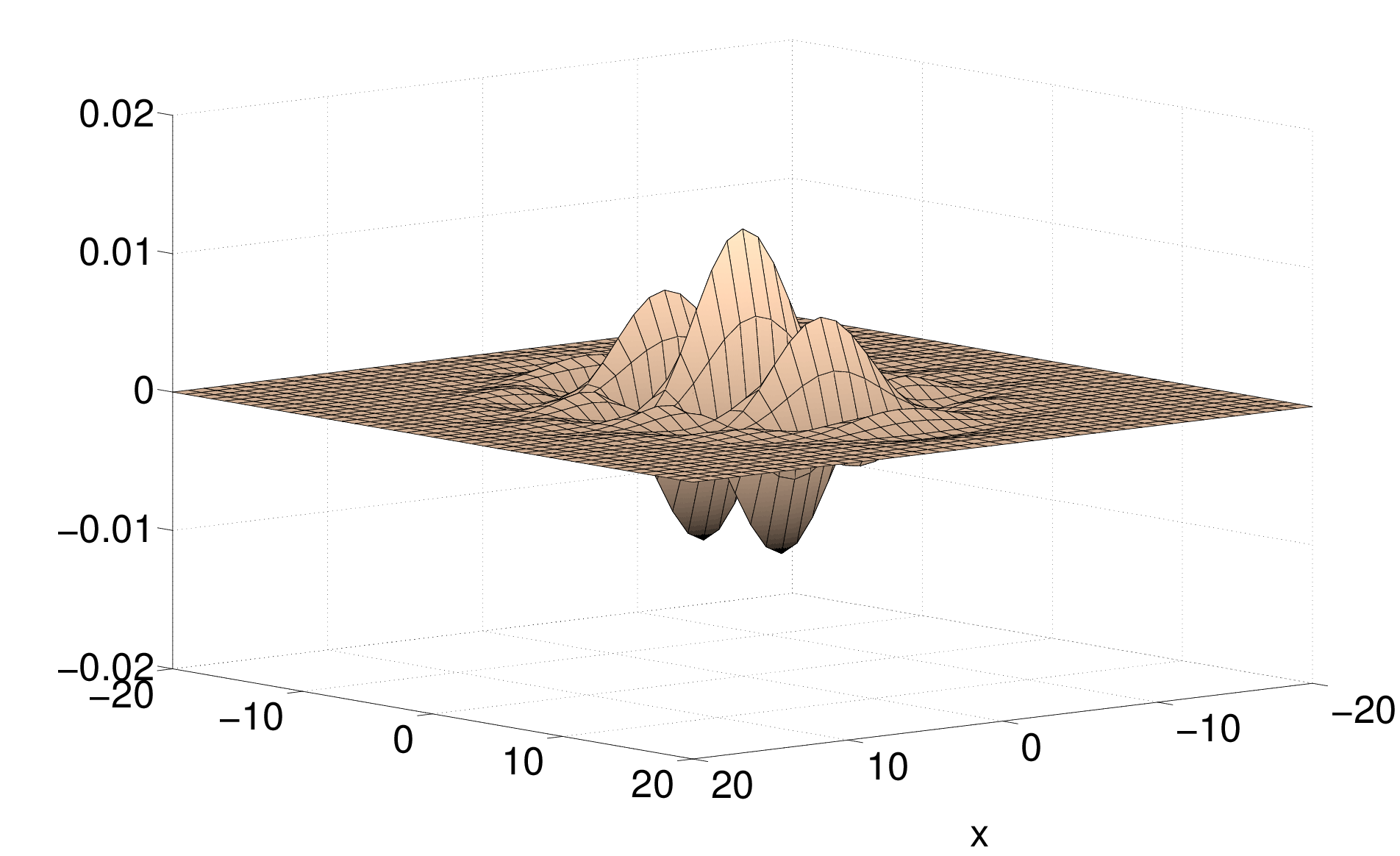}
		\caption{$\Re(e^{-i\pi y/3}H_P^n)$ for $n=100$}
		\label{fig:ex_intro_100_attractor_1}
	    \end{subfigure}}
\caption{The graphs of $\Re(e^{-i\pi y/3}H_P^n)$ for $n=10,100$.}
\label{fig:ex_intro}
\end{center}
\end{figure}

\noindent To address Question \eqref{ques:PointwiseEstimates} and obtain pointwise estimates for the $\phi^{(n)}$, we restrict our attention to those $\phi:\mathbb{Z}^d\rightarrow\mathbb{C}$ with finite support. In this article, we present two theorems concerning pointwise estimates for $|\phi^{(n)}(x)|$. The most general result, in addition to requiring finite support for $\phi$, assumes the hypotheses of Theorem \ref{thm:MainLocalLimit}; this is Theorem \ref{thm:SubExponentialEstimate}. The other result, Theorem \ref{thm:ExponentialEstimate}, additionally assumes that all $\xi\in\Omega(\phi)$ have the same corresponding drift $\alpha_\xi=\alpha\in\mathbb{R}^d$ and positive homogeneous polynomial $P=P_{\xi}$ -- a condition which is seen to be quite natural by taking a look at  Subsections \ref{subsec:ex3} and \ref{subsec:ClassicalLLT}, although not necessary, see Remark \ref{rmk:ExpHypoth}.  Theorem \ref{thm:ExponentialEstimate} extends the corresponding $1$-dimensional result, Theorem 3.1 of \cite{Diaconis2014}, to $d$-dimensions and, even in $1$-dimension, is seen to be an improvement. In addition to global pointwise estimates for $\phi^{(n)}$, in Section \ref{sec:PointwiseBounds} we present a variety of results which give global pointwise estimates for discrete space and time derivatives of $\phi^{(n)}$. In what follows, we describe the statement of Theorem \ref{thm:ExponentialEstimate} as it is the simplest.\\

\noindent For simplicity, assume that $\phi:\mathbb{Z}^d\rightarrow\mathbb{C}$ is finitely supported, satisfies $\sup_\xi|\hat\phi|=1$ and $\Omega(\phi)$ consists of only one point $\xi_0$ which is of positive homogeneous type for $\hat\phi$. In this case, we use Theorem \ref{thm:MainLocalLimit} to motivate the correct form for pointwise estimated for $\phi^{(n)}$. The theorem gives the approximation
\begin{equation}\label{eq:SingleAttractorLocalLimit}
\phi^{(n)}(x)=e^{-ix\cdot\xi_0}\hat{\phi}(\xi_0)^{n}H_P^n(x-n\alpha)+o(n^{-\mu_P})
\end{equation}
uniformly for $x\in\mathbb{Z}^d$, where $P=P_{\xi_0}$ is positive homogeneous and $\alpha=\alpha_{\xi_0}\in\mathbb{R}^d$. Pointwise estimates for the attractor $H_P$ can be deduced with the help of the Legendre-Fenchel transform, a central object in convex analysis \cite{Rockafellar1970,Touchette2007}. The Legendre-Fenchel transform of $R=\Re P$ is the function $R^{\#}:\mathbb{R}^d\rightarrow\mathbb{R}$ defined by
\begin{equation*}
R^{\#}(x)=\sup_{\xi\in\mathbb{R}^d}\{x\cdot\xi-R(\xi)\}.
\end{equation*}
It is evident that $R^{\#}(x)\geq 0$ and, for $E\in\Exp(P)$,
\begin{equation*}
tR^{\#}(x)=\sup_{\xi\in\mathbb{R}^d}\left\{tx\cdot \xi-R(t^{E}\xi)\right\}=R^{\#}\left(t^{(I-E)^*}x\right)
\end{equation*}
for all $t>0$ and $x\in\mathbb{R}^d$, i.e., $(I-E)^*\in\Exp(R^{\#})$. It turns out that $R^{\#}$ is necessarily continuous and positive definite (Proposition \ref{prop:LegendreContinuousPositiveDefinite}).  In Section \ref{sec:HomogeneousPolynomialsandAttractors}, we establish the following pointwise estimates for $H_P$. There exists positive constants $C,M$ such that
\begin{equation}\label{eq:LegendreTransformEstimateIntro}
|H_P^t(x)|\leq\frac{C}{t^{\tr E}}\exp(-MR^{\#}(t^{-E^*}x))=\frac{C}{t^{\mu_P}}\exp(-tMR^{\#}(x/t))
\end{equation}
for all $x\in\mathbb{R}^d$ and $t>0$. \\

\begin{remark}In the special case that $P(\xi)=|\xi|^{2m}$, $E=(2m)^{-1}I\in\Exp(P)$ and one can directly compute $R^{\#}(x)=C_m|x|^{2m/(2m-1)}$ where $C_m=(2m)^{-1/(2m-1)}-(2m)^{-2m/(2m-1)}>0$. Here, the estimate \eqref{eq:LegendreTransformEstimateIntro} takes the form
\begin{equation*}
H^t_{|\cdot|^{2m}}(x)\leq \frac{C}{t^{d/2m}}\exp\left(-M|x|^{2m/(2m-1)}/t^{1/(2m-1)}\right)
\end{equation*}
for $t>0$ and $x\in\mathbb{R}^d$ and so we recapture the well-known off-diagonal estimate for the semigroup $e^{-t(-\Delta)^m}$ \cite{Friedman1964,Eidelman1969,Davies1995,Davies1997}. In the context of local limit theorems, $H_{|\cdot|^{2m}}$ is seen to be the attractor of the convolution powers of $\kappa_m=\delta_0-(\delta_0-\kappa)^{(m)}$ where $\kappa$ is the probability distribution assigning $1/2$ probability to $0$ and $1/(4d)$ probability to $\pm e_j$ for $j=1,2,\dots,d$; here and in what follows, $e_1,e_2,\dots,e_d$ denote the standard euclidean basis vectors of $\mathbb{R}^d$.\\
\end{remark}

\noindent In view of \eqref{eq:SingleAttractorLocalLimit} and the preceding discussion, one expects an estimate of the form \eqref{eq:LegendreTransformEstimateIntro} to hold for $\phi^{(n)}$, although, we note that no such estimate can be established on these grounds (this is due to the error term in \eqref{eq:SingleAttractorLocalLimit}). This however motivates the correct form and we are able to establish the following result which captures, as a special case, the situation described above in which $\Omega(\phi)=\{\xi_0\}$.

\begin{theorem}\label{thm:ExponentialEstimate}
Let $\phi:\mathbb{Z}^d\rightarrow\mathbb{C}$ be finitely supported and such that $\sup_{\xi\in\mathbb{T}^d}|\hat\phi(\xi)|=1$. Suppose that every point of $\xi\in\Omega(\phi)$ is of positive homogeneous type for $\hat\phi$ and every $\xi\in\Omega(\phi)$ has the same drift $\alpha=\alpha_{\xi}\in\mathbb{R}^d$ and positive homogeneous polynomial $P=P_{\xi}$. Also let $\mu_\phi=\mu_P$ be defined by \eqref{eq:muP} and let $R^{\#}$ be the Legendre-Fenchel transform of $R=\Re P$. Then there exists $C,M>0$ for which
\begin{equation}\label{eq:ExponentialEstimate_}
|\phi^{(n)}(x)|\leq\frac{C}{n^{\mu_\phi}}\exp\left(-nMR^{\#}\left(\frac{x-n\alpha }{n}\right)\right)
\end{equation}
for all $n\in\mathbb{N}_+$ and $x\in\mathbb{Z}^d$.
\end{theorem}

\noindent Revisiting, for a final time, our motivating example, we note that $\phi$ also satisfies the hypotheses of Theorem \ref{thm:ExponentialEstimate}. An appeal to the theorem gives constants $C,M>0$ for which
\begin{equation}\label{eq:ex_intro_1}
|\phi^{(n)}(x,y)|\leq\frac{C}{n^{3/4}}\exp\left(-nM R^{\#}((x,y)/n)\right)
\end{equation}
for all $n\in\mathbb{N}_+$ and for all $(x,y)\in\mathbb{Z}^2$, where $R^{\#}$ is the Legendre-Fenchel transform of $R=\Re P=P$. Instead of finding a closed-form expression for $R^{\#}$, which is not particularly illuminating, we simply remark that
\begin{equation}\label{eq:ex_intro_2}
R^{\#}(x,y)\asymp |x|^{4/3}+|y|^2,
\end{equation}
where $\asymp$ means that the ratio of the functions is bounded above and below by positive constants (\eqref{eq:ex_intro_2} is straightforward to establish and can be seen as consequence of Corollary \ref{cor:LegendreCompareDiagonal}). Upon combining \eqref{eq:ex_intro_1} and \eqref{eq:ex_intro_2}, we obtain constants $C,M>0$ for which
\begin{align*}
|\phi^{(n)}(x,y)|&\leq\frac{C}{n^{3/4}}\exp\left(-nM \left(\left|\frac{x}{n}\right|^{4/3}+\left|\frac{y}{n}\right|^2\right)\right)\\
&\leq\frac{C}{n^{3/4}}\exp\left(-M\left(\frac{|x|^{4/3}}{n^{1/3}}+\frac{|y|^2}{n}\right)\right)
\end{align*}
for all $n\in\mathbb{N}_+$ and for all $(x,y)\in\mathbb{Z}^2$. This result illustrates the anisotropic exponential decay of $n^{3/4}|\phi^{(n)}(x,y)|$ for each $n\in\mathbb{N}_+$.\\

\noindent Back within the general setting and continuing under the assumption that $\phi:\mathbb{Z}^d\rightarrow\mathbb{C}$ is finitely supported, we come to the final question posed at the beginning of this introduction, Question \eqref{ques:Stability}. The following result extends the (affirmative) results of V. Thom\'{e}e \cite{Thomee1965} and M.V. Fedoryuk \cite{Fedoryuk1967} (see also Theorem 7.5 of \cite{Schreiber1970}).

\begin{theorem}\label{thm:StabilityIntro}
Let $\phi:\mathbb{Z}^d\rightarrow\mathbb{C}$ be finitely supported and such that $\sup_{\xi}|\hat\phi(\xi)|=1$. Suppose additionally that each $\xi\in\Omega(\phi)$ is of positive homogeneous type for $\hat\phi$. Then, there exists a positive constant $C$ for which
\begin{equation*}
\|\phi^{(n)}\|_1=\sum_{x\in\mathbb{Z}^d}|\phi^{(n)}(x)|\leq C
\end{equation*}
 for all $n\in\mathbb{N}_+$.
\end{theorem}

\noindent This article is organized as follows: Section \ref{sec:HomogeneousPolynomialsandAttractors} outlines the basic theory of positive homogeneous polynomials and their corresponding attractors. Section \ref{sec:PropertiesofHatPhi} focuses on the local behavior of $\hat\phi$ wherein necessary and sufficient condition are given to ensure that a given $\xi_0\in\Omega(\phi)$ is of positive homogeneous type for $\hat\phi$. In Section \ref{sec:LocalLimits}, we prove the main local limit theorem, Theorem \ref{thm:MainLocalLimit}, and deduce from it Theorem \ref{thm:SubDecay}. Section \ref{sec:PointwiseBounds} focuses on global space-time bounds for $\phi^{(n)}$ in the case that $\phi$ is finitely supported. In addition to the proof of Theorem \ref{thm:ExponentialEstimate}, Subsection \ref{subsec:ExponentialEstimate} contains a number of results concerning global exponential estimates for discrete space and time differences of $\phi^{(n)}$. In Subsection \ref{subsec:SubExponentialEstimate}, we prove global sub-exponential estimates for $\phi^{(n)}$ in the general case that $\phi$, in addition to being finitely supported, satisfies the hypotheses of Theorem \ref{thm:StabilityIntro}; this is Theorem \ref{thm:SubExponentialEstimate}. In Section \ref{sec:Stability}, after a short discussion on stability of numerical difference schemes in partial differential equations, we present Theorem \ref{thm:StabilityIntro} as a consequence of Theorem \ref{thm:SubExponentialEstimate}. Section \ref{sec:examples} contains a number of concrete examples, mostly in $\mathbb{Z}^2$, to which we apply our results; the reader is encouraged to skip ahead to this section as it can be read at any time. We end Section \ref{sec:examples} by showing, from our perspective, some results on the classical theory of random walks on $\mathbb{Z}^d$. The Appendix, Section \ref{sec:Appendix}, contains a number of linear-algebraic results which highlight the interplay between one-parameter contracting groups and positive homogeneous functions.\\

\noindent\textbf{Notation:} For $y\in\mathbb{Z}^d$, $\delta_y:\mathbb{Z}^d\rightarrow\{0,1\}$ is the standard delta function defined by $\delta_y(y)=1$ and $\delta_y(x)=0$ for $x\neq y$. For any subset $A$ of $\mathbb{R}$, $A_+$ denotes the subset of positive elements of $A$. Given $M\in\MdR$, its corresponding linear transformation on $\mathbb{R}^d$ is denoted by $L_M$. For any $r>0$, we denote the open unit ball with center $x\in\mathbb{R}^d$ by $B_{r}(x)$ and the closed unit ball by $\overline{B_r(x)}$. When $x=0$, we write $B_r=B_r(0)$ and denote by $S_r=\partial B_r$ the sphere of radius $r$. Further, when $r=1$, we write $B=B_1$ and $S=S_1$. We define a $d$-dimensional floor function by $\lfloor \cdot \rfloor:\mathbb{R}^d\rightarrow\mathbb{Z}^d$ by $\lfloor x\rfloor=(\lfloor x_1\rfloor,\lfloor x_2\rfloor,\dots,\lfloor x_d\rfloor)$ for $x\in\mathbb{R}^d$ where $\lfloor x_k\rfloor$ is the integer part of $x_k$ for $k=1,2,\dots d$; this is admittedly a slight abuse of notation. Given $\mathbf{n}=(n_1,n_2,\dots,n_d)\in(\mathbb{N}_+)^d=\mathbb{N}_+^d$ and a multi-index $\beta\in\mathbb{N}^d$, put
\begin{equation*}
|\beta:\mathbf{n}|=\sum_{k=1}^d\frac{\beta_k}{n_k};
\end{equation*}
this is consistent with H\"{o}rmander's notation for semi-elliptic operators and polynomials \cite{Hormander1983}. For any two real functions $f,g$ on a set $X$, we write $f\asymp g$ when there are positive constants $C$ and $C'$ for which $Cg(x)\leq f(x)\leq C'g(x)$ for all $x\in X$.

\section{Positive homogeneous polynomials and attractors}\label{sec:HomogeneousPolynomialsandAttractors}

\noindent In this section, we study positive homogeneous polynomials and their corresponding attractors; let us first give some background. In H\"{o}rmander's treatise \cite{Hormander1983}, polynomials of the form
\begin{equation*}
Q(\xi)=\sum_{|\beta:\mathbf{m}|\leq 1}a_\beta\xi^\beta
\end{equation*}
for $\mathbf{m}\in\mathbb{N}_+^d$ are called \textit{semi-elliptic} provided their principal part,
\begin{equation*}
Q_p(\xi)=\sum_{|\beta:\mathbf{m}|=1}a_{\beta}\xi^{\beta},
\end{equation*}
is non-degenerate, that is, $Q_p(\xi)\neq 0$ whenever $\xi\neq 0$. For a semi-elliptic polynomial $Q$, its corresponding partial differential operator $\Lambda_Q=Q(D)$, called a semi-elliptic operator, is hypoelliptic in the sense that all $\Lambda_Q$-harmonic distributions are smooth. What appears to be the most desirable property of semi-elliptic polynomials is the way that they scale in the sense that
\begin{align*}
\lefteqn{\hspace{-1.5cm}Q_p(t^{1/m_1}\xi_1,t^{1/m_2}\xi_2,\dots,t^{1/m_d}\xi_d)}\\
&=\sum_{|\beta:\mathbf{m}|=1}a_{\beta}\prod_{j=1}^d(t^{1/m_i}\xi_j)^{\beta_j}=\sum_{|\beta:\mathbf{m}|=1}t^{|\beta:\mathbf{m}|}a_\beta \xi^{\beta}=tQ_p(\xi)
\end{align*}
for all $t>0$ and $\xi\in\mathbb{R}^d$. This property, used explicitly by H\"{o}rmander, is precisely the statement that $E=\diag(1/m_1,1/m_2,\dots,1/m_d)\in\Exp(Q_p)$, in view of Definition \ref{def:positivehomogeneous}.
Further, the associated one-parameter group $\{T_t\}=\{t^E\}$ has the useful property that it dilates and contracts space. The following definition captures this behavior in general (see Section 1.1. of \cite{Hazod2001}).

\begin{definition}
Let $\{T_t\}_{t>0}\subseteq\GldR$ be a continuous one-parameter group. We say that $\{T_t\}$ is contracting if
\begin{equation*}
\lim_{t\rightarrow 0}\|T_t\|=0.
\end{equation*}
Here and in what follows, $\|\cdot\|$ denotes the operator norm on $\GldR$.
\end{definition}

\noindent To keep in mind, the canonical example of a contracting group is $\{t^D\}$ where $D=\diag(\gamma_1,\gamma_2,\dots,\gamma_d)\in\MdR$ with $\gamma_i>0$ for $i=1,2,\dots, d$ and here, it is easily seen that $t^D=\diag(t^{\gamma_1},t^{\gamma_2},\dots,t^{\gamma_d})$ for $t>0$. Some basic results concerning contracting groups are given in the Appendix and are used throughout this article. As we will see shortly, for any positive homogeneous polynomial $P$, $t^E$ is a contracting group for any $E\in\Exp(P)$.\\

\noindent Of interest for us is the subclass of semi-elliptic polynomials of the form
\begin{equation}\label{eq:SemiEllipticPolynomial}
P(\xi)=\sum_{|\beta:2\mathbf{m}|=1}a_\beta\xi^{\beta}=\sum_{|\beta:\mathbf{m}|=2}a_\beta\xi^{\beta},
\end{equation}
where $\mathbf{m}\in\mathbb{N}_+^d$, $\{a_{\beta}\}\subseteq\mathbb{C}$ and $\Re P$ is positive definite. For these polynomials, it is easy to see that the corresponding partial differential operator $\partial_t+\Lambda_P$ is semi-elliptic in the sense of H\"{o}rmander and hence hypoelliptic. By a slight abuse of language, any reference to a semi-elliptic polynomial is a reference to a polynomial of the form \eqref{eq:SemiEllipticPolynomial}. It is straightforward to see that all such semi-elliptic polynomials are positive homogeneous and have $D=\diag((2m_1)^{-1},(2m_2)^{-1},\dots,(2m_d)^{-1})\in \Exp(P)$. However, not all positive homogeneous polynomials are semi-elliptic as the example of Subsection \ref{subsec:ex3} illustrates.  As our first result of this section shows, every positive homogeneous polynomial has a coordinate system in which it is semi-elliptic.

\begin{proposition}\label{prop:PositiveHomogeneousPolynomialsareSemiElliptic}
Let $P$ be a positive homogeneous polynomial and let $E\in\Exp(P)$ have real spectrum. There exist $A\in\GldR$ and $\{m_1,m_2,\dots,m_d\}\subseteq\mathbb{N}_+$ for which
\begin{equation}\label{eq:EisDiagonalizable}
 A^{-1}EA=\diag((2m_1)^{-1},(2m_2)^{-1},\dots,(2m_d)^{-1})
\end{equation}
and
\begin{equation}\label{eq:PositiveHomogeneousPolynomialsareSemiElliptic}
(P\circ L_A)(\xi)=\sum_{|\beta:\mathbf{m}|=2} a_{\beta}\xi^{\beta}
\end{equation}
for $\xi\in\mathbb{R}^d$.
\end{proposition}
\begin{proof}

In light of the fact that the spectrum of $E$ is real, the characteristic polynomial for $E$ factors completely over $\mathbb{R}$ and so we may apply the Jordan-Chevally decomposition. This gives $A\in\GldR$ for which $F:=A^{-1}EA=D+N$ where $D$ is a diagonal matrix, $N$ is a nilpotent matrix and $ND=DN$. It is evident that $Q:=(P\circ L_A)$ is a polynomial and so we can write
\begin{equation}\label{eq:CoordinatePoly}
Q(\xi)=\sum_{\beta}a_\beta\xi^{\beta}
\end{equation}
for all $\xi\in\mathbb{R}^d$. In fact, our hypothesis guarantees that $Q$ is positive homogeneous and $F\in\Exp(Q)$. Our proof proceeds in three steps, first we show that $D\in\Exp(Q)$. Second, we determine the spectrum of $D$. In the final step we show that $N=0$.\\

\noindent \textit{Step 1.} We have

\begin{equation}\label{eq:PolynomialHomogeneous_1}
tQ(\xi)=Q(t^{F}\xi)=Q(t^{D+N}\xi)=Q(t^{N}t^{D}\xi)
\end{equation}
for all $t>0$ and $\xi\in\mathbb{R}^d$ where $D=\mbox{diag}(\gamma_1,\gamma_2,\dots,\gamma_d)$ for $\gamma_1,\gamma_2,\dots\gamma_d\in\mathbb{R}$. Because $N$ is nilpotent,
\begin{equation*}
t^{N}=I+\frac{\log t}{1} N+\dots+\frac{(\log t)^k}{k!}N^k
\end{equation*}
where $k+1$ is the index of $N$. Thus by \eqref{eq:PolynomialHomogeneous_1}, for all $t>0$ and $\xi\in\mathbb{R}^d$,
\begin{equation}\label{eq:PolynomialHomogeneous_2}
\begin{split}
tQ(t^{-D}\xi)&=Q\left(\xi+(\log t) N\xi+\cdots+\frac{(\log t)^k}{k!}N^k\xi\right)\\
&=Q(\xi)+S_N(\xi,\log t)
\end{split}
\end{equation}
where $S_N$ is a polynomial on $\mathbb{R}^d\times\mathbb{R}$ with no constant term. Consequently, for each $\xi\in\mathbb{R}^d$ we may write
\begin{equation}\label{eq:RemainderPolynomial}
S_N(\xi,x)=\sum_{j=1}^{l}b_j(\xi)x^j
\end{equation}
where $b_j(\xi)\in\mathbb{C}$ for each $j$.

Let us now fix a non-zero $\xi\in\mathbb{R}^d$. Combining \eqref{eq:CoordinatePoly}, \eqref{eq:PolynomialHomogeneous_2} and \eqref{eq:RemainderPolynomial} yields
\begin{equation*}
\sum_{\beta}a_{\beta}t^{(1-\beta\cdot\gamma)}\xi^{\beta}=Q(\xi)+\sum_{j=1}^{l}b_j(\xi)(\log t)^j
\end{equation*}
for all $t>0$ where $\beta\cdot\gamma=\beta_1\gamma_1+\beta_2\gamma_2+\cdots\beta_d\gamma_d$ and necessarily $Q(\xi)\neq 0$. Since distinct real powers of $t$ and $\log t$ are linearly independent as $C^{\infty}$ functions for $t>0$, it follows that $b_j(\xi)=0$ for each $j$ and more importantly,
\begin{equation}\label{eq:QSemiElliptic}
Q(\xi)=\sum_{\beta\cdot\gamma=1}a_\beta\xi^{\beta}.
\end{equation}
Since $\xi$ was arbitrary, \eqref{eq:QSemiElliptic} must hold for all $\xi\in\mathbb{R}^d$ and from this we see that
\begin{equation}\label{eq:PolynomialHomogeneous_3}
Q(t^{D}\xi)=\sum_{\beta\cdot\gamma=1}a_{\beta}(t^{D}\xi)^\beta=\sum_{\beta\cdot\gamma=1}a_{\beta}t^{\beta\cdot\gamma}(\xi)^\beta=tQ(\xi)
\end{equation}
for all $t>0$ and $\xi\in\mathbb{R}^d$; hence $D\in\Exp(Q)$. \\

\noindent \textit{Step 2.} Writing $R_Q=\Re Q$, it follows from \eqref{eq:QSemiElliptic} that
\begin{equation}\label{eq:QSemiEllipticReal}
R_Q(\xi)=\sum_{\beta\cdot\gamma=1}c_\beta\xi^{\beta}
\end{equation}
for all $\xi\in\mathbb{R}^d$ where $c_\beta=\Re a_\beta$ for each multi-index $\beta$. Now for each $i=1,2,\dots, d$, $xe_i$ is an eigenvector of $D$ with eigenvalue $\gamma_i$ for all non-zero $x\in\mathbb{R}$; here $e_i$ is that of the standard euclidean basis. Using the positive definiteness of $R_Q$,  for all $t>0$ and $x\neq 0$, we have
\begin{equation*}
tR_{Q}(x e_i)=R_Q(t^D(xe_i))=R_Q(t^{\gamma_i}xe_i)=t^{(|\beta|\gamma_i)}c_{\beta}x^{|\beta|}>0
\end{equation*}
where $\beta$ is the only surviving multi-index from the sum in \eqref{eq:QSemiEllipticReal} and necessarily $\beta$ is an integer multiple of $e_i$. From this we see that $|\beta|$ must be even for otherwise positivity would be violated and also that $1/\gamma_i=|\beta|=:2m_i$ as claimed.\\

\noindent\textit{Step 3.} In view of the previous step,
\begin{equation}\label{eq:tDisDiagonal}
t^D=\diag\Big(t^{(2m_1)^{-1}},t^{(2m_2)^{-1}},\dots,t^{(2m_d)^{-1}}\Big)
\end{equation}
for all $t>0$ and so $\{t^D\}_{t>0}$ is a one-parameter contracting group. Using the positive definiteness of $R_Q$, it follows from Proposition \ref{prop:ScalefromSphere} that
\begin{equation}\label{eq:LimitofRQ}
\lim_{|\xi|\rightarrow\infty}R_Q(\xi)\geq\lim_{t\rightarrow\infty}\inf_{\eta\in S}R_Q(t^{D}\eta)\geq\lim_{t\rightarrow\infty}t\inf_{\eta\in S}R_Q(\eta)=\infty.
\end{equation}
Now because $D$ commutes with $F$ and $D\in\Exp(R_Q)$,
\begin{equation*}
R_Q(\xi)=tt^{-1}R_Q(\xi)=R_Q(t^{F}t^{-D}\xi)=R_Q(t^{N}\xi)
\end{equation*}
for $t>0$ and $\xi\in\mathbb{R}^d$. Our goal is to show that $N=0$. For suppose that $N\neq 0$, then for some $\xi\in\mathbb{R}^d$, $\nu=N\xi\neq 0$ but $N\nu=0$. Then,
\begin{align*}
R_Q(\xi)&=R_Q(t^{N}\xi)=R_Q\left(\xi+(\log t)N\xi+\frac{(\log t)^2}{2!}(N)^2\xi+\cdots\right)\\
&=R_Q(\xi+(\log t) \nu)
\end{align*}
for all $t>0$. This however cannot hold for its validity would contradict \eqref{eq:LimitofRQ} and so $N=0$ as desired.
\end{proof}


\begin{proposition}\label{prop:SymPisCompact}
If $P$ is a positive homogeneous polynomial then $\Sym(P):=\{O\in\MdR:P(O\xi)=P(\xi)\mbox{ for all }\xi\in\mathbb{R}^d\}$ is a compact subgroup of $\GldR$ and hence a subgroup of the orthogonal group, $\OdR$.
\end{proposition}
\begin{proof}
It is clear that $I\in\Sym(P)$ and that for any $O_1,O_2\in\Sym(P)$, $O_1 O_2\in\Sym(P)$. If $O\in\Sym(P)$, $R(O\xi)=R(\xi)$ for all $\xi\in\mathbb{R}^d$ where $R=\Re P$. The positive definiteness of $R$ implies that $\Ker O$ is trivial and hence $O\in\GldR$. Consequently, $P(O^{-1}\xi)=P(OO^{-1}\xi)=P(\xi)$ for all $\xi\in\mathbb{R}^d$ and hence $O^{-1}\in\Sym(P)$.

It remains to show that $\Sym(P)$ is compact and so, in view of the Heine-Borel theorem, we show that $\Sym(P)$ is closed and bounded. To see that $\Sym(P)$ is closed, let $\{O_n\}\subseteq\Sym(P)$ be such that $O_n\rightarrow O\in\MdR$. Then the continuity of $P$ implies that for all $\xi\in\mathbb{R}^d$,
\begin{equation*}
P(O\xi)=\lim_nP(O_n\xi)=P(\xi)
\end{equation*}
and so $O\in\Sym(P)$.

To show that $\Sym(P)$ is bounded, we first make an observation from the proof of Proposition \ref{prop:PositiveHomogeneousPolynomialsareSemiElliptic}. Assuming the notation therein, we conclude from \eqref{eq:LimitofRQ} that
\begin{equation}\label{eq:SymPisCompact_1}
\lim_{|\xi|\rightarrow\infty}R(\xi)=\infty
\end{equation}
because $R(\xi)=R_Q(A^{-1}\xi)$ for all $\xi\in\mathbb{R}^d$. Finally, to reach a contradiction, we assume that $\Sym(P)$ is not bounded. Then there exist sequences $\{O_n\}\subseteq\Sym(P)$ and $\{\xi_n\}\subseteq S$ for which $\lim_n|O_n\xi_n|=\infty$. Observe however that
\begin{equation*}
R(O_n\xi_n)=R(\xi_n)\leq \sup_{\xi\in S}R(\xi)<\infty
\end{equation*}
for all $n$; in view of \eqref{eq:SymPisCompact_1} we have obtained our desired contradiction.
\end{proof}

\begin{corollary}\label{cor:UniqueTrace}
Let $P$ be a positive homogeneous polynomial. Then for any $E,E'\in\Exp(P)$,
\begin{equation*}
\tr(E)=\tr(E').
\end{equation*}
\end{corollary}
\begin{proof}
For $E,E'\in\Exp(P)$, it follows immediately that $t^{E}t^{-E'}\in\Sym(P)$ for all $t>0$. In view of Proposition \ref{prop:SymPisCompact},
\begin{equation*}
t^{\tr E-\tr E'}=|t^{\tr E}t^{-\tr E'}|=|\det(t^{E})\det(t^{-E'})|=|\det(t^{E}t^{-E'})|=1
\end{equation*}
for all $t>0$; here we have used the fact that the trace of a real matrix is real and that the determinant maps $\OdR$ into the unit circle. The corollary follows immediately.
\end{proof}

\begin{lemma}\label{lem:tEisContracting}
Let $P$ be a positive homogeneous polynomial. For any $E\in\Exp(P)$, the continuous one-parameter group $\{t^{E}\}_{t>0}$ is contracting.
\end{lemma}

\begin{proof}
First let $E_0\in \Exp(P)$ have real spectrum. In view of Proposition \ref{prop:PositiveHomogeneousPolynomialsareSemiElliptic},
\begin{equation*}
A^{-1}t^{E_0}A=\mbox{diag}(t^{\gamma_1},t^{\gamma_2},\dots,t^{\gamma_d})
\end{equation*}
for all $t>0$ where $0<\gamma_i<1/2$ for $i=1,2,\dots,d$. By inspection, we can immediately conclude that $\{t^{E_0}\}_{t>0}$ is contracting. Now for any $E\in\Exp(P)$, $t^{E}t^{-E_0}\in\Sym(P)\subseteq\OdR$ for all $t>0$ by virtue of Proposition \ref{prop:SymPisCompact}; from this it follows immediately that $\{t^{E}\}$ is contracting.
\end{proof}

\noindent We now turn to the study of the attractors appearing in Theorem \ref{thm:MainLocalLimit}; these are of the form $H_P^{(\cdot)}$, defined by \eqref{eq:HPdef}, where $P$ is a positive homogeneous polynomial.

\begin{proposition}\label{prop:PropertiesofHP}
Let $P$ be a positive homogeneous polynomial with $R=\Re P$. The following is true:
\begin{enumerate}[i)]
 \item\label{st:PropertiesofHP_1} For any $t>0$, $H_P^{(t)}(\cdot)\in\mathcal{S}(\mathbb{R}^d)$.
 \item\label{st:PropertiesofHP_2} If $E\in\Exp(P)$ then, for all $t>0$ and $x\in\mathbb{R}^d$,
 \begin{equation*}
 H_{P}^{(t)}(x)=\frac{1}{t^{\tr E}}H_P^1(t^{-E^*}x)=\frac{1}{t^{\mu_P}}H_P(t^{-E^*}x);
 \end{equation*}
where $E^*$ is the adjoint of $E$.
\item\label{st:PropertiesofHP_3} There exist constants $C, M>0$ such that
\begin{equation*}
\left|H_P^{(t)}(x)\right|\leq \frac{C}{t^{\mu_P}}\exp(-tMR^{\#}(x/t))
\end{equation*}
for all $t>0$ and $x\in\mathbb{R}^d$.
\end{enumerate}
\end{proposition}

\begin{proof}
To prove items \ref{st:PropertiesofHP_1}) and \ref{st:PropertiesofHP_2}), it suffices only to show that $H_P=H_P^1\in\mathcal{S}(\mathbb{R}^d)$. Indeed, if $H_P\in\mathcal{S}(\mathbb{R}^d)$ then, in particular, $e^{-P}\in L^1(\mathbb{R}^d)$ and so the change-of-variables formula guarantees that, for any $t>0$ and $x\in\mathbb{R}^d$,
\begin{align*}
H_P^{t}(x)&=\frac{1}{(2\pi)^d}\int_{\mathbb{R}^d}e^{-tP(\xi)}e^{-ix\cdot\xi}\,d\xi\\
&=\frac{1}{(2\pi)^d}\int_{\mathbb{R}^d}e^{-P(t^{E}\xi)}e^{-ix\cdot\xi}\,d\xi\\
&=\frac{1}{(2\pi)^d}\int_{\mathbb{R}^d}e^{-P(\xi)}e^{-ix\cdot (t^{-E}\xi)}\det(t^{-E})\,d\xi\\
&=\frac{t^{-\tr E}}{(2\pi)^d}\int_{\mathbb{R}^d}e^{-P(\xi)}e^{-i(t^{-E^*}x)\cdot\xi}\,d\xi\\
&=t^{-\mu_P}H_P(t^{-E^*}x)
\end{align*}
whenever $E\in\Exp(P)$. From this the validity of item \ref{st:PropertiesofHP_2}) is clear but moreover, the formula ensures that that $H_P^t\in\mathcal{S}(\mathbb{R}^d)$ for all $t>0$.

In view of \eqref{eq:HPdef}, $H_P\in\mathcal{S}(\mathbb{R}^d)$ if and only if $e^{-P}\in\mathcal{S}(\mathbb{R}^d)$ because the Fourier transform is an isomorphism of $\mathcal{S}(\mathbb{R}^d)$. Also, for any $A\in\GldR$, it is clear that $e^{-P}\in\mathcal{S}(\mathbb{R}^d)$ if and only if $e^{-P\circ L_A}$. Hence, to show that $H_P\in\mathcal{S}(\mathbb{R}^d)$ it suffices to show that $e^{-P\circ L_A}\in\mathcal{S}(\mathbb{R}^d)$ for some $A\in\GldR$. This is precisely what we do now: Let $E\in \Exp(P)$ have real spectrum and correspondingly, take $A\in\GldR$ as guaranteed by Proposition \ref{prop:PositiveHomogeneousPolynomialsareSemiElliptic}. As in the proof of the proposition, we write $Q=P\circ L_A$, $R_Q=\Re Q$ and $D=\diag((2m_1)^{-1},(2m_2)^{-1},\dots,(2m_d)^{-1})$. It is clear that $e^{-Q}\in C^{\infty}(\mathbb{R}^d)$. Let $\mu$ and $\beta$ be multi-indices and observe that
\begin{equation*}
\|e^{-Q}\|_{\mu,\beta}:=\sup_{\xi\in\mathbb{R}^d}\big|\xi^\mu D^\beta e^{-Q}\big|=\sup_{\xi\in\mathbb{R}^d} \big|Q_{\mu,\beta}(\xi)\exp(-Q(\xi))\big |
\end{equation*}
where $Q_{\mu,\beta}$ is a polynomial. Using Proposition \ref{prop:ScalefromSphere} and the continuity of $Q_{\mu,\beta}e^{-Q}$, it follows that
\begin{align*}
\|e^{-Q}\|_{\mu,\beta}&=\sup_{\nu\in S, t>0}\Big|Q_{\mu,\beta}(t^{D}\nu)\exp(-Q(t^{D}\nu))\Big|\\
&=\sup_{\nu\in S, t>0}\Big|Q_{\mu,\beta}(t^D\nu)\exp(-tQ(\nu))\Big|.
\end{align*}
Now because $Q$ is positive homogeneous, $Q_{\mu,\beta}$ is a polynomial and $t^D$ has the form \eqref{eq:tDisDiagonal},
\begin{equation*}
|Q_{\mu,\beta}(t^D\nu)e^{-tQ(\nu)}|\leq M_1(1+t^m)e^{-tM_2}
\end{equation*}
for all $t>0$ and $\nu\in\mathcal{S}$ where $m,M_1$ and $M_2$ are positive constants. We immediately see that
\begin{equation*}
\|e^{-Q}\|_{\mu,\beta}\leq \sup_{t>0}M_1(1+t^m)e^{-tM_2}<\infty
\end{equation*}
and therefore $e^{-Q}\in\mathcal{S}(\mathbb{R}^d)$.

The key to the proof of \ref{st:PropertiesofHP_3}) is a complex change-of-variables. For each $x\in\mathbb{R}^d$, function $z\mapsto e^{-P(z)}e^{-ix\cdot z}$ is holomorphic on $\mathbb{C}^d$ and, in view of Proposition \ref{prop:ComplexEstimate}, satisfies
\begin{equation}\label{eq:propertiesofHP}
|e^{-P(\xi-i\nu)}e^{-ix\cdot(\xi-i\nu)}|=e^{-x\cdot\nu}|e^{-P(\xi-i\nu)}|\leq e^{-x\cdot\nu+MR(\nu)}e^{-\epsilon R(\xi)}
\end{equation}
for all $z=\xi-i\nu\in\mathbb{C}^d$, where $M,\epsilon$ are positive constants. By virtue of \eqref{eq:SymPisCompact_1}, \eqref{eq:propertiesofHP} ensures that the integration in the definition of $H_P$ can be shifted to any any complex plane in $\mathbb{C}^d$ parallel to $\mathbb{R}^d$. In other words, for any $x,\nu\in\mathbb{R}^d$,
\begin{equation*}
\int_{\mathbb{R}^d}e^{-P(\xi)}e^{-ix\cdot\xi}\,d\xi=\int_{\xi\in\mathbb{R}^d}e^{-P(\xi-i\nu)}e^{-ix\cdot(\xi-i\nu)}\,d\xi
\end{equation*}
and therefore
\begin{equation*}
|H_P(x)|\leq e^{-x\cdot\nu+MR(\nu)}\frac{1}{(2\pi)^d}\int_{\mathbb{R}^d}e^{-\epsilon R(\xi)}=C\exp(-(x\cdot\nu-MR(\nu))),
\end{equation*}
where $C>0$. The natural appearance of the Legendre-Fenchel transform is now seen by infimizing over $\nu\in\mathbb{R}^d$. We have
\begin{align*}
|H_P(x)|&\leq C\inf_{\nu\in\mathbb{R}^d}\exp(-(x\cdot\nu-MR(\nu)))\\
&=C\exp\left(-\sup_{\nu\in\mathbb{R}^d}\{x\cdot\nu-MR(\nu)\}\right)\\
&=C\exp\left(-(MR)^{\#}(x)\right)\\
&\leq C\exp\left(-MR^{\#}(x)\right)
\end{align*}
for all $x\in\mathbb{R}^d$, where we have made use of Corollary \ref{cor:MovingConstant} to adjust the constant $M$. Finally, an appeal to \ref{st:PropertiesofHP_2}) and Proposition \ref{prop:LegendreContinuousPositiveDefinite}, gives
\begin{align*}
|H_P^{(t)}(x)|&\leq \frac{C}{t^{\mu_P}}\exp\left(-MR^{\#}(t^{-E^*}x)\right)\\
&=\frac{C}{t^{\mu_P}}\exp\left(-MR^{\#}(t^{(I-E)^*}(x/t))\right)\\
&=\frac{C}{t^{\mu_P}}\exp\left(-tMR^{\#}(x/t)\right)
\end{align*}
for all $t>0$ and $x\in\mathbb{R}^d$.
\end{proof}

\section{Properties of $\hat\phi$}\label{sec:PropertiesofHatPhi}

\begin{lemma}\label{lem:PisUnique}
Let $\phi\in\mathcal{S}_d$ be such that $\sup|\hat\phi|=1$ and suppose that $\xi_0\in\Omega(\phi)$ is of positive homogeneous type for $\hat\phi$. Then the expansion \eqref{eq:PosHomType}, with $\alpha_{\xi_0}\in\mathbb{R}^d$ and positive homogeneous polynomial $P_{\xi_0}$, is unique.
\end{lemma}
\begin{proof}
The fact that $|\hat\phi(\xi)|\leq 1$ ensures that the linear term in the Taylor expansion for $\Gamma_{\xi_0}$ is purely imaginary. This determines $\alpha_{\xi_0}$ uniquely. We assume that
\begin{equation*}
\Gamma_{\xi_0}(\xi)=i\alpha_{\xi_0}\cdot\xi-P_1(\xi)+\Upsilon_1(\xi)=i\alpha_{\xi_0}\cdot\xi-P_2(\xi)+\Upsilon_2(\xi)
\end{equation*}
for $\xi\in\mathcal{U}$ where $P_1$ and $P_2$ are positive homogeneous polynomials with $\Re P_1=R_1$, $\Re P_2=R_2$ and $\Upsilon_i=o(R_i)$ as $\xi\rightarrow 0$ for $i=1,2$. We shall prove that $P_1=P_2$.

Let $\epsilon>0$ and, for a fixed non-zero $\zeta\in\mathbb{R}^d$, set $\delta_i=\epsilon/2R_i(\zeta)$ for $i=1,2$. Also, take $E_i\in\Exp(P_i)$ for $i=1,2$. Because $\Upsilon_i=o(R_i)$ as $\xi\rightarrow 0$ for $i=1,2$ there is a neighborhood $\mathcal{O}$ of $0$ for which $|\Upsilon_i(\xi)|<\delta_iR_i(\xi)$ whenever $\xi\in\mathcal{O}$ for $i=1,2$. By virtue of Lemma \ref{lem:tEisContracting}, $t^{-E_1}\zeta,t^{-E_2}\zeta\in\mathcal{O}$ for some $t>0$ and therefore
\begin{align*}
|P_1(\zeta)-P_2(\zeta)|&=t|P_1(t^{-E_1}\zeta)-P_2(t^{-E_2}\zeta)|\\
&\leq t|\Upsilon_1(t^{-E_1}\zeta)|+t|\Upsilon_2(t^{-E_2}\zeta)|\\
&<t\delta_1R_1(t^{-E_1}\zeta)+t\delta_2R_2(t^{-E_2}\zeta)\\
&\leq\delta_1R_1(\zeta)+\delta_2R_2(\zeta)\\
&\leq\epsilon
\end{align*}
as required.
\end{proof}

\begin{lemma}\label{lem:UpsilonToZero}
Let $\phi\in\mathcal{S}_d$ be such that $\sup|\hat\phi|=1$ and suppose that $\xi_0\in\Omega(\phi)$ is of positive homogeneous type for $\hat\phi$ with associated positive homogeneous polynomial $P=P_{\xi_0}$ and remainder $\Upsilon=\Upsilon_{\xi_0}$. Then for any $E\in\Exp(P)$,
\begin{equation*}
\lim_{t\rightarrow\infty}t\Upsilon(t^{-E}\xi)=0.
\end{equation*}
for each $\xi\in\mathbb{R}^d$.
\end{lemma}

\begin{proof}
The assertion is clear when $\xi=0$. When $\xi\in\mathbb{R}^d$ is non-zero, we note that $t^{-E}\xi\rightarrow 0$ as $t\rightarrow 0$ by virtue of Lemma \ref{lem:tEisContracting}; in particular, $t^{-E}\xi\in\mathcal{U}$ for sufficiently large $t$. Consequently,
\begin{equation*}
\lim_{t\rightarrow\infty}\frac{\Upsilon(t^{-E}\xi)}{R(t^{-E}\xi)}=0
\end{equation*}
because $\Upsilon(\eta)=o(R(\eta))$ as $\eta\rightarrow 0$ and so it follows that
\begin{equation*}
\lim_{t\rightarrow\infty}t\Upsilon(t^{-E}\xi)=\lim_{t\rightarrow\infty}R(\xi)\frac{\Upsilon(t^{-E}\xi)}{t^{-1}R(\xi)}=R(\xi)\lim_{t\rightarrow\infty}\frac{\Upsilon(t^{-E}\xi)}{R(t^{-E}\xi)}=0
\end{equation*}
as desired.
\end{proof}

\noindent Given $\xi_0\in\Omega(\phi)$ and considering the Taylor expansion for $\Gamma_{\xi_0}$, to recognize whether or not $\xi_0$ is of positive homogeneous type for $\hat\phi$ is not always straightforward, e.g., Subsection \ref{subsec:ex3}). Nonetheless, it is useful to have a method based on the Taylor expansion for $\Gamma_{\xi_0}$ through which we can determine if $\xi_0$ is of positive homogeneous type for $\hat\phi$ and, when it is, pick out the associated positive homogeneous polynomial $P_{\xi_0}$. The remainder of this section is dedicated to doing just this.\\

\noindent Given any integer $m\geq 2$, the $m$th order Taylor expansion for $\Gamma_{\xi_0}$ is necessarily of the form
\begin{equation}\label{eq:DefinitionQm}
\Gamma_{\xi_0}(\xi)=i\alpha_{\xi_0}\cdot\xi-Q^m_{\xi_0}(\xi)+O(|\xi|^{m+1})
\end{equation}
for $\xi\in\mathcal{U}$ where $\alpha_{\xi_0}\in\mathbb{R}^d$ and $Q^m_{\xi_0}(\xi)$ is a polynomial given by
\begin{equation*}
Q^m_{\xi_0}(\xi)=\sum_{1<|\alpha|\leq m}c_\alpha\xi^{\alpha}
\end{equation*}
for $\xi\in\mathbb{R}^d$, where $\{c_{\alpha}\}\subseteq\mathbb{C}$. No constant term appears in the expansion for $\Gamma_{\xi_0}$ because $\Gamma_{\xi_0}(0)=0$.  Moreover the fact that
\begin{equation*}
\hat\phi(\xi+\xi_0)=\hat\phi(\xi_0)e^{\Gamma_{\xi_0}(\xi)}
\end{equation*}
for all $\xi\in\mathcal{U}$ and the condition that $\sup|\hat\phi(\xi)|=1$ ensure that
\begin{equation*}
\Re(i\alpha_{\xi_0}\cdot\xi-Q^m_{\xi_0}(\xi))=-\Re Q^m_{\xi_0}(\xi)\leq 0
\end{equation*}
for $\xi$ sufficiently close to $0$ (in fact, this is precisely why $\alpha_{\xi_0}\in\mathbb{R}^d$). Our final result of this section, Proposition \ref{prop:PosHomTypeChar}, provides necessary and sufficient conditions for $\xi_0$ to be of positive homogeneous type for $\hat\phi$ in terms of $Q^m_{\xi_0}$. We remark that the proposition, although quite useful for examples, is not used anywhere else in this work. As the proof is lengthy and in many ways parallels the proof of Proposition \ref{prop:PositiveHomogeneousPolynomialsareSemiElliptic}, we have placed it in the Appendix, Subsection \ref{subsec:ProofofPosHomTypeChar}.

\begin{proposition}\label{prop:PosHomTypeChar}
Let $\phi\in\mathcal{S}_d$, suppose that $\sup|\hat\phi(\xi)|=1$ and let $\xi_0\in\Omega(\phi)$. Then the following are equivalent:
\begin{enumerate}[a.]
\item The point $\xi_0$ is of positive homogeneous type for $\hat\phi$ with corresponding positive homogeneous polynomial $P_{\xi_0}$.
\item There exist $m\geq 2$ and a positive homogeneous polynomial $P$ such that, for some $C,r>0$,
\begin{equation*}
C^{-1}R(\xi)\leq \Re Q_{\xi_0}^m(\xi)\leq C R(\xi)
\end{equation*}
and
\begin{equation*}
|\Im Q_{\xi_0}^m(\xi)|\leq C R(\xi)
\end{equation*}
for all $\xi\in\overline{B_r}$, where $R=\Re P$.
\item There exist $m\geq 2$ and $E\in\MdR$ with real spectrum such that, for some $r>0$ and sequence of positive real numbers $\{t_n\}$ such that $t_n\rightarrow\infty$ as $n\rightarrow\infty$, the
sequence $\{\rho_n\}$ of polynomials defined by
\begin{equation}\label{eq:PosHomTypeCharHypothesis}
\rho_n(\xi)=t_nQ_{\xi_0}^m(t_n^{-E}\xi)
\end{equation}
converges for all $\xi\in\overline{B_r}$ as $n\rightarrow\infty$ and its limit has positive real part for all $\xi\in S_r$.
\end{enumerate}
When the above equivalent conditions are satisfied, for any $m'\geq m$,
\begin{equation*}
P_{\xi_0}(\xi)=\lim_{t\rightarrow \infty}tQ^{m'}_{\xi_0}(t^{-E}\xi)
\end{equation*}
for all $\xi\in\mathbb{R}^d$ and this convergence is uniform on all compact subsets of $\mathbb{R}^d$.
\end{proposition}

\section{Local limit theorems and $\ell^{\infty}$ estimates}\label{sec:LocalLimits}

In this section we prove Theorems \ref{thm:SubDecay} and \ref{thm:MainLocalLimit}. Our first result ensures that, under the hypotheses of Theorem \ref{thm:MainLocalLimit}, we can approximate the convolution powers of $\phi$ by a finite sum of attractors.

\begin{proposition}\label{prop:OmegaisFinite}
Let $\phi\in\mathcal{S}_d$ be such that $\sup|\hat\phi(\xi)|=1$. If each $\xi\in\Omega(\phi)$ is of positive homogeneous type for $\hat\phi$ then $\Omega(\phi)$ is discrete (and hence finite).
\end{proposition}
\begin{proof}
Let $\xi_0\in\Omega(\phi)$ be of positive homogeneous type for $\hat\phi$; it suffices to show that $\xi_0$ is an isolated point of $\Omega(\phi)$. In view of  Definitions \ref{def:positivehomogeneous} and \ref{def:PosHomType}, let $\Gamma_{\xi_0}$, $R_{\xi_0}=\Re P_{\xi_0}$ and $\Upsilon_{\xi_0}$ be associated to $\xi_0$. Because $R_{\xi_0}$ is positive definite and $\Upsilon_{\xi_0}(\eta)=o(R_{\xi_0}(\eta))$ as $\eta\rightarrow 0$, there is a neighborhood of $0$ on which $\Gamma_{\xi_0}(\xi)=0$ only when $\xi=0$. Since $\hat\phi(\xi+\xi_0)=\hat\phi(\xi_0)\exp(\Gamma_{\xi_0}(\xi))$ for all $\xi\in\mathcal{U}$, there is a neighborhood of $\xi_0$ on which $|\hat\phi(\xi)|<1$ for all $\xi\neq \xi_0$. Hence $\xi_0$ is an isolated point of $\Omega(\phi)$.
\end{proof}

\begin{remark}\label{rmk:TorusShift}
For any $\phi$ which satisfied the hypotheses of Proposition \ref{prop:OmegaisFinite}, we fix $\mathbb{T}_\phi^d=(-\pi,\pi]^d+\xi_\phi$ where $\xi_\phi\in\mathbb{R}^d$ makes $\Omega(\phi)$ live in the interior of $\mathbb{T}^d_\phi$ (as a subspace of $\mathbb{R}^d$); this can always be done in view of the proposition. We do this only to avoid non-essential technical issues arising from the difference between the topology of $\mathbb{R}^d$ and the topology of $\mathbb{T}^d$ inherited as a subspace.
\end{remark}

\begin{lemma}\label{lem:LocalLimit}
Let $\phi\in\mathcal{S}_d$ be such that $\sup|\hat\phi(\xi)|=1$ and suppose that $\xi_0\in\Omega(\phi)$ is of positive homogeneous type for $\hat\phi$. Let $\alpha=\alpha_{\xi_0}$ and $P=P_{\xi_0}$ be associated to $\hat\phi$ in view of Definition \ref{def:PosHomType} and let $\mu_P$ and $H_P^{(\cdot)}$ be defined by \eqref{eq:muP} and \eqref{eq:HPdef} respectively. Then there exists an open neighborhood $\mathcal{U}_{\xi_0}$ of $\xi_0$ such that, for any open sub-neighborhood $\mathcal{O}_{\xi_0}\subseteq\mathcal{U}_{\xi_0}$ containing $\xi_0$, the following limit holds. For all $\epsilon>0$ there exists $N\in\mathbb{N}_+$ such that
\begin{equation*}
\left|\frac{n^{\mu_P}}{(2\pi)^d}\int_{\mathcal{O}_{\xi_0}}\hat\phi(\xi)^ne^{-ix\cdot\xi}\,d\xi-n^{\mu_P}e^{-ix\cdot\xi_0}\hat\phi(\xi_0)^nH_P^n(x-n\alpha)\right|<\epsilon
\end{equation*}
for all natural numbers $n\geq N$ and for all $x\in\mathbb{R}^d$.
\end{lemma}
\begin{proof}
Given that $\xi_0\in\Omega(\phi)$ is of positive homogeneous type for $\hat\phi$,
\begin{equation}\label{eq:LocalLimit_1}
\hat\phi(\xi+\xi_0)=\hat\phi(\xi_0)e^{\Gamma(\xi)}
\end{equation}
for $\xi\in\mathcal{U}$ where
\begin{equation*}
\Gamma(\xi)=i\alpha\cdot\xi-P(\xi)+\Upsilon(\xi)
\end{equation*}
and where $\Upsilon(\xi)=o(R(\xi))$ and $R=\Re P$. If necessary, we restrict $\mathcal{U}$ further so that
\begin{equation}\label{eq:LocalLimit_2}
\big|e^{\Gamma(\xi)}\big|=e^{\Re(i\alpha\cdot\xi-P(\xi)+\Upsilon(\xi))}\leq e^{-R(\xi)/2}
\end{equation}
for all $\xi\in\mathcal{U}$ and put $\mathcal{U}_{\xi_0}=\xi_0+\mathcal{U}$. Now, let $\mathcal{O}_{\xi_0}\subseteq\mathcal{U}_{\xi_0}$ be an open set containing $\xi_0$. It is clear that $\mathcal{O}:=\mathcal{O}_{\xi_0}-\xi_0$ is open and is such that $0\in\mathcal{O}\subseteq\mathcal{U}$. Of course, \eqref{eq:LocalLimit_1} and \eqref{eq:LocalLimit_2} hold for all $\xi\in\mathcal{O}$.

Observe that, for all $x\in\mathbb{R}^d$ and $n\in\mathbb{N}_+$,
\begin{equation}\label{eq:LocalLimit_3}
\begin{split}
\lefteqn{\frac{n^{\mu_P}}{(2\pi)^d}\int_{\mathcal{O}_{\xi_0}}\hat\phi(\xi)^ne^{-ix\cdot\xi}\,d\xi-e^{-ix\cdot\xi_0}\hat\phi(\xi_0)^n n^{\mu_P}H_P^n(x-n\alpha)}\\ 
&=\frac{n^{\mu_P}}{(2\pi)^d}\int_{\mathcal{O}}\hat\phi(\xi+\xi_0)^ne^{-ix\cdot(\xi+\xi_0)}\,d\xi\\ 
&\quad  -e^{-ix\cdot\xi_0}\hat\phi(\xi_0)^n \frac{n^{\mu_P}}{(2\pi)^d}\int_{\mathbb{R}^d}e^{-nP(\xi)}e^{-i(x-n\alpha)\cdot\xi}\,d\xi \\ 
&=\frac{e^{-ix\cdot\xi_0}\hat\phi(\xi_0)^n}{(2\pi)^d}\Big(n^{\mu_P}\int_{\mathcal{O}}e^{n\Gamma(\xi)}e^{-ix\cdot\xi}\,d\xi-n^{\mu_P}\int_{\mathbb{R}^d}e^{-nP(\xi)}e^{-i(x-n\alpha)\cdot\xi}\,d\xi \Big).
\end{split}
\end{equation}
Now for $E\in\Exp(P)$,
\begin{align*}
n^{\mu_P}\int_{\mathbb{R}^d}e^{-nP(\xi)}e^{-i(x-n\alpha)\cdot\xi}\,d\xi &=n^{\mu_P}\int_{\mathbb{R}^d}e^{-P(n^E\xi)}e^{-i(x-n\alpha)\cdot\xi}\,d\xi\\
&=n^{\mu_P}\int_{n^{E}(\mathbb{R}^d)}e^{-P(\xi)}e^{-i(x-n\alpha)\cdot n^{-E}\xi}\det(n^{-E})\,d\xi\\
&=\int_{\mathbb{R}^d}e^{-P(\xi)}e^{-i(x-n\alpha)\cdot n^{-E}\xi}\,d\xi
\end{align*}
for all $x\in\mathbb{R}^d$ and $n\in\mathbb{N}_+$ where, in view of Corollary \ref{cor:UniqueTrace}, we have used the fact that $\det(n^{-E})=n^{-\tr E}=n^{-\mu_P}$. Noting the adjoint relation $(n^{-E})^*=n^{-E^*}$, and upon putting $y(n,x)=n^{-E^*}(x-n\alpha)$, we have
\begin{equation}\label{eq:LocalLimit_4}
n^{\mu_P}\int_{\mathbb{R}^d}e^{-nP(\xi)}e^{-i(x-n\alpha)\cdot\xi}\,d\xi=\int_{\mathbb{R}^d}e^{-P(\xi)}e^{-iy(n,x)\cdot\xi}\,d\xi
\end{equation}
for all $x\in\mathbb{R}^d$ and $n\in\mathbb{N}_+$.

Let $\epsilon>0$ and observe that, in view of Proposition \ref{prop:PropertiesofHP},
$e^{-P/2}\in L^1(\mathbb{R}^d)$ because $P(\xi)/2$ is a positive homogeneous polynomial. We can therefore choose a compact set $K$ for which
\begin{equation}\label{eq:LocalLimit_5}
 \int_{\mathbb{R}^d\setminus K}\big|e^{-P}\big|\,d\xi\leq\int_{\mathbb{R}^d\setminus K}e^{-R(\xi)}\,d\xi\leq \int_{\mathbb{R}^d\setminus K}e^{-R(\xi)/2}<\epsilon/3.
\end{equation}

By virtue of Proposition \ref{prop:ContractingCapturesCompact} and Lemma \ref{lem:tEisContracting}, there is $N_1\in\mathbb{N}_+$, such that $n^{-E}(K)\subseteq \mathcal{O}$ for all $n\geq N_1$. Thus
\begin{equation}\label{eq:LocalLimit_6}
\begin{split}
\lefteqn{\int_{\mathcal{O}}e^{n\Gamma(\xi)}e^{-ix\cdot\xi}\,d\xi}\\
 &=\int_{n^{-E}(K)}e^{n\Gamma(\xi)}e^{-ix\cdot\xi}\,d\xi+\int_{\mathcal{O}\setminus n^{-E}(K)}e^{n\Gamma(\xi)}e^{-ix\cdot\xi}\,d\xi\\
&=\int_{n^{-E}(K)}e^{-P(n^{E}\xi)+n\Upsilon(\xi)}e^{-i(x-n\alpha)\cdot\xi}\,d\xi+\int_{\mathcal{O}\setminus n^{-E}(K)}e^{n\Gamma(\xi)}e^{-ix\cdot\xi}\,d\xi\\
&=\frac{1}{n^{\mu_P}}\int_{K}e^{-P(\xi)+n\Upsilon(n^{-E}\xi)}e^{-iy(n,x)\cdot\xi}\,d\xi+\int_{\mathcal{O}\setminus n^{-E}(K)}e^{n\Gamma(\xi)}e^{-ix\cdot\xi}\,d\xi
\end{split}
\end{equation}
for all $n\geq N_1$ and $x\in\mathbb{R}^d$; here we have again used the fact that $\det(n^{-E})=n^{-\mu_P}$. Combining \eqref{eq:LocalLimit_3},\eqref{eq:LocalLimit_4} and \eqref{eq:LocalLimit_6} yields
\begin{equation}\label{eq:LocalLimit_7}
\begin{split}
\lefteqn{\left|\frac{n^{\mu_P}}{(2\pi)^d}\int_{\mathcal{O}_{\xi_0}}\hat\phi(\xi)^ne^{-ix\cdot\xi}\,d\xi-e^{-ix\cdot\xi_0}\hat\phi(\xi_0)^n n^{\mu_P}H_P^n(x-n\alpha)\right|}\\
&\leq \left|\int_{K}\left(e^{-P(\xi)+n\Upsilon(n^{-E}\xi)}-e^{-P(\xi)}\right)e^{-iy(n,x)}\,d\xi\right|\\\
& \quad+\int_{\mathbb{R}^d\setminus K}\left|e^{-P(\xi)}e^{-iy(n,x)\cdot \xi}\right|\,d\xi+n^{\mu_P}\left|\int_{\mathcal{O}\setminus n^{-E}(K)}e^{n\Gamma(\xi)}e^{-ix\cdot\xi}\,d\xi\right|\\
&\leq \int_{K}\left|e^{-P(\xi)+n\Upsilon(n^{-E}\xi)}-e^{-P(\xi)}\right|\,d\xi\\
&\quad +\int_{\mathbb{R}^d\setminus K}e^{-R(\xi)}\,d\xi+n^{\mu_P}\int_{\mathcal{O}\setminus n^{-E}(K)}\big|e^{\Gamma(\xi)}\big|^n\,d\xi\\
&=:I_1(n)+I_2(n)+I_3(n)
\end{split}
\end{equation}
for all $n\geq N_1$ and $x\in\mathbb{R}^d$.

It is clear that $I_2(n)<\epsilon/3$ for all $n\geq N_1$ by virtue of \eqref{eq:LocalLimit_5}. Now, in view of \eqref{eq:LocalLimit_2} and \eqref{eq:LocalLimit_5},
\begin{equation}\nonumber
I_3(n)\leq n^{\mu_P}\int_{\mathcal{O}\setminus n^{-E}(K)}e^{-nR(\xi)/2}\,d\xi\leq\int_{\mathbb{R}^d\setminus K}e^{-R(\xi)/2}\,d\xi<\epsilon/3
\end{equation}
for all $n\geq N_1$; here we have used that facts that $E\in\Exp(P)\subseteq \Exp(R)$, $\det(n^{-E})=n^{-\mu_P}$, and
\begin{equation*}
n^{E}(\mathcal{O}\setminus n^{-E}(K))= n^{E}(\mathcal{O})\setminus K\subseteq\mathbb{R}^d\setminus K.
\end{equation*}
 To estimate $I_1$, we recall that $n^{-E}(K)\subseteq \mathcal{O}$ for all $n\geq N_1$ and so the estimate \eqref{eq:LocalLimit_2} ensures that the integrand of $I_1(n)$ is bounded by $2$ for all $n\geq N_1$. In view of Lemma \ref{lem:UpsilonToZero}, an appeal to the Bounded Convergence Theorem gives a natural number $N\geq N_1$ for which $I_1(n)<\epsilon/3$ for all $n\geq N$. The desired result follows by combining our estimates for $I_1,I_2$ and $I_3$ with \eqref{eq:LocalLimit_7}.
\end{proof}

\noindent The next lemma follows directly from Lemma \ref{lem:LocalLimit} by upon recalling that $n^{\mu_P}H_P^{n}=H_P\circ L_{n^{-E^*}}\in\mathcal{S}(\mathbb{R}^d)$ for all $n\in\mathbb{N}_+$.

\begin{lemma}\label{lem:OnDiagonalBound}
Let $\phi$, $\xi_0$ and $P$ be as in the statement of Lemma \ref{lem:LocalLimit}. Under the same hypotheses of the lemma, there exists an open neighborhood $\mathcal{U}_{\xi_0}$ of $\xi_0$ such that, for any open sub-neighborhood $\mathcal{O}_{\xi_0}\subseteq\mathcal{U}_{\xi_0}$ containing $\xi_0$, there exists $C>0$ and a natural number $N$ such that
\begin{equation*}
\left|\frac{1}{(2\pi)^d}\int_{\mathcal{O}_{\xi_0}}\hat\phi(\xi)^ne^{-ix\cdot\xi}\,d\xi\right|\leq \frac{C}{n^{\mu_P}}
\end{equation*}
for all $n\geq N$ and $x\in\mathbb{R}^d$.
\end{lemma}

\begin{proof}[Proof of Theorem \ref{thm:MainLocalLimit}]
Under the hypotheses of the theorem, Proposition \ref{prop:OmegaisFinite} ensures that $\Omega(\phi)$ is finite. In line with the paragraph preceding the statement of the theorem, we label
\begin{equation*}
\Omega(\phi)=\{\xi_1,\xi_2,\dots,\xi_A, \xi_{A+1},\dots,\xi_B\}\subseteq \mathbb{T}^d
\end{equation*}
where $\mu_{P_{\xi_q}}=\mu_{\phi}$ for $q=1,2,\dots A$ and $\mu_{P_{\xi_q}}>\mu_\phi$ for $q=A+1,A+2,\dots B$. Also, we assume all additional notation from the paragraph preceding the statement of the theorem and take $\mathbb{T}_\phi^d$ as in Remark \ref{rmk:TorusShift}.

Let $\{\mathcal{O}_{\xi_q}\}_{q=1,2,\dots,B}$ be a collection of disjoint open subsets of $\mathbb{T}_\phi^d$ for which the conclusions of Lemmas \ref{lem:LocalLimit} and \ref{lem:OnDiagonalBound} hold for $q=1,2,\dots A$ and $q=A+1,A+2,\dots B$ respectively. Set
\begin{equation*}
K=\mathbb{T}_\phi^d\setminus\Big(\bigcup_{q=1}^B\mathcal{O}_{\xi_q}\Big)
\end{equation*}
and observe that
\begin{equation*}
s:=\sup_{\xi\in K}|\hat\phi(\xi)|<1
\end{equation*}
Now, in view of the Fourier inversion formula,
\begin{equation}\label{eq:MainLocalLimit_1}
\begin{split}
\phi^{(n)}(x)&= \frac{1}{(2\pi)^d}\int_{\mathbb{T}_\phi^d}\hat\phi(\xi)^ne^{-ix\cdot\xi}\,d\xi\\
&= \sum_{q=1}^{B}\frac{1}{(2\pi)^d}\int_{\mathcal{O}_{\xi_q}}\hat\phi(\xi)^ne^{-ix\cdot\xi}\,d\xi+\frac{1}{(2\pi)^d}\int_{K}\hat\phi(\xi)^ne^{-ix\cdot\xi}\,d\xi
\end{split}
\end{equation}
for all $x\in\mathbb{Z}^d$ and $n\in\mathbb{N}_+$. Appealing to Lemma \ref{lem:LocalLimit} ensures that for $q=1,2,\dots,A$,
\begin{equation}\label{eq:MainLocalLimit_2}
\frac{1}{(2\pi)^d}\int_{\mathcal{O}_{\xi_q}}\hat\phi(\xi)^ne^{-ix\cdot \xi}\,d\xi=e^{-ix\cdot\xi_q}\hat\phi(\xi_q)^nH^n_{P_q}(x-n\alpha_q)+o(n^{-\mu_\phi})
\end{equation}
uniformly for $x\in\mathbb{R}^d$. Now, for each $q=A+1,A+2,\dots,B$, Lemma \ref{lem:OnDiagonalBound} guarantees that
\begin{equation}\label{eq:MainLocalLimit_3}
\frac{1}{(2\pi)^d}\int_{\mathcal{O}_{\xi_q}}\hat\phi(\xi)^ne^{-ix\cdot \xi}\,d\xi=O(n^{-\mu_{P_{\xi_q}}})=o(n^{-\mu_{\phi}})
\end{equation}
uniformly for $x\in\mathbb{R}^d$ because $\mu_{P_{\xi_q}}>\mu_{\phi}$. Finally, we note that
\begin{equation}\label{eq:MainLocalLimit_4}
\frac{1}{(2\pi)^d}\int_{K}\hat\phi(\xi)^ne^{-ix\cdot\xi}\,d\xi=o(n^{-\mu_\phi})
\end{equation}
uniformly for $x\in\mathbb{R}^d$ because $s^n=o(n^{-\mu_\phi})$. The desired result is obtained by combining \eqref{eq:MainLocalLimit_1}, \eqref{eq:MainLocalLimit_2},\eqref{eq:MainLocalLimit_3} and \eqref{eq:MainLocalLimit_4}.
\end{proof}

\noindent As an application to Theorem \ref{thm:MainLocalLimit}, we are now in a position to prove $\ell^{\infty}(\mathbb{Z}^d)$ estimates for $\phi^{(n)}$ and thus give a partial answer to Question \eqref{ques:SupremumDecay}. We first treat a basic lemma whose proof makes use of the famous theorem of R. Dedekind (generalized by E. Artin) concerning the linear independence of characters. Interestingly enough, the statement of the lemma below mirrors a result of Dedekind appearing in the Vorlesungen \cite{Dirichlet1894} where the characters $e^{-ix\cdot \xi}$ are replaced by field isomorphisms (see p. 6 of \cite{Dean2009}).

\begin{lemma}\label{lem:CharaterMatrix}
For any distinct $\xi_1,\xi_2,\dots,\xi_A\in\mathbb{T}^d$, there exists $x_1,x_2,\dots x_A\in\mathbb{Z}^d$ such that
\begin{equation*}
V=\begin{pmatrix}
   e^{-ix_1\cdot\xi_1} & e^{-ix_1\cdot\xi_2} & \cdots &e^{-ix_1\cdot\xi_A}\\
  e^{-ix_2\cdot\xi_1} & e^{-ix_2\cdot\xi_2} & \cdots &e^{-ix_2\cdot\xi_A}\\
  \vdots & \vdots & \ddots &\vdots\\
 e^{-ix_A\cdot\xi_1} & e^{-ix_A\cdot\xi_2} & \cdots &e^{-ix_A\cdot\xi_A}\\
  \end{pmatrix}
\end{equation*}
is invertible.
\end{lemma}

\begin{proof}
The statement is obviously true when $A=1$ and so we use induction on $A$. Let $\xi_1,\xi_2,\dots,\xi_{A+1}\in\mathbb{T}^d$ be distinct and take $x_1,x_2,\dots,x_A\in\mathbb{Z}^d$ as guaranteed by the inductive hypotheses. For any $\zeta_1,\zeta_2,\dots,\zeta_A\in\mathbb{T}^d$, we define
\begin{equation*}
F(\zeta_1,\zeta_2,\dots,\zeta_A)=\det\begin{pmatrix}
   e^{-ix_1\cdot\zeta_1} & e^{-ix_1\cdot\zeta_2} & \cdots &e^{-ix_1\cdot\zeta_A}\\
  e^{-ix_2\cdot\zeta_1} & e^{-ix_2\cdot\zeta_2} & \cdots &e^{-ix_2\cdot\zeta_A}\\
  \vdots & \vdots & \ddots &\vdots\\
 e^{-ix_A\cdot\zeta_1} & e^{-ix_A\cdot\zeta_2} & \cdots &e^{-ix_A\cdot\zeta_A}\\
  \end{pmatrix}.
\end{equation*}
In this notation, our inductive hypothesis is the condition $F(\xi_1,\xi_1,\dots,\xi_A)\neq 0$.  Let $G:\mathbb{Z}^d\rightarrow\mathbb{C}$ be defined by
\begin{equation*}
G(x)=\det\begin{pmatrix}
   e^{-ix_1\cdot\xi_1} & e^{-ix_1\cdot\xi_2} & \cdots &e^{-ix_1\cdot\xi_A} & e^{-ix_1\cdot \xi_{A+1}}\\
  e^{-ix_2\cdot\xi_1} & e^{-ix_2\cdot\xi_2} & \cdots &e^{-ix_2\cdot\xi_A} & e^{-ix_1\cdot \xi_{A+1}}\\
  \vdots & \vdots & \ddots &\vdots &\vdots\\
 e^{-ix_A\cdot\xi_1} & e^{-ix_A\cdot\xi_2} & \cdots &e^{-ix_A\cdot\xi_A}& e^{-ix_A\cdot \xi_{A+1}}\\
e^{-ix\cdot\xi_1} & e^{-ix\cdot\xi_2} & \cdots &e^{-ix\cdot\xi_A}& e^{-ix\cdot \xi_{A+1}}
  \end{pmatrix}
\end{equation*}
for $x\in\mathbb{Z}^d$. Our job is to conclude that $G(x_{A+1})\neq 0$ for some $x_{A+1}\in\mathbb{Z}^d$. We assume to reach a contradiction that this is not the case, that is, for all $x\in\mathbb{Z}^d$, $G(x)=0$. Upon expanding by cofactors, we have
\begin{equation*}
G(x)=\sum_{k=1}^{A+1}(-1)^{A+1+k}F(\xi_1,\xi_2,\dots,\widehat{\xi_k},\dots,\xi_{A+1})e^{-ix\cdot \xi_k}=0
\end{equation*}
for all $x\in\mathbb{Z}^d$; here $\widehat{\xi_k}$ means that we have omitted $\xi_k$ from the list $\xi_1,\xi_2,\dots,\xi_{A+1}$. Given that $\xi_1,\xi_2,\dots,\xi_{A+1}$ are all distinct, the characters $x\mapsto e^{-ix\cdot\xi_k}$ for $k=1,2,\dots, A+1$ are distinct and so by Dedekind's independence theorem it follows that $F(\xi_1,\xi_2,\dots,\widehat{\xi_k},\dots,\xi_{A+1})=0$ for all $k=1,2,\dots,A+1$. This however contradicts our inductive hypotheses for $F(\xi_1,\xi_2,\dots,\xi_A,\widehat{\xi_{A+1}})=F(\xi_1,\xi_2,\dots,\xi_A)\neq 0$.
\end{proof}

\begin{proof}[Proof of Theorem \ref{thm:SubDecay}]
By virtue of Theorem \ref{thm:MainLocalLimit} and \eqref{eq:HPScale}, we have
\begin{equation}\label{eq:SupDecay_1}
n^{\mu_\phi}\phi^{(n)}(x)=\sum_{k=1}^A e^{-ix\cdot\xi_k}\hat\phi(\xi_k)^nH_{P_k}\left(n^{-E_k^*}\left(x-n\alpha_k\right)\right)+o(1)
\end{equation}
uniformly for $x\in\mathbb{Z}^d$ where $E_k\in\Exp(P_k)$ for $k=1,2,\dots A$. Upon recalling that the attractors $H_{P_k}\in \mathcal{S}(\mathbb{R}^d)$, the upper estimate of \eqref{eq:SupDecay} follows directly from \eqref{eq:SupDecay_1} and the triangle inequality. Showing the lower estimate of \eqref{eq:SupDecay} is trickier, for we must ensure that the sum in \eqref{eq:SupDecay_1} does not collapse at all $x\in\mathbb{Z}^d$ -- this is precisely where Lemma \ref{lem:CharaterMatrix} comes in.

For the distinct collection $\xi_1,\xi_2,\dots,\xi_A\in\mathbb{T}^d$, let $x_1,x_2,\dots,x_d\in\mathbb{Z}^d$ be as guaranteed by Lemma \ref{lem:CharaterMatrix} and, by focusing on $x$'s near $n\alpha_1$, we consider the $A\times A$ systems
\begin{equation}\label{eq:SupDecay_2}
f(n,x_j)=\sum_{k=1}^A\exp\left(-i( x_j+\lfloor n\alpha_1\rfloor)\cdot \xi_k\right)\hat\phi(\xi_k)^nH_{P_k}\left(n^{-E_k^*}(x_j+\lfloor n\alpha_1\rfloor-n\alpha_k)\right)
\end{equation}
and
\begin{equation}\label{eq:SupDecay_3}
g_j(n)=\sum_{k=1}\exp(-ix_j\cdot\xi_k) h_k(n)
\end{equation}
for $j=1,2,\dots, A$, where
\begin{equation*}
h_k(n)=\begin{cases}
       e^{-i\lfloor n\alpha_1\rfloor\cdot \xi_k}\hat\phi(\xi_k)^nH_{P_k}(0) &\mbox{ if }\alpha_1=\alpha_k\\
             0 & \mbox{ otherwise}
       \end{cases}
\end{equation*}
for $k=1,2,\dots,A$.
By virtue of Lemma \ref{lem:SpectralEstimateforContractingGroup} and Propositions \ref{prop:PositiveHomogeneousPolynomialsareSemiElliptic} and \ref{prop:SymPisCompact}, it follows that
\begin{equation*}
\lim_{n\rightarrow\infty}|n^{-E_k^*}(x_j+\lfloor n\alpha_1\rfloor -n\alpha_k)|=\begin{cases}
                                                                              0 & \mbox{ if }\alpha_k=\alpha_1\\
                                                                              \infty & \mbox{ otherwise}.
                                                                             \end{cases}
\end{equation*}
for all $j,k=1,2,\dots, A$. Again using the fact that each $H_{P_k}\in\mathcal{S}(\mathbb{R}^d)$, the above limit ensures that, for all $\epsilon>0$, there exists $N_{\epsilon}\in\mathbb{N}_+$ for which
\begin{equation}\label{eq:SupDecay_4}
|f(n,x_j)-g_j(n)|<\epsilon
\end{equation}
for all $j=1,2,\dots A$ and $n\geq N_{\epsilon}$. The system \eqref{eq:SupDecay_3} can be rewritten in the form
\begin{equation*}
\begin{pmatrix}
g_1(n)\\
g_2(n)\\
\vdots\\
g_A(n)
\end{pmatrix}
=\begin{pmatrix}
   e^{-ix_1\cdot\xi_1} & e^{-ix_1\cdot\xi_2} & \cdots &e^{-ix_1\cdot\xi_A}\\
  e^{-ix_2\cdot\xi_1} & e^{-ix_2\cdot\xi_2} & \cdots &e^{-ix_2\cdot\xi_A}\\
  \vdots & \vdots & \ddots &\vdots\\
 e^{-ix_A\cdot\xi_1} & e^{-ix_A\cdot\xi_2} & \cdots &e^{-ix_A\cdot\xi_A}\\
  \end{pmatrix}
\begin{pmatrix}
h_1(n)\\
h_2(n)\\
\vdots\\
h_A(n)
\end{pmatrix}
\end{equation*}
or equivalently
\begin{equation}\label{eq:SupDecay_5}
g(n)=V h(n)
\end{equation}
for $n\in\mathbb{N}_+$ where $V$ is that of Lemma \ref{lem:CharaterMatrix}. Taking $\mathbb{C}^A$ to be equipped with the maximum norm, the matrix $V$ determines a linear operator $L_V:\mathbb{C}^A\rightarrow\mathbb{C}^A$ which is bounded below by virtue of the lemma. So, in view of \eqref{eq:SupDecay_6}, there is a constant $\delta>0$ for which
\begin{equation}\label{eq:SupDecay_6}
\max_{j=1,2,\dots, A}|g_j(n)|\geq \delta \max_{j=1,2,\dots, A}|h_j(n)|\geq \delta |H_{P_1}(0)|=:3C>0
\end{equation}
for all $n\in\mathbb{N}_+$. Upon combining \eqref{eq:SupDecay_1}, \eqref{eq:SupDecay_4} and \eqref{eq:SupDecay_6}, we obtain $N\in\mathbb{N}_+$ for which
\begin{equation*}
n^{\mu_\phi}\|\phi^{(n)}\|_{\infty}\geq \max_{j=1,2,\dots A}|n^{\mu_\phi}\phi^{(n)}(x_j+\lfloor n\alpha_1\rfloor)|\geq C
\end{equation*}
for all $n\geq N$. The theorem now follows by, if necessary, adjusting the constant $C$ for $n<N$.
\end{proof}

\section{Pointwise bounds for $\phi^{(n)}$}\label{sec:PointwiseBounds}

Throughout this section, we assume that $\phi:\mathbb{Z}^d\rightarrow \mathbb{C}$ is finitely supported. In this case, $\hat\phi(z)$ is a trigonometric polynomial on $\mathbb{C}^d$. As usual, we assume that $\sup_{\xi\in\mathbb{T}^d}|\hat\phi(\xi)|=\sup_{\xi\in\mathbb{R}^d}|\hat\phi(\xi+0i)|=1$.

\subsection{Generalized exponential bounds}\label{subsec:ExponentialEstimate}

\noindent In this subsection, we prove Theorem \ref{thm:ExponentialEstimate} and present a variety of results concerning discrete space and time differences of convolution powers. The estimate of the following lemma, Lemma \ref{lem:ComplexExponentialEstimate}, is crucial to our arguments to follow; its analogue when $d=1$ can be found the proof of Theorem 3.1 of \cite{Diaconis2014}. We note that in \cite{Diaconis2014}, the analogue of Lemma \ref{lem:ComplexExponentialEstimate} is used to deduce Gevrey-type estimates from which the desired estimates follow in one dimension. Such arguments are troublesome when the decay is anisotropic for $d>1$. By contrast, our off-diagonal estimates are found by applying Lemma \ref{lem:ComplexExponentialEstimate} following a complex change-of-variables.

\begin{lemma}\label{lem:ComplexExponentialEstimate}
Let $\phi:\mathbb{Z}^d\rightarrow\mathbb{C}$ be finitely supported and such that $\sup_{\xi\in\mathbb{T}^d}|\hat\phi(\xi)|=1$. Suppose that $\xi_0\in\Omega(\phi)$ is of positive homogeneous type for $\hat\phi$ with associated $\alpha\in\mathbb{R}^d$ and positive homogeneous polynomial $P$. Define $f_{\xi_0}:\mathbb{C}^d\rightarrow\mathbb{C}$ by
\begin{equation}\label{eq:ComplexExponentialEstimate}
f_{\xi_0}(z)=\hat\phi(\xi_0)^{-1}e^{-\alpha\cdot (z+\xi_0)}\hat\phi(z+\xi_0)
\end{equation}
for $z\in\mathbb{C}^d$. For any compact set $K\subseteq\mathbb{R}^d$ containing an open neighborhood of $0$ for which $|\phi(\xi+\xi_0)|<1$ for all non-zero $\xi\in K$, there exist $\epsilon,M>0$ for which
\begin{equation*}
|f_{\xi_0}(z)|\leq \exp(-\epsilon R(\xi)+M R(\nu))
\end{equation*}
for all $z=\xi-i\nu$ such that $\xi\in K$ and $\nu\in\mathbb{R}^d$.
\end{lemma}

\begin{proof}
Write $f=f_{\xi_0}$ and denote by $\pi_{r}$ the canonical projection from $\mathbb{C}^d$ onto $\mathbb{R}^d$. We first estimate $f(z)$ on a neighborhood of $0$ in $\mathbb{C}^d$.

Our assumption that $\xi_0\in\Omega(\phi)$ ensures that the expansion \eqref{eq:PosHomType} is valid on an open set $\mathcal{U}\in\mathbb{C}^d$ such that $0\in \pi_r(\mathcal{U})\subseteq K$. By virtue of Proposition \ref{prop:ComplexEstimate}, we can further restrict $\mathcal{U}$ to ensure that, for some $\epsilon'>0$ and $M>0$,
\begin{equation}\label{eq:ComplexExponentialEstimate_1}
|f(z)|\leq e^{-\epsilon' R(\xi)+MR(\nu)}
\end{equation}
for $z=\xi-i\nu\in\mathcal{U}$.

We now estimate $f(z)$ on a cylinder of $K$ in $\mathbb{C}^d$. Since $|\hat\phi(\xi)|<1$ for all non-zero $\xi\in K$, the compactness $K\setminus\pi_r(\mathcal{U})$ ensures that, for some $0<\epsilon\leq\epsilon'$, the continuous function $h:\mathbb{C}^d\rightarrow\mathbb{C}$, defined by
\begin{equation*}
h(z)=e^{\epsilon R(\xi)}f(z)=\exp(-\epsilon (R\circ \pi_r)(z))f(z)
\end{equation*}
for $z=\xi-i\nu\in\mathbb{C}^d$, is such that $|h(\xi)|<1$ for all $\xi\in K\setminus \pi_r(\mathcal{U})$. Because $h$ is continuous, there exists $\delta>0$ for which $|h(z)|\leq 1$ for all $z=\xi-i\nu$ such that $\xi\in K\setminus\pi_r(\mathcal{U})$ and $|\nu|\leq\delta$. Consequently,
\begin{equation}\label{eq:ComplexExponentialEstimate_2}
|h(z)|\leq e^{-\epsilon R(\xi)}\leq e^{-\epsilon R(\xi)+M R(\nu)}
\end{equation}
for all $z=\xi-i\nu$ such that $\xi\in K\setminus\pi_r(\mathcal{U})$ and $|\nu|\leq \delta$. Upon possibly further restricting $\delta>0$, a combination of the estimates \eqref{eq:ComplexExponentialEstimate_1} and \eqref{eq:ComplexExponentialEstimate_2} ensures that
\begin{equation}\label{eq:ComplexExponentialEstimate_3}
|f(z)|\leq e^{-\epsilon R(\xi)+MR(\nu)}
\end{equation}
for all $z=\xi-i\nu\in\mathbb{C}$ such that $\xi\in K$ and $|\nu|\leq \delta$.

Finally, we estimate $f(z)=f(\xi-i\nu)$ for unbounded $\nu$. Because $\hat\phi$ is a trigonometric polynomial, $f(z)$ has exponential growth on the order of $|\nu|$ for $z=\xi-i\nu\in\mathbb{C}^d$ when $\xi$ is restricted to $K$. Therefore,
\begin{equation}\label{eq:ComplexExponentialEstimate_4}
|f(z)|\leq e^{-\epsilon R(\xi)+|\nu|+C}
\end{equation}
for all $z=\xi-i\nu$ such that $\xi\in K$ and $\nu\in\mathbb{R}^d$. Because $|\nu|+C$ is dominated by $R(\nu)$ by virtue of Corollary \ref{cor:RDominatesNorm}, the lemma follows immediately from the estimates \eqref{eq:ComplexExponentialEstimate_3} and \eqref{eq:ComplexExponentialEstimate_4}.
\end{proof}

\begin{lemma}\label{lem:TimeExponentialEstimate}
Let $\phi:\mathbb{Z}^d\rightarrow\mathbb{C}$ be finitely supported and such that $\sup_{\xi\in\mathbb{T}^d}|\hat\phi(\xi)|=1$. Assume additionally that $\Omega(\phi)=\{\xi_0\}$ and $\xi_0$ is of positive homogeneous type for $\hat\phi$ with corresponding $\alpha\in\mathbb{R}^d$ and positive homogeneous polynomial $P$ and let $\mathbb{T}_\phi^d$ be as in Remark \ref{rmk:TorusShift}. Define $g_{(\cdot)}:\mathbb{N}_+\times\mathbb{C}^d\rightarrow \mathbb{C}$ by $g_l(z)=1-f_{\xi_0}(z)^l$ for $l\in\mathbb{N}_+$ and $z\in\mathbb{C}^d$ where $f_{\xi_0}$ is given by \eqref{eq:ComplexExponentialEstimate}. There exist positive constants $C$ and $M$ for which
\begin{equation*}
|g_l(z)|\leq lC(R(\nu)+R(\xi))e^{lMR(\nu)}
\end{equation*}
for all $l\in\mathbb{N}_+$ and $z=\xi-i\nu$ such that $\xi\in\mathbb{T}_{\phi}^d$ and $\nu\in\mathbb{R}^d$.
\end{lemma}
\begin{proof}
By making similar arguments to those in the preceding lemma's proof, we obtain positive constants $C$ and $M$ for which $|1-f_{\xi_0}(z)|\leq C(R(\xi)+R(\nu))e^{MR(\nu)}$ for all $z=\xi-i\nu$ such that $\xi\in\mathbb{T}_{\phi}^d$ and $\nu\in\mathbb{R}^d$. The desired estimate now follows from Lemma \ref{lem:ComplexExponentialEstimate} (where $K=\overline{T_\phi^d}$) by writing $g_l=(1-f_{\xi_0})\sum_{k=0}^{l-1}f_{\xi_0}^k$ and making use of the triangle inequality.
\end{proof}

\noindent We are now in a position to prove Theorem \ref{thm:ExponentialEstimate}.
\begin{proof}[Proof of Theorem \ref{thm:ExponentialEstimate}] In view of the hypotheses, there exist $\alpha\in\mathbb{R}^d$ and a positive homogeneous polynomial $P$ such that each $\xi\in\Omega(\phi)$ is of positive homogeneous type for $\hat\phi$ with corresponding $\alpha_\xi=\alpha$ and $P_\xi=P$. We write $\Omega(\phi)=\{\xi_1,\xi_2,\dots,\xi_Q\}$ in view of Proposition \ref{prop:OmegaisFinite} and take $\mathbb{T}_\phi^d$ as in Remark \ref{rmk:TorusShift}. Because $\Omega(\phi)$ is finite and lives on the interior of $\mathbb{T}_\phi^d$, there exits a collection of mutually disjoint and relatively compact sets $\{K_q\}_{q=1}^Q$ such that $\mathbb{T}_\phi^d=\cup_{q=1}^Q K_q$ and, for each $q=1,2,\dots,Q$, $K_q$ contains an open neighborhood of $\xi_q$. We now establish two important uniform estimates. First, upon noting that $|\hat\phi(\xi+\xi_q)|<1$ for all $\xi\in \overline{K_q-\xi_q}$ for each $q=1,2,\dots,Q$, by virtue of Lemma \ref{lem:ComplexExponentialEstimate} there are positive constants $M$ and $\epsilon$ such that, for each $q=1,2,\dots,Q$,
\begin{equation}\label{eq:ExponentialEstimate_1}
|f_{\xi_q}(\xi-i\nu)|\leq \exp(-\epsilon R(\xi)-MR(\nu))
\end{equation}
for all $\xi\in K_q-\xi_q$ and $\nu\in\mathbb{R}^d$. Also, by a similar argument to those given in the proof of Lemma \ref{lem:LocalLimit}, we observe that
\begin{equation}\label{eq:ExponentialEstimate_2}
\begin{split}
n^{\mu_P}\int_{K_q-\xi_q}e^{-\epsilon nR(\xi)}\,d\xi&= n^{\mu_P}\int_{K_q-\xi_q}e^{-\epsilon R(n^{-E}\xi)}\,d\xi\\
&=\int_{n^{E}(K_q-\xi_q)}e^{-\epsilon R(\xi)}\,d\xi\\
&\leq \int_{\mathbb{R}^d}e^{-\epsilon R(\xi)}\,d\xi=:C<\infty
\end{split}
\end{equation}
for all $n\in\mathbb{N}_+$ and $q=1,2,\dots,Q$.

Now, let $\nu\in\mathbb{R}^d$ be arbitrary but fixed. Because $\hat\phi$ is a trigonometric polynomial (and so periodic on $\mathbb{C}^d$), it follows that
\begin{equation}\label{eq:ExponentialEstimate_3}
\begin{split}
\phi^{(n)}(x)&=\frac{1}{(2\pi)^d}\int_{\mathbb{T}_\phi^d}e^{-ix\cdot(\xi-i\nu)}\hat\phi(\xi-i\nu)^n\,d\xi\\
&=\frac{1}{(2\pi)^d}\sum_{q=1}^Q\int_{K_q}e^{-ix\cdot(\xi-i\nu)}\hat\phi(\xi-i\nu)^n\,d\xi
\end{split}
\end{equation}
for all $x\in\mathbb{Z}^d$ and $n\in\mathbb{N}_+$. Our aim is to uniformly estimate the integrals over $K_q$. To this end, for each $q=1,2,\cdots,Q$, we observe that
\begin{align*}
\lefteqn{\int_{K_q}e^{-ix\cdot(\xi-i\nu)}\hat\phi(\xi-i\nu)^n\,d\xi}\\
&=\int_{K_q-\xi_q}e^{-ix\cdot(\xi_q+\xi-i\nu)}\hat\phi(\xi_q)^ne^{-in\alpha\cdot(\xi_0+\xi-i\nu)}f_{\xi_q}(\xi-i\nu)^n\,d\xi\\
&=e^{-ny_n(x)\cdot\nu}\int_{K_q-\xi_q}\left(e^{-iy_n(x)\cdot(\xi_q+\xi)}\hat\phi(\xi_q)\right)^nf_{\xi_q}(\xi-i\nu)^n\,d\xi
\end{align*}
for all $x\in\mathbb{Z}^d$ and $n\in\mathbb{N}_+$, where $y_n(x):=(x-n\alpha)/n$. In view of the estimates \eqref{eq:ExponentialEstimate_1} and \eqref{eq:ExponentialEstimate_2}, we have
\begin{equation}\label{eq:ExponentialEstimate_4}
\begin{split}
\left|\int_{K_q}e^{-ix\cdot(\xi-i\nu)}\hat\phi(\xi-i\nu)\,d\xi\right|&\leq e^{-ny_n(x)\cdot\nu}\int_{K_q-\xi_q}|f_{\xi_q}(\xi-i\nu)|^n\,d\xi\\
&\leq \frac{C}{n^{\mu_\phi}}\exp(-n(y_n(x)\cdot\nu-MR(\nu)))
\end{split}
\end{equation}
for all $x\in\mathbb{Z}^d$, $n\in\mathbb{N}_+$ and $q=1,2,\dots,Q$ where the constants $M$ and $C$ are independent of $\nu$. Upon setting $C'=(2\pi)^d/Q$ and combining \eqref{eq:ExponentialEstimate_3} and \eqref{eq:ExponentialEstimate_4}, we obtain the estimate
\begin{equation*}
|\phi^{(n)}(x)|\leq \frac{C'}{n^{\mu_\phi}}\exp(-n(y_n(x)\cdot\nu-MR(\nu)))
\end{equation*}
which holds uniformly for $x\in\mathbb{Z}^d$ and $n\in\mathbb{N}_+$ and $\nu\in\mathbb{R}^d$. Consequently,
\begin{align*}
|\phi^{(n)}(x)|&\leq \inf_{\nu\in\mathbb{R}^d}\frac{C'}{n^{\mu_\phi}}\exp(-n(y_n(x)\cdot\nu-MR(\nu)))\\
&\leq\frac{C'}{n^{\mu_\phi}}\exp\left(-n\sup_{\nu}(y_n(x)\cdot\nu-MR(\nu))\right)\\
&=\frac{C'}{n^{\mu_\phi}}\exp\left(-n(MR)^{\#}(y_n(x))\right)
\end{align*}
for all $x\in\mathbb{Z}^d$ and $n\in\mathbb{N}_+$. The desired result follows upon noting that $(MR)^{\#}\asymp R^{\#}$ in view of Corollary \ref{cor:MovingConstant}.
\end{proof}

\begin{remark}\label{rmk:ExpHypoth}
The essential hypothesis of Theorem \ref{thm:ExponentialEstimate} (essential for a global exponential bound) is that each $\xi\in\Omega(\phi)$ has the same drift $\alpha$; this can be seen by looking at the example of Subsection \ref{subsec:ex2} wherein the convolution powers $\phi^{(n)}$ exhibit two ``drift packets'' which drift away from one another. The hypothesis that all of the corresponding positive homogeneous polynomials are the same can be weakened to include, at least, the condition that $R_\xi=\Re P_\xi\asymp R$ for all $\xi\in\Omega(\phi)$, where $R$ is some fixed real valued positive homogeneous polynomial. In any case, the theorem's hypotheses are seen to be natural when $\phi$ has some form of ``periodicity'' as can be seen in the example of Subsection \ref{subsec:ex3}. Also, the hypotheses are satisfied for all finitely supported and genuinely $d$-dimensional probability distributions on $\mathbb{Z}^d$, see Subsection \ref{subsec:ClassicalLLT}.\\
\end{remark}

\noindent For the remainder of this subsection, we restrict our attention further to finitely supported functions $\phi:\mathbb{Z}^d\rightarrow\mathbb{C}$ which satisfy $\sup_{\xi}|\hat\phi|=1$ and where this supremum is attained at only one point in $\mathbb{T}^d$, i.e., $\Omega(\phi)=\{\xi_0\}$. In this setting, we obtain global estimates for discrete space and time derivatives of convolution powers. Our first result concerns only discrete spatial derivatives of $\phi^{(n)}$ and is a useful complement to Theorem \ref{thm:ExponentialEstimate}. For related results, see Theorem 3.1 of \cite{Diaconis2014} and Theorem 8.2 of \cite{Thomee1969}, the latter being due to O. B. Widlund \cite{Widlund1965,Widlund1966}. For $w\in\mathbb{Z}^d$ and $\psi:\mathbb{Z}^d\rightarrow\mathbb{C}$, define $D_w \psi:\mathbb{Z}^d\rightarrow\mathbb{C}$ by
\begin{equation*}
D_w\psi(x)=\psi(x+w)-\psi(x)
\end{equation*}
for $x\in\mathbb{Z}^d$.
\begin{theorem}\label{thm:DerivativeEstimate}
Let $\phi:\mathbb{Z}^d\rightarrow\mathbb{C}$ be finitely supported and such that $\sup_{\xi\in\mathbb{T}^d}|\hat\phi(\xi)|=1$. Additionally assume that $\Omega(\phi)=\{\xi_0\}$ and that $\xi_0$ is of positive homogeneous type for $\hat\phi$ with corresponding $\alpha=\alpha_{\xi_0}\in\mathbb{R}^d$ and positive homogeneous polynomial $P=P_{\xi_0}$. Also let $\mu_\phi$ be defined by \eqref{eq:MuPhiDefinition} (or equivalently \eqref{eq:muP}), let $R^{\#}$ be the Legendre-Fenchel transform of $R=\Re P$ and take $E\in\Exp(P)$. There exists $M>0$ such that, for any $B>0$ and $m\in\mathbb{N}_+$, there exists $C_m>0$ such that, for any $w_1,w_2,\dots, w_m\in\mathbb{Z}^d$,
\begin{equation}\label{eq:DerivativeEstimate}
\left|D_{w_1}D_{w_2}\cdots D_{w_m}\left(\hat\phi(\xi_0)^{-n}e^{ix\cdot\xi_0}\phi^{(n)}(x)\right)\right|
\leq \frac{C_m}{n^{\mu_{\phi}}}\left(\prod_{j=1}^m |n^{-E^*}w_j|\right)\exp\left(-nMR^{\#}\left(\frac{x-n\alpha}{n}\right)\right)
\end{equation}
for all $x\in\mathbb{Z}^d$ and $n\in\mathbb{N}_+$ such that $|n^{-E^*}w_j|\leq B$ for $j=1,2,\dots, m$.
\end{theorem}

\noindent We remark that all constants in the statement of the theorem are independent of $E\in\Exp(P)$ in view of Proposition \ref{prop:SymPisCompact}. The appearance of the prefactor $\hat\phi(\xi_0)^{-n}e^{ix\cdot\xi_0}$ in the left hand side of the estimate is used to remove the highly oscillatory behavior which appears, for instance, in the example outlined in the introduction. That which remains of $\phi^{(n)}$ is well-behaved when this oscillatory prefactor is removed and this is loosely what the theorem asserts. Let us further note that, in contrast to Theorem \ref{thm:ExponentialEstimate}, Theorem \ref{thm:DerivativeEstimate} does not apply to the example illustrated in Subsection \ref{subsec:ex3} (where $\Omega(\phi)$ consists of two points) and, in fact, the latter theorem's conclusion does not hold for this $\phi$. See Subsection \ref{subsec:ex3} for further discussion.

\begin{lemma}\label{lem:DerivativeEstimate}
Given $A>0$, $\epsilon>0$ and $m\in\mathbb{N}_+$, there exists $C>0$ such that the function
\begin{equation*}
Q_{w_1,w_2,\dots,w_m}(z)=\prod_{i=1}^m(e^{iw_i\cdot z}-1)
\end{equation*}
satisfies
\begin{equation}\label{eq:DerivativeEstimate_1}
|Q_{w_1,w_2,\dots,w_m}(\xi-i\nu)|\leq C\left(\prod_{i=1}^{m}|n^{-E^*}w_i|\right)e^{n(\epsilon R(\xi)+R(v))}
\end{equation}
for all $z=\xi-i\nu\in\mathbb{C}^d$, $n\in\mathbb{N}_+$ and $w_1,w_2,\dots, w_m\in\mathbb{Z}^d$ for which $|n^{-E^*}w_i|\leq A$ for all $i=1,2,\dots,m$.
\end{lemma}
\begin{proof}
We observe that, for $M=m(B+1)$,
\begin{equation}\label{eq:DerivativeEstimate_2}
\begin{split}
|Q_{w_1,w_2,\dots,w_m}(z)|&\leq \prod_{j=1}^m |w_j\cdot z|e^{|w_j\cdot z|}\\
&\leq\prod_{j=1}^m |n^{-E^*}w_j||n^{E}z|e^{B|n^{E}z|}\\
&\leq\left(\prod_{j=1}^m|n^{-E^*}w_j|\right)e^{M|n^{E}z|}
\end{split}
\end{equation}
for all $z\in\mathbb{C}^d$, $n\in\mathbb{N}_+$ and $w_1,w_2,\dots, w_m\in\mathbb{Z}^d$ for which $|n^{-E^*}w_j|\leq B$ for all $j=1,2,\dots,m$. Given $\epsilon>0$, an appeal to Proposition \ref{prop:RDominatesNorm} ensures that, for some $M'>0$,
\begin{equation}\label{eq:DerivativeEstimate_3}
M|n^Ez|\leq M'+\epsilon R(n^E\xi)+R(n^E\nu)=M'+n(\epsilon R(\xi)+R(\nu))
\end{equation}
for all $z=\xi-i\nu\in\mathbb{C}^d$ and $n\in\mathbb{N}_+$. The desired estimate is obtained by combining \eqref{eq:DerivativeEstimate_2} and \eqref{eq:DerivativeEstimate_3}.
\end{proof}

\begin{proof}[Proof of Theorem \ref{thm:DerivativeEstimate}]
By replacing $\phi(x)$ by $\hat\phi(\xi_0)^{-1}e^{ix\cdot\xi_0}\phi(x)$, we assume without loss of generality that $\xi_0=0$ and $\hat\phi(\xi_0)=1$. For any $x,w_1,w_2,\dots,w_m\in\mathbb{Z}^d$ and $\nu\in\mathbb{R}^d$, we invoke the periodicity of $\hat\phi$ to see that
\begin{equation}\label{eq:DerivativeEstimate_4}
\begin{split}
\lefteqn{D_{w_1}D_{w_2}\cdots D_{w_m}\phi^{(n)}(x)}\\
&= D_{w_1}D_{w_2}\cdots D_{w_m}\frac{1}{(2\pi)^d}\int_{\mathbb{T}^d}e^{-ix\cdot (\xi-i\nu)}(\hat\phi(\xi-i\nu))^n\,d\xi\\
&= \frac{e^{-ny_n(x)\cdot \nu}}{(2\pi)^d}\int_{\mathbb{T}^d}e^{-iny_n(x)\cdot \xi}Q_{w_1,w_2,\dots,w_m}(\xi-i\nu)f(\xi-i\nu)^n\,d\xi,
\end{split}
\end{equation}
where $y_n(x)=(x-n\alpha)/n$ and $f(z)=f_{\xi_0}(z)=e^{-i\alpha\cdot z}\hat\phi(z)$ is that of Lemma \ref{lem:ComplexExponentialEstimate}. An appeal to the lemma shows that, for some $\epsilon>0$ and $M\geq 1$,
\begin{equation}\label{eq:DerivativeEstimate_5}
|f(\xi-i\nu)|\leq e^{-2\epsilon R(\xi)+(M-1)R(\nu)}
\end{equation}
for all $\xi\in\mathbb{T}^d$ and $\nu\in\mathbb{R}^d$; note that these constants are independent of $m$. By combining the estimates \eqref{eq:ExponentialEstimate_2}, \eqref{eq:DerivativeEstimate_1},\eqref{eq:DerivativeEstimate_4} and \eqref{eq:DerivativeEstimate_5} we obtain, for $\nu\in\mathbb{R}^d$ and $w_1,w_2,\dots,w_m\in\mathbb{Z}^d$,
\begin{align*}
\lefteqn{|D_{w_1}D_{w_2}\cdots D_{w_m}\phi^{(n)}(x)|}\\
&\leq e^{-ny_n(x)\cdot \nu}\int_{\mathbb{T}^d}|Q_{w_1,w_2,\dots,w_m}(\xi-i\nu)||f(\xi-i\nu)|^n\,d\xi\\
&\leq C_m'\left(\prod_{j=1}^m|n^{-E^*}w_j|\right)\exp(-ny_n(x)\cdot\nu+nMR(\nu))\int_{\mathbb{T}^d}e^{-n\epsilon R(\xi)}\,d\xi \\
&\leq  \frac{C C_m'}{n^{\mu_\phi}}\left(\prod_{j=1}^m|n^{-E^*}w_j|\right)\exp(-n(y_n(x)\cdot\nu-MR(\nu)))
\end{align*}
 for all $x\in\mathbb{Z}^d$ and $n\in\mathbb{N}_+$ for which $|n^{-E^*}w_j|\leq B$ for all $j=1,2,\dots,m$. As all constants are independent of $\nu$, the desired estimate is obtained by repeating the same line of reasoning of the proof of Theorem \ref{thm:ExponentialEstimate}.
\end{proof}

\noindent For a collection $v=\{v_1,\dots,v_d\}\in\mathbb{Z}^d$ and a multi-index $\beta$, consider the discrete spatial operator
\begin{equation}\label{eq:SpatialOperator}
D_v^{\beta}=(D_{v_1})^{\beta_1}(D_{v_2})^{\beta_2}\cdots (D_{v_d})^{\beta_d}.                                                                                                                                                                                                                                                                                                                                                                                                                                                                                                                                                                                                                                                                                                                                                                                                                                                                                               \end{equation}
Our next result, a corollary to Theorem \ref{thm:DerivativeEstimate}, gives estimates for $D_v^{\beta}\phi^{(n)}$ in the case that $n^{-E^*}$ acts diagonally on $v_j$ for $j=1,2,\dots,d$ and, in this case, the term involving $w$'s in \eqref{eq:DerivativeEstimate} simplifies considerably. We first give a definition.

\begin{definition}
Let $P:\mathbb{R}^d\rightarrow\mathbb{C}$ be a positive homogeneous polynomial and let $A\in\GldR$ and $\mathbf{m}=(m_1,m_2,\dots,m_d)\in\mathbb{N}_+^d$ be as given by Proposition \ref{prop:PositiveHomogeneousPolynomialsareSemiElliptic}. An ordered collection $v=\{v_1,v_2,\dots,v_d\}\subseteq \mathbb{Z}^d$ is said to be $P$-fitted if $A^*v_j\in\mbox{span}(e_j)$ for $j=1,2,\dots,d$. In this case we say that $\mathbf{m}$ is the weight of $v$.
\end{definition}

\noindent Let us make a few remarks about the above definition. First, for a $P$-fitted collection $v=\{v_1,v_2,\dots,v_d\}$ of weight $\mathbf{m}$, by virtue of Proposition \ref{prop:PositiveHomogeneousPolynomialsareSemiElliptic}, $t^{-E^*}v_j=t^{-1/(2m_j)}v_j$ for all $t>0$ and $j=1,2,\dots d$, where $E=ADA^{-1}\in\Exp(P)$. Our definition does not require the $v_j's$ to be non-zero and, in fact, it is possible that the only $P$-fitted collection to a given positive homogeneous polynomial $P$ is the zero collection. We note however that every positive homogeneous polynomial $P$ seen in this article admits a $P$-fitted collection $v$ which is also a basis of $\mathbb{R}^d$ and, in fact, whenever $P$ is semi-elliptic, every $P$-fitted collection is of the form $v=\{x^1e_1,x^2e_2,\dots,x^de_d\}$ where $x^1,x^2,\dots,x^d\in\mathbb{Z}$.

\begin{corollary}\label{cor:DerivativeEstimate}
Let $\phi:\mathbb{Z}^d\rightarrow\mathbb{C}$ be finitely supported and such that $\sup_{\xi\in\mathbb{T}^d}|\hat\phi(\xi)|=1$. Additionally assume that $\Omega(\phi)=\{\xi_0\}$ and that $\xi_0$ is of positive homogeneous type for $\hat\phi$ with corresponding $\alpha=\alpha_{\xi_0}\in\mathbb{R}^d$ and positive homogeneous polynomial $P=P_{\xi_0}$. Define $\mu_\phi$ by \eqref{eq:MuPhiDefinition} (or equivalently \eqref{eq:muP}), let $\mathbf{m}$ (and $A$) be as in Proposition \ref{prop:PositiveHomogeneousPolynomialsareSemiElliptic} and denoted by $R^{\#}$, the Legendre-Fenchel transform of $R=\Re P$. There exists $M>0$ such that, for any $B>0$ and multi-index $\beta$, there is a positive constant $C_\beta$ such that, for any $P$-fitted collection $v=\{v_1,v_2,\dots,v_d\}$ of weight $\mathbf{m}$,
\begin{equation}\label{eq:PrincipalAxesDerivativeEstimate}
\left|D_{v}^{\beta}\left(\hat\phi(\xi_0)^{-n}e^{ix\cdot\xi_0}\phi^{(n)}(x)\right)\right|
\leq \frac{C_\beta\prod_{j=1}^d|v_j|^{\beta_j}}{n^{\mu_\phi+|\beta:2\mathbf{m}|}}\exp\left(-nM R^{\#}\left(\frac{x-n\alpha}{n}\right)\right)
\end{equation}
for all $x\in\mathbb{Z}^d$ and $n\in\mathbb{N}_+$  such that $|v_j|\leq Bn^{1/(2m_j)}$ for $j=1,2,\dots,d$.
\end{corollary}

\begin{proof}
As we previously remarked,
\begin{equation*}
\left|n^{-E^*}v_j\right|=|a_j|\left|n^{-E^*}(A^*)^{-1}e_j\right|=|a_j|\left|(A^*)^{-1}n^{-D}e_j\right|=n^{-1/(2m_j)}|v_k|
\end{equation*}
for $j=1,2,\dots, d$ and $n\in\mathbb{N}_+$, where $D=\diag\left((2m_1)^{-1},(2m_2)^{-1},\dots,(2m_d)^{-1}\right)$ and $E=ADA^{-1}$. Considering the operator $D_v^{\beta}$, the term involving $w$'s appearing in the right hand side of \eqref{eq:DerivativeEstimate} is, in our case,
\begin{equation}\label{eq:PrincipalAxesDerivativeEstimate_1}
\prod_{j=1}^d\left(\left|n^{-E^*}v_j\right|\right)^{\beta_j}=\prod_{j=1}^d|v_j|^{\beta_j}n^{-\beta_j/(2m_j)}=n^{-|\beta:2\mathbf{m}|}\prod_{j=1}^d|v_j|^{\beta_j}
\end{equation}
for all $n\in\mathbb{N}_+$. The desired estimate now follows by inserting \eqref{eq:PrincipalAxesDerivativeEstimate_1} into \eqref{eq:DerivativeEstimate}.
\end{proof}

\noindent Our next theorem concerns discrete time estimates for convolution powers. Given $\phi:\mathbb{Z}^d\rightarrow\mathbb{C}$ which satisfies the hypotheses of Theorem \ref{thm:DerivativeEstimate} with corresponding $\alpha\in\mathbb{R}^d$. For any $l\in\mathbb{N}_+$, the theorem provides pointwise estimates for $\phi^{(n)}-\phi^{(l+n)}$ and analogous higher-order differences. Because, in general, the peak of the convolution powers drifts according to $\alpha$, to compare $\phi^{(n)}$ and $\phi^{(l+n)}$, one needs to account for this drift by re-centering $\phi^{(l+n)}$ but, in doing this, a possible complication arises: If $l\alpha\not\in \mathbb{Z}^d$, one cannot re-center $\phi^{(l+n)}$ in a way that keeps it on the lattice. For this reason, the theorem requires $l\alpha\in\mathbb{Z}^d$ and in this case $\left(\delta_{-l\alpha}\ast\phi^{(l)}\right)\ast\phi^{(n)}(x)=\phi^{(l+n)}(x+l\alpha)$ which can then be compared to $\phi^{(n)}(x)$. Assuming that $\phi$ satisfies the hypotheses of Theorem \ref{thm:DerivativeEstimate} (with $\xi_0\in\mathbb{T}^d$ and $\alpha\in\mathbb{Z}^d$), for any $l\in\mathbb{N}_+$ such that $l\alpha\in\mathbb{Z}^d$, we define the discrete time difference operator $\partial_l=\partial_{l}(\phi,\xi_0,\alpha)$ by
\begin{equation}\label{eq:TimeOperator}
\partial_{l}\psi=\left(\delta-\hat\phi(\xi_0)^{-l}\left(\delta_{-l\alpha}\ast\phi^{(l)}\right)\right)\ast\psi=\psi-\hat\phi(\xi_0)^{-l}\left(\delta_{-l\alpha}\ast\phi^{(l)}\right)\ast\psi
\end{equation}
for $\psi\in\ell^{1}(\mathbb{Z}^d)$.
\begin{theorem}\label{thm:TimeExponentialEstimate}
Let $\phi:\mathbb{Z}^d\rightarrow\mathbb{C}$ be finitely supported and such that $\sup_{\xi\in\mathbb{T}^d}|\hat\phi(\xi)|=1$. Additionally assume that $\Omega(\phi)=\{\xi_0\}$ and that $\xi_0$ is of positive homogeneous type for $\hat\phi$ with corresponding $\alpha=\alpha_{\xi_0}\in\mathbb{R}^d$ and positive homogeneous polynomial $P=P_{\xi_0}$. Define $\mu_\phi$ by \eqref{eq:MuPhiDefinition} (or equivalently \eqref{eq:muP}) and denote by $R^{\#}$, the Legendre-Fenchel transform of $R=\Re P$. There are positive constants $C$ and $M$ such that, for any $l_1,l_2,\dots,l_k\in\mathbb{N}_+$ such that $l_q\alpha\in\mathbb{Z}^d$ for $q=1,2,\dots,k$ (assume $k\geq 1$),
\begin{equation}\label{eq:TimeExponentialEstimate}
|\partial_{l_1}\partial_{l_2}\cdots\partial_{l_k} \phi^{(n)}(x)|\leq \frac{C^k k!\prod_{q=1}^kl_q}{n^{\mu_\phi+k}}\exp\left(-(n+l_1+l_2+\cdots+ l_k) M R^{\#}\left(\frac{x-n\alpha}{n+l_1+l_2+\cdots+ l_k}\right)\right)
\end{equation}
for all $x\in\mathbb{Z}^d$ and $n\in\mathbb{N}_+$.
\end{theorem}

\begin{proof}
As in the proofs of Theorems \ref{thm:ExponentialEstimate} and \ref{thm:DerivativeEstimate}, we fix $\nu\in\mathbb{R}^d$ and invoke the periodicity of $\hat\phi$ to see that
\begin{align*}
\lefteqn{\partial_{l_1}\partial_{l_2}\cdots\partial_{l_k}\phi^{(n)}(x)}\\
&=\frac{1}{(2\pi)^d}\int_{\xi\in\mathbb{T}_\phi^d}\prod_{q=1}^k\left(1-\left(\hat\phi(\xi_0)^{-1}e^{-\alpha\cdot(\xi_0+z)}\hat\phi(\xi_0+z)\right)^{l_q}\right)\hat\phi(\xi_0+z)^ne^{-i x\cdot(\xi_0+z)}\,d\xi\\
&=\frac{1}{(2\pi)^d}\int_{\xi\in\mathbb{T}_\phi^d}\prod_{q=1}^k g_{l_q}(z) \hat\phi(\xi_0)^nf(z)^n e^{-i (x-n\alpha)\cdot(\xi_0+z)}\,d\xi
\end{align*}
for all $x\in\mathbb{Z}^d$ and $n\in\mathbb{N}_+$, where $z=\xi-i\nu$; here, $f=f_{\xi_0}$ is defined by \eqref{eq:ComplexExponentialEstimate} and $g_{l_1}g_{l_2},\dots,g_{l_k}$ are those of Lemma \ref{lem:TimeExponentialEstimate}. Put $s_k=l_1+l_2+\cdots+l_k$, take $\epsilon,M$ and $C$ as guaranteed by Lemmas \ref{lem:ComplexExponentialEstimate} and \ref{lem:TimeExponentialEstimate} and set $C_1=(2C/\epsilon)$. Observe that
\begin{align*}
\lefteqn{\left|\partial_{l_1}\partial_{l_2}\cdots\partial_{l_k}\phi^{(n)}(x)\right|}\\
&\leq \frac{C_1^k k!\prod_{q=1}^kl_q}{n^k} e^{-(x-n\alpha)\cdot\nu}\\
&\quad \times\int_{\xi\in\mathbb{T}_\phi^d}\frac{1}{k!}\left(\frac{n\epsilon}{2}(R(\nu)+R(\xi))\right)^ke^{s_kMR(\nu)} \exp(-n\epsilon R(\xi)+nMR(\nu))\,d\xi\\
&\leq\frac{C_1^k k!\prod_{q=1}^kl_q}{n^k}e^{-(x-n\alpha)\cdot\nu}\\
&\quad \times\int_{\xi\in\mathbb{T}_\phi^d} \exp(n\epsilon(R(\xi)+R(\nu))/2) \exp((n+s_k)MR(\nu)-n\epsilon R(\xi))\,d\xi
\end{align*}
for $x\in\mathbb{Z}^d$ and $n\in\mathbb{N}_+$. Upon setting $y_{n,s_k}(x)=(x-n\alpha)/(n+s_k)$ and replacing $M$ by $M+\epsilon/2$, we can write
\begin{align*}
\lefteqn{\left|\partial_{l_1}\partial_{l_2}\cdots\partial_{l_k}\phi^{(n)}(x)\right|}\\
&\leq\frac{C_1^k k!\prod_{q=1}^kl_q}{n^k}\exp(-(n+s_k)\left(y_{n,s_k}(x)\cdot\nu-MR(\nu))\right)\int_{\xi\in\mathbb{T}_\phi^d}\exp(-n\epsilon R(\xi)/2)\,d\xi
\end{align*}
for $x\in\mathbb{Z}^d$ and $n\in\mathbb{N}_+$. Now, as we observed in the proof of Theorem \ref{thm:ExponentialEstimate}, the integral over $\xi$ is bounded above by $C_2 n^{-\mu_\phi}\leq C_2^kn^{-\mu_\phi}$ for some constant $C_2\geq 1$ and so we obtain the estimate
\begin{equation*}
 \left|\partial_{l_1}\partial_{l_2}\cdots\partial_{l_k}\phi^{(n)}(x)\right|\leq\frac{(C_1 C_2)^k k!\prod_{q=1}^kl_q}{n^{\mu_\phi+k}}\exp(-(n+s_k)\left(y_{n,s_k}(x)\cdot\nu-MR(\nu))\right)\\
\end{equation*}
for all $x\in\mathbb{Z}^d$ and $n\in\mathbb{N}_+$. Once again, the desired result is obtained by infimizing over $\nu\in\mathbb{R}^d$.
\end{proof}

\begin{remark}
If one allows the constant $M$ to depend on $l_1,l_2,\dots,l_k$, then \eqref{eq:TimeExponentialEstimate} can be written
\begin{equation*}
|\partial_{l_1}\partial_{l_2}\cdots\partial_{l_k} \phi^{(n)}(x)|\leq \frac{C^k k!\prod_{q=1}^kl_q}{n^{\mu_\phi+k}}\exp\left(-nM_{l_1,l_2,\dots,l_k} R^{\#}\left(\frac{x-n\alpha}{n}\right)\right)
\end{equation*}
for all $x\in\mathbb{Z}^d$ and $n\in\mathbb{N}_+$. Indeed, set $s_k=l_1+l_2+\cdots +l_k$ and observe that
\begin{align*}
-(n+s_k)R^{\#}\left(\frac{x-n\alpha}{n+s_k}\right)&=-n\sup_{\nu\in\mathbb{R}^d}\left\{\left(\frac{x-n\alpha}{n}\right)\cdot \nu-\frac{n+s_k}{n}R(\nu)\right\}\\
&\leq -n\sup_{\nu}\left\{\left(\frac{x-n\alpha}{n}\right)\cdot \nu-(1+k_s)R(\nu)\right\}\\
&=-n\Big((1+s_k)R\Big)^{\#}\left(\frac{x-n\alpha}{n}\right)\\
&\leq -nM_{s_k}R^{\#}\left(\frac{x-n\alpha}{n}\right)
\end{align*}
where we have used Corollary \ref{cor:MovingConstant} to obtain $M_{s_k}=M_{l_1,l_2,\dots,l_k}$.\\
\end{remark}

\noindent In view the remark above, the following corollary is a special case of Theorem \ref{thm:TimeExponentialEstimate} when $\alpha=0$, $\hat\phi(\xi_0)=1$ and we only consider one discrete time derivative; it applies to the example in the introduction and the examples of Subsections \ref{subsec:ex1} and \ref{subsec:ex5}.

\begin{corollary}\label{cor:TimeDifferenceEstimate}
Let $\phi:\mathbb{Z}^d\rightarrow\mathbb{C}$ be finitely supported and such that $\sup|\hat\phi(\xi)|=1$. Suppose that $\Omega(\phi)=\{\xi_0\}$ is of positive homogeneous type for $\hat\phi$ with corresponding $\alpha\in\mathbb{R}^d$ and positive homogeneous polynomial $P$. Also let $\mu_\phi$ be defined by \eqref{eq:MuPhiDefinition} (or equivalently \eqref{eq:muP}) and let $R^{\#}$ be the Legendre-Fenchel transform of $R=\Re P$.  Additionally assume that $\alpha=0$ and $\hat\phi(\xi_0)=1$.  There exists a positive constant $C$ and, to each $l\in\mathbb{N}_+$, a positive constant $M_l$ such that
\begin{equation*}
\left|\phi^{(n)}(x)-\phi^{(l+n)}(x)\right|\leq \frac{Cl}{n^{\mu_\phi+1}}\exp(-nM_lR^{\#}(x/n))
\end{equation*}
for all $x\in\mathbb{Z}^d$ and $n\in\mathbb{N}_+$.
\end{corollary}

\noindent Our final theorem of this subsection concerns both time and space differences for convolution powers.
\begin{theorem}\label{thm:TimeSpaceDerivatives}
Let $\phi:\mathbb{Z}^d\rightarrow\mathbb{C}$ be finitely supported and such that $\sup_{\xi\in\mathbb{T}^d}|\hat\phi(\xi)|=1$. Additionally assume that $\Omega(\phi)=\{\xi_0\}$ and that $\xi_0$ is of positive homogeneous type for $\hat\phi$ with corresponding $\alpha=\alpha_{\xi_0}\in\mathbb{R}^d$ and positive homogeneous polynomial $P=P_{\xi_0}$. Define $\mu_\phi$ by \eqref{eq:MuPhiDefinition} (or equivalently \eqref{eq:muP}), let $\mathbf{m}$ (and $A$) be as guaranteed by Proposition \ref{prop:PositiveHomogeneousPolynomialsareSemiElliptic} and denote by $R^{\#}$, the Legendre-Fenchel transform of $R=\Re P$. There are positive constants  $M$ and $C_0$ and, to each $B>0$ and multi-index $\beta$, a positive constant $C_\beta$ such that, for any $P$-fitted collection $v=\{v_1,v_2,\dots,v_d\}$ of weight $\mathbf{m}$ and $l_1,l_2,\dots,l_k\in\mathbb{N}_+$ such that $l_q\alpha\in\mathbb{Z}^d$ for $q=1,2,\dots,k$,
\begin{align*}
\lefteqn{\hspace{-1cm}\left|\partial_{l_1}\partial_{l_2}\cdots\partial_{l_k}D_{v}^{\beta}(\hat\phi(\xi_0)^{-1}e^{ix\cdot\xi_0}\phi^{(n)}(x))\right|}\\
&\leq\frac{C_{\beta}C_0^k k!\prod_{q=1}^kl_q\prod_{j=1}^d|v_j|^{\beta_j}}{n^{\mu_\phi+|\beta:2\mathbf{m}|+k}}\exp\left(-(n+l_1+l_2+\cdots+ l_k)M R^{\#}\left(\frac{x-n\alpha}{n+l_1+l_2+\cdots+ l_k}\right)\right)
\end{align*}
for all $x\in\mathbb{Z}^d$ and $n\in\mathbb{N}_+$ such that $|v_k|\leq Bn^{1/(2m_k)}$ for $k=1,2,\dots,d$.
\end{theorem}

\begin{proof}
By replacing $\phi(x)$ by $\hat\phi(\xi_0)^{-1}e^{ix\cdot\xi_0}\phi(x)$ we can assume without loss of generality that $\xi_0=0$ and $\hat\phi(\xi_0)=1$. Assuming the notation of Lemma \ref{lem:ComplexExponentialEstimate} (with $f=f_{\xi_0}$) and Lemma \ref{lem:TimeExponentialEstimate}, we fix $\nu\in\mathbb{R}^d$ and observe that
\begin{equation*}
\partial_{l_1}\partial_{l_2}\cdots\partial_{l_k}D_v^{\beta}\phi^{(n)}(x)=\frac{1}{(2\pi)^d}\int_{\mathbb{T}^d}\prod_{q=1}^k g_{l_q}(z)Q(z)f(z)^ne^{-i(n+s_k)y_{s_k,n}(x)\cdot z}\,d\xi
\end{equation*}
for all $x\in\mathbb{Z}^d$ and $n\in\mathbb{N}_+$, where $z=\xi-i\nu$, $s_k=l_1+l_2+\cdots+l_k$, $y_{s_k,n}(x)=(x-n\alpha)/(n+s_k)$ and $Q(z)=\prod_{j=1}^d (e^{iv_j\cdot z}-1)^{\beta_j}$ is the subject of Lemma \ref{lem:DerivativeEstimate}.  The desired estimate is now established by virtually repeating the arguments in the proof of Theorems \ref{thm:DerivativeEstimate} and \ref{thm:TimeExponentialEstimate} while making use of Lemmas \ref{lem:ComplexExponentialEstimate}, \ref{lem:TimeExponentialEstimate} and \ref{lem:DerivativeEstimate} and noting, as was done in the proof of Corollary \ref{cor:DerivativeEstimate}, that $|n^{-E^*}v_j|=n^{-1/(2m_j)}|v_j|$ for $j=1,2,\dots,d$.
\end{proof}

\subsection{Sub-exponential bounds}\label{subsec:SubExponentialEstimate}

In this subsection, we again consider a finitely supported function $\phi:\mathbb{Z}^d\rightarrow\mathbb{C}$ such that $\sup_{\xi\in\mathbb{T}^d}|\hat\phi(\xi)|=1$ and each $\xi\in\Omega(\phi)$ is of positive homogeneous type for $\hat\phi$. In contrast to the previous subsection, we do not require any relationship between the drifts $\alpha_\xi$ and positive homogeneous polynomials $P_\xi$ for those $\xi\in\Omega(\phi)$; a glimpse into Subsections \ref{subsec:ex2} and \ref{subsec:ex4} shows this situation to be a natural one. As was noted in \cite{Diaconis2014}, the optimization procedure which yielded the exponential-type estimates of the previous subsection is no longer of use. Here we have the following result concerning sub-exponential estimates.

\begin{theorem}\label{thm:SubExponentialEstimate}
Let $\phi:\mathbb{Z}^d\rightarrow\mathbb{C}$ be finitely supported and such that \break $\sup_{\xi\in\mathbb{T}^d}|\hat\phi(\xi)|=1$. Suppose additionally each $\xi\in\Omega(\phi)$ is of positive homogeneous type for $\hat\phi$ and hence $\Omega(\phi)=\{\xi_1,\xi_2,\dots,\xi_Q\}.$ Let $\alpha_q\in\mathbb{R}^d$ and positive homogeneous polynomial $P_{q}$ be those associated to $\xi_q$ for $q=1,2,\dots, Q$. Moreover, for each $q=1,2,\dots,Q$, set $\mu_q=\mu_{P_q}$ and let $E_q\in\Exp(P_q)$. Then, for any $N\geq 0$, there is a positive constant $C_N$ such that
\begin{equation}\label{eq:SubExponentialEstimate_1}
|\phi^{(n)}(x)|\leq C_N\sum_{q=1}^Q\frac{1}{n^{\mu_q}}(1+|n^{-E_q^*}(x-n\alpha_q)|)^{-N}
\end{equation}
for all $x\in\mathbb{Z}^d$ and $n\in\mathbb{N}_+$. The constant $C_N$ is independent of $E_q\in\Exp(P_q)$ for $q=1,2,\dots Q$.
\end{theorem}
\begin{proof}
In view of Proposition \ref{prop:OmegaisFinite} and Remark \ref{rmk:TorusShift}, there exist relatively open subsets $\mathcal{B}_1,\mathcal{B}_2,\dots,\mathcal{B}_Q$ of $\mathbb{T}_\phi^d$ satisfying the following properties:
\begin{enumerate}
\item For each $q=1,2,\dots Q$, $\mathcal{B}_q$ contains $\xi_q$.
\item $\mathcal{B}_1$ contains the boundary of $\mathbb{T}_\phi^d$ (as a subset of $\mathbb{R}^d$).
\item The closed sets $\{\overline{\mathcal{B}}_1,\overline{\mathcal{B}}_2,\dots,\overline{\mathcal{B}}_Q\}$ are mutually disjoint.
\end{enumerate}
For $q=1,2,\dots Q$, define
\begin{equation*}
\mathcal{O}_q=\mathbb{T}_\phi^d\setminus\left(\bigcup_{r\neq q}\overline{\mathcal{B}_r}\right)
\end{equation*}
and observe that each $\mathcal{O}_q$ is an open neighborhood of $\xi_q$ (in the relative topology). Let $\{u_q\}_{q=1}^Q$ be a smooth partition of unity subordinate to $\{\mathcal{O}_q\}_{q=1}^Q$. By construction, $u_1\equiv 1$ on the boundary of $\mathbb{T}_\phi^d$  and, for each $q=1,2,\dots Q$, $u_q$ is compactly supported in $\mathcal{O}_q$. We note that, for each $q\neq 1$, $\supp (u_q)$ is also a compact subset of $\mathbb{R}^d$ because the boundary of $\mathbb{T}_\phi^d$ is contained in $\mathcal{B}_1$ (the relative topology of $\mathbb{T}_\phi^d$ is only seen in $\supp(u_1)$). Set
\begin{equation*}
\delta=\frac{\min_{q=1,2,\dots,Q}\mbox{dist}(\supp (u_q),\partial\mathcal{O}_q)}{2 \sqrt{d}}>0.
\end{equation*}
Observe that, for any $x\in\mathbb{Z}^d$ and $n\in\mathbb{N}_+$,
\begin{equation}\label{eq:SubExponentialEstimate_2}
\begin{split}
\phi^{(n)}(x)&=\frac{1}{(2\pi)^d}\int_{\mathbb{T}_\phi^d}e^{-ix\cdot\xi}\hat\phi(\xi)^n\,d\xi\\
&=\sum_{q=1}^Q\frac{1}{(2\pi)^d}\int_{\mathcal{O}_q}e^{-ix\cdot\xi}\hat\phi(\xi)^nu_q(\xi)\,d\xi\\
&=\sum_{q=1}^Q\frac{e^{ix\cdot\xi_q}\hat\phi(\xi_q)^n}{n^{\mu_q}(2\pi)^d}\int_{\mathcal{U}_{q,n}}e^{-iy_n(x)\cdot\xi}f_{q,n}(\xi)u_{q,n}(\xi)\,d\xi\\
&=\sum_{q=1}^Q\frac{e^{ix\cdot\xi_q}\hat\phi(\xi_q)^n}{n^{\mu_q}(2\pi)^d}\mathcal{I}_{q,n}(x),
\end{split}
\end{equation}
where we have set $y_{q,n}(x)=n^{-E_q^*}(x-n\alpha_q)$, $\mathcal{U}_{q,n}=n^{E_q}(\mathcal{O}_q)-\xi_q$, defined
\begin{equation*}
u_{q,n}(\xi)=u_q(n^{-E_q}\xi),
\end{equation*}
and
\begin{equation*}
f_{q,n}(\xi)=(\hat\phi(\xi_q)^{-1}e^{-i\alpha\cdot n^{-E_q}\xi}\hat\phi(n^{-E_q}\xi+\xi_q))^n
\end{equation*}
for $\xi\in \mathcal{U}_{q,n}$, and put
\begin{equation*}
\mathcal{I}_{q,n}(x)=\int_{\mathcal{U}_{q,n}}e^{-iy_n(x)\cdot\xi}f_{q,n}(\xi)u_{q,n}(\xi)\,d\xi.
\end{equation*}
Of course, for each $n$ and $q$,  $f_{n,q}$ extends to an entire function on $\mathbb{C}^d$; we make no distinction between this function and $f_{n,q}$. We will soon obtain the desired estimates by integrating $\mathcal{I}_{n,q}$ by parts. For this purpose, it is useful to estimate the derivatives of $f_{q,n}$ and this is done in the lemma below. The idea behind the lemma's proof is to look at $f_{q,n}$ on small neighborhoods in $\mathbb{C}^d$ of $\zeta\in\supp( u_{q,n})\subseteq\mathbb{R}^d$. On such complex neighborhoods, Lemma \ref{lem:ComplexExponentialEstimate} gives tractable estimates for $f_{q,n}$ to which Cauchy's $d$-dimensional integral formula can be applied to estimate $D^\alpha f_{q,n}(\zeta)$.

\begin{lemma}\label{lem:ExponentialEstimate}
For each $q=1,2,\dots, Q$, there exist positive constants $C_q$ and $\epsilon_q$ such that, for each multi-index $\beta$,
\begin{equation*}
|D^{\beta}f_{q,n}(\zeta)|\leq C_q\frac{\beta!}{\delta^{|\beta|}}\exp(-\epsilon_q R_q(\zeta))
\end{equation*}
for all $n\in\mathbb{N}_+$ and $\zeta\in\supp (u_{q,n})$.
\end{lemma}

\begin{subproof}[Proof of Lemma \ref{lem:ExponentialEstimate}.]
Our choice of the open cover $\{\mathcal{O}_q\}$ guarantees that $|\hat\phi(\eta+\xi_q)|<1$ for all non-zero $\eta$ in the compact set $\overline{\mathcal{O}_q}-\xi_q$. An appeal to Lemma \ref{lem:ComplexExponentialEstimate} gives $\epsilon'_q,M'_q>0$ such that
\begin{equation}\label{eq:SubExponentialEstimate_3}
\begin{split}
|f_{q,n}(z)|&\leq \exp\big(-\epsilon'_q R_q\big(n^{-E_q}\eta\big)+M'_q R_q\big(n^{-E_q}\nu\big)\big)^n\\
&\leq\exp(-\epsilon'_q R_q(\eta)+M'_q R_q(\nu))
\end{split}
\end{equation}
for all $n\in\mathbb{N}_+$ and $z=\eta-i\nu\in\mathbb{C}^d$ for which $\eta\in\overline{\mathcal{U}_{q,n}}$.

We claim that there are constants $\epsilon_q,M_q>0$ for which
\begin{equation}\label{eq:SubExponentialEstimate_4}
-\epsilon_q'R_q(\eta)+M_q'R_q(\nu)\leq -\epsilon_q R_q(\zeta)+M_q
\end{equation}
for all $z=\eta-i\nu\in\mathbb{C}^d$ and $\zeta\in\mathbb{R}^d$ such that $|z_i-\zeta_i|=\delta$ for $i=1,2,\dots d$. Indeed, it is clear that $R_q(\nu)$ is bounded for all possible values of $\nu$. An appeal to Proposition \ref{prop:RSumEstimate} ensures that, there are $M'_q,\epsilon_q>0$ for which
\begin{equation*}
-\epsilon_q'R_q(\eta)=-\epsilon_q'R_q(\zeta+(\eta-\zeta))\leq -\epsilon_q R_q(\zeta)+M'_q
\end{equation*}
for all $\eta,\zeta\in\mathbb{R}^d$ provided $|\eta_i-\zeta_i|\leq|z_i-\zeta_i|=\delta$ for all $i=1,2,\dots d$. This proves the claim.

By combining \eqref{eq:SubExponentialEstimate_3} and \eqref{eq:SubExponentialEstimate_4}, we deduce that, for all $n\in\mathbb{N}_+$, $\zeta\in\mathbb{R}^d$, and $z=\eta-i\nu\in\mathbb{C}^d$ for which $\eta\in\overline{\mathcal{U}_{q,n}}$,
\begin{equation}\label{eq:SubExponentialEstimate_5}
|f_{q,n}(z)|\leq \exp(-\epsilon_q R_q(\zeta)+M_q)
\end{equation}
whenever $|z_i-\zeta_i|=\delta$ for all $i=1,2,\dots d$. Our aim is to combine Cauchy's $d$-dimensional integral formula,
\begin{equation}\label{eq:SubExponentialEstimate_6}
D^{\beta}f_{q,n}(\zeta)=\frac{\beta!}{(2\pi i)^d}\int_{C_1}\int_{C_2}\cdots\int_{C_d}\frac{f_{q,n}(z)\,dz_1dz_2\dots dz_d}{(z-\zeta)^{(\beta_1+1,\beta_2+1,\dots,\beta_d+1)}},
\end{equation}
with \eqref{eq:SubExponentialEstimate_5} to obtain our desired bound for $\zeta\in\supp (u_{q,n})$; here, $C_i=\{z:|z_i-\zeta_i|=\delta\}$ for $i=1,2,\dots,d$. To do this, we must verify that $z=\eta-i\nu$ is such that $\eta\in\overline{\mathcal{U}_{q,n}}$ whenever $|z_i-\zeta_i|=\delta$ for $i=1,2,\dots,d$. This is easy to see, for if $\zeta\in\supp (u_{q,n})$ and $z$ is such that $|z_i-\zeta_i|=\delta$ for $i=1,2,\dots,d$,
\begin{equation*}
|z-\zeta|=\sqrt{d}\delta<\mbox{dist}(\supp (u_q),\partial\mathcal{O}_q)\leq \mbox{dist}(\supp (u_{n,q}),\partial\mathcal{U}_{n,q})
\end{equation*}
for all $n\in\mathbb{N}_+$ (the distance only increases with $n$ because $\{t^{E_q}\}$ is contracting). Consequently, a combination of \eqref{eq:SubExponentialEstimate_5} and \eqref{eq:SubExponentialEstimate_6} shows that, for any multi-index $\beta$,
\begin{equation*}
|D^{\beta}f_{q,n}(\zeta)|\leq\frac{\beta!}{\delta^{|\beta|}}\exp(-\epsilon_q R_q(\zeta)+M_q)
\end{equation*}
for all $n\in\mathbb{N}_+$ and $\zeta\in\supp (u_{q,n})$ and thus the desired result holds.
\end{subproof}

We now finish the proof of Theorem \ref{thm:SubExponentialEstimate}. We assert that, for each $q=1,2,\dots, Q$ and multi-index $\beta$, there exists $C_{\beta}>0$ such that
\begin{equation}\label{eq:SubExponentialEstimate_7}
|y_{q,n}(x)^\beta\mathcal{I}_{q,n}(x)|\leq C_{\beta}
\end{equation}
for all $x\in\mathbb{Z}^d$ and $n\in\mathbb{N}_+$. By inspecting \eqref{eq:SubExponentialEstimate_2}, we see that the desired estimate, \eqref{eq:SubExponentialEstimate_1}, follow directly from \eqref{eq:SubExponentialEstimate_7} and so we prove \eqref{eq:SubExponentialEstimate_7}.

We have, for any multi-index $\beta$,
\begin{align*}
(iy_{q,n}(x))^{\beta}\mathcal{I}_{q,n}(x)&=\int_{\mathcal{U}_{q,n}}D_\xi^{\beta}(e^{-iy_n(x)\cdot\xi})f_{q,n}(\xi)u_{q,n}(\xi)\,d\xi\\
&=(-1)^{|\beta|}\int_{\mathcal{U}_{q,n}}e^{-iy_n(x)\cdot\xi}D^{\beta}(f_{q,n}(\xi)u_{q,n}(\xi))\,d\xi
\end{align*}
for all $n\in\mathbb{N}_+$ and $x\in\mathbb{Z}^d$ where we have integrated by parts and made explicit use of our partition of unity $\{u_q\}$ to ensure that all boundary terms vanished. To see the absence of boundary contributions, note that when $q\neq 1$, $u_{q,n}$ and its derivatives are identically zero on a neighborhood of $\partial\mathcal{U}_{q,n}$. When $q=1$, $\supp (u_{1,n})\cap \partial\mathcal{U}_{1,n}=\partial( n^E\mathbb{T}^d)$ and because $u_{1,n}\equiv 1$ on a neighborhood of $\partial(n^E\mathbb{T}_\phi^d)$, the periodicity of $f_{q,n}$ and its derivatives (which are directly inherited form the periodicity of $\hat\phi(\xi)$) ensure that the integral over the $\partial\mathcal{U}_{1,n}$ is zero. Consequently,
\begin{equation*}
|y_{q,n}(x))^{\beta}\mathcal{I}_{q,n}(x)|\leq\int_{\supp (u_{q,n})}\big|D^{\beta}(f_{q,n}(\xi)u_{q,n}(\xi))\big|\,d\xi
\end{equation*}
for, $q=1,2,\dots Q$, $n\in\mathbb{N}_+$ and $x\in\mathbb{Z}^d$. Once it is observed that derivatives of $u_{q,n}$ are well-behaved as $n$ increases, the estimate \eqref{eq:SubExponentialEstimate_7} follows immediately from Lemma \ref{lem:ExponentialEstimate}. The fact that $C_N$ is independent of $E_q\in\Exp(P_q)$ for $q=1,2,\dots,Q$ follows by a direct application of Proposition \ref{prop:SymPisCompact}.
\end{proof}

\section{Stability theory}\label{sec:Stability}

\noindent We now turn to the stability of convolution operators. In this brief section, we show that Theorem \ref{thm:StabilityIntro} is a consequence of of estimates of the preceding section. Let $\phi:\mathbb{Z}^d\rightarrow\mathbb{C}$ be finitely supported and define the operator $A_\phi$ on $L^p =L^p(\mathbb{R}^d)$ for $1\leq p\leq \infty$ by
\begin{equation}\label{eq:TPhiDefinition}
(A_\phi f)(x)=\sum_{y\in\mathbb{Z}^d}\phi(y)f(x-y).
\end{equation}
Such operators arise in the theory of finite difference schemes for partial differential equations in which they produce extremely accurate numerical approximations to solutions for initial value problems, e.g., \eqref{eq:HeatTypeEquation}. We encourage the reader to see \cite{Richtmyer1967} and \cite{Trefethen1996} for readable introductions to this theory; Thom\'{e}e's survey \cite{Thomee1969} is also an excellent reference. In this framework, the operator $A_\phi$ is known as an explicit constant-coefficient difference operator. General explicit difference operators are produced by allowing $\phi$ to depend on a real parameter $h>0$ which is usually the grid size of an associated spatial discretization for the initial value problem.\\

\noindent The operator $A_\phi$ is said to be \textit{stable} in $L^p$ if the collection of successive powers of $A_\phi$ is uniformly bounded on $L^p$, i.e., there is a positive constant $C$ for which
\begin{equation*}
\|A_\phi^nf\|_{L^p}\leq C\|f\|_{L_p}
\end{equation*}
for all $f\in L^p$ and $n\in\mathbb{N}_+$; this property has profound consequences for difference schemes of partial differential equations as we discussed in the introduction. For example, the Lax equivalence theorem states that a consistent approximate difference scheme for \eqref{eq:HeatTypeEquation} is stable in $L^p$ if and only if the difference scheme converges to the true solution \eqref{eq:SemigroupRepresentation} \cite{Thomee1969,Trefethen1996}. In the $L^2$ setting, checking stability is straightforward. Using the Fourier transform, one finds that $A_\phi$ is stable in $L^2$ if and only if $\sup_{\xi}|\hat\phi(\xi)|\leq 1$; this is a special case of the von Neumann condition \cite{Thomee1965}. When $p\neq 2$, the question of stability for $A_\phi$ is more subtle. It follows directly from the definition \eqref{eq:TPhiDefinition} that $A_\phi^{n}=A_{\phi^{(n)}}$ for all $n\in\mathbb{N}_+$ and so by Minkowski's inequality we see that
\begin{equation}\label{eq:OperatorBoundsforTPhi}
\|A_{\phi}^nf\|_{L^p}=\|A_{\phi^{(n)}}\|_{L^p}\leq\|\phi^{(n)}\|_1\|f\|_{L^p}
\end{equation}
for all $f\in L^p$ and $n\in\mathbb{N}_+$. This allows us to formulate a sufficient condition for stability in $L^p$ for $1\leq p\leq \infty$ in terms of the convolution powers of $\phi$ (which is consistent with Question \eqref{ques:Stability} of Section \ref{sec:Introduction}) as follows: $A_\phi$ is stable in $L^p$ whenever there is a positive constant $C$ for which
\begin{equation}\label{eq:StabilityCondition}
\|\phi^{(n)}\|_1=\sum_{x\in\mathbb{Z}^d}|\phi^{(n)}(x)|\leq C
\end{equation}
for all $n\in\mathbb{N}_+$. The condition \eqref{eq:StabilityCondition} is also necessary when $p=\infty$ and so it is called the condition of max-norm stability. Originally investigated by John \cite{John1952} and Strang \cite{Strang1962}, this theory for difference schemes has been further developed by many authors, see for example \cite{Thomee1969,Thomee1965, Fedoryuk1967,Serdyukova1967}. In one dimension $(d=1)$, the question of stability in the max-norm was completely sorted out by Thom\'{e}e \cite{Thomee1965}. Thom\'{e}e showed that a sufficient condition of Strang was also necessary; this is summarized in the following theorem.
\begin{theorem}[Thom\'{e}e 1965]\label{thm:Thomee}
The operator $A_\phi$ is stable in $L^\infty(\mathbb{R})$ if and only if one of the following conditions is satisfied:
\begin{enumerate}[(a)]
 \item $\hat{\phi}(\xi)=ce^{i x\xi}$ for some $x\in \mathbb{Z}$ and $|c|=1$.
\item\label{con:ThomeeCondition} $|\hat\phi(\xi)|<1$ except for at most a finite number of points $\xi_1,\xi_2,\dots,\xi_Q$ in $\mathbb{T}$ where $|\hat\phi(\xi)|=1$. For $q=1,2,\dots Q$, there are constants $\alpha_q,\gamma_q,m_q$, where $\alpha_q\in\mathbb{R}$, $\Re\gamma_q>0$ and where $m_q\in\mathbb{N}_+$, such that
\begin{equation}\label{eq:ThomeeCondition}
\hat\phi(\xi+\xi_q)=\hat\phi(\xi_q)\exp(i\alpha_q\xi-\gamma_k\xi^{2m_q}+o(\xi^{2m_q}))
\end{equation}
as $\xi\rightarrow 0$.
\end{enumerate}
\end{theorem}

\noindent Thom\'{e}e's characterization makes use of the fact that the level sets of non-constant holomorphic functions on $\mathbb{C}$ have no accumulation points -- a fact that breaks down in all other dimensions, e.g., $f(z)=f(z_1,z_2)=\cos(z_1-z_2)$. When $\phi:\mathbb{Z}\rightarrow\mathbb{C}$ is finitely supported and such that $\sup_\xi|\hat\phi(\xi)|=1$, the reader should note that the condition (\ref{con:ThomeeCondition}) of Theorem \ref{thm:Thomee} is equivalent to the hypotheses of Theorem \ref{thm:SubExponentialEstimate} for, in one dimension, every positive homogeneous polynomial is necessarily of the form $P(\xi)=\gamma\xi^{2m}$ where $\Re\gamma>0$ and $m\in\mathbb{N}_+$. In $\mathbb{Z}^d$, we have the following result.

\begin{corollary}\label{cor:Stability}
Let $\phi:\mathbb{Z}^d\rightarrow\mathbb{C}$ satisfy the hypotheses of Theorem \ref{thm:SubExponentialEstimate} and define $A_\phi$ by \eqref{eq:TPhiDefinition}. Then $A_\phi$ is stable in $L^\infty$ and hence stable in $L^p(\mathbb{R}^d)$ for all $1\leq p\leq\infty$.
\end{corollary}

\begin{proof}
An application of Theorem \ref{thm:SubExponentialEstimate} with $N\geq d+1$ yields the uniform estimate \eqref{eq:StabilityCondition} after summing over $x\in\mathbb{Z}^d$.
\end{proof}

\begin{proof}[Proof of Theorem \ref{thm:StabilityIntro}]
This is simply Corollary \ref{cor:Stability} translated into the language of Section \ref{sec:Introduction}.
\end{proof}

\noindent In \cite{Thomee1965}, Thom\'{e}e also showed that when $\sup|\hat\phi|=1$ but the leading non-linear term in the expansion \eqref{eq:ThomeeCondition} was purely imaginary, the corresponding difference scheme was unstable. As was discussed in \cite{Diaconis2014} and \cite{Randles2015}, such expansions give rise to local limit theorems in which the corresponding attractors are bounded but not in $L^2$ and hence not in $S(\mathbb{R})$; for instance, the Airy function. In the spirit of \cite{Thomee1965}, M. V. Fedoryuk explored stability and instability in higher dimensions \cite{Fedoryuk1967}. Fedoryuk's affirmative result assumes that, for $\xi_0\in\Omega(\phi)$, the leading quadratic polynomial in the expansion for $\Gamma_{\xi_0}$ has positive definite real part. Because any quadratic polynomial $P$ with positive definite real part is positive homogeneous ($2^{-1}I\in\Exp(P)$), Corollary \ref{cor:Stability} (equivalently, Theorem \ref{thm:StabilityIntro}) extends the affirmative result of \cite{Fedoryuk1967}.

\section{Examples}\label{sec:examples}

\subsection{A well-behaved real valued function on $\mathbb{Z}^2$}\label{subsec:ex1}

\noindent This example illustrates the case in which $\hat\phi$ is maximized only at $0$ which is of positive homogeneous type for $\hat\phi$ with corresponding $P$. In this case, the local limit theorem for $\phi$ yields one attractor with no oscillatory prefactor. The positive homogeneous polynomial $P$ is a semi-elliptic polynomial of the form \eqref{eq:SemiEllipticPolynomial} and the corresponding attractor exhibits small oscillations and decays anisotropically.\\

\noindent Consider $\phi:\mathbb{Z}^2\rightarrow\mathbb{R}$ defined by $\phi=(\phi_1+\phi_2)/512$, where
\begin{equation*}
\phi_1(x,y)=
\begin{cases}
326 & (x,y)=(0,0)\\
20 & (x,y)=(\pm 2,0)\\
1 & (x,y)=(\pm 4,0)\\
64 & (x,y)=(0,\pm 1)\\
-16 & (x,y)=(0,\pm 2)\\
0 & \mbox{otherwise}
\end{cases}
\hspace{.2cm}\mbox{ and }\hspace{.2cm}
\phi_2(x,y)=
\begin{cases}
76 & (x,y)=(1,0)\\
52 & (x,y)=(-1,0)\\
\mp 4 & (x,y)=(\pm 3,0)\\
\mp 6 & (x,y)=(\pm 1,1)\\
\mp 6 & (x,y)=(\pm 1,-1)\\
\pm 2 & (x,y)=(\pm 3,1)\\
\pm 2 & (x,y)=(\pm 3,-1)\\
0 & \mbox{otherwise}.
\end{cases}
\end{equation*}
The graphs of $\phi^{(n)}$ on the domain $[-50,50]\times[-50,50]$ for $n=100$, $n=1,000$ and $n=10,000$ are shown in Figure \ref{fig:ex1triple}; in particular, the figure illustrates the decay in $\|\phi^{(n)}\|_{\infty}$. Figure \ref{fig:ex1quad} depicts $\phi^{(n)}(x,y)$ when $n=10,000$ from various angles and clearly illustrates its non-Gaussian anisotropic nature.\\

\begin{figure}[h!]
\begin{center}
\resizebox{\textwidth}{!}{
	    \begin{subfigure}[5cm]{0.5\textwidth}
		\includegraphics[width=\textwidth]{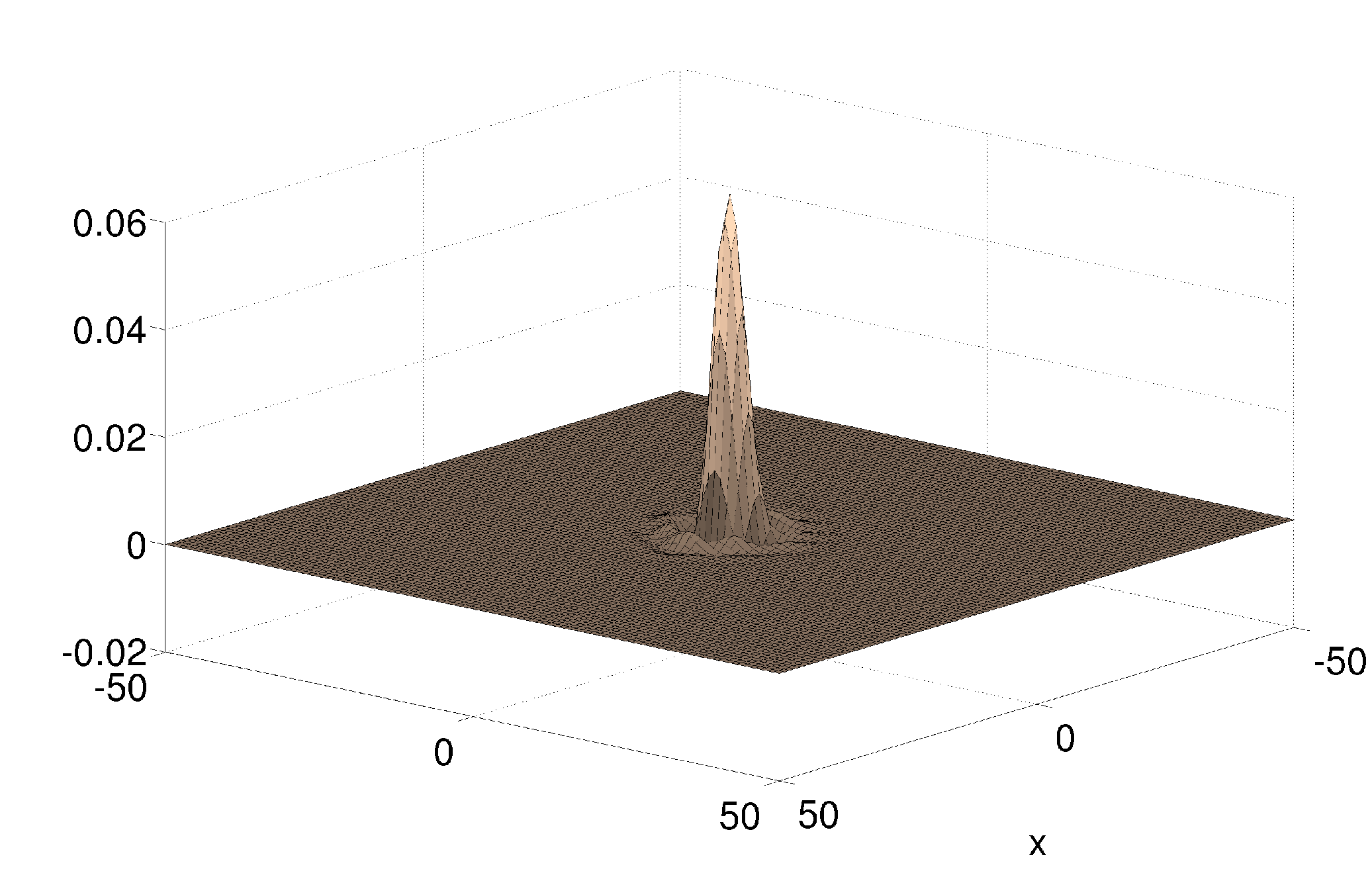}
		\caption{$n=100$}
	    \end{subfigure}
	    \begin{subfigure}[5cm]{0.5\textwidth}
		\includegraphics[width=\textwidth]{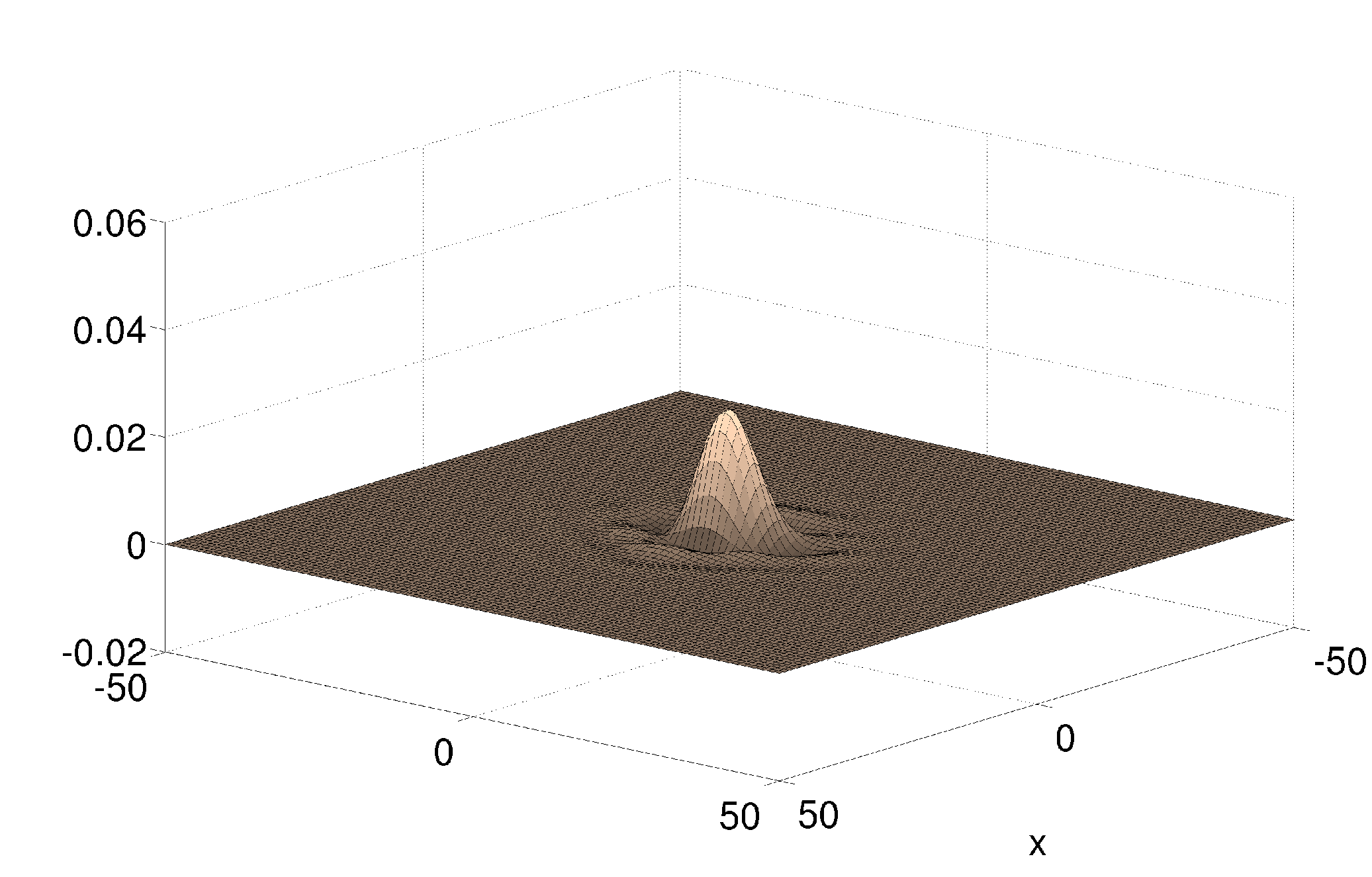}
		\caption{$n=1,000$}
	    \end{subfigure}}	
	    \begin{subfigure}[5cm]{0.5\textwidth}
		\includegraphics[width=\textwidth]{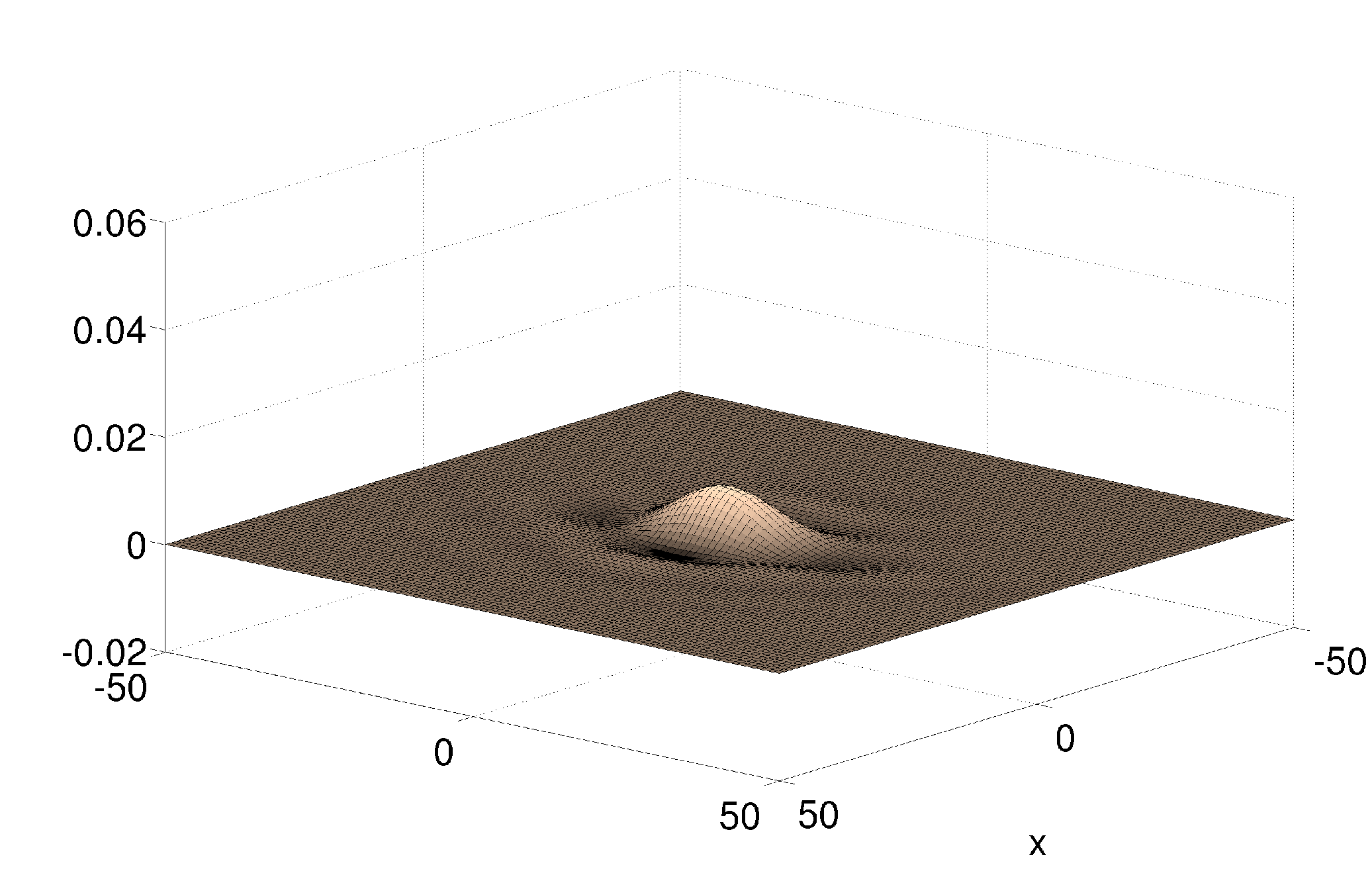}
		\caption{$n=10,000$}
	    \end{subfigure}
\caption{The graphs of $\phi^{(n)}$ for $n=100,$ $n=1,000$ and $n=10,000$.}
\label{fig:ex1triple}
\end{center}
\end{figure}

\begin{figure}[h!]
\begin{center}
\resizebox{\textwidth}{!}{
	    \begin{subfigure}[5cm]{0.5\textwidth}
		\includegraphics[width=\textwidth]{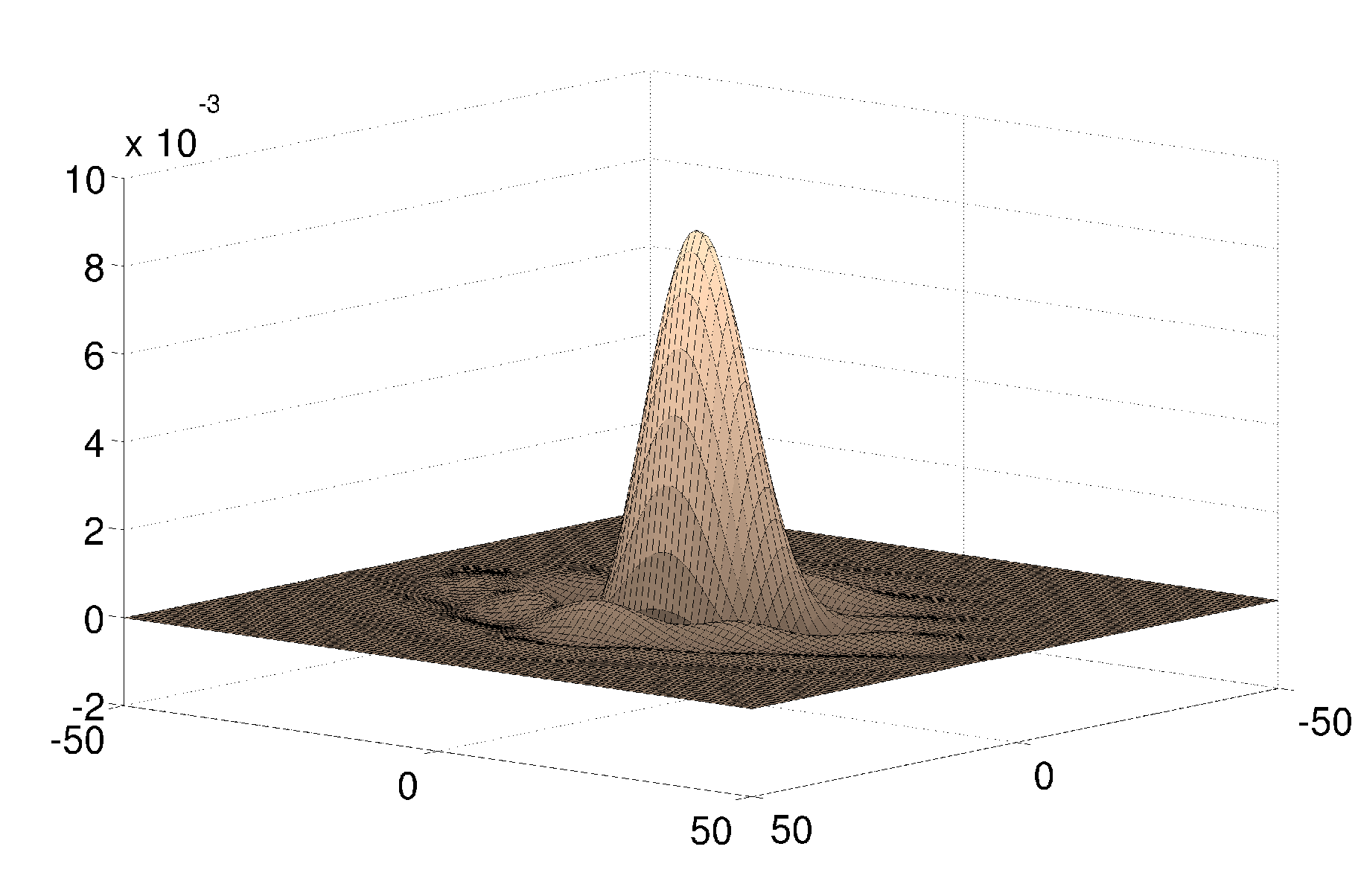}
	    \end{subfigure}
	    \begin{subfigure}[5cm]{0.5\textwidth}
		\includegraphics[width=\textwidth]{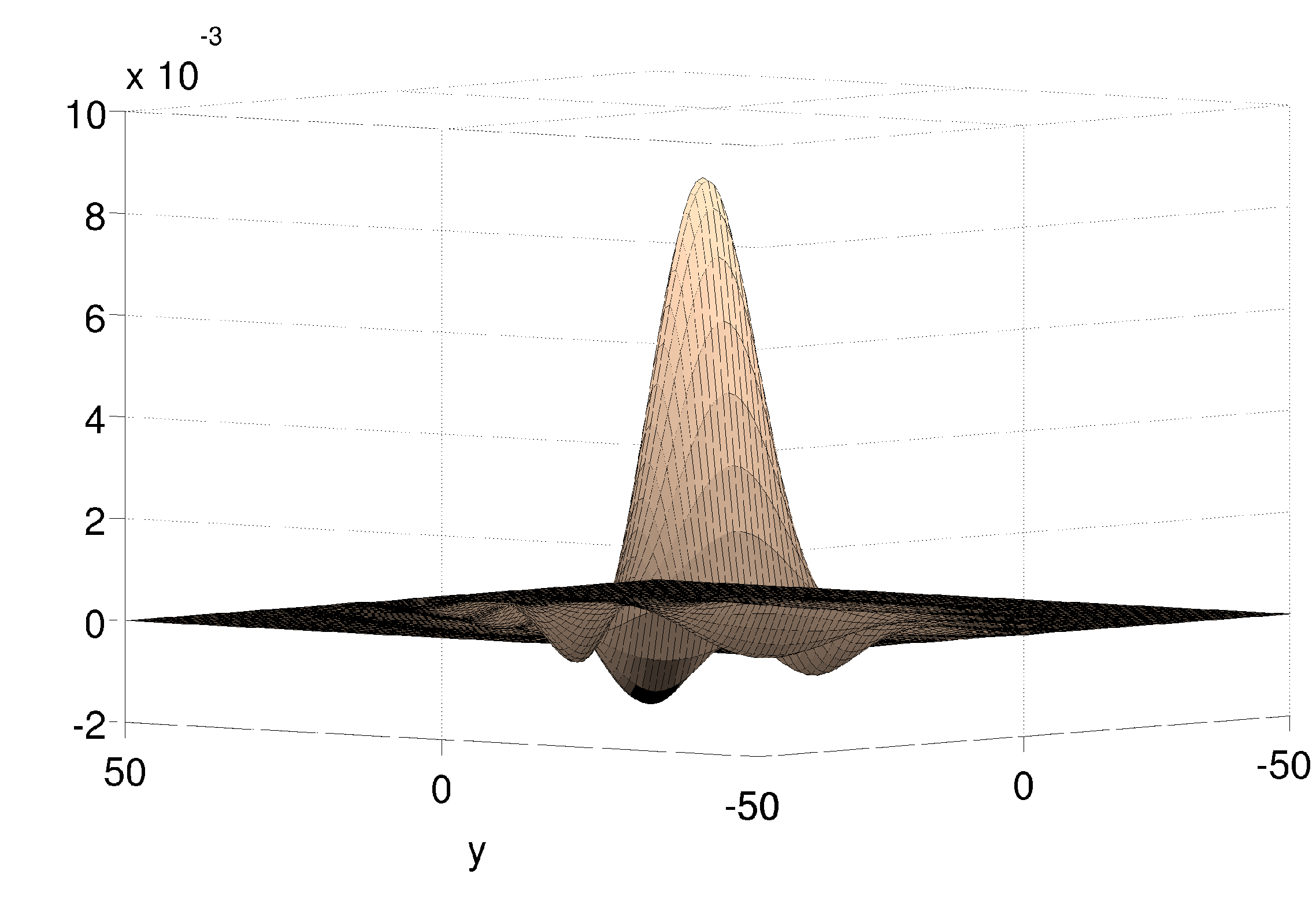}
	    \end{subfigure}}
\resizebox{\textwidth}{!}{	
	    \begin{subfigure}[5cm]{0.5\textwidth}
		\includegraphics[width=\textwidth]{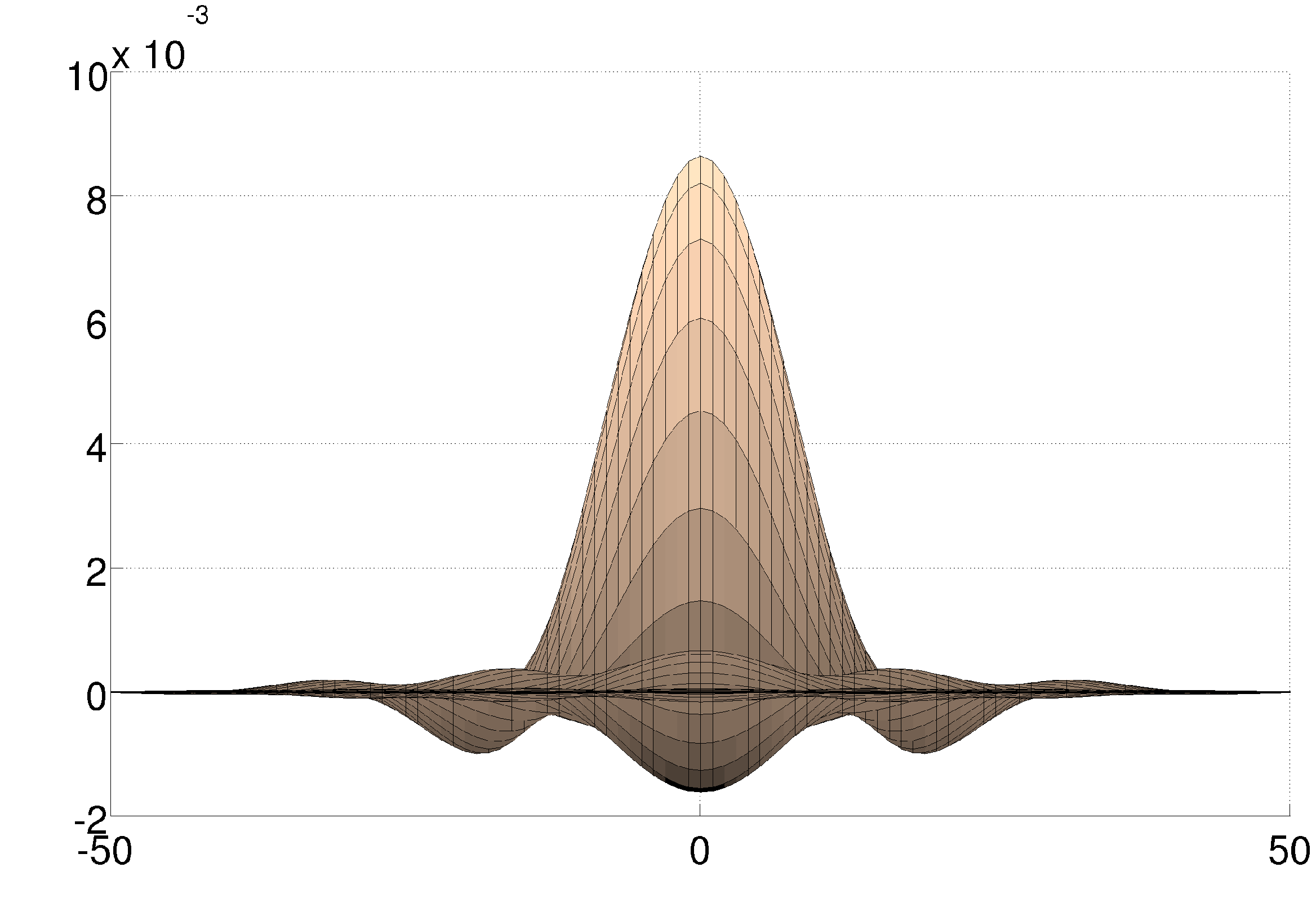}
	    \end{subfigure}
	    \begin{subfigure}[5cm]{0.5\textwidth}
		\includegraphics[width=\textwidth]{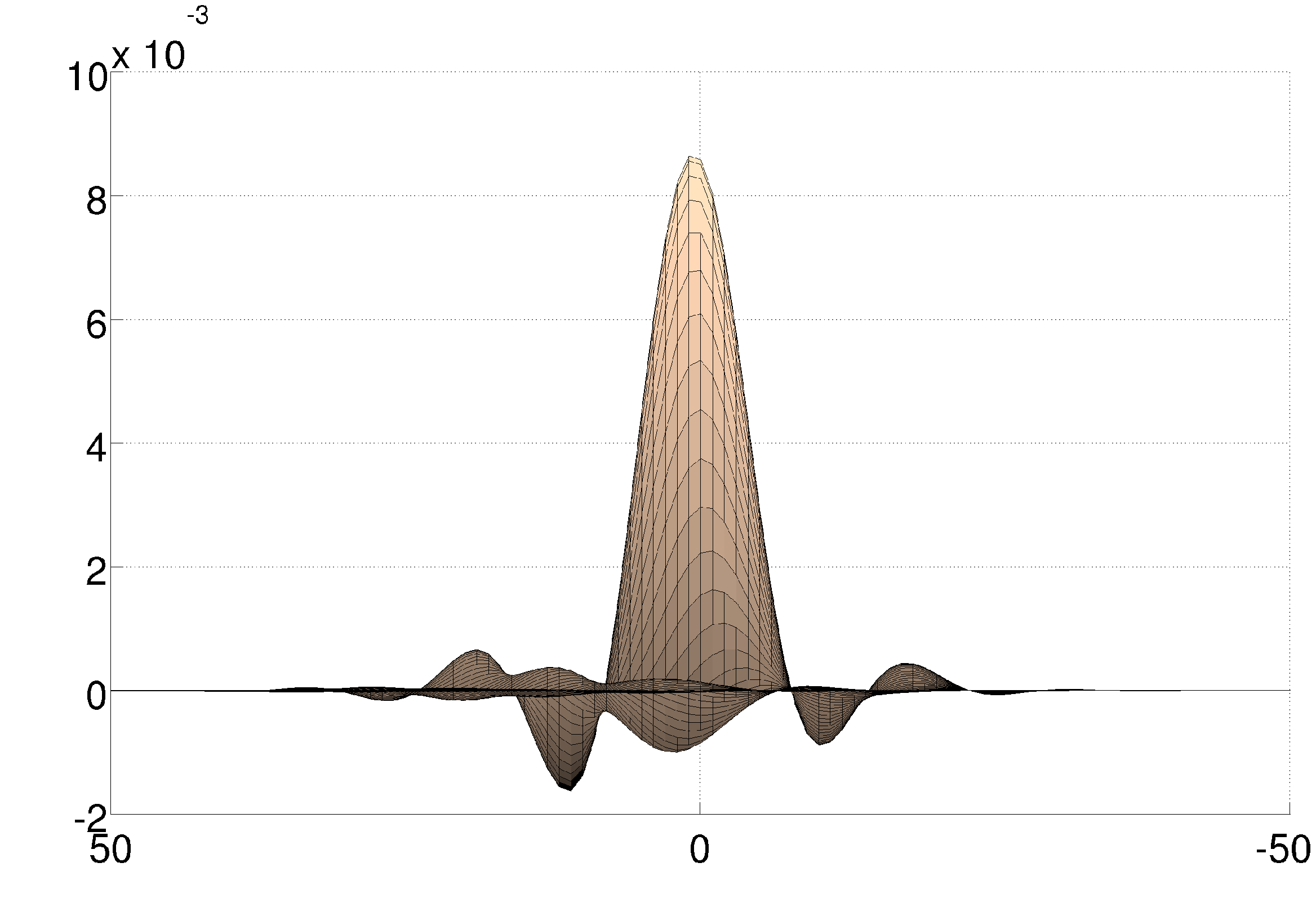}
	    \end{subfigure}}
\caption{The graph of $\phi^{(n)}$ for $n=10,000$}
\label{fig:ex1quad}
\end{center}
\end{figure}

\begin{figure}[h!]
\begin{center}
\resizebox{\textwidth}{!}{
	    \begin{subfigure}[5cm]{0.5\textwidth}
		\includegraphics[width=\textwidth]{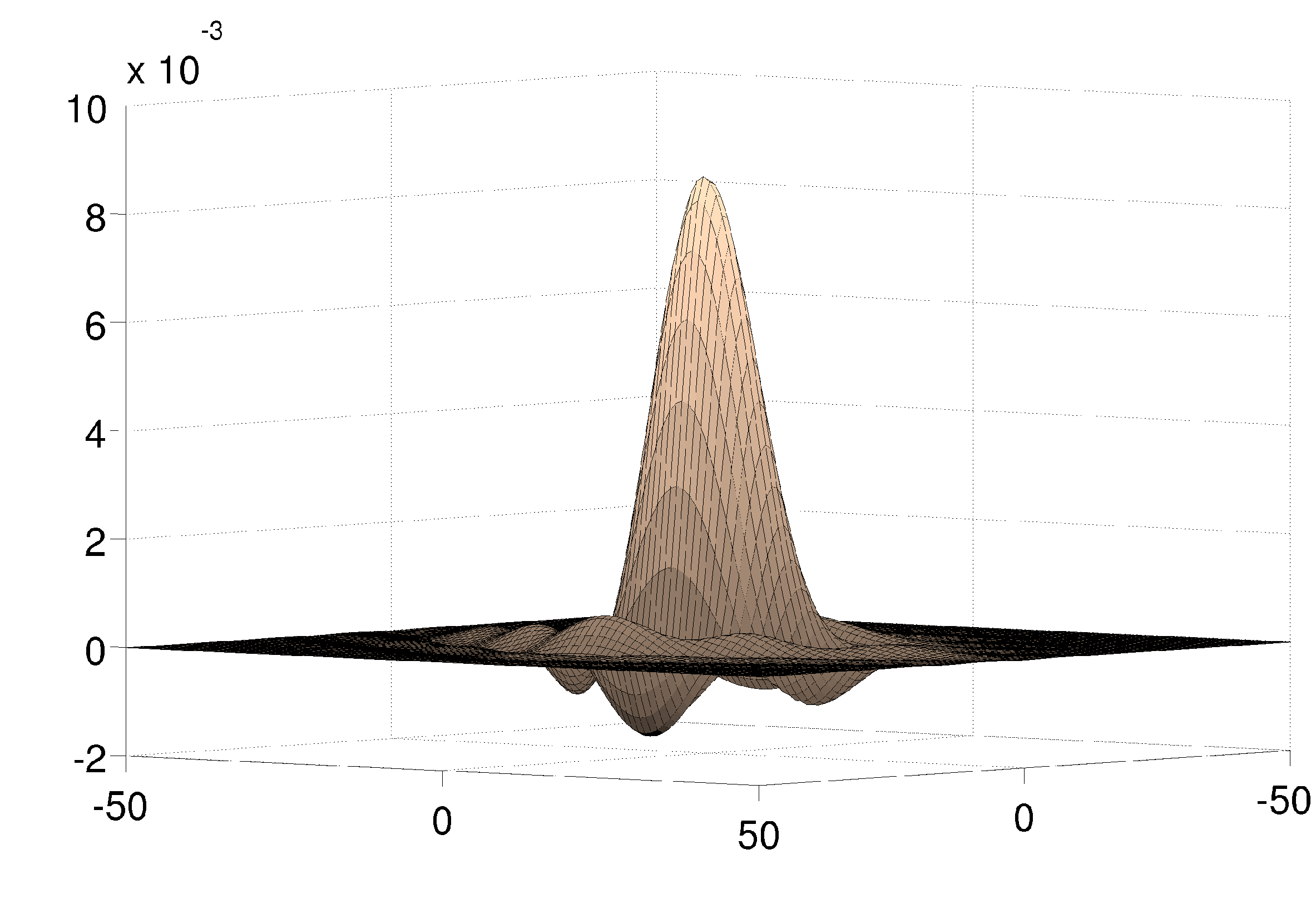}
		\caption{$\phi^{(n)}$ for $n=10,000$}
		\label{fig:ex5phi100_1}
	    \end{subfigure}
	    \begin{subfigure}[5cm]{0.5\textwidth}
		\includegraphics[width=\textwidth]{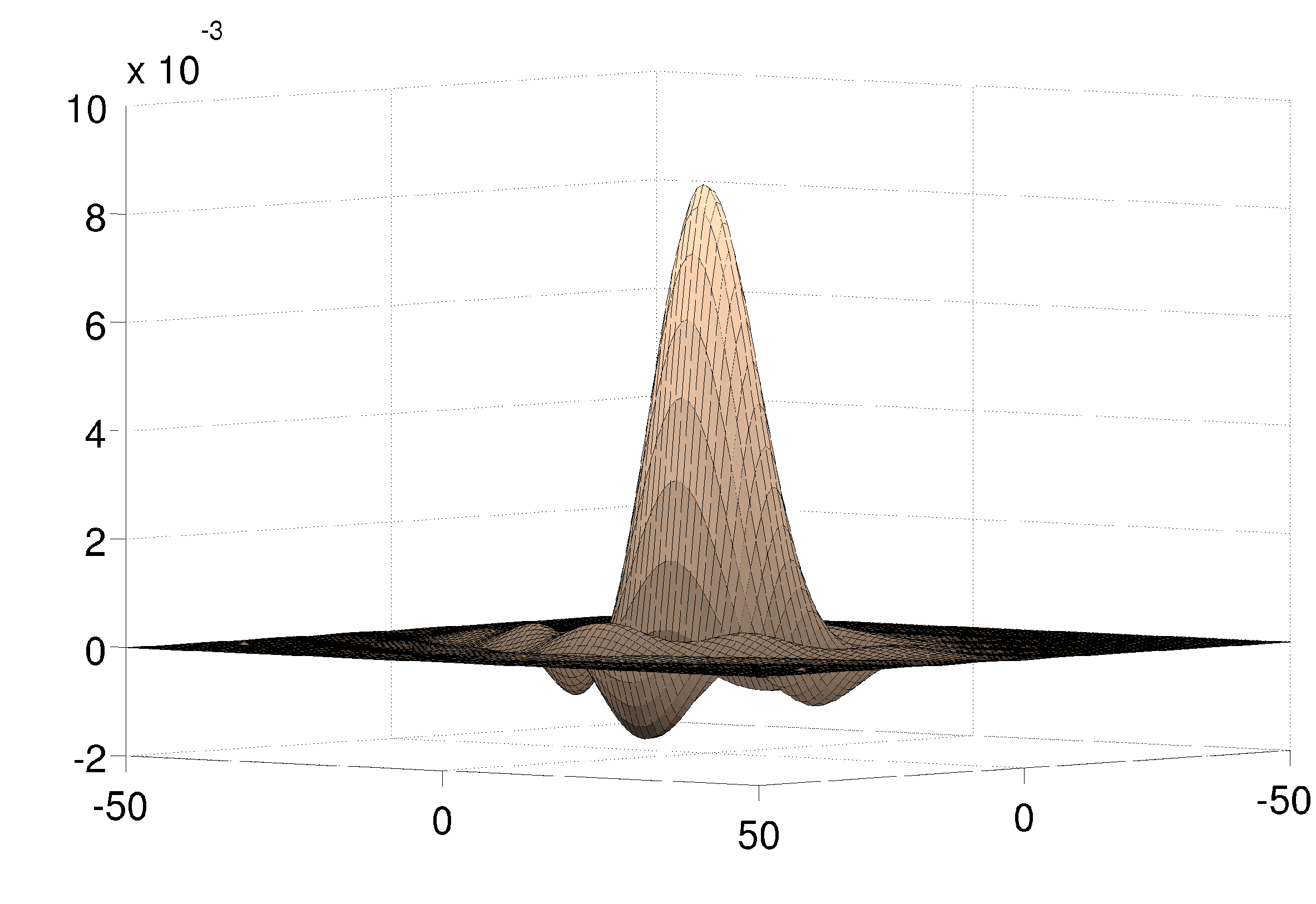}
		\caption{$H_{P_{\xi_1}}^n$ for $n=10,000$}
		\label{fig:ex5HP100_1}
	    \end{subfigure}}
\caption{The graphs of $\phi^{(n)}$ and $H_P^n$ for $n=10,000$.}
\label{fig:ex1compare}
\end{center}
\end{figure}

\noindent Given that $\phi$ is supported on $21$ points, it is clear that $\phi\in\mathcal{S}_2$. An easy computation shows that $\sup|\hat\phi(\xi)|=1$ and this supremum is only attained at $\xi=(\eta,\zeta)=(0,0)$, where $\phi(0,0)=1$, and hence $\Omega(\phi)=\{(0,0)\}$. Expanding the logarithm of $\phi(\eta,\zeta)/\phi(0,0)$ about $(0,0)$ we find that, as $(\eta,\zeta)\rightarrow (0,0)$,
\begin{equation*}
\Gamma(\eta,\zeta)=-\frac{1}{64}\left(\eta^6+2\zeta^4-2i\eta^3\zeta^2\right)+O(|\eta^7|+|\zeta^5|+|\eta^3\zeta^4|+|\eta^5\zeta^2|+|\eta^6\zeta^5|).
\end{equation*}
It is easy to see that the polynomial which leads the expansion,
\begin{equation*}
P(\eta,\zeta)=\frac{1}{64}\left(\eta^6+2\zeta^4-2i\eta^3\zeta^2\right),
\end{equation*}
has positive definite real part,
\begin{equation*}
R(\eta,\zeta)=\Re P(\eta,\zeta)=\frac{1}{64}\left(\eta^6+2\zeta^4\right).
\end{equation*}
Moreover
\begin{equation*}
P(t^E(\eta,\zeta))=P(t^{1/6}\eta,t^{1/4}\zeta)=tP(\eta,\zeta) \hspace{.5cm}\mbox{ with }\hspace{.5cm}E=
\begin{pmatrix}
   \frac{1}{6} & 0\\
  0 & \frac{1}{4}
  \end{pmatrix}\in\Exp(P)
\end{equation*}
for all $t>0$ and $(\eta,\zeta)\in\mathbb{R}^2$ and therefore $P$ is a positive homogeneous polynomial (it is also semi-elliptic). Further, we can rewrite the error to see that
\begin{equation*}
\Gamma(\eta,\zeta)=-P(\eta,\zeta)+\Upsilon(\eta,\zeta)
\end{equation*}
where $\Upsilon(\eta,\zeta)=o(R(\eta,\zeta))$ as $(\eta,\zeta)\rightarrow (0,0)$ and so we conclude that $(0,0)$ is of positive homogeneous type for $\hat\phi$ with corresponding $\alpha=(0,0)\in\mathbb{R}^2$ and positive homogeneous polynomial $P$. Clearly, $\mu_\phi=\mu_P=\tr E=5/12$ and so Theorem \ref{thm:SubDecay} gives positive constants $C$ and $C'$ for which
\begin{equation*}
C'n^{-5/12}\leq \|\phi^{(n)}\|_{\infty}\leq C n^{-5/12}
\end{equation*}
for all $n\in\mathbb{N}_+$. An appeal to Theorem \ref{thm:MainLocalLimit} shows that
\begin{equation}\label{eq:LocalLimitExample1}
\phi^{(n)}(x,y)=H_P^n(x,y)+o(n^{-5/12})
\end{equation}
uniformly for $(x,y)\in\mathbb{Z}^2$ where,
\begin{equation*}
H_P^n(x,y)=\frac{1}{(2\pi)^2}\int_{\mathbb{R}^2}e^{-i(x,y)\cdot(\xi_1,\xi_2)-nP(\xi_1,\xi_2)}\,d\xi_1\,d\xi_2=\frac{1}{n^{5/12}}H_P(n^{-1/6}x,n^{-1/4}y)
\end{equation*}
for $n\in\mathbb{N}_+$ and $(x,y)\in\mathbb{R}^2$. The local limit \eqref{eq:LocalLimitExample1} is illustrated in Figure \ref{fig:ex1compare} when $n=10,000$. We also make an appeal to Theorem \ref{thm:ExponentialEstimate} to deduce pointwise estimates for $\phi^{(n)}$ (in fact, all results of Section \ref{sec:PointwiseBounds} are valid for this $\phi$). Upon noting that
\begin{equation*}
R^{\#}(x,y)=\frac{5}{3^{6/5}}x^{6/5}+\left(1-\frac{1}{2^5}\right)y^{4/3}
\end{equation*}
for $(x,y)\in\mathbb{R}^2$, the theorem gives positive constants $C$ and $M$ for which
\begin{equation*}
|\phi^{(n)}(x,y)|\leq \frac{C}{n^{5/12}}\exp\left(-nM\left(\left(\frac{x}{n}\right)^{6/5}+\left(\frac{y}{n}\right)^{4/3}\right)\right)
\end{equation*}
for all $n\in\mathbb{N}_+$ and $(x,y)\in\mathbb{Z}^2$.

\subsection{Two drifting packets}\label{subsec:ex2}

\noindent In this example, we study a complex valued function on $\mathbb{Z}^2$ whose convolution powers $\phi^{(n)}$ exhibit two \textit{packets} which drift apart as $n$ increases. This behavior is easily described by applying Theorem \ref{thm:MainLocalLimit} in which two distinct $\alpha$'s appear.\\

\begin{figure}[h!]
\begin{center}
\resizebox{\textwidth}{!}{
	    \begin{subfigure}[5cm]{0.5\textwidth}
		\includegraphics[width=\textwidth]{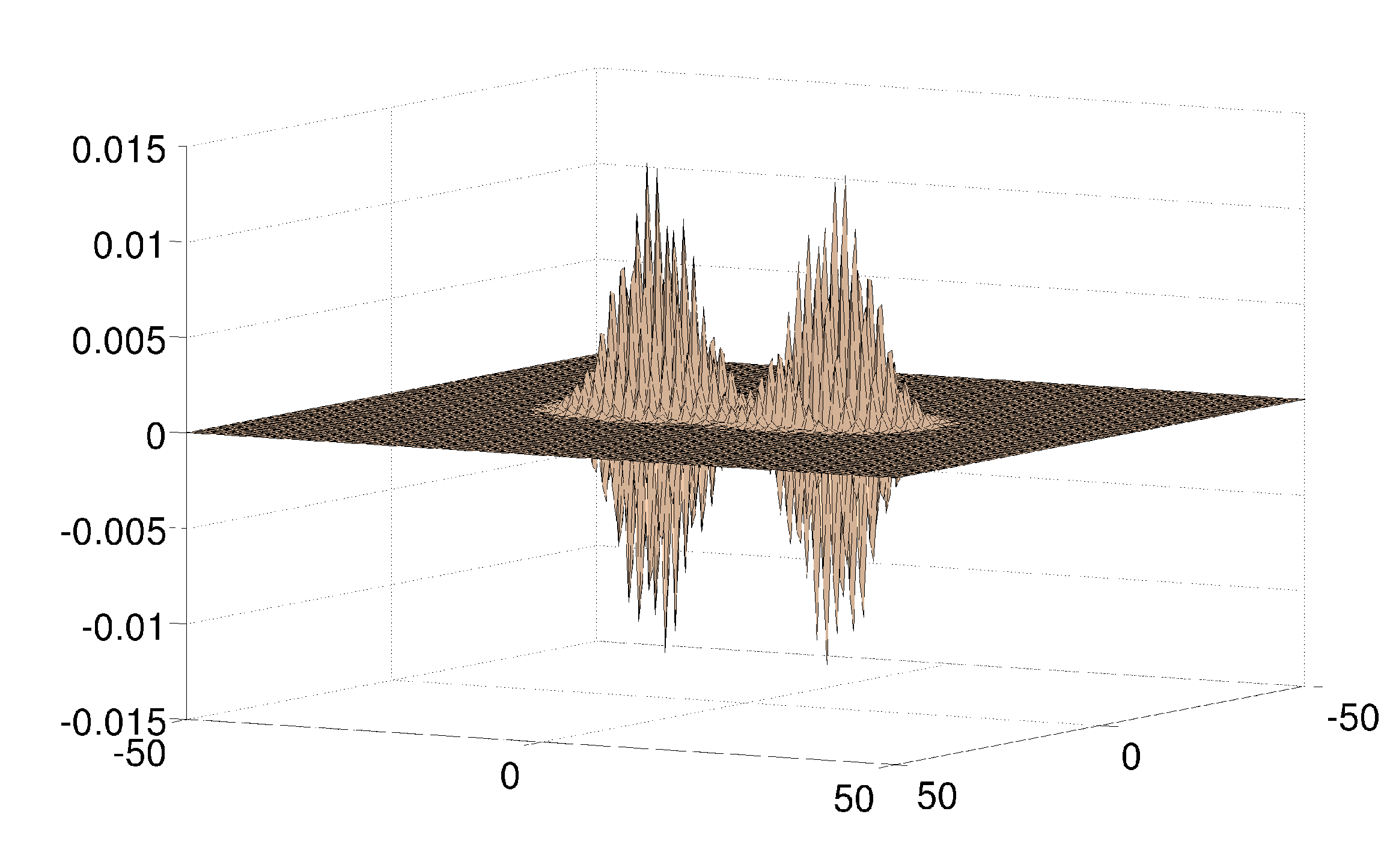}
		\caption{$\Re(\phi^{(n)})$ for $n=30$}
		\label{fig:ex2phi30}
	    \end{subfigure}
	    \begin{subfigure}[5cm]{0.5\textwidth}
		\includegraphics[width=\textwidth]{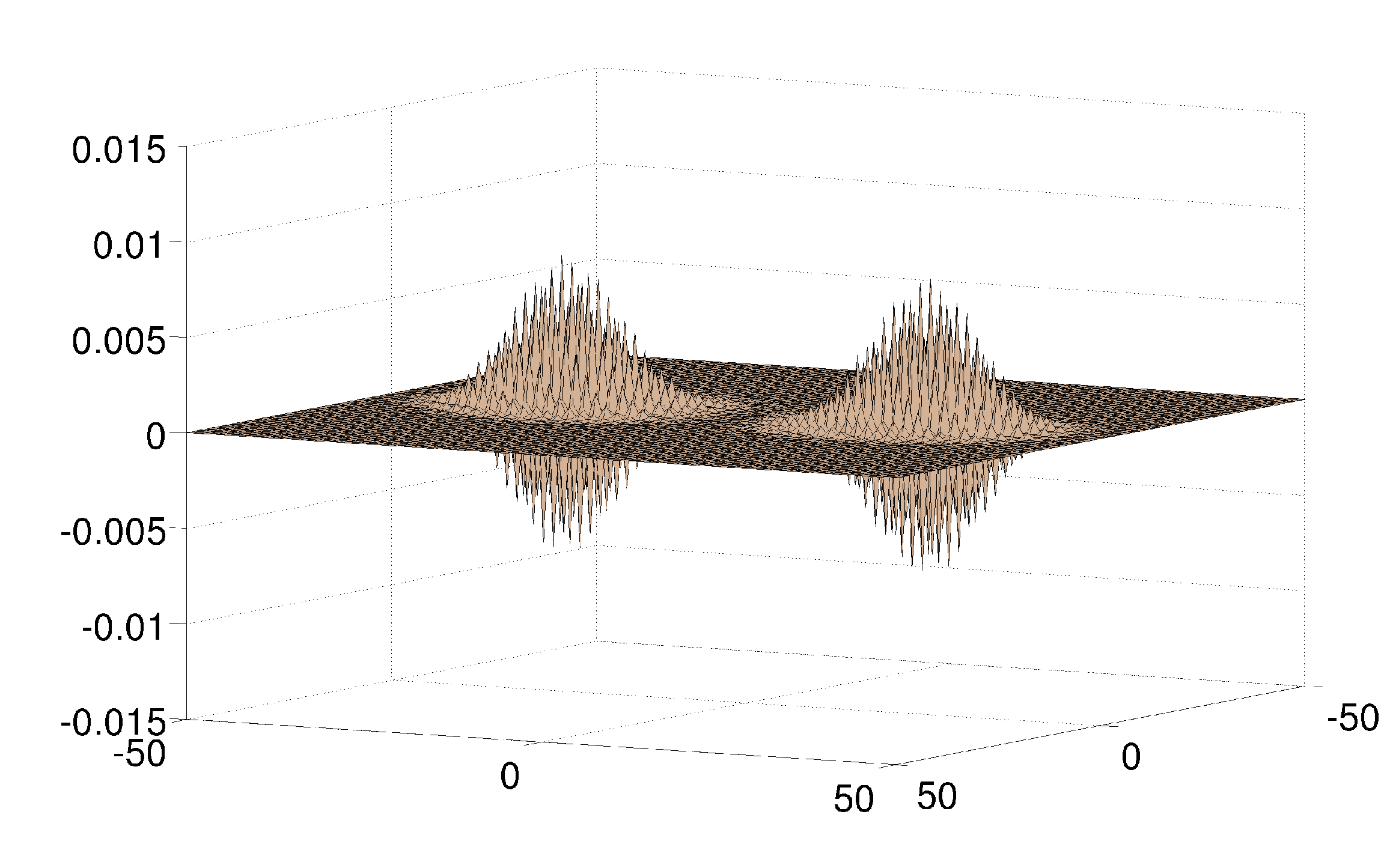}
		\caption{$\Re(\phi^{(n)})$ for $n=60$}
		\label{fig:ex2phi60}
	    \end{subfigure}}
\resizebox{\textwidth}{!}{	
	    \begin{subfigure}[5cm]{0.5\textwidth}
		\includegraphics[width=\textwidth]{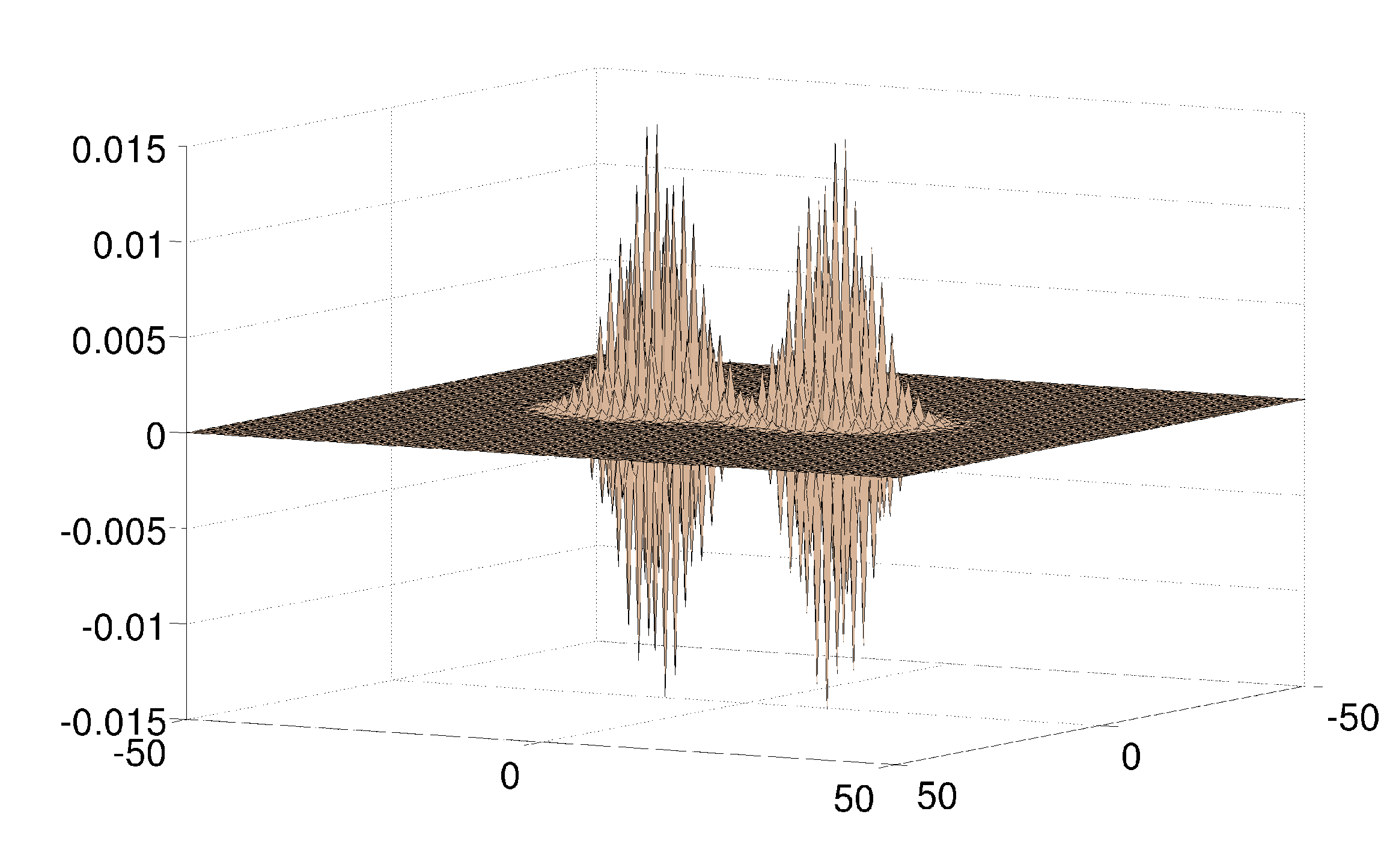}
		\caption{$Re(f_n)$ for $n=30$}
		\label{fig:ex2f30}
	    \end{subfigure}
	    \begin{subfigure}[5cm]{0.5\textwidth}
		\includegraphics[width=\textwidth]{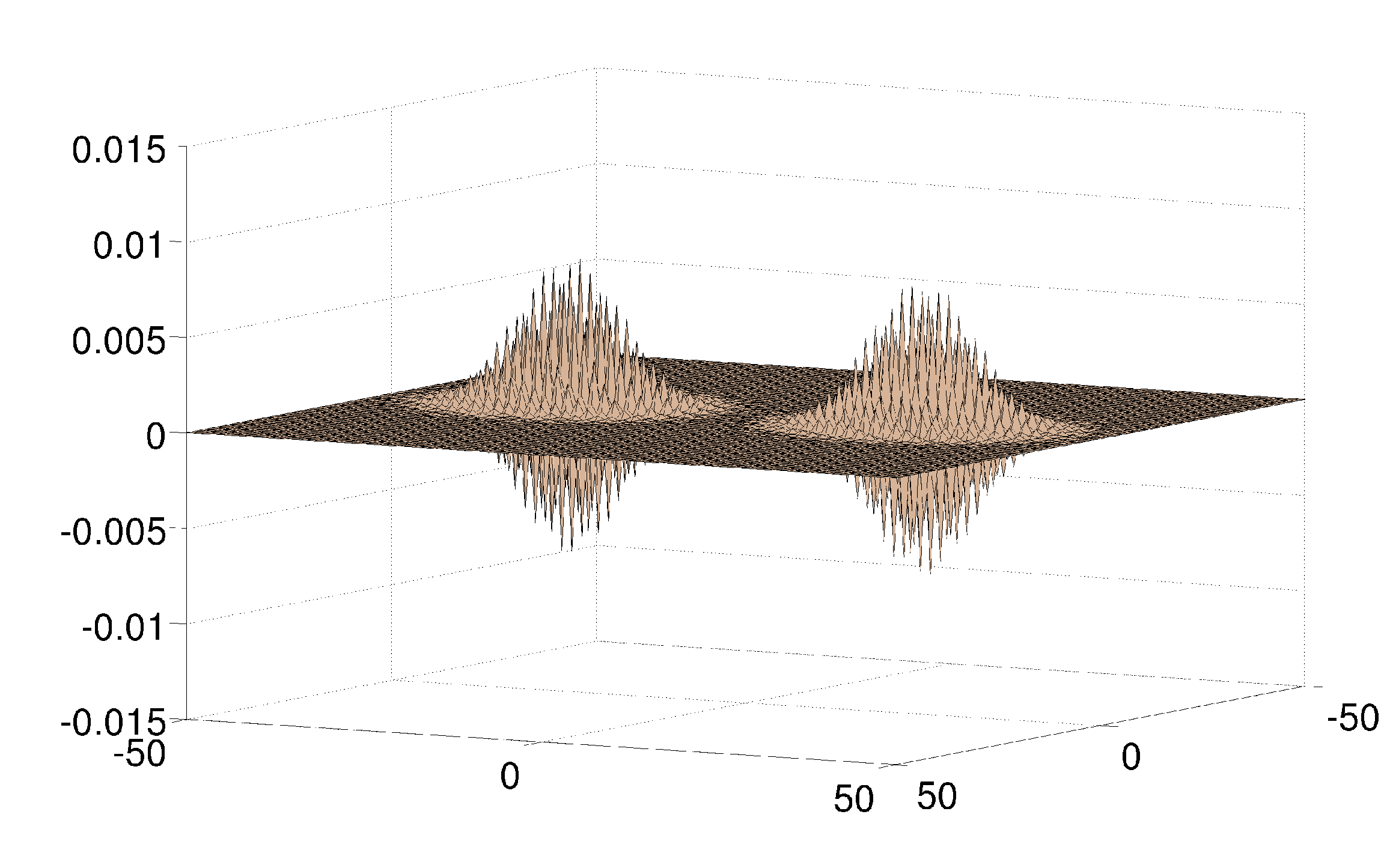}
		\caption{$\Re(f_n)$ for $n=60$}
		\label{fig:ex2f60}
	    \end{subfigure}}
\caption{The graphs of $\Re(\phi^{(n)})$ and $\Re(f_n)$ for $n=30,60$.}
\label{fig:ex2}
\end{center}
\end{figure}

\noindent Consider $\phi:\mathbb{Z}^2\rightarrow\mathbb{C}$ defined by
\begin{equation*}
\phi(x,y)=\begin{cases}
           \frac{1+i}{4a} & (x,y)=(-1,\pm 1)\\
           -\frac{1+i}{4a} & (x,y)=(1,\pm 1)\\
           \pm\frac{1}{\sqrt{2}a} & (x,y)=(0,\pm 1)\\
           0 & \mbox{otherwise}.
          \end{cases}
\end{equation*}
where $a=\sqrt{2+\sqrt{2}}$. The graphs of $\Re(\phi^{(n)})$ for $n=30,60$ are shown in Figures \ref{fig:ex2phi30} and \ref{fig:ex2phi60} respectively; observe the appearance of the drifting packets.\\

\noindent In computing the Fourier transform of $\hat\phi$, we find that $\sup|\hat\phi|=1$ and
\begin{equation*}
\Omega(\phi)=\{\xi_1,\xi_2,\xi_3,\xi_4\}=\{(\pi/2,3\pi/4),(\pi/2,-\pi/4),(-\pi/2,-3\pi/4),(-\pi/2,\pi/4)\},
\end{equation*}
where
\begin{equation*}
\hat\phi(\xi_1)=\hat\phi(\xi_3)=(i)^{5/4}\hspace{.5cm}\mbox{ and }\hspace{.5cm}\hat\phi(\xi_2)=\hat\phi(\xi_4)=-(i)^{5/4}.
\end{equation*}
Set $\gamma=\sqrt{2}-1$ and
\begin{equation*}
P(\eta,\zeta)=\frac{1+i\gamma}{4}\eta^2+\gamma \zeta^2.
\end{equation*}
As in the previous example, we expand the logarithm of $\hat\phi$ near $\xi_j$ for $j=1,2,3,4$. We find that each element of $\Omega(\phi)$ is of positive homogeneous type for $\hat\phi$ with  $\alpha_{\xi_1}=\alpha_{\xi_2}=(0,\gamma)$, $\alpha_{\xi_3}=\alpha_{\xi_4}=(0,-\gamma)$ and $P_{\xi_1}=P_{\xi_2}=P_{\xi_3}=P_{\xi_4}=P$. Note that $P$ is obviously positive homogeneous with $E=(1/2)I\in\Exp(P)$ and hence
\begin{equation}\label{eq:ex2}
\mu_{\phi}=\mu_{P_{\xi_1}}=\mu_{P_{\xi_2}}=\mu_{P_{\xi_3}}=\mu_{P_{\xi_4}}=\mu_P=1.
\end{equation}
An appeal to Theorem \ref{thm:SubDecay} gives positive constants $C$ and $C'$ for which
\begin{equation*}
C n^{-1}\leq\|\phi^{(n)}\|_{\infty}\leq C' n^{-1}
\end{equation*}
for all $n\in\mathbb{N}_+$. In view of \eqref{eq:ex2}, let us note that the contribution from all points $\xi_1,\xi_2,\xi_3,\xi_4\in\Omega(\phi)$ appear in the local limit given by Theorem \ref{thm:MainLocalLimit}. An application of the theorem gives
\begin{align*}
\phi^{(n)}(x,y)&=(i)^{5n/4}\Big(e^{-i(x,y)\cdot\xi_1}H_P^n(x,y-n\gamma)+(-1)^n e^{i(x,y)\cdot\xi_2}H_P^n(x,y-n\gamma)\\
&\quad+e^{-i(x,y)\cdot\xi_3}H_P^n(x,y+n\gamma)+(-1)^n e^{i(x,y)\cdot\xi_4}H_P^n(x,y+n\gamma)\Big)+o(n^{-1})\\
&= (i)^{5n/4}\Big((-1)^y+(-1)^n\Big)\Big(e^{-i\pi x/2}e^{i\pi y/4}H_P^n(x,y-\gamma n)\\
& \quad+e^{i\pi x/2}e^{i3\pi y/4}H_P^n(x,y+\gamma n)\Big)+o(n^{-1})
\end{align*}
which holds uniformly for $(x,y)\in\mathbb{Z}^2$. In this special case that $P$ is of second order, we can write
\begin{align*}
H_P^{n}(x,y)&=\frac{1}{(2\pi)^2}\int_{\mathbb{R}^2}e^{-i(\eta,\zeta)\cdot(x,y)-P(\eta,\zeta)}\,d\eta d\zeta\\
&=\frac{1}{2\pi n\sqrt{\gamma(1+i\gamma)}}\exp\left(-\frac{x^2}{n(1+i\gamma)}-\frac{y^2}{4n \gamma}\right)
\end{align*}
for $(x,y)\in\mathbb{R}^2$ and from this, it is easily seen that $\phi^{(n)}$ is approximated by two  generalized Gaussian packets respectively centered  at $\pm(0,\gamma n)$ for $n\in\mathbb{N}_+$. For comparison, Figures \ref{fig:ex2phi30} and \ref{fig:ex2phi60} illustrate the approximation
\begin{align*}
f_n(x,y)&:=(i)^{5n/4}\Big((-1)^y+(-1)^n\Big)\\
&\qquad\times\Big(e^{-i\pi x/2}e^{i\pi y/4}H_P^n(x,y-\gamma n)+e^{i\pi x/2}e^{i3\pi y/4}H_P^n(x,y+\gamma n)\Big)
\end{align*} to $\phi^{(n)}$ for $n=30$ and $60$.

\subsection{A supporting lattice misaligned with $\mathbb{Z}^2$}\label{subsec:ex3}

In this example, we study a real valued function $\phi$ whose support is not well-aligned with the principal coordinate axes. Here, the points at which $\hat\phi$ is maximized are of positive homogeneous type for $\hat\phi$ but the corresponding positive homogeneous polynomials are not semi-elliptic. In this way, we have a concrete example to illustrate the conclusion of Proposition \ref{prop:PositiveHomogeneousPolynomialsareSemiElliptic}. In writing out the local limit theorem for $\phi$, we also see the appearance of a multiplicative prefactor which gives us information concerning the support of $\phi^{(n)}$. Finally, the validity of global space-time exponential-type estimates is discussed.\\

\noindent Consider $\phi:\mathbb{Z}^2\rightarrow\mathbb{R}$ defined by
\begin{equation*}
\phi(x,y)=\begin{cases}
	    3/8 & (x,y)=(0,0)\\
           1/8 & (x,y)=\pm(1,1)\\
           1/4 & (x,y)=\pm(1,-1)\\
            -1/16& (x,y)=\pm( 2,-2)\\
           0 & \mbox{otherwise}.
          \end{cases}
\end{equation*}
Figures \ref{fig:ex3_100_1} and \ref{fig:ex3_100_grid} illustrate the graph and heat map of $\phi^{(n)}$ respectively when $n=100$.\\
\begin{figure}[h!]
\begin{center}
\resizebox{\textwidth}{!}{
	    \begin{subfigure}[8cm]{0.5\textwidth}
		\includegraphics[width=\textwidth]{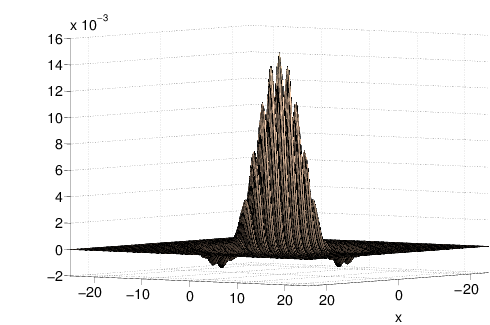}
		\caption{}
		\label{fig:ex3_100_1}
	    \end{subfigure}
	    \begin{subfigure}[8cm]{0.5\textwidth}
		\includegraphics[width=\textwidth]{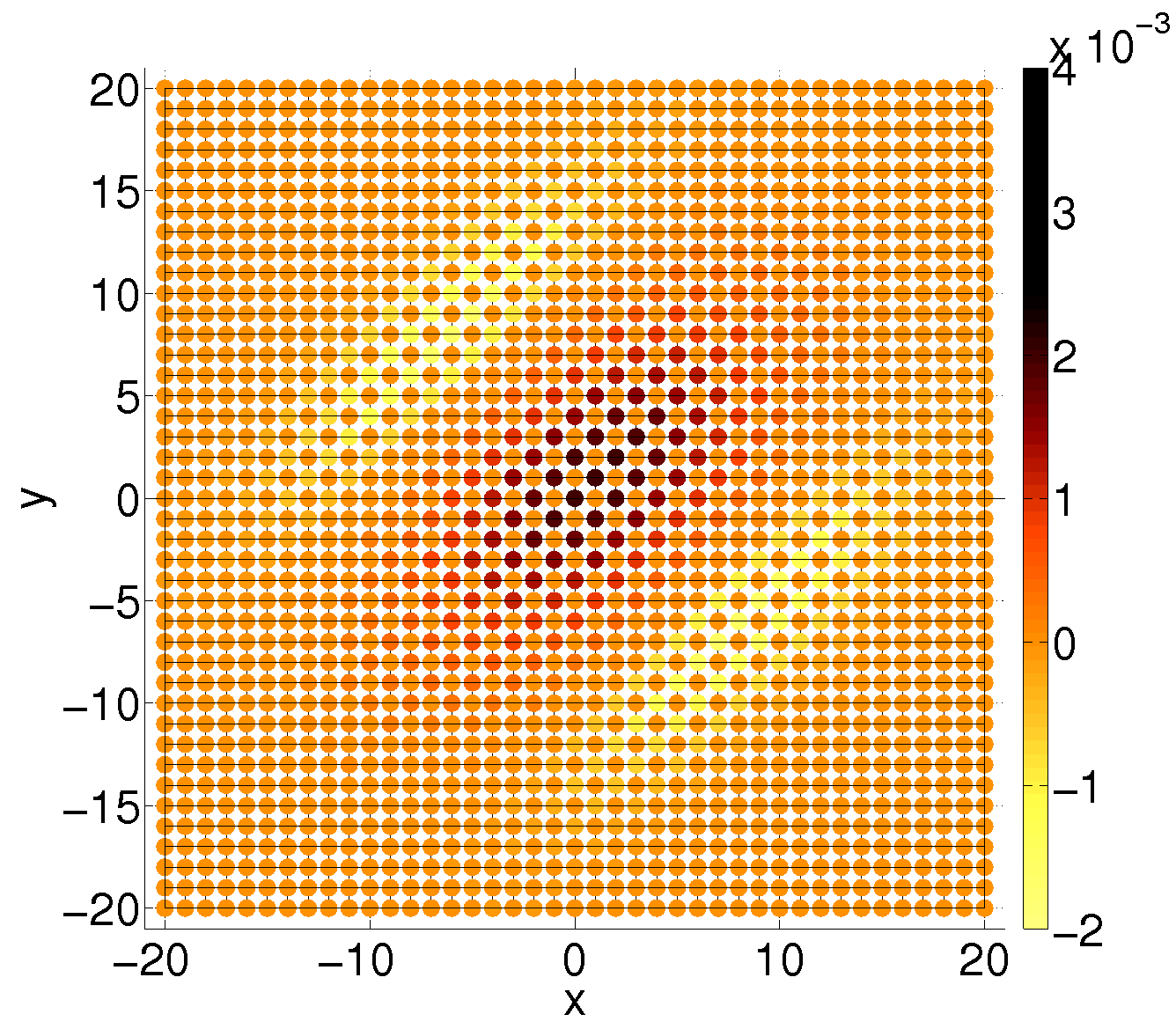}
		\caption{}
		\label{fig:ex3_100_grid}
	    \end{subfigure}}
\caption{$\phi^{(n)}$ for $n=100$}
\label{fig:ex3}
\end{center}
\end{figure}

\noindent We compute the Fourier transform of $\phi$ and find by a routine calculation that $\sup|\hat\phi|=1$ and this maximum is attained at only two points in $\mathbb{T}^2$, $(0,0)$ and $(\pi,\pi)$. We write this as
\begin{equation*}
\Omega(\phi)=\{\xi_1,\xi_2\}=\{(0,0),(\pi,\pi)\},
\end{equation*}
and note that $\phi(\xi_1)=\phi(\xi_2)=1$. For $\xi_1=(0,0)$, we have
\begin{align*}
\Gamma(\eta,\zeta)&=\log\left(\frac{\hat\phi(\xi+\xi_1)}{\phi(\xi_1)}\right)\\
&=-\frac{\eta^2}{8}-\frac{23\eta^4}{384}-\frac{\eta\zeta}{4}+\frac{25\eta^3\zeta}{96}-\frac{\zeta^2}{8}-\frac{23\eta^2\zeta^2}{64}+\frac{25 \eta\zeta^3}{96}-\frac{23\zeta^4}{384}\\
&\quad+o(|(\eta,\zeta)|^4)
\end{align*}
as $(\eta,\zeta)\rightarrow (0,0)$. In seeking a positive homogeneous polynomial to lead the expansion, we first note the appearance of the second order polynomial $\eta^2/8+\eta\zeta/4+\zeta^2/8$. We might be tempted to choose this as our candidate, however, it is not positive definite because it vanishes on the line $\eta=-\zeta$. Upon closer study, we write
\begin{align*}
\Gamma(\eta,\zeta)&=-\frac{1}{8}(\eta+\zeta)^2-\frac{23}{384}(\eta-\zeta)^4+o(|(\eta,\zeta)|^4)\\
                  &=-P(\eta,\zeta)+o(P(\eta,\zeta))
\end{align*}
as $(\eta,\zeta)\rightarrow (0,0)$, where the polynomial
\begin{equation*}
P(\eta,\zeta)=\frac{1}{8}(\eta+\zeta)^2+\frac{23}{384}(\eta-\zeta)^4.
\end{equation*}
is positive definite. Fortunately, it is also a positive homogeneous polynomial as can be seen by observing that, for
\begin{equation*}
E=\begin{pmatrix}
   3/8 & 1/8\\
   1/8 & 3/8
  \end{pmatrix},
\end{equation*}
\begin{align*}
P(t^E(\eta,\zeta))&=P\left(t^{1/2}(\eta+\zeta)/2+t^{1/4}(\eta-\zeta)/2,t^{1/2}(\eta+\zeta)/2-t^{1/4}(\eta-\zeta)/2\right)\\
&=\frac{1}{8}\left(t^{1/2}(\eta+\zeta)\right)^2+\frac{23}{384}\left(t^{1/4}(\eta-\zeta)\right)^4\\
&=tP(\eta,\zeta)
\end{align*}
for all $t>0$ and $(\eta,\zeta)\in\mathbb{R}^2$. In contrast to the previous examples, $P$ is not semi-elliptic. However, observe that
\begin{equation*}
A^{-1}EA=\begin{pmatrix}
 1/\sqrt{2} & 1/\sqrt{2}\\
  -1/\sqrt{2} & 1/\sqrt{2}
\end{pmatrix}
\begin{pmatrix}
 3/8 & 1/8\\
   1/8 & 3/8
\end{pmatrix}
\begin{pmatrix}
  1/\sqrt{2} & -1/\sqrt{2}\\
  1/\sqrt{2} & 1/\sqrt{2}
 \end{pmatrix}
=\begin{pmatrix}
  1/2 & 0\\
  0 & 1/4
 \end{pmatrix}
\end{equation*}
and
\begin{equation*}
(P\circ L_A)(\eta,\zeta)=P\left(\frac{\eta-\zeta}{\sqrt{2}},\frac{\eta+\zeta}{\sqrt{2}}\right)=\frac{1}{8}(\sqrt{2}\eta)^2+\frac{23}{384}(-\sqrt{2}\zeta)^4=\frac{1}{4}\eta^2+\frac{23}{96}\zeta^4
\end{equation*}
which is semi-elliptic; this illustrates the conclusion of Proposition \ref{prop:PositiveHomogeneousPolynomialsareSemiElliptic}.

We have shown that $\xi_1$ is of positive homogeneous type for $\hat\phi$ with corresponding $\alpha_{\xi_1}=(0,0)$ and positive homogeneous polynomial $P=P_{\xi_1}$. By expanding the logarithm of $\hat\phi$ near $\xi_2$, a similar argument shows that $\xi_2$ is also of positive homogeneous type for $\hat\phi$ with corresponding $\alpha_{\xi_2}=(0,0)$ and the same positive homogeneous polynomial $P=P_{\xi_2}$. It then follows immediately that $\phi$ meets they hypotheses of Theorems \ref{thm:SubDecay} and \ref{thm:MainLocalLimit} where
\begin{equation*}
\mu_\phi=\mu_P=\tr E=3/4.
\end{equation*}
An appeal to Theorem \ref{thm:SubDecay} gives positive constants $C$ and $C'$ for which
\begin{equation*}
C'n^{-3/4}\leq \|\phi^{(n)}\|_{\infty}\leq Cn^{-3/4}
\end{equation*}
for all $n\in\mathbb{N}_+$. By an appeal to Theorem \ref{thm:MainLocalLimit}, we also have
\begin{equation}\label{eq:ex_3}
\begin{split}
\phi^{(n)}(x,y)&=\hat\phi(\xi_1)^ne^{-i\xi_1\cdot (x,y)}H_P^n(x,y)+\hat\phi(\xi_2)e^{-i\xi_2\cdot (x,y)}H_P^n(x,y)+o(n^{-3/4})\\
&=\left(1+e^{i\pi(x+y)}\right)H_P^n(x,y)+o(n^{-3/4})\\
&=(1+\cos(\pi(x+y)))H_P^n(x,y)+o(n^{-3/4})
\end{split}
\end{equation}
uniformly for $(x,y)\in\mathbb{Z}^2$. Upon closely inspecting the prefactor $1+\cos(\pi(x+y)$, it is reasonable to assert that
\begin{equation*}
\supp\left(\phi^{(n)}\right)\subseteq\{(x,y)\in\mathbb{Z}^2:x\pm y\in 2\mathbb{Z}\}=:\mathcal{L}
\end{equation*}
for all $n\in\mathbb{N}_+$ (Figure \ref{fig:ex3_100_grid} also gives evidence for this when $n=100$). The assertion is indeed true, for it is easily verified that $\supp(\phi)\subseteq \mathcal{L}$ and, because $\mathcal{L}$ is an additive group, induction shows that
\begin{equation*}
\supp\left(\phi^{(n+1)}\right)=\supp\left(\phi^{(n)}\ast\phi\right)\subseteq \supp\left(\phi^{(n)}\right)+\supp(\phi)\subseteq \mathcal{L}+\mathcal{L}=\mathcal{L}
\end{equation*}
for all $n\in\mathbb{N}_+$. Thus, the prefactor $(1+\cos(\pi(x+y))$ gives us information about the support of the convolution powers. In Section \ref{subsec:ClassicalLLT}, we will see that this situation is commonplace when $\phi$ is a probability distribution.

Let us finally note that, because $\alpha_{\xi_1}=\alpha_{\xi_2}=(0,0)$ and $P_{\xi_1}=P_{\xi_2}=P$, $\phi$ satisfies the hypotheses of Theorem \ref{thm:ExponentialEstimate}. A straightforward computation shows that $R^{\#}(x,y)\asymp |x+y|^2+|x-y|^{4/3}$ where $R=\Re P$ and so, by an appeal to Theorem \ref{thm:ExponentialEstimate}, there are positive constants $C$ and $M$ for which
\begin{equation*}
|\phi^{(n)}(x,y)|\leq \frac{C}{n^{3/4}}\exp\left(-M\left(\frac{|x+y|^2}{n}+\frac{|x-y|^{4/3}}{n^{1/3}}\right)\right)
\end{equation*}
for all $(x,y)\in\mathbb{Z}^2$ and $n\in\mathbb{N}_+$. We note however that because $\Omega(\phi)=\{\xi_1,\xi_2\}$, $\phi$ does not satisfy the hypotheses of Theorem \ref{thm:DerivativeEstimate} and, by closely inspecting Figure \ref{fig:ex3_100_1}, this should come at no surprise. In fact, by a direct application of \eqref{eq:ex_3}, it is easily shown that $|\phi^{(n)}(0,0)|\geq \epsilon n^{-3/4}$ for some $\epsilon>0$ whereas $\phi^{(n)}(0,1)=0$ for all $n\in\mathbb{N}_+$. Consequently, $|D_{(0,1)}\phi^{(n)}(0,0)|\geq \epsilon n^{-3/4}$ for all $n\in\mathbb{N}_+$ from which it is evident that the conclusion to Theorem \ref{thm:DerivativeEstimate}, \eqref{eq:DerivativeEstimate}, doesn't hold.

\subsection{Contribution from non-minimal decay exponent}\label{subsec:ex4}

\noindent In the present example, we study a real valued function $\phi$ on $\mathbb{Z}^2$ with $\Omega(\phi)=\{\xi_1,\xi_2\}$. Although both $\xi_1$ and $\xi_2$ are of positive homogeneous type for $\hat\phi$ with corresponding positive homogeneous polynomials $P_{\xi_1}$ and $P_{\xi_2}$, we find that $\mu_{\phi}=\mu_{P_{\xi_1}}<\mu_{P_{\xi_2}}$ which is in contrast to the preceding examples. Consequently, only the contribution from $\xi_1$ appears in the local limit.\\

\noindent Consider $\phi:\mathbb{Z}^2\rightarrow\mathbb{R}$ be defined by
\begin{equation}\label{ex4phi}
\phi(x,y)=\begin{cases}
	    19/128 & (x,y)=(0,0)\\
	    19/256 & (x,y)=(0,\pm 1)\\
	    1/4 & (x,y)=(\pm 1,0)\\
            1/8 & (x,y)=(\pm 1,\pm 1)\\
            -5/64 & (x,y)=(\pm 2,0)\\
             -5/128 & (x,y)=(\pm 2,\pm 1)\\
            1/256& (x,y)=(\pm 4,0)\\
            1/512 & (x,y)=(\pm 4,\pm 1)\\
           0 & \mbox{otherwise}.
          \end{cases}
\end{equation}\begin{figure}[h!]
\begin{center}
\resizebox{\textwidth}{!}{
	    \begin{subfigure}[5cm]{0.5\textwidth}
		\includegraphics[width=\textwidth]{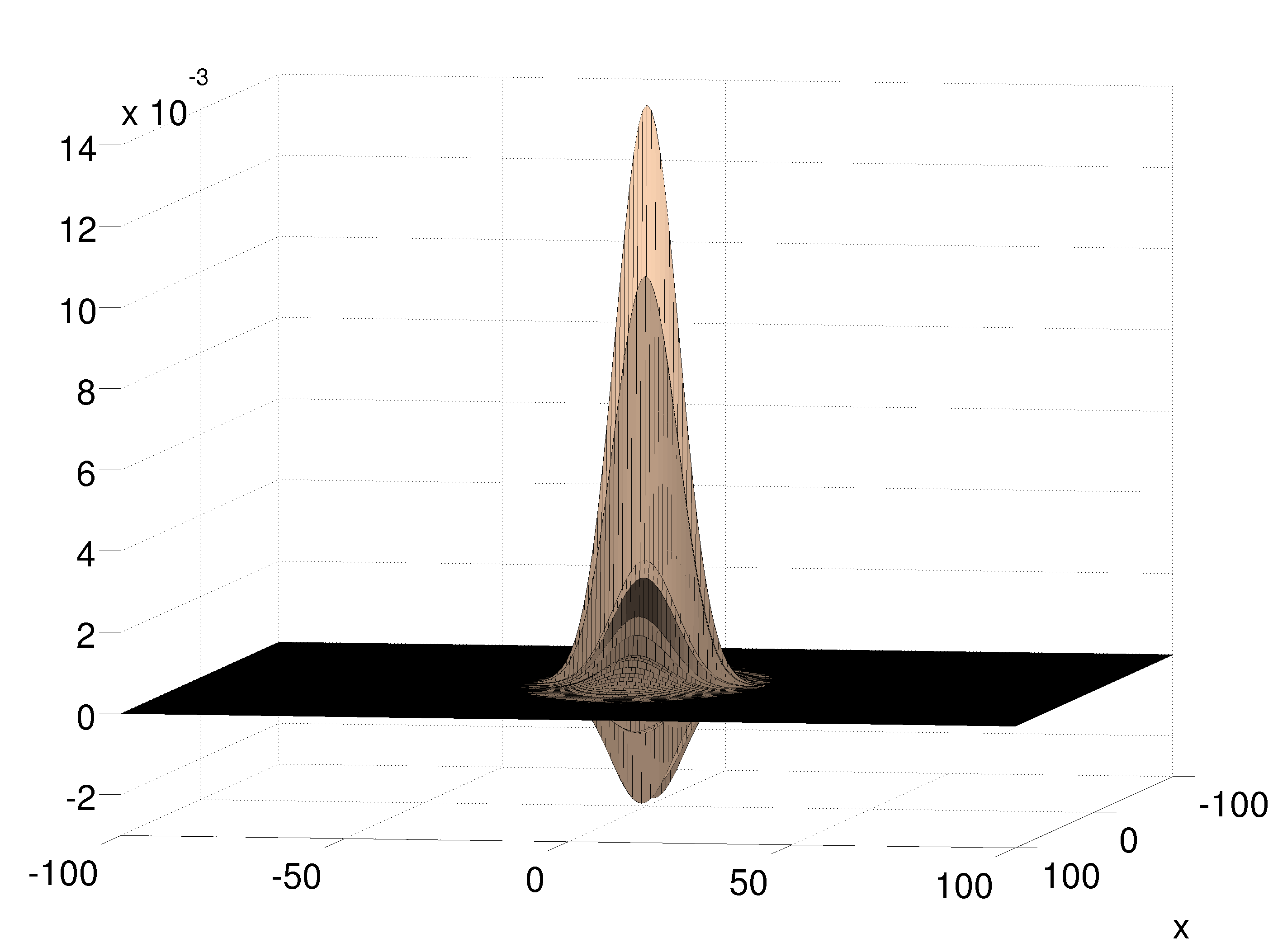}
		\caption{$\phi^{(n)}$ for $n=100$}
		\label{fig:ex4phi100_1}
	    \end{subfigure}
	    \begin{subfigure}[5cm]{0.5\textwidth}
		\includegraphics[width=\textwidth]{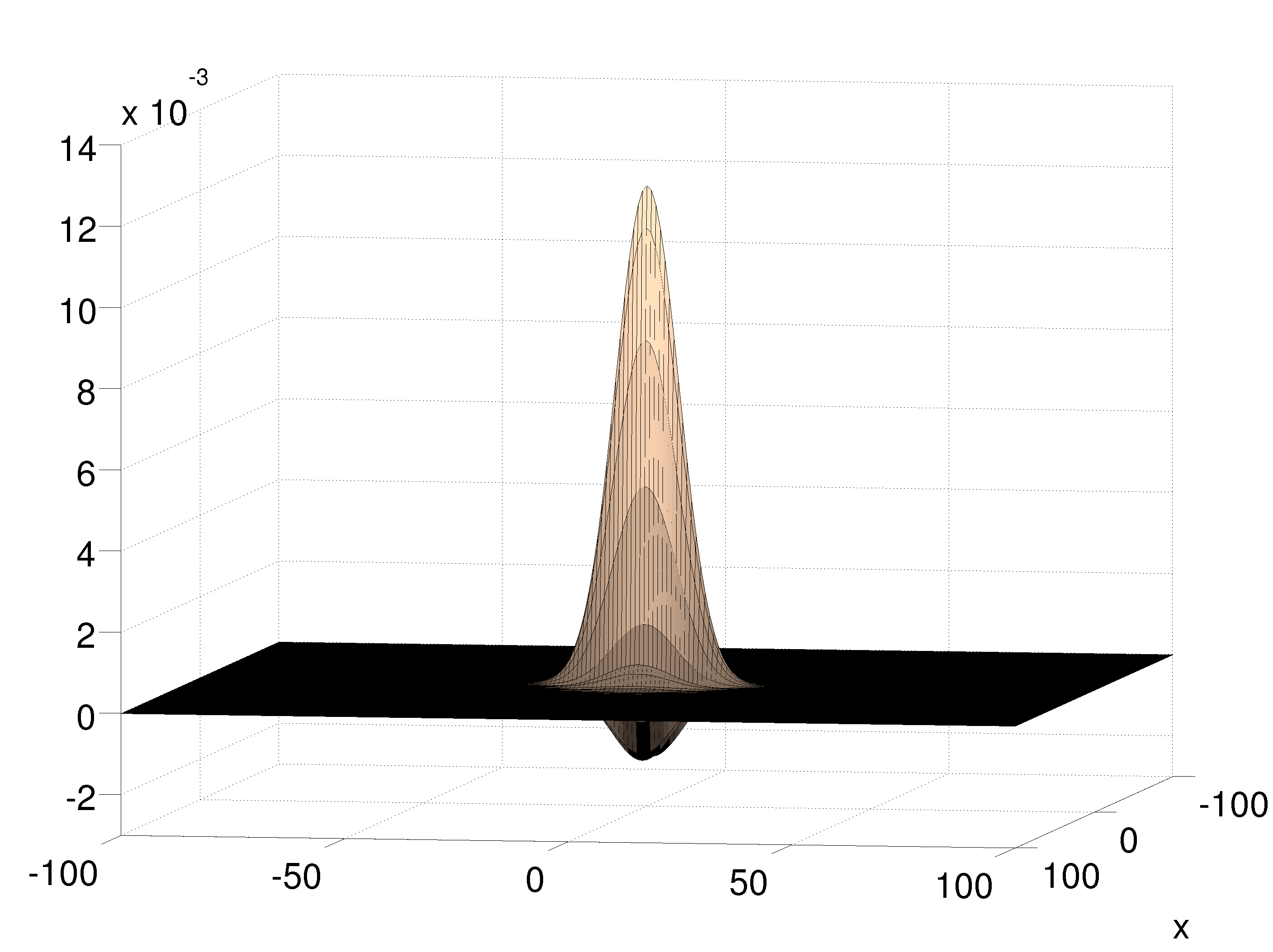}
		\caption{$H_{P_{\xi_1}}^n$ for $n=100$}
		\label{fig:ex4HP100_1}
	    \end{subfigure}}
\resizebox{\textwidth}{!}{	
	    \begin{subfigure}[5cm]{0.5\textwidth}
		\includegraphics[width=\textwidth]{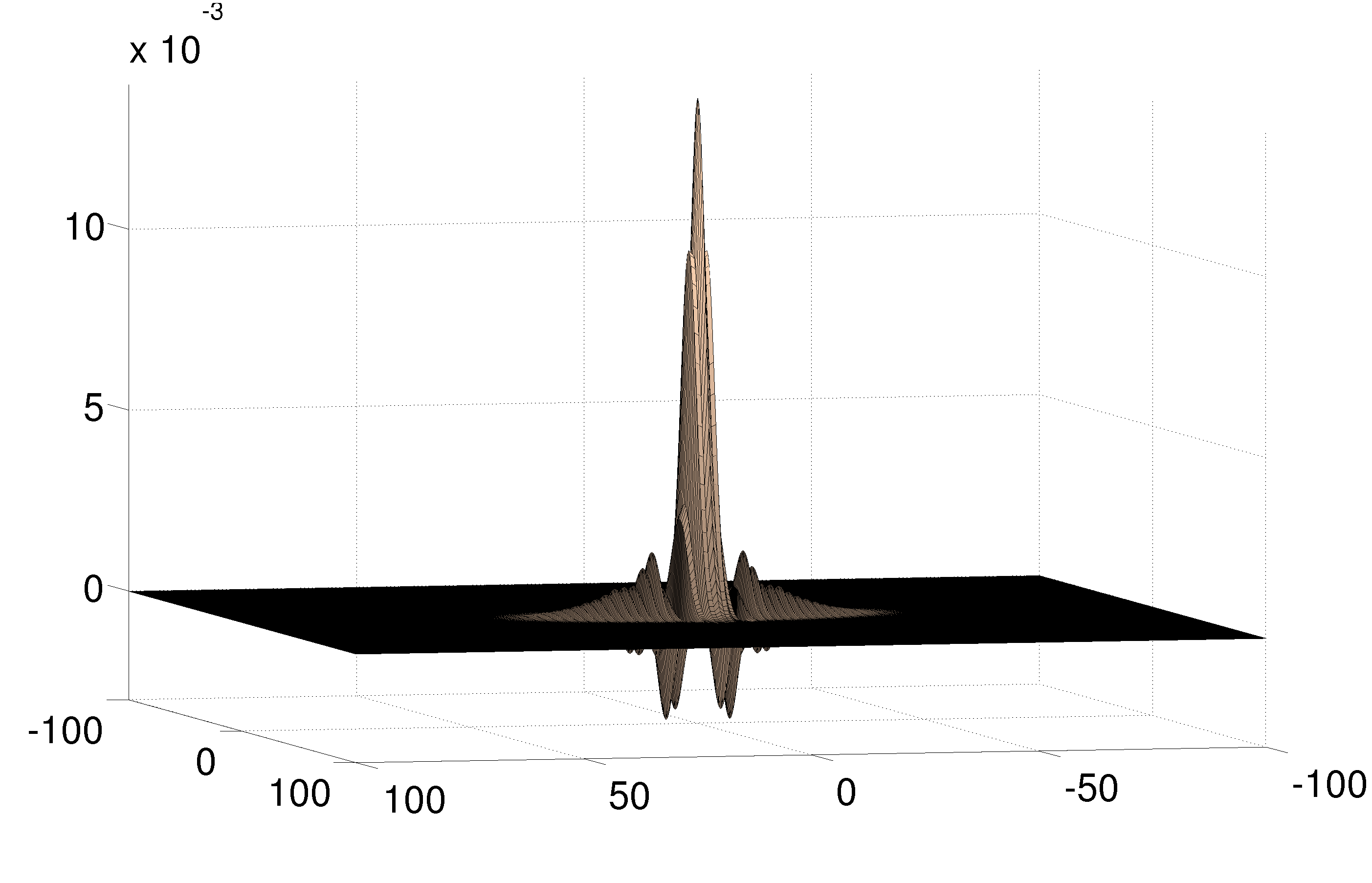}
		\caption{$\phi^{(n)}$ for $n=100$}
		\label{fig:ex4phi100_2}
	    \end{subfigure}
	    \begin{subfigure}[5cm]{0.5\textwidth}
		\includegraphics[width=\textwidth]{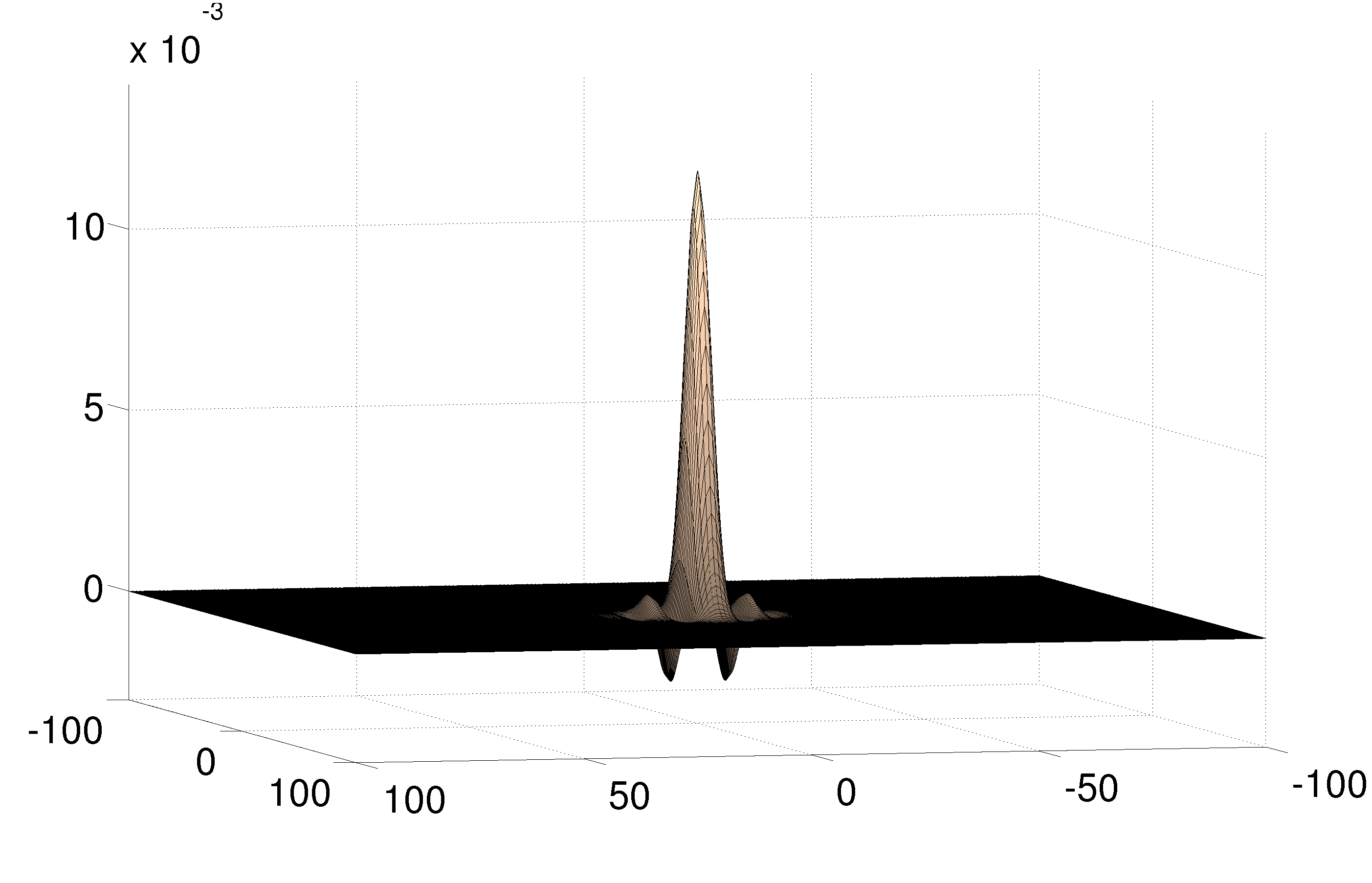}
		\caption{$H_{P_{\xi_1}}^n$ for $n=100$}
		\label{fig:ex4HP100_2}
	    \end{subfigure}}
\caption{The graphs of $\phi^{(n)}$ and $H^n_{P_{\xi_1}}$ for $n=100$}
\label{fig:ex4_100}
\end{center}
\end{figure}
\begin{figure}[h!]
\begin{center}
\resizebox{\textwidth}{!}{
	    \begin{subfigure}[5cm]{0.5\textwidth}
		\includegraphics[width=\textwidth]{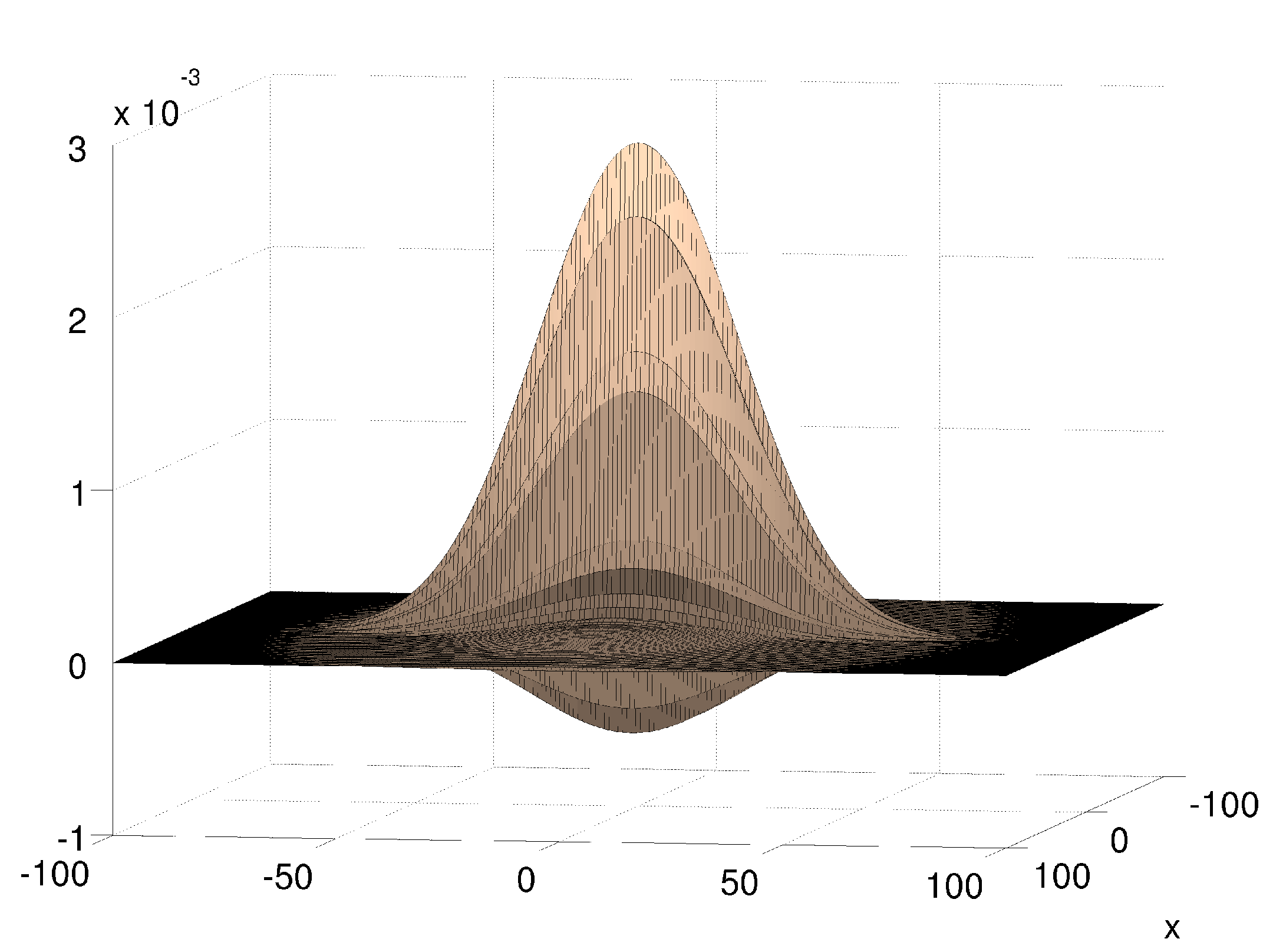}
		\caption{$\phi^{(n)}$ for $n=1,000$}
		\label{fig:ex4phi1000_1}
	    \end{subfigure}
	    \begin{subfigure}[5cm]{0.5\textwidth}
		\includegraphics[width=\textwidth]{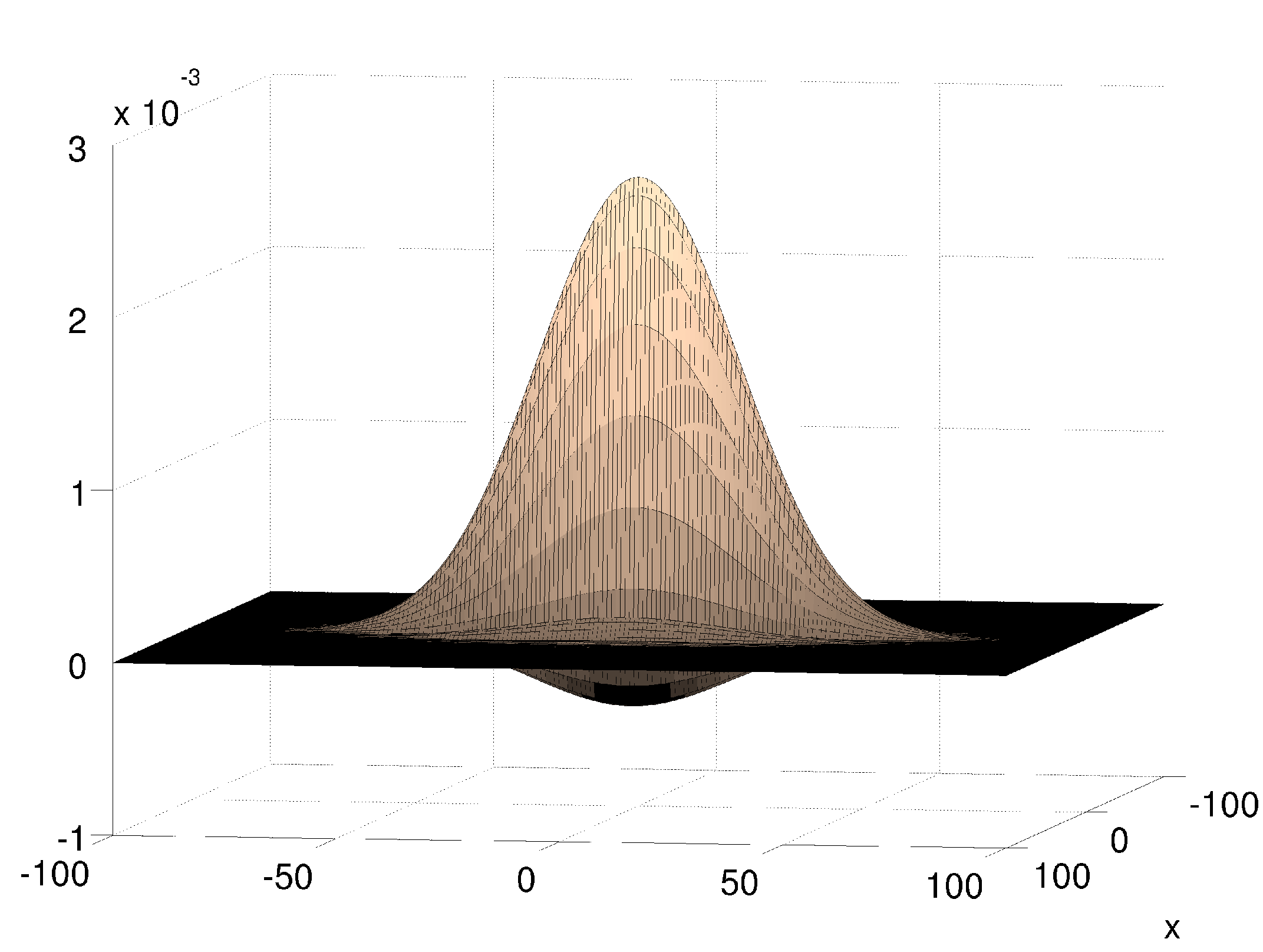}
		\caption{$H_{P_{\xi_1}}^n$ for $n=1,000$}
		\label{fig:ex4HP1000_1}
	    \end{subfigure}}
\resizebox{\textwidth}{!}{	
	    \begin{subfigure}[5cm]{0.5\textwidth}
		\includegraphics[width=\textwidth]{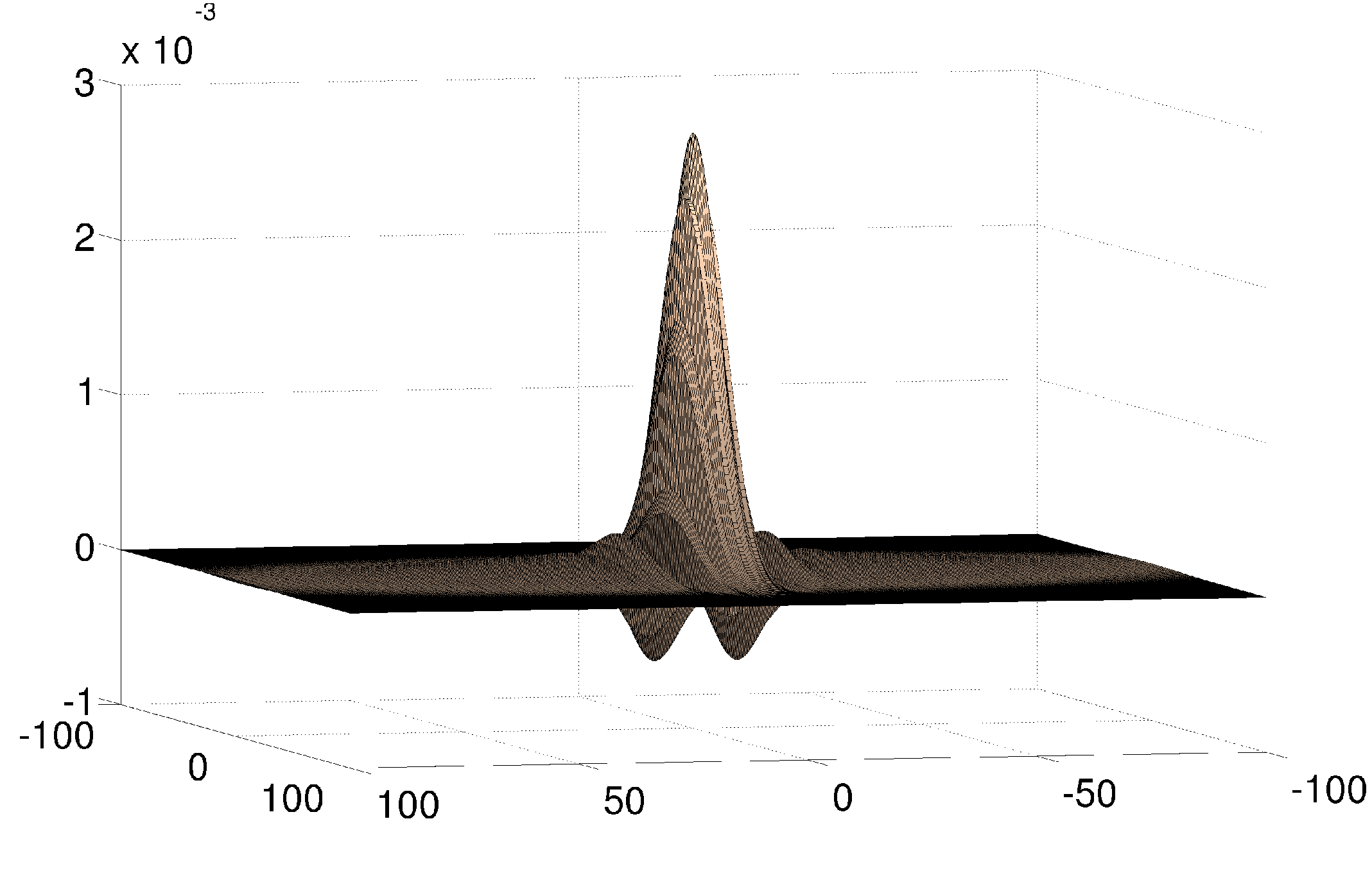}
		\caption{$\phi^{(n)}$ for $n=1,000$}
		\label{fig:ex4phi1000_2}
	    \end{subfigure}
	    \begin{subfigure}[5cm]{0.5\textwidth}
		\includegraphics[width=\textwidth]{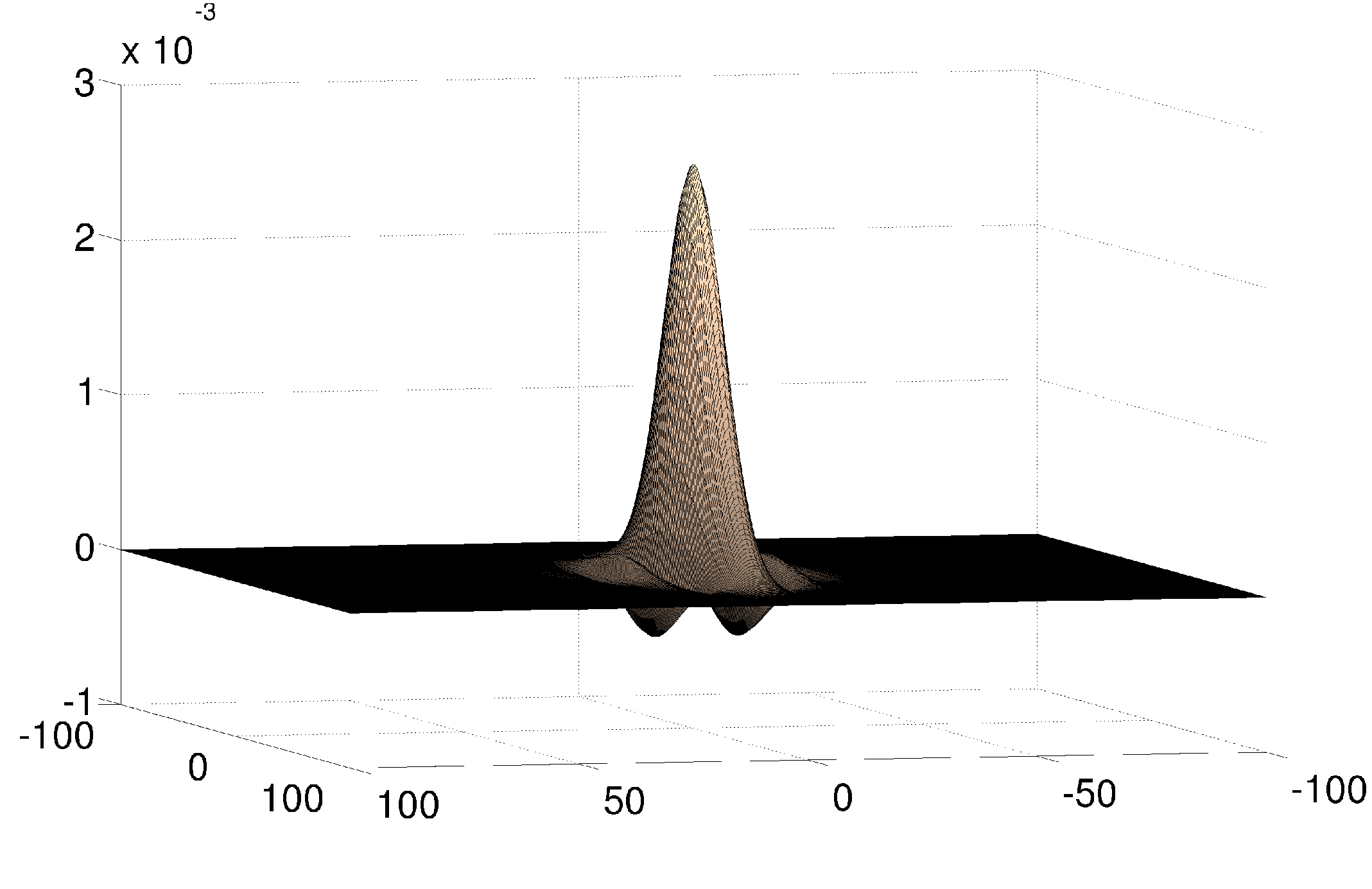}
		\caption{$H_{P_{\xi_1}}^n$ for $n=1,000$}
		\label{fig:ex4HP1000_2}
	    \end{subfigure}}
\caption{The graphs of $\phi^{(n)}$ and $H^n_{P_{\xi_1}}$ for $n=1,000$}
\label{fig:ex4_1000}
\end{center}
\end{figure}

\noindent The graphs of $\phi^{(n)}$ for $(x,y)\in\mathbb{Z}^2$ such that $-100\leq x,y\leq 100$ are displayed in Figures \ref{fig:ex4phi100_1} and \ref{fig:ex4phi100_2} for $n=100$ and Figures \ref{fig:ex4phi1000_1} and \ref{fig:ex4phi1000_2} for $n=1,000$. Upon considering the Fourier transform of $\phi$, we find that $\sup|\hat\phi|=1$ and this maximum is attained at exactly two points in $\mathbb{T}^2$. Specifically,
\begin{equation*}
\Omega(\phi)=\{\xi_1,\xi_2\}=\{(0,0),(\pi,0)\},
\end{equation*}
where $\hat\phi(\xi_1)=1$ and $\hat\phi(\xi_2)=-1$. In expanding the logarithm of $\hat\phi(\xi+\xi_1)/\hat\phi(\xi_1)$ about $(0,0)$, we find that $\xi_1=(0,0)$ is of positive homogeneous type for $\hat\phi$ with $\alpha_{\xi_1}=(0,0)$ and
\begin{equation*}
P_{\xi_1}(\eta,\zeta)=\frac{\eta^6}{16}+\frac{\zeta^2}{4}.
\end{equation*}
Clearly $P_{\xi_1}$ is positive homogeneous with $E_1=\diag(1/6,1/2)\in\Exp(P_{\xi_1})$ thus $\mu_{P_{\xi_1}}=\tr E_1=2/3$. Now, upon expanding the logarithm of $\hat\phi(\xi+\xi_2)/\hat\phi(\xi_2)$ we find that $\xi_2=(\pi,0)$ is also of positive homogeneous type for $\hat\phi$ with $\alpha_{\xi_2}=(0,0)$ and positive homogeneous polynomial
\begin{equation*}
P_{\xi_2}(\eta,\zeta)=\eta^2+\frac{\zeta^2}{4};
\end{equation*}
Here, $E_2=(1/2)I\in\Exp(P_{\xi_2})$ and thus $\mu_{P_{\xi_2}}=\tr E_2=1$. In this case
\begin{equation*}
\mu_\phi=\min_{i=1,2}\mu_{P_{\xi_i}}=\mu_{P_{\xi_1}}=2/3
\end{equation*}
and so, in light of the paragraph preceding the statement of Theorem \ref{thm:MainLocalLimit}, we restrict our attention to $\xi_1$, in which case the theorem describes the approximation of $\phi^{(n)}$ by a single attractor $H_{P_{\xi_1}}$. This is the local limit
\begin{equation}\label{ex4locallimit}
\phi^{(n)}(x,y)=H_{P_{\xi_1}}^n(x,y)+o(n^{-2/3})
\end{equation}
which holds uniformly for $(x,y)\in\mathbb{Z}^2$. Figures \ref{fig:ex4HP100_1}, \ref{fig:ex4HP100_2}, \ref{fig:ex4HP1000_1} and \ref{fig:ex4HP1000_2} illustrate this result. It should be noted that the approximations shown in Figures \ref{fig:ex4_100} and \ref{fig:ex4_1000} appear more coarse than those of the previous examples. For instance, Figure \ref{fig:ex4phi100_2} depicts minor oscillations in the graph of $\phi^{(n)}$ which do not appear in the approximation illustrated in Figure \ref{fig:ex4HP100_2}. These oscillations are due to the influence on $\phi^{(n)}$ by $\hat\phi$ near $\xi_2$ which for $n=1,000$ is not yet sufficiently scaled out. As demonstrated in the proof of Theorem \ref{thm:MainLocalLimit}, this influence dies out on the relative order of $n^{1-2/3}=n^{-1/3}$ and thus the influence is not negligible when $n=1,000$.\\

\noindent As a final remark, we note that $\phi$ is the tensor product of two functions mapping $\mathbb{Z}$ into $\mathbb{C}$. Specifically, $\phi=\phi_1\otimes\phi_2$ where,
\begin{equation*}
\phi_1(x)=\begin{cases}
           19/64 & x=0\\
           1/2 &x=\pm 1\\
           -5/32 & x=\pm 2\\
           1/128 & x=\pm 4\\
            0 &\mbox{ otherwise}
          \end{cases}
\hspace{.5cm}\mbox{and}\hspace{.5cm}
\phi_2(y)=\begin{cases}
           1/2 &y=0\\
           1/4 &y=\pm 1\\
           0 &\mbox{ otherwise}.
          \end{cases}
\end{equation*}
In fact, had we studied the functions $\phi_1$ and $\phi_2$ separately, we would have found that
\begin{equation*}
\phi_1^{(n)}(x)=H^n_{\eta^6/16}(x)+o(n^{-1/6})
\hspace{.5cm}\mbox{ and }\hspace{.5cm}
\phi_1^{(n)}(y)=H^n_{\zeta^2/4}(y)+o(n^{-1/2})
\end{equation*}
uniformly for $x,y\in\mathbb{Z}$ and from this deduced the local limit \eqref{ex4locallimit} because $\phi^{(n)}=\phi_1^{(n)}\otimes\phi_2^{(n)}$ and $H_{P_{\xi_1}}=H_{\eta^6/16}\otimes H_{\zeta^2/4}$ (note also that $\mu_\phi=1/6+1/2=\mu_{\phi_1}+\mu_{\phi_2}$). In general, tensor products can be used to create a wealth of examples in any dimension to which the results of lower dimensions can be applied. For instance, by applying the much stronger theory of one dimension (in light of \cite{Randles2015}), one can deduce stronger results than are given here for the class of finitely supported functions on $\mathbb{Z}^d$ of the form
\begin{equation*}
\phi=\phi_1\otimes\phi_2\otimes\cdots\otimes\phi_d
\end{equation*}
where $\phi_k:\mathbb{Z}\mapsto\mathbb{C}$ is finitely supported for $k=1,2,\dots,d$. How to place these examples in a $d$-dimensional theory is an open question.

\subsection{A simple class of real valued functions}\label{subsec:ex5}

\noindent In this subsection we consider a class of real valued and finitely supported functions $\phi_{\mathbf{m},\lambda}$ determined by two multi-parameters $\mathbf{m}\in\mathbb{N}_+$ and $\lambda\in\mathbb{R}_+^d$, c.f, Subsection 2.4 of \cite{Diaconis2014}. Here, $\Omega(\phi_{\mathbf{m},\lambda})$ contains only $0\in\mathbb{T}^d$ which is of positive homogeneous type for $\hat\phi_{\mathbf{m},\lambda}$ with no drift and whose associated positive homogeneous polynomial is a semi-elliptic polynomial with no ``mixed'' terms. In this setting, our methods yield easily $\ell^{\infty}$-asymptotics and local limit theorems for $\phi_{\mathbf{m},\lambda}^{(n)}=(\phi_{\mathbf{m},\lambda})^{(n)}$. Moreover, all of the results of Section \ref{sec:PointwiseBounds} concerning global space-estimates for $\phi_{\mathbf{m},\lambda}^{(n)}$ and its discrete differences are valid and we apply them.\\

\noindent Let $\mathbf{m}=(m_1,m_2,\dots,m_d)\in\mathbb{N}_+^d$ and $\lambda=(\lambda_1,\lambda_2,\dots,\lambda_d)$ be such that $\lambda_j\in(0,2^{1-m_j}/d]$ for $j=1,2,\dots,d$ with at least one $\lambda_j<2^{1-m_j}/d$. Define
\begin{equation}\label{eq:DefofPhiMLambda}
\phi_{\mathbf{m},\lambda}=\delta_0-\sum_{j=1}^d\lambda_j(\delta_0-\rho_j)^{(m_j)}
\end{equation}
where $\rho_j=(1/2)(\delta_{e_j}+\delta_{-e_j})$ is the Bernoulli walk on the $j$th coordinate axis. By a straightforward computation, we have
\begin{equation*}
\hat\phi_{\mathbf{m},\lambda}(\xi)=1-\sum_{j=1}^d\lambda_j(1-\cos(\xi_j))^{m_j}
\end{equation*}
for $\xi=(\xi_1,\xi_2,\dots,\xi_d)\in\mathbb{R}^d$ and from this it is easily seen that $\sup_{\xi}|\hat\phi_{\mathbf{m},\lambda}(\xi)|=1$ which is attained only at $0\in\mathbb{T}^d$, i.e., $\Omega(\phi_{\mathbf{m},\lambda})=\{0\}$. Here, $\hat\phi_{\mathbf{m},\lambda}(0)=1$ and it is easily seen that
\begin{equation*}
\Gamma(\xi)=\log(\hat\phi_{\mathbf{m},\lambda}(\xi))=-P_{\mathbf{m},\lambda}(\xi)+o(P_{\mathbf{m},\lambda}(\xi))
\end{equation*}
as $\xi\rightarrow 0$, where
\begin{equation*}
P_{\mathbf{m},\lambda}(\xi)=\sum_{j=1}^d\frac{\lambda_j}{2^{m_j}}\xi_j^{2m_j}
\end{equation*}
for $\xi=(\xi_1,\xi_2,\dots,\xi_d)\in\mathbb{R}^d$. Note that $P_{\mathbf{m},\lambda}(\xi)$ is a semi-elliptic polynomial of the form \eqref{eq:SemiEllipticPolynomial} with $D_{\mathbf{m}}=\diag((2m_1)^{-1},(2m_2)^{-1},\dots,(2m_d)^{-1})\in\Exp(P_{\mathbf{m},\lambda})$ and hence
\begin{equation*}
\mu_{\phi_{\mathbf{m},\lambda}}=\mu_{P_{\mathbf{m},\lambda}}=(2m_1)^{-1}+(2m_2)^{-1}+\cdots+(2m_d)^{-1}=|\mathbf{1}:2\mathbf{m}|,
\end{equation*}
where $\mathbf{1}=(1,1,\dots,1)\in\mathbb{N}^d$. \\

\noindent For any $l\in\mathbb{N}$, recall from Section \ref{sec:PointwiseBounds} the discrete time difference operator $\partial_l=\partial_l(\phi_{\mathbf{m},\lambda},\xi_0,\alpha)$ which, in this case, is given by
\begin{equation*}
\partial_l\psi=(\delta-\phi_{\mathbf{m},\lambda}^{(l)})\ast\psi
\end{equation*}
for $\psi\in\ell^1(\mathbb{Z}^d)$. For any multi-index $\beta\in\mathbb{N}^d$, consider the difference operator $D^\beta=D^{\beta}_e$ defined for any $\psi\in\ell^1(\mathbb{Z}^d)$ by
\begin{equation*}
D^{\beta}\psi=(D_{e_1})^{\beta_1}(D_{e_2})^{\beta_2}\cdots (D_{e_d})^{\beta_d}\psi
\end{equation*}
where $D_{e_j}\psi(x)=\psi(x+e_j)-\psi(x)$ for $x\in\mathbb{Z}^d$ and $(D_{e_j})^0$ is the identity. We note that $e=\{e_1,e_2,\dots,e_d\}$ is $P_{\mathbf{m},\lambda}$-fitted of weight $\mathbf{m}$ in view of the discussion preceding Corollary \ref{cor:DerivativeEstimate}.  Finally, define
\begin{equation}\label{eq:mHomogeneousNormDef}
|x|_{\mathbf{m}}=\sum_{j=1}^d|x_j|^{2m_i/(2m_i-1)}
\end{equation}
for $x=(x_1,x_2,\dots,x_d)\in\mathbb{R}^d$ and observe that 
\begin{equation*}
|n^{-D_{\mathbf{m}}}x|_\mathbf{m}=\sum_{j=1}^d|x_j|^{2m_j/(2m_j-1)}/n^{1/(2m_j-1)}
\end{equation*} for $x\in\mathbb{R}^d$ and $n\in\mathbb{N}_+$.

\begin{proposition}
Let $\phi_{\mathbf{m},\lambda}$ be defined by \eqref{eq:DefofPhiMLambda}, assume the notation above and write $(\phi_{\mathbf{m},\lambda})^{(n)}=\phi_{\mathbf{m},\lambda}^{(n)}$ for $n\in\mathbb{N}_+$. There are positive constants $C$ and $C'$ for which
\begin{equation}\label{eq:BernoulliSemiElliptic_1}
Cn^{-|\mathbf{1}:2\mathbf{m}|}\leq \|\phi_{\mathbf{m},\lambda}^{(n)}\|_{\infty}\leq C'n^{-|\mathbf{1}:2\mathbf{m}|}
\end{equation}
for all $n\in\mathbb{N}_+$. We have
\begin{equation}\label{eq:BernoulliSemiElliptic_2}
\begin{split}
\phi_{\mathbf{m},\lambda}^{(n)}(x)&=n^{-|\mathbf{1}:2\mathbf{m}|}H_{P_{\mathbf{m},\lambda}}\left(n^{-D_{\mathbf{m}}}x\right)+o(n^{-|\mathbf{1}:2\mathbf{m}|})\\
&=n^{-|\mathbf{1}:2\mathbf{m}|}H_{P_{\mathbf{m},\lambda}}\left(\frac{x_1}{n^{1/(2m_1)}},\frac{x_2}{n^{1/(2m_2)}},\dots,\frac{x_d}{n^{1/(2m_d)}}\right)\\
&\quad+o(n^{-|\mathbf{1}:2\mathbf{m}|})
\end{split}
\end{equation}
uniformly for $x=(x_1,x_2,\dots,x_d)\in\mathbb{Z}^d$, where $H_{P_{\mathbf{m},\lambda}}$ is defined by \eqref{eq:HPdef}. There are positive constants $C_0,C_1,M_0$ and $M_1$ for which
\begin{equation}\label{eq:BernoulliSemiElliptic_3}
|\phi_{\mathbf{m},\lambda}^{(n)}(x)|\leq \frac{C_0}{n^{|\mathbf{1}:2\mathbf{m}|}}\exp\left(-M_0\left|n^{-D_{\mathbf{m}}}x\right|_{\mathbf{m}}\right)
\end{equation}
and
\begin{equation}\label{eq:BernoulliSemiElliptic_4}
|\phi_{\mathbf{m},\lambda}^{(n+1)}(x)-\phi_{\mathbf{m},\lambda}^{(n)}(x)|\leq \frac{C_1}{n^{|\mathbf{1}+2\mathbf{m}:2\mathbf{m}|}}\exp\left(-M_1\left|n^{-D_{\mathbf{m}}}x\right|_{\mathbf{m}}\right)
\end{equation}
for all $x\in\mathbb{Z}^d$ and $n\in\mathbb{N}_+$. Further, there are positive constants $C_0$ and $M$ and, to each multi-index $\beta$, a positive constant $C_\beta$ such that, for any $l_1,l_2,\dots,l_k\in\mathbb{N}_+^d$,
\begin{equation}\label{eq:BernoulliSemiElliptic_5}
|\partial_{l_1}\partial_{l_2}\cdots\partial_{l_j}D^{\beta}\phi_{\mathbf{m},\lambda}^{(n)}(x)|\leq\frac{C_\beta C_0^kk!\prod_{q=1}^kl_q}{n^{|\mathbf{1}+\beta+2k\mathbf{m}:2\mathbf{m}|}}\exp\left(-M\left|(n+s_k)^{-D_{\mathbf{m}}}x\right|_{\mathbf{m}}\right)
\end{equation}
for all $x\in\mathbb{Z}^d$ and $n\in\mathbb{N}_+$, where $s_k=l_1+l_2+\cdots+l_k$.
\end{proposition}
\begin{remark}
For simplicity, we have not treated the critical case in which $\lambda_j=2^{1-m_j}/d$ for $j=1,2,\dots,d$ in the proposition above, however, our methods handle this easily. In this case, the local limit \eqref{eq:BernoulliSemiElliptic_2} instead contains the prefactor $1+\exp(i\pi(n-x_1-x_2-\cdots-x_d))$. The estimate \eqref{eq:BernoulliSemiElliptic_4} is also valid here but \eqref{eq:BernoulliSemiElliptic_5} and \eqref{eq:BernoulliSemiElliptic_3} fail to hold (for reasons similar to those of Subsection \ref{subsec:ex3}).
\end{remark}
\begin{proof}
In view of the discussion proceeding the proposition, straightforward applications of Theorems \ref{thm:SubDecay} and \ref{thm:MainLocalLimit} yield \eqref{eq:BernoulliSemiElliptic_1} and \eqref{eq:BernoulliSemiElliptic_2} respectively. To see the global space-time estimates, we first observe that
\begin{align*}
\lefteqn{P_{\mathbf{m},\lambda}^{\#}(x_1,x_2,\dots,x_d)}\\
&=\sum_{j=1}^d\left(\left(\frac{1}{2m_j}\right)^{1/(2m_j-1)}-\left(\frac{1}{2m_j}\right)^{2m_j/(2m_j-1)}\right)\left(\frac{2^{m_j}|x_j|^{2m_j}}{\lambda_j}\right)^{1/(2m_j-1)}
\end{align*}
for $x=(x_1,x_2,\dots,x_d)$. From this it is easily checked that $|\cdot|_\mathbf{m}\asymp P_{\mathbf{m},\lambda}^{\#}$ (this can also be seen with the help of Corollary \ref{cor:LegendreCompareDiagonal}). Using the fact $0\in\Omega(\phi_{\mathbf{m},\lambda})$ has corresponding $\alpha_0=0$ and $P_0=P_{\mathbf{m},\lambda}$ which is semi-elliptic, $\phi_{\mathbf{m},\lambda}$ meets hypotheses  of Theorem \ref{thm:ExponentialEstimate}, Corollary \ref{cor:TimeDifferenceEstimate} and Theorem \ref{thm:TimeExponentialEstimate}. The estimates \eqref{eq:BernoulliSemiElliptic_3} follows immediately from Theorem \ref{thm:ExponentialEstimate}. Upon noting that $\mu_\phi+1=|\mathbf{1}:2\mathbf{m}|+|2\mathbf{m}:2\mathbf{m}|=|\mathbf{1}+2\mathbf{m}:2\mathbf{m}|$,  \eqref{eq:BernoulliSemiElliptic_4} follows from Corollary \ref{cor:TimeDifferenceEstimate}. Finally, the estimate \eqref{eq:BernoulliSemiElliptic_5} follows from Theorem \ref{thm:TimeSpaceDerivatives} once it is observed that $\mu_\phi+|\beta:2\mathbf{m}|+k=|\mathbf{1}+\beta+2k\mathbf{m}:2\mathbf{m}|$, $e=\{e_1,e_2,\dots,e_d\}$ is $P_{\mathbf{m},\lambda}$-fitted with weight $\mathbf{m}$ and $\prod_{j=1}|e_j|^{\beta_j}=1$.
\end{proof}

\subsection{Random walks on $\mathbb{Z}^d$: A look at the classical theory}\label{subsec:ClassicalLLT}

\noindent In this short subsection, we revisit the classical theory of random walks on $\mathbb{Z}^d$. We denote by $\mathcal{M}^1_d$, the set functions $\phi:\mathbb{Z}^d\rightarrow [0,1]$ satisfying
\begin{equation*}
\|\phi\|_1=\sum_{x\in\mathbb{Z}^d}\phi(x)=1,
\end{equation*}
i.e., $\mathcal{M}_d^1$ is the set of probability distributions on $\mathbb{Z}^d$. As discussed in the introduction, each $\phi\in\mathcal{M}^1_d$ drives a random walk on $\mathbb{Z}^d$ whose $n$th-step transition kernel $k_n$ is given by $k_n(x,y)=\phi^{(n)}(y-x)$ for $x,y\in\mathbb{Z}^d$. Taking our terminology from p. 72 of \cite{Spitzer1964}, we say that $\phi\in\mathcal{M}_d^1$ is \textit{genuinely $d$-dimensional} if $\supp(\phi)$ is not contained in any $(d-1)$-dimensional affine subspace of $\mathbb{R}^d$; in this case, we also say that the associated random walk is genuinely $d$-dimensional. Our main focus throughout this subsection is on subset of $\phi\in\mathcal{M}_d^1$ which are genuinely $d$-dimensional with finite second moments. In contrast to the standard literature, we make no assumptions concerning periodicity/aperiodicity/irreducibility, c.f., \cite{Spitzer1964,Lawler2010}. In this generality, our formulation of the (classical) local limit theorem, Theorem \ref{thm:classicLLT}, naturally contains a prefactor $\Theta$ which nicely describes the support of $\phi^{(n)}$ and hence the random walk's periodic structure.\\

\noindent Our first two results, Lemma \ref{lem:constantphase} and Proposition \ref{prop:OmegaisSubgroup} are stated for the general class of $\phi\in\mathcal{M}_d^1$; one should note that both results fail to hold in the case that $\phi$ is generally complex valued. The lemma and proposition highlight the importance of the set $\Omega(\phi)$ and, in particular, its inherent group structure.  This intrinsic structure (and much more) was also recognized by B. Schreiber in his study of (complex valued) measure algebras on locally compact abelian groups \cite{Schreiber1970}. In fact, Schreiber's results can be used to prove Lemma \ref{lem:constantphase} and Proposition \ref{prop:OmegaisSubgroup}; although, in our context, the proofs are straightforward and so we proceed directly.

\begin{lemma}\label{lem:constantphase}
Let $\phi\in\mathcal{M}_d^1$. Then $\Omega(\phi)$ depends only on $\supp(\phi)$ in the sense that, if $\supp(\phi_1)=\supp(\phi_2)$ for $\phi_1,\phi_2\in\mathcal{M}_d^1$, then $\Omega(\phi_1)=\Omega(\phi_2)$. Furthermore, for each $\xi\in\Omega(\phi)$,  there exists $\omega(\xi)\in(-\pi,\pi]$ such that
\begin{equation*}
\hat\phi(\xi)=e^{i\omega(\xi)}=e^{ix\cdot\xi}
\end{equation*}
for all $x\in\supp(\phi)$.
\end{lemma}
\begin{proof}
We shall use the following property of complex numbers. If $\{z_1,z_2,\dots\}\subseteq\mathbb{C}$ satisfy
\begin{equation*}
\sum_k^\infty |z_k|=1=\Big|\sum_{k=1}^\infty z_k\Big|,
\end{equation*}
then, for some $\omega\in(-\pi,\pi]$, $z_k=r_ke^{i\omega}$ for all $k$. Thus, whenever $\xi\in\Omega(\phi)$, i.e.,
\begin{equation*}
|\hat\phi(\xi)|=\Big|\sum_{x\in\mathbb{Z}^d}\phi(x)e^{ix\cdot\xi}\Big|=1,
\end{equation*}
there exists $\omega=\omega(\xi)\in(-\pi,\pi]$ for which
\begin{equation}\label{eq:constantphase_1}
e^{ix\cdot\xi}=e^{i\omega(\xi)}
\end{equation}
for all $x\in\supp(\phi)$. In particular, this shows that $\Omega(\phi)$ depends only on $\supp(\phi)$. Further, observe that
\begin{equation}\label{eq:constantphase_2}
\hat\phi(\xi)=\sum_{x\in\mathbb{Z}^d}\phi(x)e^{ix\cdot\xi}=e^{i\omega(\xi)}\sum_{x\in\mathbb{Z}^d}\phi(x)=e^{i\omega(\xi)}
\end{equation}
and so the result follows upon combining \eqref{eq:constantphase_1} and \eqref{eq:constantphase_2}.
\end{proof}

\begin{proposition}\label{prop:OmegaisSubgroup}
Let $\phi\in\mathcal{M}_d^1$. Then $\Omega(\phi)$ is a subgroup of $\mathbb{T}^d$ and
\begin{equation*}
\hat\phi\Big\vert_{\Omega(\phi)}:\Omega(\phi)\rightarrow \mathbb{S}^1
\end{equation*}
is a homomorphism of groups; here, $\mathbb{T}^d$ is taken to have the canonical group structure and $\mathbb{S}^1=\{z\in\mathbb{C}:|z|=1\}$.
\end{proposition}

\begin{proof}
It is obvious that $0\in\Omega(\phi)$; hence $\Omega(\phi)$ is non-empty. Let $\xi_1,\xi_2\in\Omega(\phi)$ and, in view of Lemma \ref{lem:constantphase},
\begin{equation*}
\hat\phi(\xi_2-\xi_1)=\sum_{x\in\supp(\phi)}\phi(x)e^{ix\cdot(\xi_1-\xi_2)}=\sum_{x\in\supp(\phi)}\phi(x)\hat\phi(\xi_2)\hat\phi(\xi_1)^{-1}=\hat\phi(\xi_2)\hat\phi(\xi_1)^{-1}
\end{equation*}
and thus $\xi_2-\xi_1\in\Omega(\phi)$ because $|\hat\phi(\xi_2-\xi_1)|=|\hat\phi(\xi_2)\hat\phi(\xi_1)^{-1}|=1$. As $\Omega(\phi)$ is non-empty and closed under subtraction, we conclude at once that $\Omega(\phi)$ is a subgroup of $\mathbb{T}^d$ and the restriction of $\hat\phi$ to $\Omega(\phi)$ is a homomorphism.
\end{proof}

\noindent We now begin to develop what is needed to recapture and reformulate the classical local limit theorem in the general case that $\phi\in\mathcal{M}_d^1$ is genuinely $d$-dimensional and has finite second moments. In this case, the \textit{mean} $\alpha_\phi\in\mathbb{R}^d$ and \textit{covariance} $C_\phi\in\MdR$ of $\phi$ are defined respectively by
\begin{equation*}
\{\alpha_\phi\}_k=\sum_{x\in\mathbb{Z}^d}x_k\phi(x)\hspace{.5cm}\mbox{ for }k=1,2,\dots,d
\end{equation*}
and
\begin{equation*}
\{C_{\phi}\}_{k,l}=\sum_{x\in\mathbb{Z}^d}(x_k-\{\alpha_\phi\}_k)(x_l-\{\alpha_\phi\}_l)\phi(x)\hspace{.5cm}\mbox{ for }k,l=1,2,\dots,d.
\end{equation*}

\begin{proposition}\label{prop:ProbabilityPosHomType}
Let $\phi\in\mathcal{M}_d^1$ be genuinely $d$-dimensional with finite second moments and let $\alpha_\phi$ and $C_\phi$ be the mean and covariance of $\phi$ as defined above. Set
\begin{equation*}
P_\phi(\xi)=\frac{1}{2}\xi\cdot C_\phi\xi
\end{equation*}
for $\xi\in\mathbb{R}^d$. Then each $\xi_0$ is of positive homogeneous type for $\hat\phi$ with $\alpha_{\xi_0}=\alpha_\phi$ and positive homogeneous polynomial $P_{\xi_0}=P_\phi$. In particular, $\mu_{\phi}=\mu_{P_\phi}=d/2$.
\end{proposition}
\begin{proof}
When $\phi$ is genuinely $d$-dimensional, it is well-known that the covariance form
\begin{equation*}
\xi\mapsto \mbox{Cov}(\phi)(\xi)=\xi\cdot C_\phi\xi
\end{equation*}
is positive definite (when $\alpha_\phi=0$, $\supp(\phi)$ contains a basis of $\mathbb{R}^d$ and when $\alpha_\phi\neq 0$, an appropriate shift does the trick). Upon noting that $2^{-1}I\in\Exp(P_\phi)$, we conclude that $P_\phi$ is a positive homogeneous polynomial. Observe that, for $\Gamma(\xi)=\log(\hat\phi(\xi+\xi_0)/\hat\phi(\xi_0))$,
\begin{align*}
\partial_k\Gamma(0)&=\frac{\partial_k\hat\phi(\xi_0)}{\hat\phi(\xi_0)}\\
&=\frac{1}{\hat\phi(\xi_0)}\sum_{x\in\supp(\phi)}ix_k\phi(x)e^{ix\cdot\xi_0}\\
&=\frac{1}{\hat\phi(\xi_0)}\sum_{x\in\supp(\phi)}ix_k\phi(x)e^{i\omega(\xi_0)}\\
&=\frac{e^{i\omega(\xi_0)}}{\hat\phi(\xi_0)}\sum_{x\in\supp(\phi)}ix_k\phi(x)\\
&=i\{\alpha_\phi\}_k
\end{align*}
for all $k=1,2\dots d$, where we have used Lemma \ref{lem:constantphase}. By analogous reasoning, which again makes use of the lemma, $\partial_{k,l}\Gamma(0)=-\{C_\phi\}_{k,l}$ for $k,l=1,2,\dots,d$. Consequently,
\begin{equation}\label{eq:ProbabilityPosHomType}
\begin{split}
\Gamma(\xi)&=\sum_{k=1}^d\partial_k\Gamma(0)\xi_k+\sum_{k,l=1}^d\frac{1}{2}\partial_{k,l}\Gamma(0)\xi_k\xi_l+o(|\xi|^2)\\
&=i\alpha_\phi\cdot\xi-P_\phi(\xi)+o(|\xi|^2),
\end{split}
\end{equation}
as $\xi\rightarrow 0$, where we have used the positive definiteness $P_\phi$ to rewrite the error. From this it follows immediately that $\xi_0$ is of positive homogeneous type for $\hat\phi$ with $\alpha_{\xi_0}=\alpha_\phi$ and positive homogeneous polynomial $P_{\xi_0}=P_\phi$.
\end{proof}

\noindent We now present the classical local limit theorem in a new form. Assuming the notation of the previous proposition, the attractor $G_\phi=H_{P_\phi}$ which appears below is the generalized Gaussian density given by \eqref{eq:Gaussian} (see p. 25 of \cite{Lawler2010}). Let us also note that, in view of the previous proposition and Proposition \ref{prop:OmegaisFinite}, $\Omega(\phi)$ is finite.

\begin{theorem}\label{thm:classicLLT}
Let $\phi\in\mathcal{M}_d^1$ be genuinely $d$-dimensional with finite second moments. Then there exists positive constants $C$ and $C'$ for which
\begin{equation}\label{eq:classicLLT_1}
Cn^{-d/2}\leq \sup_{x\in\mathbb{Z}^d}\phi^{(n)}(x)\leq C'n^{-d/2}
\end{equation}
for all $n\in\mathbb{N}_+$. Furthermore,
\begin{equation}\label{eq:classicLLT_2}
\phi^{(n)}(x)=n^{-d/2}\Theta(n,x)G_\phi\left(\frac{x-n\alpha_\phi}{\sqrt{n}}\right)+o(n^{-d/2})
\end{equation}
uniformly for $x\in\mathbb{Z}^d$, where $\Theta:\mathbb{N}_+\times\mathbb{Z}^d$ is dependent only on $\supp(\phi)$ in the sense of Lemma \ref{lem:constantphase} and is given (equivalently) by
\begin{equation}\label{eq:classicLLT_3}
\Theta(n,x)=\sum_{\xi\in\Omega(\phi)}e^{i(n\omega(\xi)-x\cdot\xi)}=\sum_{\xi\in\Omega(\phi)}\cos(n\omega(\xi)-x\cdot\xi);
\end{equation}
here, $\omega(\xi)\in(-\pi,\pi]$ is that given by Lemma \ref{lem:constantphase} for each $\xi\in\Omega(\phi)$.
\end{theorem}
\begin{proof}
The hypotheses of the present theorem are weaker than those of Theorems \ref{thm:SubDecay} and \ref{thm:MainLocalLimit} as the latter theorems require $\phi$ to have finite moments of all orders. However, what is really used in the proof of the Theorem \ref{thm:MainLocalLimit} is the condition that, for each $\xi_0\in\Omega(\phi)$, $\Gamma_{\xi_0}$ can be written in the form \eqref{eq:ProbabilityPosHomType} where $P_{\xi_0}=P_\phi$ is positive definite (in the general case that $\phi$ is complex valued, it is not known a priori how many terms in the Taylor expansion for $\Gamma_{\xi_0}$ are needed for this to be true). Under the present hypotheses and in view of Proposition \ref{prop:ProbabilityPosHomType}, the proof of Theorem \ref{thm:MainLocalLimit} pushes through with no modification and so we apply it (or simply its conclusion). As an immediate consequence, we obtain \eqref{eq:classicLLT_1} because Theorem \ref{thm:SubDecay} follows directly from Theorem \ref{thm:MainLocalLimit}. It remains to show that the local limit yielded by Theorem \ref{thm:MainLocalLimit} can be written in the form \eqref{eq:classicLLT_2}.

By virtue of Proposition \ref{prop:ProbabilityPosHomType}, we have  $\alpha_\xi=\alpha_\phi$, $P_{\xi}=P_\phi$ for all $\xi\in\Omega(\phi)$ and, moreover $\mu_{\phi}=\mu_P=d/2$. Noting that all $\xi\in\Omega(\phi)$ have corresponding positive homogeneous polynomials of the same order (because the polynomials are identical), all appear in the local limit. Consequently,
\begin{align*}
\phi^{(n)}(x)&=\sum_{\xi\in\Omega(\phi)}e^{-ix\cdot\xi}(\hat\phi(\xi))^nH_{P_\phi}^n\left(x-n\alpha_\phi\right)+o(n^{-d/2})\\\nonumber
&=\left(\sum_{\xi\in\Omega(\phi)}e^{-ix\cdot\xi}(\hat\phi(\xi))^n\right)n^{-d/2}H_{P_\phi}\left(n^{-1/2}\left(x-n\alpha_\phi\right)\right)+o(n^{-d/2})\\
&=n^{-d/2}\left(\sum_{\xi\in\Omega(\phi)}e^{i(n\omega(\xi)-x\cdot\xi)}\right)G_\phi\left(n^{-1/2}\left(x-n\alpha_\phi\right)\right)+o(n^{-d/2})\\
&=n^{-d/2}\Theta(n,x)G_\phi\left(\frac{x-n\alpha_\phi}{\sqrt{n}}\right)+o(n^{-d/2})
\end{align*}
uniformly for $x\in\mathbb{Z}^d$. In view of Lemma \ref{lem:constantphase}, it is clear that $\Theta$ depends only on $\supp(\phi)$ and so to complete the proof, we need only to verify the second equality in \eqref{eq:classicLLT_3}. Using the fact that $\Omega(\phi)$ is a subgroup of $\mathbb{T}^d$ in view of Proposition \ref{prop:OmegaisSubgroup}, for each $\xi\in\Omega(\phi)$, $-\xi\in\Omega(\phi)$ and therefore
\begin{align*}
\Theta(n,x)&=\frac{1}{2}\left(\sum_{\xi\in\Omega(\phi)}e^{i(n\omega(\xi)-x\cdot\xi)}+\sum_{\xi\in\Omega(\phi)}e^{i(n\omega(-\xi)-x\cdot(-\xi))}\right)\\
&=\sum_{\xi\in\Omega(\phi)}\cos(n\omega(\xi)-x\cdot\xi)
\end{align*}
where we have noted that $\omega(-\xi)=-\omega(\xi)$ for each $\xi\in\Omega(\phi)$.
\end{proof}

\noindent By close inspection of the theorem, one expects that in general $\Theta$ can help us describe the support of $\phi^{(n)}$ and hence the periodicity of the associated random walk. This turns out to be the case as our next theorem shows.

\begin{theorem}\label{thm:PeriodicityofRandomWalk}
Let $\phi\in\mathcal{M}_d^1$ be genuinely $d$-dimensional with finite second moments and define $\Theta:\mathbb{N}_+\times\mathbb{Z}^d\rightarrow\mathbb{R}$ by \eqref{eq:classicLLT_3}. Then
\begin{equation}\label{eq:suppcondition}
\supp\left(\phi^{(n)}\right)\subseteq\supp(\Theta(n,\cdot))
\end{equation}
for all $n\in\mathbb{N}_+$. Further, if
\begin{equation*}
\limsup_n|\Theta(n,x+\lfloor n\alpha_\phi\rfloor)|>0
\end{equation*}
for $x\in\mathbb{Z}^d$, then
\begin{equation*}
\limsup_n n^{\mu_\phi}\phi^{(n)}(x+\lfloor n\alpha_\phi\rfloor)>0.
\end{equation*}
\end{theorem}

\begin{proof}
In view of Lemma \ref{lem:constantphase}, for any $x_0\in\supp(\phi)$, $\omega(\xi)=x_0\cdot\xi$ for all $\xi\in\Omega(\phi)$. Therefore, for any $x_0\in\supp(\phi)$,
\begin{equation*}
\Theta(n,x)=\sum_{\xi\in\Omega(\phi)}\cos((nx_0-x)\cdot\xi)
\end{equation*}
for all $n\in\mathbb{N}_+$ and $x\in\mathbb{Z}^d$ and, in particular,
\begin{equation*}
\Theta(1,x_0)=\sum_{\xi\in\Omega(\phi)}\cos(0)=\#(\Omega(\phi))>0
\end{equation*}
whence $\supp(\phi)\subseteq \supp(\Theta(1,\cdot))$. The inclusion \eqref{eq:suppcondition} follows straightforwardly by induction. For the second conclusion, an appeal to Theorem \ref{thm:classicLLT} shows that, for sufficiently large $n$,
\begin{equation*}
n^{d/2}\phi^{(n)}(x+\lfloor n\alpha_\phi\rfloor)\geq |\Theta(n,x+\lfloor n\alpha_\phi\rfloor)G_\phi(n^{-1/2}(x+\lfloor n\alpha_\phi\rfloor-n\alpha_\phi))|/2
\end{equation*}
for all $x\in\mathbb{R}^d$. Of course, for any fixed $x$, 
\begin{equation*}\lim_{n\rightarrow\infty}|G_\phi(n^{-1/2}(x+\lfloor n\alpha_\phi\rfloor-n\alpha_\phi))|=G_\phi(0)>0
\end{equation*}
 and from this the assertion follows without trouble.

\end{proof}

\noindent To illustrate the utility of the function $\Theta$, we consider a class of examples which generalizes simple random walk. For a fixed $\mathbf{m}=(m_1,m_2,\dots,m_2)\in\mathbb{N}_+^d$ define $\phi_\mathbf{m}\in\mathcal{M}_d^1$ by
\begin{equation*}
\phi_{\mathbf{m}}(m_je_j)=\phi_{\mathbf{m}}(-m_je_j)=\frac{1}{2d}
\end{equation*}
for $j=1,2,\dots,d$ and set $\phi_{\mathbf{m}}(x)=0$ otherwise; here, $\{e_1,e_2,\dots,e_d\}$ is the standard euclidean basis in $\mathbb{R}^d$. This generates the random walk with statespace $\{(k_1 m_1,k_2m_2,\dots, k_d m_d):k_j\in\mathbb{Z}\mbox{ for }j=1,2,\dots,d\}$. We have:
\begin{proposition}
Let $\Theta_{\mathbf{m}}:\mathbb{N}_+\times\mathbb{Z}^d\rightarrow\mathbb{R}$ be that associated to $\phi_{\mathbf{m}}$ by \eqref{eq:classicLLT_3}. Then
\begin{align*}
\lefteqn{\Theta_{\mathbf{m}}(n,x)}\\
&=
\begin{cases}
2\left(\prod_{j=1}^dm_j\right) & \mbox{ if }m_j\vert x_j \mbox{ for all }j=1,2,\dots d\mbox{ and }n-|x:\mathbf{m}|\mbox{ is even}\\
0 & \mbox{ otherwise}.
\end{cases}
\end{align*}
\end{proposition}
\begin{proof}
For notational convenience, we write $\phi=\phi_{\mathbf{m}}$ and $\Theta=\Theta_{\mathbf{m}}$. Observe that $\hat\phi(\xi)=(1/d)\sum_{j=1}^d\cos(m_j\xi_j)$ for $\xi=(\xi_1,\xi_2,\dots,\xi_d)\in\mathbb{T}^d$ and so by a direct computation,
\begin{align*}
\Omega(\phi)&=\Omega_{\mbox{e}}\,\dot\cup\,\Omega_{\mbox{o}}\\
&=\left\{\pi\left(\frac{k_1}{m_2},\frac{k_2}{m_2},\dots,\frac{k_d}{m_d}\right):\mathbf{k}\in Z_{\mbox{e}}\right\}\dot\bigcup\left\{\pi\left(\frac{k_1}{m_2},\frac{k_2}{m_2},\dots,\frac{k_d}{m_d}\right):\mathbf{k}\in Z_{\mbox{o}}\right\},
\end{align*}
where
\begin{equation*}
Z_{\mbox{e}}=\{\mathbf{k}\in\mathbb{Z}^d:-m_j<k_j\leq m_j\mbox{ and }k_j\mbox{ is even for }j=1,2,\dots, d\}
\end{equation*}
and
\begin{equation*}
Z_{\mbox{o}}=\{\mathbf{k}\in\mathbb{Z}^d:-m_j<k_j\leq k_j\mbox{ and }m_j\mbox{ is odd for }j=1,2,\dots, d\}.
\end{equation*}
With this decomposition, we immediately observe that
\begin{equation*}
\omega(\xi)=
\begin{cases}
0 & \mbox{ if }\xi\in\Omega_{\mbox{e}}\\
\pi & \mbox{ if }\xi\in\Omega_{\mbox{o}}.
\end{cases}
\end{equation*}
In the case that $m_j\big\vert x_j$ for $j=1,2,\dots,d$,
\begin{align*}
\Theta(n,x)&=\sum_{\xi\in\Omega_{\mbox{\tiny e}}}e^{i(0n-x\cdot \xi)}+\sum_{\xi\in\Omega_{\mbox{\tiny o}}}e^{i(\pi n-x\cdot \xi)}\\
&=\sum_{\mbox{k}\in Z_{\mbox{\tiny e}}}\exp\left(-i\pi \sum_{j=1}^d \frac{k_jx_j}{m_j}\right)+\sum_{\mbox{k}\in Z_{\mbox{\tiny o}}}\exp\left(i\pi\left(n-\sum_{j=1}^d \frac{k_jx_j}{m_j}\right)\right)\\
&=\#(Z_{\mbox{e}})+\exp\left(i\pi\left(n-\sum_{j=1}^d\frac{x_j}{m_j}\right)\right)\#(Z_{\mbox{o}})
\end{align*}
where we have used \eqref{eq:classicLLT_3}. Now $\#(Z_{\mbox{e}})=\#(Z_{\mbox{o}})=\prod_{j=1}^d m_j$ and so it follows that
\begin{equation*}
\Theta(n,x)=\left (1+e^{i\pi(n-|x:\mathbf{m}|)}\right)\prod_{j=1}^d m_j=
\begin{cases}
2\left(\prod_{j=1}^d m_j\right) & \mbox{ if }n-|x:\mathbf{m}|\mbox{ is even}\\
0 & \mbox{ if }n-|x:\mathbf{m}| \mbox{ is odd.}
\end{cases}
\end{equation*}
In the case that $m_l\not \big\vert x_l$ for some $l=1,2,\dots,d$, observe that
\begin{equation}\label{eq:mSRW_1}
\begin{split}
\Theta(n,x)&=\sum_{\xi\in\Omega_{\mbox{\tiny e}}}e^{-i\xi\cdot x}+\sum_{\xi\in\Omega_{\mbox{\tiny o}}}e^{i(\pi n-\xi\cdot x)}\\
&=\prod_{j=1}^d \sum_{\substack{m_j<k_j\leq m_j\\ k_j\mbox{\tiny even}}}\exp\left(-i\pi\frac{x_jk_j}{m_j}\right)\\
&\quad+e^{i\pi n}\prod_{j=1}^d \sum_{\substack{m_j<k_j\leq m_j\\ k_j\mbox{\tiny odd}}}\exp\left(-i\pi\frac{x_jk_j}{m_j}\right).
\end{split}
\end{equation}
Focusing on the $l$th multiplicand in the first term, it is straightforward to see that
\begin{align*}
\lefteqn{\left(e^{-2\pi ix_l/m_l}-1\right)\sum_{\substack{m_l<k_j\leq m_l\\ k_l\mbox{\tiny even}}}\exp\left(-i\pi\frac{x_lk_l}{m_l}\right)}\\
&=\sum_{\substack{m_l<k_j\leq m_l\\ k_l\mbox{\tiny even}}}\exp\left(-i\pi\frac{x_l(k_l+2)}{m_l}\right)-\exp\left(-i\pi\frac{x_lk_l}{m_l}\right)=0
\end{align*}
and since $m_l\not\big\vert x_l$, we can immediately conclude that
\begin{equation*}
\sum_{\xi\in\Omega_{\mbox{\tiny e}}}e^{-i\xi\cdot x}=\sum_{\substack{m_l<k_j\leq m_l\\ k_l\mbox{\tiny even}}}\exp\left(-i\pi\frac{x_lk_l}{m_l}\right)\prod_{j\neq l}\sum_{\substack{m_j<k_j\leq m_j\\ k_j\mbox{\tiny even}}}\exp\left(-i\pi\frac{x_jk_j}{m_j}\right)=0.
\end{equation*}
An analogous argument shows that $\sum_{\xi\in\Omega_{\mbox{\tiny o}}}e^{i(\pi n-\xi\cdot x)}=0$ and so, in view of \eqref{eq:mSRW_1}, it follows that $\Theta(n,x)=0$ as desired.
\end{proof}

\noindent Simple random walk is, of course, the random walk defined by $\phi_{\mathbf{m}}$ where $\mathbf{m}=(1,1,\dots,1)$. In this case, the proposition yields
\begin{equation*}
\Theta_{(1,1,\dots,1)}(n,x)=\begin{cases}
             2 & \mbox{if }n-x_1-x_2-\cdots-x_d\mbox{ is even}\\
             0 & \mbox{if }n-x_1-x_2-\cdots-x_d\mbox{ is odd};
            \end{cases}
\end{equation*}
this captures the walk's well-known periodicity.\\

\noindent We end this section by showing that Theorem \ref{thm:ExponentialEstimate} provides a Gaussian (upper) bound in the case that $\phi\in\mathcal{M}_d^1$ is finitely supported and genuinely $d$-dimensional. To obtain a matching lower bound, it is necessary to make some assumptions concerning aperiodicity.

\begin{theorem}\label{thm:RWGaussianBound} Let $\phi\in\mathcal{M}_d^1$ be finitely supported and genuinely $d$-dimensional with mean $\alpha_\phi\in\mathbb{R}^d$. Then, there exist positive constants $C$ and $M$ for which
\begin{equation*}
\phi^{(n)}(x)\leq\frac{C}{n^{d/2}}\exp\left(-M|x-n\alpha_\phi|^2/n\right)
\end{equation*}
for all $n\in\mathbb{N}_+$ and $x\in\mathbb{Z}^d$.
\end{theorem}
\begin{proof} In view of Proposition \ref{prop:ProbabilityPosHomType}, our hypotheses guarantee that every $\xi\in\Omega(\phi)$ is of positive homogeneous type with corresponding $\alpha_\xi=\alpha_\phi$ and positive homogeneous polynomial $P_{\xi}=P_\phi$; here $\mu_\phi=\mu_{P_\phi}=d/2$ and $R_\phi=\Re P_\phi=P_\phi$. An appeal to Theorem \ref{thm:ExponentialEstimate} gives positive constants $C$ and $M$ for which
\begin{equation*}
\phi^{(n)}(x)=|\phi^{(n)}(x)|\leq \frac{C}{n^{d/2}}\exp\left(-nMP_\phi^{\#}\left((x-n\alpha_\phi\right)/n)\right)
\end{equation*}
for all $n\in\mathbb{N}_+$ and $x\in\mathbb{Z}^d$. Upon noting that $P_\phi^{\#}$ is necessarily quadratic and positive definite by virtue of Proposition \ref{prop:LegendreContinuousPositiveDefinite}, we conclude that $P_\phi^{\#}\asymp|\cdot|^2$ and the theorem follows at once.
\end{proof}

\section{Appendix}\label{sec:Appendix}

\subsection{Properties of contracting one-parameter groups}

\noindent The following proposition is standard \cite{Hazod2001}.
\begin{proposition}\label{prop:tEProperties}
Let $E,G\in \MdR$ and $A\in \GldR$. Also, let $E^*\in\MdR$ denote the adjoint of $E$. Then for all $t,s>0$, the following statements hold:\\

\vspace{.1cm}
\begin{tabular}{lllll}
$\bullet$ $1^E=I$ &  $\bullet$ $t^{E^*}=(t^E)^*$ &   $\bullet$ If $EG=GE$ then $t^Et^G=t^{E+G}$\\
\vspace{.1cm}\\
$\bullet$ $(st)^E=s^Et^E$ & $\bullet$ $At^EA^{-1}=t^{AEA^{-1}}$&  $\bullet$ $\det(t^E)=t^{\tr E}$\\
\vspace{.1cm}\\
$\bullet$ $(t^E)^{-1}=t^{-E}$ & 
\end{tabular}
\end{proposition}

\begin{lemma}\label{lem:OperatorBoundsforContractingGroup}
Let $\{T_t\}\subseteq\GldR$ be a continuous one-parameter contracting group. Then, for some $E\in \GldR$, $T_t=t^E$ for all $t>0$. Moreover, there exists a positive constant $C$ for which
\begin{equation*}
\|T_t\|\leq C+t^{\|E\|}
\end{equation*}
for all $t>0$.
\end{lemma}
\begin{proof}
The representation $T_t=t^E$ for some $E\in\MdR$ follows from the Hille-Yosida construction. If for some non-zero vector $\eta$, $E\eta=0$, then $t^E\eta=\eta$ for all $t>0$ and this would contradict our assumption that $\{T_t\}$ is contracting. Hence $E\in\GldR$ and, in particular, $\|E\|>0$. From the representation $T_t=t^E$ it follows immediately that $\|T_t\|\leq t^{\|E\|}$ for all $t\geq 1$ and so it remains to estimate $\|T_t\|$ for $t<1$. Given that $\{T_t\}$ is continuous and contracting, the map $t\mapsto \|T_t\|$ is continuous and approaches $0$ as $t\rightarrow 0$ and so it is necessarily bounded for $0<t\leq 1$.
\end{proof}

\begin{lemma}\label{lem:SpectralEstimateforContractingGroup}
Let $E\in\GldR$ be diagonalizable with strictly positive spectrum. Then $\{t^E\}$ is a continuous one-parameter contracting group. Moreover, there is a positive constant $C$ such that
\begin{equation*}
\|t^E\|\leq Ct^{\lambda_{\mbox{\tiny{max}}}}
\end{equation*}
for all $t\geq 1$ and
\begin{equation*}
\|t^E\|\leq Ct^{\lambda_{\mbox{\tiny{min}}}}
\end{equation*}
for all $0<t<1$, where $\lambda_{\mbox{\tiny{max}}}=\max(\Spec(E))$ and $\lambda_{\mbox{\tiny{min}}}=\min(\Spec(E))$.
\end{lemma}
\begin{proof}
Let $A\in\GldR$ be such that $A^{-1}EA=D=\diag(\lambda_1,\lambda_2,\dots,\lambda_d)$ where necessarily $\Spec(E)=\Spec(D)=\{\lambda_1,\lambda_2,\dots,\lambda_d\}\subseteq (0,\infty)$. It follows from the spectral mapping theorem that $\Spec(t^D)=\{t^{\lambda_1},t^{\lambda_2},\dots,t^{\lambda_d}\}$ for all $t>0$ and moreover, because $t^D$ is symmetric,
\begin{equation*}
\|t^D\|\leq \max(\{t^{\lambda_1},t^{\lambda_2},\dots,t^{\lambda_d}\})
=\begin{cases}
t^{\lambda_{\mbox{\tiny{max}}}} & \mbox{if }t\geq 1\\
t^{\lambda_{\mbox{\tiny{min}}}} & \mbox{if }t<1.
\end{cases}
\end{equation*}
By virtue of Proposition \ref{prop:tEProperties}, we have
\begin{equation*}
\|t^E\|=\|At^DA^{-1}\|\leq \|A\|\|t^D\|\|A^{-1}\|\leq C\|t^D\|=C\times
\begin{cases}
t^{\lambda_{\mbox{\tiny{max}}}} & \mbox{if }t\geq 1\\
t^{\lambda_{\mbox{\tiny{min}}}} & \mbox{if }t<1
\end{cases}
\end{equation*}
for $t>0$ where $C=\|A\|\|A^{-1}\|$; in particular, $\{t^E\}$ is contracting because $\lambda_{\mbox{\tiny{min}}}>0$.
\end{proof}

\begin{proposition}\label{prop:ContractingLimits}
Let $\{T_t\}_{t>0}\subseteq\GldR$ be a continuous one-parameter contracting group.  Then, for all non-zero $\xi\in\mathbb{R}^d$,
\begin{equation*}
\lim_{t\rightarrow 0}|T_t \xi|=0\hspace{.5cm}\mbox{ and }\hspace{.5cm}\lim_{t\rightarrow\infty}|T_t \xi|=\infty.
\end{equation*}
\end{proposition}
\begin{proof}
The validity of the first limit is clear. Upon noting that $|\xi|=|T_{1/t}T_t\xi|\leq \|T_{1/t}\||T_t \xi|$ for all $t>0$, the second limit follows at once.
\end{proof}


\begin{proposition}\label{prop:ScalefromSphere}
Let $\{T_t\}_{t>0}\subseteq\GldR$ be a continuous one-parameter contracting group. There holds the following:
\begin{enumerate}[a)]
\item\label{st:ScalefromSphere_1} For each non-zero $\xi\in \mathbb{R}^d$, there exists $t>0$ and $\eta\in S$ for which $T_t\eta=\xi$. Equivalently,
\begin{equation*}
\mathbb{R}^d\setminus\{0\}=\{T_t\eta:t>0\mbox{ and }\eta\in S\}.
\end{equation*}
\item\label{st:ScalefromSphere_2} For each sequence $\{\xi_n\}\subseteq\mathbb{R}^d$ such that $\lim_n|\xi_n|=\infty$, $\xi_n=T_{t_n}\eta_n$ for each $n$, where $\{\eta_n\}\subseteq S$ and $t_n\rightarrow\infty$ as $n\rightarrow\infty$.
\item\label{st:ScalefromSphere_3} For each sequence $\{\xi_n\}\subseteq\mathbb{R}^d$ such that $\lim_n|\xi_n|=0$, $\xi_n=T_{t_n}\eta_n$ for each $n$, where $\{\eta_n\}\subseteq S$ and $t_n\rightarrow 0$ as $n\rightarrow\infty$.
\end{enumerate}
\end{proposition}
\begin{proof}
In view of Proposition \ref{prop:ContractingLimits}, the assertion \ref{st:ScalefromSphere_1}) is a straightforward application of the intermediate value theorem. For \ref{st:ScalefromSphere_2}), suppose that $\{\xi_n\}\subseteq\mathbb{R}^d$ is such that $|\xi_n|\rightarrow \infty$ as $n\rightarrow\infty$. In view of \ref{st:ScalefromSphere_1}), take $\{\eta_n\}\subseteq S$ and $\{t_n\}\subseteq (0,\infty)$ for which $\xi_n=T_{t_n}\eta_n$ for each $n$. In view of Lemma \ref{lem:OperatorBoundsforContractingGroup},
\begin{equation*}
\infty=\liminf_n |\xi_n|\leq\liminf_n \left(C+t_n^M\right)|\eta_n|\leq C+\liminf_nt_n^M,
\end{equation*}
where $C,M>0$, and therefore $t_n\rightarrow\infty$. If instead $\lim_n\xi_n=0$,
\begin{equation*}
\infty=\lim_{n\rightarrow\infty}\frac{|\eta_n|}{|\xi_n|}=\lim_{n\rightarrow\infty}\frac{|T_{1/t_n}\xi_n|}{|\xi_n|}\leq\limsup_n\|T_{1/t_n}\|\leq\limsup_n(C+(1/t_n)^M)
\end{equation*}
from which we see that $t_n\rightarrow 0$, thus proving \ref{st:ScalefromSphere_3}).
\end{proof}

\begin{proposition}\label{prop:ContractingCapturesCompact}
Let $\{T_t\}$ be a continuous contracting one-parameter group. Then for any open neighborhood $\mathcal{O}\subseteq\mathbb{R}^d$ of the origin and any compact set $K\subseteq\mathbb{R}^d$, $K\subseteq T_t(\mathcal{O})$ for sufficiently large $t$.
\end{proposition}
\begin{proof}
Assume, to reach a contradiction, that there are sequences $\{\xi_n\}\subseteq K$ and $t_n\rightarrow\infty$ for which $\xi_n\notin T_{t_n}(\mathcal{O})$ for all $n$. Because $K$ is compact, $\{\xi_n\}$ has a subsequential limit and so by relabeling, let us take sequences $\{\zeta_k\}\subseteq K$ and $\{r_k\}\subseteq (0,\infty)$ for which $\zeta_k\rightarrow \zeta$, $r_k\rightarrow\infty$ and $\zeta_k\notin T_{r_k}(\mathcal{O})$ for all $k$. Setting $s_k=1/r_k$ and using the fact that $\{T_t\}$ is a one-parameter group, we have $T_{s_k}\zeta_k\notin\mathcal{O}$ for all $k$ and so $\liminf_{k}|T_{s_k}\zeta_k|>0$, where $s_k\rightarrow 0$. This is however impossible because $\{T_t\}$ is contracting and so
\begin{equation*}
\lim_{k\rightarrow\infty}|T_{s_k}\zeta_k|\leq\lim_{k\rightarrow \infty}|T_{s_k}(\zeta_k-\zeta)|+\lim_{k\rightarrow\infty}|T_{s_k}\zeta|\leq C\lim_{k\rightarrow\infty}|\zeta_k-\zeta|+0=0
\end{equation*}
in view of Lemma \ref{lem:OperatorBoundsforContractingGroup}.
\end{proof}

\subsection{Properties of homogeneous functions on $\mathbb{R}^d$}

\begin{proposition}\label{prop:ComparableFunctions}
Let $\{T_t\}\subseteq \GldR$ be a contracting one-parameter group and let $R,Q:\mathbb{R}^d\rightarrow\mathbb{R}$ be continuous and homogeneous with respect to $\{T_t\}$. If $R$ is positive definite, then there exists $C>0$ such that
\begin{equation}\label{eq:ComparableFunctions_1}
|Q(\xi)|\leq C R(\xi)
\end{equation}
for all $\xi\in\mathbb{R}^d$. If both $Q$ and $R$ are positive definite, then
\begin{equation}\label{eq:ComparableFunctions_2}
Q\asymp R.
\end{equation}
\end{proposition}
\begin{proof}
Upon reversing the roles of $R$ and $Q$, it is clear that the \eqref{eq:ComparableFunctions_2} follows from \eqref{eq:ComparableFunctions_1} and so it suffices to prove \eqref{eq:ComparableFunctions_1}. Because $R$ is continuous and positive definite, it is strictly positive on $S$ and so, given that $Q$ is also continuous,
\begin{equation*}
C:=\sup_{\eta\in S}\frac{|Q(\eta)|}{R(\eta)}<\infty.
\end{equation*}
For any non-zero $\xi\in\mathbb{R}^d$, let $t>0$ be such that $\xi=T_t\nu$ for $\nu\in S$ in view of Proposition \ref{prop:ScalefromSphere} and observe that
\begin{equation*}
|Q(\xi)|=|Q(T_t\eta)|=t|Q(\eta)|\leq tCR(\eta)=CR(T_t\eta)=CR(\xi).
\end{equation*}
By invoking the continuity of $R$ and $Q$, the above estimate must also hold for $\xi=0$.
\end{proof}

\begin{proposition}\label{prop:NormGammaLittleOhR}
Let $E\in\GldR$ be diagonalizable with strictly positive spectrum and suppose that $R:\mathbb{R}^d\rightarrow\mathbb{R}$ is continuous, positive definite, and homogeneous with respect to $\{t^E\}$. Then, for any $\gamma>(\min(\Spec(E)))^{-1}$,
\begin{equation*}
|\xi|^{\gamma}=o(R(\xi))\mbox{ as }\xi\rightarrow 0.
\end{equation*}
\end{proposition}
\begin{proof}
By virtue of Lemma \ref{lem:SpectralEstimateforContractingGroup}, we know that $\{t^E\}$ is contracting and $\|t^E\|\leq Ct^{\lambda}$ for all $t<1$ where $\lambda=\min(\Spec(E))$ and $C>0$. Let $\{\xi_n\}$ be such that $\lim_n\xi_n\rightarrow 0$ and, in view of Proposition \ref{prop:ScalefromSphere}, let $\{\eta_n\}\subseteq S$ and $t_n\rightarrow 0$ be such that $\xi_n=t_n^E \eta_n$. Then
\begin{equation*}
\limsup_n\frac{|\xi_n|^\gamma}{R(\xi_n)}=\limsup_n\frac{|t_n^E\eta_n|^{\gamma}}{t_n R(\eta_n)}\leq\limsup_n \frac{(Ct_n^{\lambda}|\eta_n|)^{\gamma}}{t_nR(\eta_n)}\leq M\lim_n t_n^{\gamma\lambda-1}=0,
\end{equation*}
where
\begin{equation*}
M:=\sup_{\eta\in S}\frac{C^\gamma|\eta|^{\gamma}}{R(\eta)}
\end{equation*}
is finite because $R$ is continuous and positive definite.
\end{proof}

\begin{lemma}\label{lem:LittleOh}
Let $\mathbf{m}=(m_1,m_2,\dots,m_d)\in\mathbb{N}_+^d$, $D=\diag(m_1^{-1},m_2^{-1},\dots,m_d^{-1})\in\GldR$ and suppose that $R:\mathbb{R}^d\rightarrow\mathbb{R}$ is continuous, positive definite and homogeneous with respect to $\{t^D\}$. Then for any multi-index $\beta$ such that $|\beta:\mathbf{m}|>1$,
\begin{equation*}
\xi^{\beta}=o(R(\xi))\hspace{1cm}\mbox{as }\xi\rightarrow 0.
\end{equation*}
\end{lemma}
\begin{proof}
Put $\gamma=|\beta:\mathbf{m}|-1>0$ observe that
\begin{equation*}
\sup_{\eta\in S}\frac{|\eta^{\beta}|}{R(\eta)}:=M<\infty
\end{equation*}
because $R$ is continuous and non-zero on $S$. Let $\{\xi_n\}\subseteq \mathbb{R}^d$ be such that $|\xi_n|\rightarrow 0$ as $n\rightarrow\infty$. By virtue of Proposition \ref{prop:ScalefromSphere}, there are sequences $\{\eta_n\}\subseteq S$ and $\{t_n\}\subseteq(0,\infty)$ for which $t_n\rightarrow 0$ and $\xi_n=t_n^D\eta_n$ for all $n$. We see that
\begin{equation*}
\xi_n^{\beta}=(t_n^D\eta_n)^{\beta}=\Big(t_n^{(m_1)^{-1}}\eta_1\Big)^{\beta_1}\Big(t_n^{(m_2)^{-1}}\eta_2\Big)^{\beta_2}\cdots\Big(t_n^{(m_d)^{-1}}\eta_d\Big)^{\beta_d}=t^{|\beta:\mathbf{m}|}\eta_n^{\beta}
\end{equation*}
for each $n$. Therefore
\begin{equation*}
\limsup_n\frac{|\xi_n^{\beta}|}{R(\xi_n)}=\limsup_n\frac{t^{|\beta:\mathbf{m}|}|\eta_n^{\beta}|}{tR(\eta_n)}\leq \limsup_n M t_n^{\gamma}=0
\end{equation*}
as desired.
\end{proof}

\noindent For the remainder of this appendix, $P$ is a positive homogeneous polynomial and $R=\Re P$.
\begin{proposition}\label{prop:RSumEstimate}
For any compact set $K$, there are positive constants $M$ and $M'$ such that
\begin{equation*}
M R(\xi)\leq R(\xi+\zeta)+M'
\end{equation*}
for all $\xi\in\mathbb{R}^d$ and $\zeta\in K$.
\end{proposition}
\begin{proof}
Set $U=\overline{B}_{3/2}\setminus B_{1/2}=\{u\in\mathbb{R}^d:1/2\leq |u|\leq 3/2\}$ and
\begin{equation*}
M=\inf_{\eta\in S,u\in U}\frac{R(u)}{R(\eta)};
\end{equation*}
$M$ is necessarily positive because $R$ is continuous and positive definite.
For $E\in\Exp(P)$, $\{t^E\}$ is contracting and so it follows that for some $T>0$, $\left(\eta+t^{-E}\zeta\right)\in U$ for all $\eta\in S$, $\zeta\in K$ and $t> T$. Consider the set $V=\{\xi=t^E\eta\in\mathbb{R}^d:t>T,\eta\in S\}$ and observe that for any $\xi\in V$ and $\zeta\in K$, $t^{-E}(\xi+\zeta)=\eta+t^{-E}\zeta\in U$ for some $t>T$ and consequently
\begin{equation*}
\frac{R(\xi+\zeta)}{R(\xi)}=\frac{tR(\eta+t^{-E}\zeta)}{tR(\eta)}\geq M.
\end{equation*}
We have shown that $MR(\xi)\leq R(\xi+\zeta)$ for all $\xi\in V$ and $\zeta\in K$. To complete the proof, it remains to show that $R$ is bounded on $V^{\mathsf{C}}=\mathbb{R}^d\setminus V$; however, given the continuity of $R$, we need only verify that the set $V^{\mathsf{C}}$ is bounded. By virtue of Proposition \ref{prop:ScalefromSphere}, we can write
\begin{equation*}
V^{\mathsf{C}}=\{0\}\cup\left\{\xi=t^E\eta:t\leq T,\eta\in S\right\}.
\end{equation*}
and so, by an appeal to Lemma \ref{lem:OperatorBoundsforContractingGroup}, we see that $|\xi|\leq C+T^{\|E\|}$ for all $\xi\in V^{\mathsf{C}}$.
\end{proof}

\noindent Our final three results in this subsection concern estimates for $P$ and $R$ regarded as functions on $\mathbb{C}^d$. In what follows, $|\cdot|$ denotes the standard euclidean norm on $\mathbb{C}^d=\mathbb{R}^{2d}$ and $S$ denotes the $2d$-sphere.
\begin{proposition}\label{prop:RDominatesNorm}
For any $M,M'>0$, there exists $C>0$, for which
\begin{equation*}
|z|\leq C+M R(\xi)+M' R(\nu).
\end{equation*}
for all $z=\xi-i\nu\in\mathbb{C}^d$.
\end{proposition}
\begin{proof}
Define $Q(\xi,\nu)=MR(\xi)+M'R(\nu)$ for $(\xi,\nu)=z\in\mathbb{R}^{2d}$ and observe that $Q$ is positive definite. It suffices to show that there exists a set $V$ with bounded complement $V^{\mathsf{C}}=\mathbb{R}^{2d}\setminus V $ such that
\begin{equation}\label{eq:RDominatesNorm}
|z|=|(\xi,\nu)|\leq Q(\xi,\nu)
\end{equation}
for all $(\xi,\nu)\in V$. To this end, set
\begin{equation*}
N=\sup_{(\eta,\zeta)\in S}\frac{|(\eta,\zeta)|}{Q(\eta,\zeta)}
\end{equation*}
which is finite because $Q$ is strictly positive on $S$. Let $E\in\Exp(P)$ have real spectrum and recall that $E$ is diagonalizable with $\lambda:=\max(\Spec(E))<1$ in view of Proposition \ref{prop:PositiveHomogeneousPolynomialsareSemiElliptic}. An appeal to Lemma \ref{lem:SpectralEstimateforContractingGroup} gives $C>0$ for which $\|t^E\|\leq Ct^\lambda$ for all $t\geq 1$; the lemma also guarantees that $\{t^{E\oplus E}\}\subseteq\mbox{Gl}_{2d}(\mathbb{R})$ is contracting. Set $T=\max(\{1,(CN)^{1/(1-\lambda}\})$ and consider the set $V=\{(\xi,\nu)=t^{E\oplus E}(\eta,\zeta)\in\mathbb{R}^{2d}:t>T,(\eta,\zeta)\in S\}$. For any $(\xi,\nu)\in V$, we have
\begin{equation*}
\frac{|(\xi,\nu)|}{Q(\xi,\nu)}=\frac{|(t^E\eta,t^E\zeta)|}{Q(t^{E\oplus E}(\eta,\zeta))}\leq \frac{Ct^{\lambda}|(\eta,\zeta)|}{tQ(\eta,\zeta)}\leq Ct^{\lambda-1}N<N^{-1}N=1
\end{equation*}
and therefore \eqref{eq:RDominatesNorm} is satisfied. To see that $V^{\mathsf{C}}$ is bounded, one simply repeats the argument given in the proof of Proposition \ref{prop:RDominatesNorm} where, in this case, Proposition \ref{prop:ScalefromSphere} and Lemma \ref{lem:OperatorBoundsforContractingGroup} are applied to the one-parameter contracting group $\{t^{E\oplus E}\}$.
\end{proof}

\noindent By considering only real arguments $\xi\in\mathbb{R}^d$, Proposition \ref{prop:RDominatesNorm} ensures that, for some constant $C_1>0$, $|\xi|\leq C_1+R(\xi)$ for all $\xi\in\mathbb{R}^d$. Upon noting that $R$ is strictly positive on any sphere of radius $\delta$, one easily obtains the following corollary.

\begin{corollary}\label{cor:RDominatesNorm}
For each $C,\delta>0$, there exists $M>0$ for which
\begin{equation*}
|\xi|+C\leq M R(\xi)
\end{equation*}
for all $|\xi|>\delta$.
\end{corollary}

\begin{proposition}\label{prop:ComplexEstimate}
Let $P$ be a positive homogeneous polynomial with $R=\Re P$. There exist $\epsilon>0$ and $M>0$ such that \begin{equation}\label{eq:ComplexEstimate_1}
-\Re P(z)\leq -\epsilon R(\xi)+M R(\nu)
\end{equation}
and
\begin{equation}\label{eq:ComplexEstimate_2}
|P(z)|\leq M (R(\xi)+R(\nu))
\end{equation}
for all $z=\xi-i\nu\in\mathbb{C}^d$.
\end{proposition}
\begin{proof}
Let $E\in\Exp(P)$ have strictly real spectrum and, by virtue of Proposition \ref{prop:PositiveHomogeneousPolynomialsareSemiElliptic}, let $A$ be such that $D=A^{-1} EA=\diag((2m_1)^{-1},(2m_2)^{-1},\dots,(2m_d)^{-1})$ and
\begin{equation*}
P_A(\xi):=(P\circ L_A)(\xi)=\sum_{|\alpha:\mathbf{m}|=2}a_\alpha\xi^{\alpha},
\end{equation*}
where $\mathbf{m}=(m_1,m_2,\dots,m_d)\in\mathbb{N}_+^d$. Because $A\in\GldR\subseteq\mbox{Gl}_d(\mathbb{C})$, it suffices to verify the estimates \eqref{eq:ComplexEstimate_1} and \eqref{eq:ComplexEstimate_2} for $P_A$ and $R_A=\Re P_A$. As in the proof of the previous proposition, we identify $\mathbb{C}^d=\mathbb{R}^{2d}$ by $z=(\xi,\nu)$, and observe that $\{t^{D\oplus D}\}\subseteq \mbox{Gl}_{2d}(\mathbb{R})$ is contracting. Consequently, by considering $T_t=t^{D\oplus D}$, the estimate \eqref{eq:ComplexEstimate_2} follows directly from Proposition \ref{prop:ComparableFunctions}.

An appeal to the multivariate binomial theorem shows that for all $z=\xi-i\nu\in\mathbb{C}^d$,
\begin{equation}\label{eq:ComplexEstimate_3}
P_A(\xi-i\nu)=P_A(\xi)+Q(\xi,\nu),
\end{equation}
where
\begin{equation*}
Q(\xi,\nu)=\sum_{|\alpha:\mathbf{m}|=2}a_{\alpha}\sum_{\gamma<\alpha}\binom{\alpha}{\gamma}\xi^{\gamma}(-i\nu)^{\alpha-\gamma}=\sum_{\substack{|\alpha:\mathbf{m}|=2\\ \gamma<\alpha}}b_{\alpha,\gamma}\xi^{\gamma}\nu^{\alpha-\gamma};
\end{equation*}
here, $\{b_{\alpha,\gamma}\}\subseteq\mathbb{C}$. We claim that for each $\delta>0$, there exists $M>0$ such that
\begin{equation}\label{eq:ComplexEstimate_4}
|Q(\xi,\nu)|\leq \delta R_A(\xi)+MR_A(\nu)
\end{equation}
for all $\xi,\nu\in\mathbb{R}^d$. For the moment, let us accept the validity of the claim. By choosing $\delta<1$, a combination of \eqref{eq:ComplexEstimate_3} and \eqref{eq:ComplexEstimate_4} yields
\begin{equation*}
-\Re(P_A(\xi-i\nu))+R_A(\xi)\leq \delta R_A(\xi)+MR_A(\nu)
\end{equation*}
for all $\xi,\nu\in\mathbb{R}^d$ and from this we see that \eqref{eq:ComplexEstimate_1} is satisfied with $\epsilon=1-\delta$. It therefore suffices to prove \eqref{eq:ComplexEstimate_4}.

For any multi-indices $\beta$ and $\gamma$ for which $|\beta:\mathbf{m}|=2$ and $\gamma<\beta$, it is straightforward to see that
\begin{equation*}
(t^D\xi)^{\gamma}(t^D\nu)^{\beta-\gamma}=t^{|\gamma:2\mathbf{m}|}t^{|\beta-\gamma:2\mathbf{m}|}\xi^{\gamma}\nu^{\beta-\gamma}=t^{|\beta:2\mathbf{m}|}\xi^{\gamma}\nu^{\beta-\gamma}=t\xi^{\gamma}\nu^{\beta-\gamma}
\end{equation*}
for all $\xi,\nu\in\mathbb{R}^d$ and so the map $(\xi,\nu)\mapsto\xi^{\gamma}\nu^{\beta-\gamma}$ is homogeneous with respect to the contracting group $\{t^{D\oplus D}\}\subseteq\mbox{Gl}_{2d}(\mathbb{R})$. Consequently, an application of Proposition \ref{prop:ComparableFunctions} gives $C>0$ for which
\begin{equation*}
|\xi^{\gamma}\nu^{\beta-\gamma}|\leq C( R_A(\xi)+R_A(\nu))
\end{equation*}
for all $\xi,\nu\in\mathbb{R}^d$. By invoking the homogeneity of $\xi^\gamma$ and $R_A(\xi)$ with respect to $\{t^D\}\subseteq\GldR$, the above estimate ensures that, for all $t>0$,
\begin{align*}
|\xi^{\gamma}\nu^{\beta-\gamma}|&=|t^{|\gamma:2\mathbf{m}|}(t^{-D}\xi)^{\gamma}\nu^{\beta-\gamma}|\\
&\leq t^{|\gamma:2\mathbf{m}|}C(R_A(t^{-D}\xi)+R_A(\nu))\\
&=Ct^{|\gamma:2\mathbf{m}|-1}R_A(\xi)+Ct^{|\gamma:2\mathbf{m}|}R_A(\nu)
\end{align*}
for all $\xi,\nu\in\mathbb{R}^d$. Noting that $|\gamma:2\mathbf{m}|-1<0$ because $\gamma<\beta$, we can make the coefficient of $R_A(\xi)$ in the above estimate arbitrarily small by choosing $t$ sufficiently large. Consequently, for any $\delta>0$ there exists $M>0$ for which
\begin{equation*}
|\xi^{\gamma}\nu^{\beta-\gamma}|\leq \delta R_A(\xi)+M R_A(\nu)
\end{equation*}
for all $\xi,\nu\in\mathbb{R}^d$. The claim \eqref{eq:ComplexEstimate_4} now follows by a simple application of the triangle inequality.
\end{proof}

\subsection{Properties of the Legendre-Fenchel transform of a positive homogeneous polynomial}
\begin{lemma}\label{lem:LegendreNormEstimate}
Let $P$ be a positive homogeneous polynomial and let $R=\Re P$. For $E\in\Exp(P)$ with real spectrum let $\lambda_{\mbox{\tiny{max}}}=\max(\Spec(E))$ and $\lambda_{\mbox{\tiny{min}}}=\min(\Spec(E))$ (note that $0<\lambda_{\mbox{\tiny{min}}},\lambda_{\mbox{\tiny{max}}}\leq 1/2$ by Proposition \ref{prop:PositiveHomogeneousPolynomialsareSemiElliptic}) and set
\begin{equation*}
\mathcal{N}_E(x)=\begin{cases}
          |x|^{1/(1-\lambda_{\mbox{\tiny{max}}})} & \mbox{if }|x|\geq 1\\
          |x|^{1/(1-\lambda_{\mbox{\tiny{min}}})} & \mbox{if }|x|<1
         \end{cases}
\end{equation*}
for $x\in\mathbb{R}^d$. There are positive constants $M,M'$ for which
\begin{equation}\label{eq:LegendreNormEstimate}
|x|-M\leq R^{\#}(x)\leq M'\mathcal{N}_E(x)
\end{equation}
for all $x\in\mathbb{R}^d$.
\end{lemma}
\begin{proof}
Set $M=\sup_{\xi\in S}R(\xi)$ and observe that, for any non-zero $x\in\mathbb{R}^d$,
\begin{equation*}
R^{\#}(x)=\sup_{\xi\in\mathbb{R}^d}\{x\cdot \xi-R(\xi)\}\geq x\cdot\frac{x}{|x|}-R\left(\frac{x}{|x|}\right)\geq |x|-M.
\end{equation*}
The lower estimate in \eqref{eq:LegendreNormEstimate} now follows from the observation that $R^{\#}(0)=0$ which is true because $R$ is positive definite. We now focus on the upper estimate. In view of Lemma \ref{lem:SpectralEstimateforContractingGroup} and Proposition \ref{prop:PositiveHomogeneousPolynomialsareSemiElliptic}, let $C\geq 1$ be such that $\|t^E\|\leq Ct^{\lambda_{\mbox{\tiny{max}}}}$ for $t\geq 1$ and $\|t^E\|\leq Ct^{\lambda_{\mbox{\tiny{min}}}}$ for $t\leq 1$. An appeal to Proposition \ref{prop:RDominatesNorm} gives $M'>0$ for which $C|\xi|\leq R(\xi)+M'$ for all $\xi\in\mathbb{R}^d$. In the case that $|x|\geq 1$, we set $t=|x|^{1/(1-\lambda_{\mbox{\tiny{max}}})}$ and observe that
\begin{align*}
x\cdot\xi&\leq |x||t^Et^{-E}\xi|\\
&\leq |x|\|t^E\||t^{-E}\xi|\\
& \leq|x|t^{\lambda_{\mbox{\tiny{max}}}}\left(R(t^{-E}\xi)+M'\right)\\
&= |x|t^{\lambda_{\mbox{\tiny{max}}}-1}R(\xi)+M'|x|t^{\lambda_{\mbox{\tiny{max}}}}\\
&= R(\xi)+M'|x|^{1/(1-\lambda_{\mbox{\tiny{max}}})}
\end{align*}
for all $\xi\in\mathbb{R}^d$ and therefore
\begin{equation*}
R^{\#}(x)=\sup_{\xi\in\mathbb{R}^d}\{x\cdot\xi-R(\xi)\}\leq M'|x|^{1/(1-\lambda_{\mbox{\tiny{max}}})}=M'\mathcal{N}_E(x).
\end{equation*}
When $|x|\leq 1$, we repeat the argument above to find that 
\begin{equation*}R^{\#}(x)\leq M'|x|^{1/(1-\lambda_{\mbox{\tiny{min}}})}=M'\mathcal{N}_E(x)
\end{equation*}
as desired.
\end{proof}

\begin{proposition}\label{prop:LegendreContinuousPositiveDefinite}
Let $P$ be a positive homogeneous polynomial with $R=\Re P$. Then $R^{\#}$ is continuous, positive definite, and for any $E\in\Exp(P)$, $F=(I-E)^*\in\Exp(R^{\#})$.
\end{proposition}
\begin{proof}
Since $R^{\#}$ is the Legendre-Fenchel transform of $R:\mathbb{R}^d\rightarrow\mathbb{R}$ it is convex (and lower semi-continuous). Furthermore, the upper estimate in Lemma \ref{lem:LegendreNormEstimate} guarantees that $R^{\#}$ is finite on $\mathbb{R}^d$ and therefore continuous in view of Corollary 10.1.1 of \cite{Rockafellar1970}.

Given that $R$ is positive definite and homogeneous with respect to $\{t^E\}$, it follows directly from the definition of the Legendre-Fenchel transform that $R^{\#}$ is non-negative, homogeneous with respect to $\{t^F\}$ where $F=(I-E)^*$ and has $R^{\#}(0)=0$. To complete the proof, it remains to show that $R^{\#}(x)\neq 0$ for all non-zero $x\in\mathbb{R}^d$. Using the lower estimate in Lemma \ref{lem:LegendreNormEstimate}, we have
\begin{equation}\label{eq:LegendreContinuousPositiveDefinite}
\lim_{x\rightarrow\infty}R^{\#}(x)=\infty.
\end{equation}
By virtue of Proposition \ref{prop:PositiveHomogeneousPolynomialsareSemiElliptic}, $F$ is diagonalizable with $\Spec(F)\subseteq [1/2,1)$; in particular, $\{t^F\}$ is contracting in view of Lemma \ref{lem:SpectralEstimateforContractingGroup}. Now if for some non-zero $x\in\mathbb{R}^d$, $R^{\#}(x)=0$,
\begin{equation*}
0=\lim_{t\rightarrow \infty}tR^{\#}(x)=\lim_{t\rightarrow\infty}R^{\#}(t^Fx),
\end{equation*}
which is impossible in view of Proposition \ref{prop:ContractingLimits} and \eqref{eq:LegendreContinuousPositiveDefinite}.
\end{proof}

\begin{corollary}\label{cor:LegendreCompareDiagonal}
Let $P$ be a positive homogeneous polynomial of the form \eqref{eq:SemiEllipticPolynomial} for $\mathbf{m}=(m_1,m_2,\dots,m_d)\in\mathbb{N}_+^d$ and $\{a_\beta\}\subseteq\mathbb{C}$. That is, the conclusion of Proposition \ref{prop:PositiveHomogeneousPolynomialsareSemiElliptic} holds where $A=I\in\GldR$. Let $R=\Re P$, let $R^{\#}$ be the Legendre-Fenchel transform of $R$ and define $|\cdot|_\mathbf{m}:\mathbb{R}^d\rightarrow[0,\infty)$ by \eqref{eq:mHomogeneousNormDef} for $x\in\mathbb{R}^d$. Then
\begin{equation*}
R^{\#}(x)\asymp |x|_{\mathbf{m}}.
\end{equation*}
\end{corollary}
\begin{proof}
Let us first note that $|\cdot|_{\mathbf{m}}$ is continuous, positive definite and homogeneous with respect to the one-parameter contracting group $\{t^F\}$ where $F=\diag((2m_1-1)/(2m_1),(2m_2-1)/(2m_2),\dots,(2m_d-1)/(2m_d))$. Because $E=\diag\left((2m_1)^{-1},(2m_2)^{-1},\dots,(2m_d)^{-1}\right)\in\Exp(R)$, Proposition \ref{prop:LegendreContinuousPositiveDefinite} ensures that $R^{\#}$ is continuous, positive definite and has $F=(I-E)^*\in\Exp(R^{\#})$. The desired result follows directly by an appeal to Proposition \ref{prop:ComparableFunctions}.
\end{proof}

\noindent Another application of Proposition \ref{prop:LegendreContinuousPositiveDefinite} and \ref{prop:ComparableFunctions} yields the following corollary.

\begin{corollary}\label{cor:MovingConstant}
Let $P$ be a positive homogeneous polynomial with $R=\Re P$. For any constant $M>0$,
$(MR)^{\#}\asymp R^{\#}$.
\end{corollary}

\subsection{The proof of Proposition \ref{prop:PosHomTypeChar}}\label{subsec:ProofofPosHomTypeChar}

\begin{proof}[Proof of Proposition \ref{prop:PosHomTypeChar}]
\noindent $(a\Rightarrow b)$ Let $P=P_{\xi_0}$, take $E\in\Exp(P_{\xi_0})$ with strictly real spectrum and set $m=\max_{i=1,2\dots,d} 2m_i$ in view of Proposition \ref{prop:PositiveHomogeneousPolynomialsareSemiElliptic}. Noting that $E$ is diagonalizable,   $m+1>(\min(\Spec(E)))^{-1}$ and $Q_{\xi_0}^{m}(\xi)+O(|\xi|^{m+1})=P_{\xi_0}(\xi)+\Upsilon_{\xi_0}(\xi)$ for $\xi$ sufficiently close to $0$, our result follows from Proposition \ref{prop:NormGammaLittleOhR}.\\

\noindent $(b\Rightarrow c)$ Let $E\in\Exp(P)$ have real spectrum and observe that, for all $n\in\mathbb{N}_+$,
\begin{equation}\label{eq:PosHomTypeChar_1}
C^{-1}R(\xi)\leq n \Re Q_{\xi_0}^m(n^{-E}\xi)\leq CR(\xi) \hspace{.2cm}\mbox{ and }\hspace{.2cm}|n\Im Q_{\xi_0}^m(n^{-E}\xi)|\leq C R(\xi)
\end{equation}
for $\xi\in \overline{B_r}$. It follows that the sequence $\{\rho_n\}\subseteq C(\overline{B_r})$ of degree $m$ polynomials, defined by $\rho_n(\xi)=nQ_{\xi_0}^m(n^{-E}\xi)$ for all $n\in\mathbb{N}_+$ and $\xi\in\overline{B_r}$, is bounded. As the subspace of degree $m$ polynomials is a finite dimensional subspace of $C(\overline{B_r})$, $\{\rho_n\}$ must have a convergent subsequence. Moreover, because $R(\xi)$ is positive definite, \eqref{eq:PosHomTypeChar_1} ensures that the subsequential limit has positive real part on $S_r$.\\

\noindent $(c\Rightarrow a)$ The proof of this implication is lengthy and will be shown using a sequence of lemmas. We fix $E\in\MdR$ with real spectrum and for which the condition \eqref{eq:PosHomTypeCharHypothesis} is satisfied. As the characteristic polynomial of $E$ completely factors over $\mathbb{R}$, the Jordan-Chevally decomposition for $E$ gives $A\in\GldR$ for which $F:=A^{-1}EA=D+N$ where $D$ is diagonal, $N$ is nilpotent and $DN=ND$. Upon setting $Q_A=Q_{\xi_0}^m\circ L_A$, it follows that
\begin{equation*}
Q_A(\xi)=\sum_{1<|\beta|\leq m}a_{\beta}\xi^{\beta}
\end{equation*}
for $\xi\in\mathbb{R}^d$ where $\{a_\beta\}\subseteq\mathbb{C}$. Define $\rho_A:(0,\infty)\times\mathbb{R}^d\rightarrow\mathbb{C}$ by $\rho_A(t,\xi)=tQ_A(t^{-F}\xi)$ for $t>0$ and $\xi\in\mathbb{R}^d$. The hypotheses \eqref{eq:PosHomTypeCharHypothesis} ensures that, for each $\xi\in A^{-1}\overline{B_r}$,
\begin{equation}\label{eq:PosHomTypeChar_2}
P_A(\xi):=\lim_{n\rightarrow\infty}\rho_A(t_n,\xi)
\end{equation}
exists and is such that $\Re P_A(\xi)>0$ whenever $\xi\in A^{-1}S_r$.

\begin{lemma}\label{lem:PosHomTypeChar_1}
Under the hypotheses \eqref{eq:PosHomTypeCharHypothesis}, $\lim_{t\rightarrow\infty}\rho_A(t,\xi)$ exists for all $\xi\in\mathbb{R}^d$ and the convergence is uniform on all compact sets of $\mathbb{R}^d$. In particular, $P_A$ extends uniquely to $\mathbb{R}^d$ (which we also denote by $P_A$) by
\begin{equation}\label{eq:PosHomTypeChar_3}
P_A(\xi)=\lim_{t\rightarrow\infty}\rho_A(t,\xi)=\lim_{n\rightarrow\infty}\rho_A(t_n,\xi)
\end{equation}
for all $\xi\in\mathbb{R}^d$. Moreover, $P_A:\mathbb{R}^d\rightarrow\mathbb{C}$ is a positive homogeneous polynomial with the representation
\begin{equation}\label{eq:PosHomTypeChar_4}
P_A(\xi)=\sum_{|\beta:\mathbf{m}|=2}a_\beta\xi^{\beta}
\end{equation}
for some $\mathbf{m}=(m_1,m_2,\dots,m_d)\in\mathbb{N}_+^d$ where $m\geq 2m_i$ for $i=1,2,\dots d$ and
\begin{equation}\label{eq:PosHomTypeChar_5}
F=D=\diag((2m_1)^{-1},(2m_2)^{-1},\dots,(2m_d)^{-1})\in \Exp(P_A).
\end{equation}
Furthermore
\begin{equation}\label{eq:PosHomTypeChar_6}
Q_A(\xi)=\sum_{|\beta:2\mathbf{m}|\geq 1}a_{\beta}\xi^{\beta}=P_A(\xi)+\sum_{|\beta:2\mathbf{m}|>1}a_{\beta}\xi^{\beta}
\end{equation}
for $\xi\in\mathbb{R}^d$.
\end{lemma}
\begin{subproof}[Proof of Lemma \ref{lem:PosHomTypeChar_1}.] Our proof is broken into three steps. In the first step we show that the representation \eqref{eq:PosHomTypeChar_4} is valid on $A^{-1}\overline{B_r}$ and the first equality in \eqref{eq:PosHomTypeChar_6} holds on $\mathbb{R}^d$. The first step also ensures the validity of the second equality in \eqref{eq:PosHomTypeChar_5}. In the second step, we define $P_A:\mathbb{R}^d\rightarrow\mathbb{C}$ by the right hand side of \eqref{eq:PosHomTypeChar_4} and check that $P_A$ is a positive homogeneous polynomial with $D\in \Exp(P_A)$. In the third step we show that $N=0$ and hence $F=D$ and in the fourth step we show that the limit \eqref{eq:PosHomTypeChar_3} converges uniformly on any compact set $K\subseteq \mathbb{R}^d$. The second inequality in \eqref{eq:PosHomTypeChar_6} follows directly by combining the results.\\

\noindent\textit{Step 1.} Write $D=\diag(\gamma_1,\gamma_2,\dots,\gamma_d)$ and put $\gamma=(\gamma_1,\gamma_2,\dots,\gamma_d)\in\mathbb{R}^d$.  We fix $\xi\in A^{-1}\overline{B_r}$ and observe that
\begin{equation}\label{eq:PosHomTypeChar_7}
\begin{split}
\rho_A(t,\xi)&=\sum_{1<|\beta|\leq m}a_{\beta}t\Big(t^{-D}\Big(I+\log t N\xi+\dots+\frac{(\log t)^k}{k!}N^k\xi\Big)\Big)^{\beta}\\
&=\sum_{1<|\beta|\leq m}a_{\beta}t^{1-\gamma\cdot \beta}\xi^{\beta}+\sum_{j=1}^lb_{j}t^{\omega_j}(\log t)^j
\end{split}
\end{equation}
for all $t>0$ where, by invoking the multinomial theorem, we have simplified the expression so that  $\omega_1,\omega_2,\dots,\omega_l\in \mathbb{R}$ are distinct and $b_j=b_j(\xi;N)\in\mathbb{C}$ for $j=1,2,\dots l=km$. Considering the sum
\begin{equation}\label{eq:PosHomTypeChar_8}
\sum_{j=1}^lb_{j}t^{\omega_j}(\log t)^j
\end{equation}
we see that, as $t\rightarrow \infty$, the summands must either converge to $0$ or diverge to $\infty$ in absolute value. Moreover, the distinctness of the collection $\{\omega_1,\omega_2,\dots,\omega_l\}$ and the presence of positive powers of $\log t$ ensure that this convergence or divergence happens at distinct rates. Consequently, as $t_n\rightarrow \infty$ the divergence of even a single summand would force the non-existence of the limit \eqref{eq:PosHomTypeChar_2}. Consequently, the expression \eqref{eq:PosHomTypeChar_8} converges to $0$ as $t\rightarrow\infty$ and so
\begin{equation}\label{eq:PosHomTypeChar_9}
P_A(\xi)=\lim_{n\rightarrow\infty}\rho_A(t_n,\xi)=\lim_{t\rightarrow\infty}\rho_A(t,\xi)=\lim_{t\rightarrow\infty}\sum_{1<|\beta|\leq m}a_{\beta}t^{1-\gamma\cdot\beta}\xi^{\beta}.
\end{equation}
Since $\xi$ was arbitrary, \eqref{eq:PosHomTypeChar_9} must hold for all $\xi\in A^{-1}\overline{B_r}$. This is the only part of the argument in which the subsequence $\{t_n\}$ appears.

We claim that, for all multi-indices $\beta$ for which $a_\beta\neq 0$, $\beta\cdot \gamma=\beta_1\gamma_1+\beta_2\gamma_2+\cdots+\beta_d\gamma_d\geq 1$. Indeed, fix $\kappa=\min(\{\beta\cdot\gamma:a_\beta\neq 0\})$, set $\mathcal{I_{\kappa}}=\{\beta:a_{\beta}\neq 0\mbox{ and }\beta\cdot\gamma=\kappa\}$ and define $B_{\kappa}:\mathbb{R}^d\rightarrow\mathbb{C}$ by
\begin{equation*}
B_{\kappa}(\xi)=\sum_{\beta\in \mathcal{I}_{\kappa}}a_{\beta}\xi^{\beta}
\end{equation*}
for $\xi\in\mathbb{R}^d$. The linear independence of the monomials $\{\xi^{\beta}\}_{\beta\in\mathcal{I}_\kappa}$ ensures that $B_{\kappa}(\xi')\neq 0$ for some $\xi'\in A^{-1}\overline{B_r}$. It follows from \eqref{eq:PosHomTypeChar_9} that $\lim_{t\rightarrow\infty} \rho_A(t,\xi')=\lim_{t\rightarrow\infty}t^{1-\kappa}B_{\kappa}(\xi')$ from which we conclude that $\kappa=1$; the hypotheses that $P_A$ has positive real part on $A^{-1}S_r$ rules out the possibility that $\kappa>1$.

From the claim it is now evident that
\begin{equation}\label{eq:PosHomTypeChar_10}
P_A(\xi)=\sum_{\beta\cdot\gamma=1}a_{\beta}\xi^{\beta}
\end{equation}
for $\xi\in A^{-1}\overline{B_r}$ and
\begin{equation}\label{eq:PosHomTypeChar_11}
Q_{A}(\xi)=\sum_{\beta\cdot\gamma\geq 1}a_{\beta}\xi^{\beta}
\end{equation}
for $\xi\in\mathbb{R}^d$.

It is straightforward to see that the set $A^{-1}S_r$ intersects each coordinate axis at exactly two antipodal points. That is, for each $j=1,2,\dots d$, there exists $x_j>0$ for which $\{\pm x_j e_j\}=A^{-1}S_r\cap\{xe_j: x\in\mathbb{R}\}$. Upon evaluating $\Re (P_A)$ at such points and recalling that $\Re P_A>0$ on $A^{-1}S_r$, one sees by the same argument given in \textit{Step 2} of the proof of Proposition \ref{prop:PositiveHomogeneousPolynomialsareSemiElliptic} that $1/\gamma_j$ is an even natural number which cannot be greater than $m$. Therefore, for each $j=1,2,\dots, d$, $1/\gamma_j=2m_j\geq m$ for some $m_j\in\mathbb{N}_+$. The representation \eqref{eq:PosHomTypeChar_4} on $A^{-1}\overline{B_r}$ and the first equality in \eqref{eq:PosHomTypeChar_6} now follow from \eqref{eq:PosHomTypeChar_10} \eqref{eq:PosHomTypeChar_11} and the observation that $\beta\cdot\gamma=\sum_{j=1}^d\beta_j/2m_j=|\beta:2\mathbf{m}|$. Moreover,
\begin{equation}\label{eq:PosHomTypeChar_12}
D=\diag((2m_1)^{-1},(2m_2)^{-1},\dots,(2m_d)^{-1}).
\end{equation}

\noindent\textit{Step 2.} We define $P_A:\mathbb{R}^d\rightarrow\mathbb{C}$ by the right hand side of \eqref{eq:PosHomTypeChar_4}. It is clear that $D\in\Exp(P_A)$ and so, to prove that $P_A$ is positive homogeneous, it suffices to show that that $R_A(\xi)=\Re P_A(\xi)>0$ whenever $\xi\neq 0$. To this end, let $\xi\in\mathbb{R}^d$ be non-zero and find $t>0$ for which $t^D\xi\in A^{-1}S_r$; this can be done because $\{t^D\}$ is contracting in view of \eqref{eq:PosHomTypeChar_12}. From the previous step we know that \eqref{eq:PosHomTypeChar_4} holds on $A^{-1}S_r$ and thus by invoking \eqref{eq:PosHomTypeChar_2}, we find that $R_A(\xi)=t^{-1}\Re P_A(t^{D}\xi)>0$ as claimed.\\

\noindent\textit{Step 3.} We now show that $F\in \Exp(P_A)$ and use it to conclude that $N=0$. As we will see, this assertion relies on $P_A$ being originally defined via a ``scaling'' limit. Indeed, for any $\xi\in\mathbb{R}^d$ and $t>0$, find $u>0$ for which both $u^{-D}\xi$ and $u^{-D}t^{F}\xi$ belong to $A^{-1}\overline{B_r}$; this can be done because $A^{-1}\overline {B_r}$ necessarily contains an open neighborhood of $0$. In view of \eqref{eq:PosHomTypeChar_9},
\begin{align*}
tP_A(\xi)&=tuP_A(u^{-D}\xi)=ut\lim_{s\rightarrow\infty}s\rho_A(s,u^{-D}\xi)=ut\lim_{s\rightarrow\infty}sQ_A(s^{-F}u^{-D}\xi)\\
&=u\lim_{s\rightarrow\infty}stQ_A(s^{-F}t^{-F}t^{F}u^{-D}\xi)=u\lim_{(st)\rightarrow\infty}(st)Q_A((st)^{-F}u^{-D}t^{F}\xi)\\
&=u\lim_{v\rightarrow\infty}\rho_A(v,u^{-D}t^{F}\xi)=(uP_A(u^{-D}t^F\xi)=P_A(t^F\xi)
\end{align*}
where we have used Proposition \ref{prop:tEProperties} and the fact that $D\in\Exp(P_A)$. Consequently $F\in\Exp(P_A)$ and since $P_A$ is a positive homogeneous polynomial, the same argument given in \textit{Step 3} of the proof of Proposition \ref{prop:PositiveHomogeneousPolynomialsareSemiElliptic} ensures that $N=0$.\\

\noindent\textit{Step 4.} Let $K\subseteq\mathbb{R}^d$ be compact and note that $t^{-F}K\subseteq A^{-1}\overline{B_r}$ for sufficiently large $t$ by virtue of Proposition \ref{prop:ContractingCapturesCompact}. Thus by invoking \eqref{eq:PosHomTypeChar_4}, which we know to be valid on $A^{-1}\overline{B_r}$, we have
\begin{align*}
|\rho_A(t,\xi)-P_A(\xi)|&=|tQ_A(t^{-F}\xi)-tP_A(t^{-F}\xi)|\\
&=\Big|t\sum_{|\beta:2\mathbf{m}|>1}a_{\beta}(t^{-F}\xi)^{\beta}\Big|\\
&\leq\sum_{|\beta:2\mathbf{m}|>1}t^{1-|\beta:2\mathbf{m}|} |a_{\beta}\xi^{\beta}\\
&\leq t^{\omega}\sum_{|\beta:2\mathbf{m}|>1}|a_{\beta}\xi^{\beta}|
\end{align*}
for all $\xi\in K$ and sufficiently large $t$ where $\omega<0$ is independent of $K$. The assertion concerning the uniform limit follows at once because $\sum_{|\beta:2\mathbf{m}|>1}|a_{\beta}\xi^{\beta}|$ is necessarily bounded on $K$.
\end{subproof}

\noindent We shall henceforth abandon using the symbol $D$ and write
\begin{equation*}
F=A^{-1}EA=\diag((2m_1)^{-1},(2m_2)^{-1},\dots,(2m_d)^{-1})\in\Exp(P_A).
\end{equation*}

\begin{lemma}\label{lem:PosHomTypeChar_2}
Under the hypotheses of Lemma \ref{lem:PosHomTypeChar_1}, $Q_A(\xi)-P_A(\xi)=o(R_A(\xi))$ as $\xi\rightarrow 0$.
\end{lemma}
\begin{subproof}[Proof of Lemma \ref{lem:PosHomTypeChar_2}.] In view of Lemma \ref{lem:PosHomTypeChar_1},
\begin{equation*}
|Q_A(\xi)-P_A(\xi)|\leq\sum_{|\beta:2\mathbf{m}|>1}|a_{\beta}\xi^{\beta}|
\end{equation*}
for all $\xi\in\mathbb{R}^d$. The desired result now follows directly from Lemma \ref{lem:LittleOh}.
\end{subproof}

\noindent We now define $P_{\xi_0}:\mathbb{R}^d\rightarrow\mathbb{C}$ by $P_{\xi_0}=P_A\circ L_{A^{-1}}$. By virtue of our results above, it is clear that $P_{\xi_0}$ is positive homogeneous with $E\in\Exp(P_{\xi_0})$. We have
\begin{equation*}
\Upsilon_{\xi_0}(\xi)=\Gamma_{\xi_0}(\xi)-i\alpha_{\xi_0}\cdot\xi+P_{\xi_0}(\xi)= P_{\xi_0}(\xi)-Q^m_{\xi_0}(\xi)+O(|\xi|^{(m+1)})
\end{equation*}
as $\xi\rightarrow 0$. Because $R_{\xi_0}=\Re P_{\xi_0}= R_A\circ L_{A^{-1}}$, it follows from Lemma \ref{lem:PosHomTypeChar_2} that $P_{\xi_0}(\xi)-Q_{\xi_0}(\xi)=o(R_{\xi_0}(\xi))$ as $\xi\rightarrow 0$. Moreover, because $E$ is diagonalizable and $m+1>2m_i\geq (\min(\Spec(E)))^{-1}$, $|\xi|^{(m+1)}=o(R_{\xi_0}(\xi))$ as $\xi\rightarrow 0$ by virtue of Proposition \ref{prop:NormGammaLittleOhR}. Therefore
\begin{equation*}
\Gamma_{\xi_0}(\xi)=i\alpha_{\xi_0}-P_{\xi_0}(\xi)+\Upsilon_{\xi_0}(\xi)
\end{equation*}
where $\Upsilon_{\xi_0}=o(R_{\xi_0})$ as $\xi\rightarrow 0$ and thus completing the proof of the implication $(c\Rightarrow a)$.

 To finish the proof of Proposition \ref{prop:PosHomTypeChar}, it remains to prove that, for any $m'\geq m$,
\begin{equation*}
P_{\xi_0}(\xi)=\lim_{t\rightarrow\infty}tQ_{\xi_0}^{m'}(t^{-E}\xi)
\end{equation*}
for all $\xi\in\mathbb{R}^d$ and this limit is uniform on all compact subsets of $\mathbb{R}^d$. Indeed, Let $K\subseteq\mathbb{R}^d$ be compact. By virtue of Lemma \ref{lem:PosHomTypeChar_1},
\begin{equation}\label{eq:PosHomTypeChar_13}
\begin{split}
P_{\xi_0}(\xi)&=P_A(A^{-1}\xi)\\
&=\lim_{t\rightarrow\infty} \rho_A(t,A^{-1}\xi)\\
&=\lim_{t\rightarrow\infty} tQ_A(A^{-1}t^{-E}\xi)\\
&=\lim_{t\rightarrow\infty} tQ_{\xi_0}^m(t^{-E}\xi)
\end{split}
\end{equation}
uniformly for $\xi\in K$. Observe that for any $m'>m$, there exists $M>0$ for which
\begin{align*}
\big|tQ_{\xi_0}^{m'}(t^{-E}\xi)-tQ_{\xi_0}^m(t^{-E}\xi)\big|&\leq\sum_{m<|\beta|\leq m'}t\big|c_{\beta}(t^{-E}\xi)^{\beta}\big|\\
&=\sum_{m<|\beta|\leq m'}t\big|c_{\beta}(At^{-F}A^{-1}\xi)^{\beta}\big|\\
&\leq M\sum_{m<|\gamma|\leq m'}t\big|(t^{-F}A^{-1}\xi)^{\gamma}\big|\\
&=\sum_{m<|\gamma|\leq m'}t^{1-|\gamma:2\mathbf{m}|}\big|(A^{-1}\xi)^{\gamma}\big|
\end{align*}
for all $\xi\in\mathbb{R}^d$ and $t>0$. Noting that $|\gamma:2\mathbf{m}|>1$ whenever $m<|\gamma|\leq m'$, by repeating the argument given in \textit{Step 4} of Lemma \ref{lem:PosHomTypeChar_1}, we observe that
\begin{equation}\label{eq:PosHomTypeChar_14}
\lim_{t\rightarrow \infty}\big(tQ_{\xi_0}^{m'}(t^{-E}\xi)-tQ_{\xi_0}^m(t^{-E}\xi)\big)=0
\end{equation}
uniformly for $\xi\in K$. The desired result follows by combining \eqref{eq:PosHomTypeChar_13} and \eqref{eq:PosHomTypeChar_14}.
\end{proof}

\noindent{\small\bf Acknowledgment:} We thank David Eriksson for his help and expertise while producing the figures for this paper. We also thank Bernie Randles for her careful proofreading.

\noindent Evan Randles\footnote{This material is based upon work supported by the National Science Foundation Graduate Research Fellowship under Grant No. DGE-1144153} : Center for Applied Mathematics, Cornell University, Ithaca, NY 14853.
\newline E-mail: edr62@cornell.edu\\

\noindent Laurent Saloff-Coste\footnote{This material is based upon work supported by the National Science Foundation under Grant No. DMS-1404435}: Department of Mathematics, Cornell University, Ithaca NY 14853.
\newline E-mail: lsc@math.cornell.edu

\end{document}